\documentclass[11pt]{amsart}
\usepackage{amsmath,amssymb}
\usepackage{graphicx}
\usepackage{pictexwd}
\usepackage[top=1in,bottom=1in,left=1in,right=1in]{geometry}
\newtheorem{thm}{Theorem}[section]

\newtheorem{lem}[thm]{Lemma}
\newtheorem{prop}[thm]{Proposition}
\newtheorem{defn}{Definition}[section]
\theoremstyle{definition}
\newtheorem{rem}{Remark}[section]
\numberwithin{equation}{section}

\newcommand{\eps}{\varepsilon}

\newcommand{\tsfr}{\ts\frac12}

\newcommand{\hu}{\hat u}

\newcommand{\ds}{\displaystyle}
\newcommand{\ts}{\textstyle}

\newcommand{\be}{\begin{equation} \label}
\newcommand{\ee}{\end{equation}}
\newcommand{\bas}{\begin{eqnarray*}}
\newcommand{\eas}{\end{eqnarray*}}
\newcommand{\R}{\mathbb{R}}
\newcommand{\N}{\mathbb{N}}

\usepackage{color}

\def\ph{\varphi}

\newcommand{\hA}{\hat A}
\newcommand{\hB}{\hat B}

\newcommand{\tT}{\widetilde T}
\newcommand{\tv}{\widetilde v}

\newcommand{\tA}{\widetilde A}
\newcommand{\tB}{\widetilde B}

\newcommand{\tu}{\widetilde u}

\begin{document}

\title[classification in GBU and RBC]{Complete classification of gradient blow-up \\
and recovery of boundary condition for \\ the viscous Hamilton-Jacobi equation}

\author[Mizoguchi]{Noriko Mizoguchi}%
\address{Department of Mathematics, Tokyo Gakugei University,
Koganei, Tokyo 184-8501, Japan }
\email{mizoguti@u-gakugei.ac.jp}

\author[Souplet]{Philippe Souplet}%
\address{Universit\'e Sorbonne Paris Nord,
CNRS UMR 7539, Laboratoire Analyse, G\'{e}om\'{e}trie et Applications,
93430 Villetaneuse, France}
\email{souplet@math.univ-paris13.fr}

 \begin{abstract}  
It is known that the Cauchy-Dirichlet problem for the super\-qua\-dratic viscous Hamilton-Jacobi equation
$u_t-\Delta u=|\nabla u|^p$,
which has important applications in stochastic control theory,
admits a unique, global viscosity solution.
Solutions thus exist in the weak sense after the appearance of singularity in finite time, 
which occurs through gradient blow-up (GBU) on the boundary.
The solutions eventually become classical again for large time, but in-between they
may undergo losses and recoveries of boundary conditions at multiple times
(as well as GBU at multiple times).

\vskip 1.5pt

In this paper we give a complete classification, namely rates and space-time profiles in one dimensional case when viscosity solutions undergo gradient blow-up (GBU) or recovery of boundary condition (RBC) at any time when such a phenomenon occurs.
These results can be modified in radial domains in general dimensions.	
Previously, upper and lower estimates of GBU or RBC rates were available only in
a special case when the basic comparison principle can be used. 
Even for type II blow-up in other PDEs, as far as we know, 
there has been no complete classification except \cite{mizo_cpam}, in which the argument 
relies on features peculiar to chemotaxis system. 
 Whereas there are many results on construction of special type II blow-up solutions of PDEs with investigation of stability/instability of bubble, determination of stability/instability of space-time profile for general solutions has not been done. In this paper, we determine whether the space-time profile for each general solution is stable or unstable.
\vskip 1.5pt

A key in our proofs is to focus on algebraic structure with respect to vanishing intersections with the singular steady state, 
as time approaches a GBU or RBC time of a viscosity solution. In turn, the GBU and RBC rates and profiles, as well as their stabillity/instability, 
can be completely characterized by the number of vanishing intersections.
We construct special solutions in bounded and unbounded intervals in both GBU and RBC cases, 
based on methods from \cite{HV94pre}, and then we apply braid group theory to get upper and lower estimates of the rates. After that, we rule out oscillation of the rates, which leads us to the complete space-time profile. 
In the process, careful construction of special solutions with specific behaviors in intermediate and outer regions, 
which is far from bubble and the RBC point, plays an essential role.
The application of such techniques to viscosity solutions is completely new.

 \end{abstract}

\maketitle

\setcounter{tocdepth}{1}
\tableofcontents

\section{Introduction and background}\label{section1} 

\subsection{The problem}
Let $p>2$ and consider the viscous Hamilton-Jacobi equation 
\begin{equation}\label{maineq}
\left\{
\begin{aligned}
u_t-\Delta u&=|\nabla u|^p && \quad  x\in\Omega, t>0,\\ 
u(x,t)&=0, &&\quad x\in \partial\Omega,t>0,\\
u(x,0)&=u_0(x), &&\quad x\in \Omega,
\end{aligned}
\right.
\end{equation}
where $\Omega$ is a smooth proper subdomain of $\R^n$.

Problem \eqref{maineq} has a rich background. 
First of all, let us recall that \eqref{maineq} arises in 
stochastic control. Namely, denoting by $(W_s)_{s>0}$ a standard Brownian motion, 
it is known from \cite{BB} that the (unique global viscosity) solution of \eqref{maineq} gives the value function of the optimal control problem associated with the stochastic differential system $dX_s =\alpha_s ds+dW_s$, with control $\alpha_s$,
distribution of rewards $u_0$ and cost function $|\alpha_s|^{p/(p-1)}$ (see~e.g.~\cite{AttSou2} for more details).
As another motivation, \eqref{maineq} corresponds to the so-called deterministic KPZ equation,
arising in a well-known model of surface growth by ballistic deposition (see \cite{kpzhang}, \cite{KS}).

Let 
\[
\mathcal{W}=\bigl\{u_0\in W^{1,\infty}(\Omega);\ u_0\ge 0,\ u_0=0 \mbox{ on } \partial\Omega\bigr\}. 
\]
For $p>1$ and $u_0\in \mathcal{W}$,
it is well known that \eqref{maineq} admits a unique, maximal classical solution $u\ge 0$ 
and that $u$ satisfies
\be{bounduMP}
\sup_{t\in(0,T^*)} \|u(\cdot,t)\|_\infty\le \|u_0\|_\infty,
\ee
by the maximum principle, where $T^*=T^*(u_0)\in (0,\infty]$ denotes its maximal existence time.
Moreover, if $p>2$ and the initial data is suitably large, then $T^*<\infty$ 
and the solution undergoes gradient blow-up 
(GBU), i.e.,
$$\lim_{t\to T_-} \|\nabla u(\cdot, t)\|_\infty=\infty$$
(see \cite{Alaa, ABG, souplet2002, HM04}; on the contrary all solutions are global and classical for $p\in (1,2]$).
However the solution survives after the blow-up time and can be continued as a {\it generalized viscosity solution}. 
More precisely, by \cite{BdaLio}, problem \eqref{maineq} admits a unique, global nonnegative solution
$$u \in C^{1,2}(\Omega\times(0,\infty))\cap C(\overline\Omega\times[0,\infty))$$
which solves the PDE in the pointwise sense in $\Omega\times(0,\infty)$
but only satisfies the boundary condition in the viscosity sense, i.e.,
\be{defviscBC}
\min\bigl(u,u_t-\Delta u-|\nabla u|^p\bigr)\le 0 \quad\hbox{ on $\partial\Omega\times(0,\infty)$,}
\ee
where, for each $(x_0,t)\in\partial\Omega\times(0,\infty)$, $[u_t-\Delta u-|\nabla u|^p](x_0,t)\le 0$ is understood in the viscosity sense, i.e.,
for any smooth function~$\psi$, if $\psi$ touches $u$ from above at $(x_0,t)$, then 
$[\psi_t-\Delta\psi-|\nabla\psi|^p](x_0,t)\le 0$.
Moreover (see \cite[Section~3]{PSloss}), $u$ still satisfies
\be{boundunifvisc}
\sup_{t\ge 0}\ \|u(\cdot,t)\|_\infty \le \|u_0\|_\infty.
\ee
This solution coincides with the (unique) classical solution in~$(0,T)$, so
throughout this paper, we shall also denote it by $u$, without risk of confusion.

\subsection{Known results on GBU behavior} 
The main issues for the description of the blow-up behavior as $t\to T_-$ are the blow-up set, time rates and space-time 
profiles.
The location of the blowup set has been studied in \cite{Esteve, LiSouplet, SZ}.
As a consequence of interior gradient estimates \cite{SZ}, it is known that GBU for problem \eqref{maineq}
 can only take place on the boundary.
Concerning the question of the gradient blowup rates as $t\to T_-$, 
it is known that the lower estimate
\begin{equation}\label{lowerGBU}
\|\nabla  u(\cdot,t)\|_\infty\geq C (T- t)^{-1/(p-2)},\quad 0<t<T,
\end{equation}
is true for any GBU solution in any dimension (cf.~\cite{PS3} and references therein).
This in particular implies that GBU is always of type~II, i.e.,~it does not follow the natural self-similar scaling of the equation
(which would lead to the smaller exponent $1/2(p-1)$ instead of $1/(p-2)$).
The corresponding upper bound for the GBU rate, first conjectured in \cite{CG} on the basis of numerical simulations,
 is known to hold for certain classes of solutions.
 See \cite{GH}, \cite{QSbook07},  \cite{ZhangLi}, \cite{PS3} for one dimensional results (for all $p>2$) and 
 the recent paper \cite{AttSou2} for the (nonradial) higher dimensional case, which is understood only for $p\in (2,3]$.
 Roughly speaking these results guarantee that the two-sided estimate
 \begin{equation}\label{GBUratecC}
C_1(T-t)^{-\frac{1}{p-2}}\leq \|\nabla u(\cdot,t)\|_\infty \leq C_2(T-t)^{-\frac{1}{p-2}},\quad 0<t<T,
\end{equation}
is valid for solutions that are increasing in time in a neighborhood of the boundary
(and some of the results cover variants of the problem involving inhomogeneous 
 terms on the right hand side or in the boundary condition, which allows 
 the existence of solutions that are time increasing everywhere).
The question whether \eqref{GBUratecC} should hold for any GBU solution of \eqref{maineq}
was answered negatively in \cite{PS3}.
Namely, it was shown that for $n=1$ and $\Omega=(0,1)$, there exists a class of solutions such that
\begin{equation}\label{upperGBUinfty}
\lim_{t\to T_-} (T- t)^{1/(p-2)}\|u_x(\cdot,t)\|_\infty=\infty.
\end{equation}
For those solutions, the more precise lower bound 
\be{uxmoresingular}
\|u_x(\cdot,t)\|_\infty\ge C(T-t)^{-2/(p-2)},\quad T/2<t<T,
\ee
 was then obtained in \cite{AttSou2},
but the existence of solutions satisfying the corresponding two-sided bound has remained as an open question.

\subsection{Known results on post GBU behavior} 

The global viscosity solution $u$ of \eqref{maineq}, whose definition was recalled in Section 1.1,
may {\it lose the boundary condition in the classical sense.} 
Indeed such a possibility was first suggested in \cite{BdaLio} and confirmed in \cite{PSloss, QR16} 
where it was shown that, for suitably large initial data, 
the solution undergoes a loss of boundary conditions (LBC) at some times $t>T^*(u_0)$, i.e.,\footnote{We see that LBC solutions, which are meant to satisfy zero boundary conditions in the generalized viscosity sense,
nevertheless have to continuously take on some positive boundary values.
This apparently paradoxical situation can however be interpreted in a more intuitive way,
when one recalls that the global viscosity solution can also be obtained as the limit 
of a sequence of global classical solutions of regularized versions of problem \eqref{maineq}, with truncated nonlinearity
(see e.g.~\cite{PSloss}). Since this convergence is monotone increasing but not uniform up to the boundary,
LBC can in this framework be seen as a more familiar
 boundary layer phenomenon.}
 $$\sup_{x\in\partial\Omega} u(x,t)>0.$$
 However, some exceptional GBU solutions without LBC were also shown to exist in \cite{PSloss, PS3},
found as separatrices between global solutions and GBU solutions with LBC (see also \cite{FPS19}).
On the other hand, it was shown in \cite{PZ} that any solution
becomes classical again for all sufficiently large time, i.e.~there exists $\tilde T=\tilde T(u_0)\ge T^*$ such that 
$$\hbox{$u\in C^{2,1}(\overline\Omega\times (\tilde T,\infty))$,\quad
with $u=0$ on $\partial\Omega\times [\tilde T,\infty)$ in the classical sense,}$$
 and furthermore $u$ decays exponentially in $C^1(\overline\Omega)$ as $t\to\infty$. 

In view of these results, a natural and important question is thus to describe the behavior of $u$ in the intermediate time range $[T^*,\tilde T]$.
In this respect, the authors in \cite{MizSou} showed that, in any space dimension, GBU, LBC and recovery of boundary condition (RBC), and regularization occur at multiple times in the time interval~$(T^*,\tilde T)$. 
Moreover, in one dimension and in radial domains in higher dimensions,
they obtained a complete classification at each time. Namely, given a boundary point, say $x=0$, 
there are only finitely many times $t>0$ such that 
$u(0,t)=0$ without $u$ being $C^1$ up to $x=0$. We call such times {\it transition times}
and denote their set by $\mathcal{T}$. Between any two consecutive elements of $\mathcal{T}$,
the solution is either classical with $u=0$ at $x=0$, or remains positive at $x=0$ (LBC).
We thus see that a time $t\in \mathcal{T}$ can be of four types:

\smallskip
$\bullet$ time of GBU with LBC: $u(0,s)=0$ for $s\to t_-$ and $u(0,s)>0$ for $s\to t_+$;

\vskip 1pt
$\bullet$ time of GBU without LBC: $u(0,s)=0$ for $s\to t_\pm$;

\vskip 1pt
$\bullet$ bouncing time (i.e.~time of RBC and GBU with LBC): $u(0,s)>0$ for $s\to t_\pm$;

\vskip 1pt
$\bullet$ time of RBC: $u(0,s)>0$ for $s\to t_-$ and $u(0,s)=0$ for $s\to t_+$.

\smallskip

Furthermore for each $ m \ge 2 $ and arbitrarily given combination of GBU types with/without LBC at $ m $ times, 
we have constructed in \cite{MizSou} a viscosity solution undergoing this exact combination of GBU.
The boundary behavior of a typical solution with GBU and LBC at multiple times is depicted in~Fig.~\ref{FigMixed}.

\begin{figure}[h]
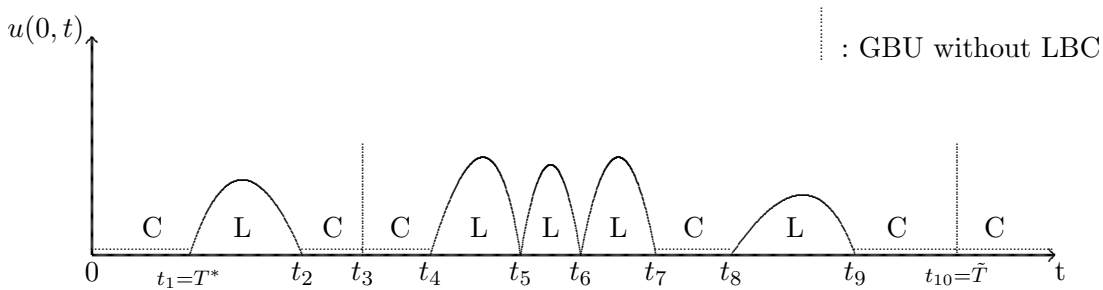

\[
\beginpicture
\setcoordinatesystem units <1cm,1cm>
\setplotarea x from -6 to 6, y from -1 to 4

\setdots <0pt>
\linethickness=1pt
\putrule from -5 0 to 7.8 0
\putrule from -5 0 to -5 2.9

\setdots <1.4pt> 
\setlinear
\plot -5 0.075 -3.7 0.075 /
\plot -2.2 0.075 -1.4 0.075 /
\plot -1.4 0.075 -0.5 0.075 /
\plot 2.5 0.075 3.5 0.075 /
\plot 5.15 0.075 6.5 0.075 /
\plot 6.5 0.075 7.8 0.075 /

\setdots <1pt> 

\plot -1.4 0  -1.4 1.5 /
\plot 6.5 0  6.5 1.5 /

\setdots <0pt> 

\setquadratic

\plot -3.7 0 -3 1 -2.2 0  /

\plot -0.5 0 0.2 1.3 0.7 0  /
\plot 0.7 0 1.1 1.2 1.5 0  /
\plot 1.5 0 2 1.3 2.5 0  /

\plot 3.5 0 4.45 0.8 5.15 0  /

\put {t} [lt] at 7.8 -.1
\put {$u(0,t)$} [rc] at -5.1 3

\setlinear 
\setdots <0pt> 
\plot 7.725 -.075 7.8 0 /
\plot 7.725 .075 7.8 0 /
\plot -5.075 2.825 -5 2.9 /
\plot -5 2.9 -4.925 2.825 /

\put {$0$}  [ct] at -5 -.1
\put {C}  [ct] at -4.2 0.5
\put {$^{t_1=T^*}$}  [ct] at -3.7 -.2
\put {L}  [ct] at -3 0.5
\put {$t_2$}  [ct] at -2.2 -.1
\put {C}  [ct] at -1.8 0.5
\put {$t_3$}  [ct] at -1.4 -.1
\put {C}  [ct] at -0.9 0.5
\put {$t_4$}   [ct] at -0.5 -.1
\put {L}  [ct] at 0.15 0.5
\put {$t_5$}   [ct] at 0.7 -.1
\put {L}  [ct] at 1.1 0.5
\put {$t_6$}   [ct] at 1.5 -.1
\put {L}  [ct] at 2 0.5
\put {$t_7$}   [ct] at 2.5  -.1
\put {C}  [ct] at 3 0.5
\put {$t_8$}   [ct] at 3.5 -.1
\put {L}  [ct] at 4.35 0.5
\put {$t_9$}   [ct] at 5.15 -.1
\put {C}  [ct] at 5.7 0.5
\put {$^{t_{10}=\tilde T}$}   [ct] at 6.5 -.1
\put {C}  [ct] at 7 0.5

\setlinear
\setdots <1.4pt> 
\plot 4.7 2.6 4.7 3.3 /
\put {:~GBU without LBC }  [ct] at 6.75 2.9

\endpicture
\] 
\caption{A solution with mixed behaviors (2 LBC, 2 GBU without LBC and 1 double bouncing);
here $\mathcal{T}=\{t_1,\dots,t_{10}\}$.}
\label{FigMixed}
\end{figure}

In the RBC case, the behavior of $u(\cdot,s)>0$ as $s\to \tau_-$, where $ \tau $ is the RBC time, is unknown in general,
except for a special case obtained earlier in \cite{PS3}.
More precisely, in the simplest case, the RBC rate 
was stated to be linear, namely $u(0,t)\sim\tau-t$.
 It turns out in our theorem below that RBC rate is not linear in general. It happened that the general RBC rate coincides with the linear rate in the special case.
In addition, the phenomenon of RBC for weak solutions does not seem to have been found in other PDEs, whereas blow-up has been studied in many papers.
As seen in our theorems below, the behavior of viscosity solutions in the RBC case is quite different from that in the GBU case. 

\subsection{Known results on type II blow-up solutions} 
 
As a general result, it was proved in \cite{Matano-Merle:CPAM}, \cite{mm_jfa11}, \cite{mizo_ma07}  that radial type II blow-up solutions to the semilinear heat equation $u_t-\Delta u=|u|^{p-1}u$, the so called Fujita equation, converge to the singular steady state locally uniformly in the spatial domain except the origin.   
Unfortunately, this does not give 
detailed information on bubbling phenomenon, which is one of the most important features in type II blow-up. 

On the other hand, there have been many papers on construction of special type II blow-up solutions, most of which 
dealt with radial solutions, with exact behavior of bubble, i.e., exact blow-up rate and space-time profile near blow-up point, in various partial differential equations.
It was originated by Herrero and Vel\'azquez in \cite{HV94}, \cite{HV94pre} for the Fujita equation.
Their argument was based on the linearization around a radial singular steady state and the comparison principle.
The method was applied to the dead-core problem, which is essentially the same as the Fujita equation,
 and to the harmonic heat flow, and just formally to the chemotaxis system
(see e.g.\cite{HVchemo}, \cite{BierSeki}). 

Merle and Rapha\"el later invented a method of construction of special type II blow-up solutions based on the linearization around a quasi-stationary  
state which is the first approximation of the bubble (the feature of solution in inner region in terms of Herrero-Vel\'azquez). 
The method has been applied to various equations in many papers.
Their method is universal in the sense that it works well for equations of essentially different types, for example
Schr\"{o}dinger equation, heat equation, harmonic map heat flow, chemotaxis system and Navier-Stokes equations, 
since they do not rely on comparison principle.
 Another advantage of the method of Merle and Rapha\"el is that it showed the stability/instability of bubble for their special solutions. 
We note that no known results based on the method of Herrero-Vel\'azquez included such information.
On the other hand, it seems that this 
method needs much tougher and longer computation than that of Herrero and Vel\'azquez (see e.g. \cite{MerRapIM}, \cite{MerRapCMP}, \cite{MerRapAM}, \cite{MerRapAJM}, \cite{DuKeMer}, \cite{RapSchHarM}, \cite{RapSchChemo}, \cite{MerRapRod}, \cite{ColMerRap}, \cite{ColGhoMasNgu}, \cite{MerRapRodSze}). 

Another way of construction was given by del Pino, Musso and Wei. They dealt with so called critical case in Fujita equation and chemotaxis system 
(see e.g. \cite{dPiMusWeiAMS}, \cite{JuadPiWei}, \cite{dPiMusWeiJFA}).

On the other hand, in \cite{mizo_ma07} (partially), \cite{mat_cm07}, \cite{mizo_tams11},  blow-up rate of all type II blow-up solutions with radial symmetry was determined, that is, for each radial type II blow-up solution $ u $, there exists $ C_1, C_2, m > 0 $ such that 
\[
C_1 (T^*-t)^{-m} \le \|u(\cdot, t)\|_\infty \le C_2 (T^*-t)^{-m}, \quad t \in (0,T^*),
\]
where $T^* < \infty $ is the blowup time of $ u $.  
We note that the coefficients $ C_1, C_2 $ in the upper and lower estimates may be  different, which prevented them from deriving the space-time profile of bubbling.
There have been no results on complete description of bubbling for all type II blow-up solutions except \cite{mizo_cpam} as far as we know. 
Since features peculiar to chemotaxis system were essentially used in \cite{mizo_cpam} and blow-up rate
 there was unique, 
 which implies nonexistence of complicated algebraic structure, the situation is simpler than ours in some sense. 
  Moreover, the stability of bubble was not treated there.

\section{Main results: GBU and RBC rates and space-time profiles}

In what follows, we consider \eqref{maineq} on a bounded interval or on the half-line, namely,
\be{equ}
\left\{\ 
\begin{aligned}
	u_t&=u_{xx}+|u_x|^p,&&\quad x\in \Omega=(0,R),\ t>0, \\
	u&=0,&&\quad x\in \partial\Omega,\ t>0, \\
	u(x,0)&=u_0(x),&&\quad x\in \Omega,
\end{aligned}
\right.
\ee
where $ u_0 \in \mathcal{W} $ and $0<R\le \infty$. 
All the results below turn out to hold true in radial domains in general dimensional space by easy modification of the arguments.

Let $u$ be the global viscosity solution of \eqref{equ}.
It is known that $u_x$ can become unbounded only near $x=0$, or near $x=R$ if $R<\infty$. 
For $T\in(0,\infty)$, we say that {\it $u$ undergoes GBU at $(x,t)=(0,T)$}
if there exists $\eta\in(0,T)$ such that
\be{defGBUtime}
\hbox{$u_x\in C([0,R)\times[T-\eta,T))$ \quad and \quad $\ds\limsup_{x\to 0,\ t\to T_-} |u_x|=\infty$.}
\ee
Note that the first part of \eqref{defGBUtime} ensures that $u(0,t)=0$ on $[T-\eta,T]$ in the classical sense (cf.~\cite{BdaLio,PS3}).
Moreover, we then have (see Proposition~\ref{prop:PS})
\be{propMPux}
\sup_{x\in (0,R/2)} |u_x(x,t)|=u_x(0,t), \quad\hbox{ for $t$ close enough to $T$.}
\ee
For $\tau\in(0,\infty)$, we say that {\it $u$ undergoes RBC at $(x,t)=(0,\tau)$}
 if there exists $\eta\in(0,\tau)$ such that, in the classical sense,
 \be{defRBCtime}
\hbox{$u(0,t)>0$ for all $t\in(\tau-\eta,\tau)$ \quad and \quad $u(0,\tau)=0$.}
\ee
When $R<\infty$, the case of GBU or RBC at $x=R$ is similar by setting $x'=R-x$.

Our aim is to completely classify the behavior of $ u $, namely, rates and space-time profile as $ t \to T_- $ and $ t \to \tau_- $.
A key in our proof is to focus an algebraic structure of solutions as $ t \to T_- $ and $ t \to \tau_- $ with respect to vanishing intersections with the singular steady state 
at boundary point at those times. 
Since the number of intersections with the singular steady state does not change as $ t \to T_+ $ or $ t \to \tau_+ $, the behaviors of viscosity solutions in those cases are not subject of this paper.  
We need delicate considerations coming from treatment of weak solutions instead of classical solutions.

\subsection{Main theorems} \label{secmainres}

In view of subsequent results, we introduce the scaling parameters 
$$\beta:=\ts\frac{1}{p-1}\in (0,1),\qquad k:=\frac{p-2}{2(p-1)}=\frac{1-\beta}{2}\in (0,\ts\frac12)$$
and the singular and regular steady states of \eqref{equ}, respectively given by
\be{defU}
U(x):=c_p x^{1-\beta},\quad x>0,
\quad\hbox{ where }
c_p:=(1-\beta)^{-1}\beta^\beta
\ee
and
\be{defUa}
U_a(x):=U(a+x)-U(a),\quad x>0\qquad (a>0).
\ee
 
When a viscosity solution $ u $ of \eqref{equ} undergoes GBU or RBC at $ (x,t)=(0,T) $ 
there exist $r\in (0,R]$, $t_0<T$ and an integer $m \ge 1$ such that
$$ 
\hbox{for all $t\in(t_0,T)$, \ $u(\cdot,t)-U$ has exactly $m$ zeros on $(0,r)$} 
$$ 
 and, denoting $0<x_1(t)<\dots<x_m(t)$ the zeros of $u(\cdot,t)-U$ on~$(0, r)$,
 we have $ \liminf_{t\to T_-}x_1(t)=0$ 
(see Proposition~\ref{ZeroNumberConst}). 
We call $n\in\{1,\dots,m\}$ 
the number of vanishing intersections between $ u(\cdot,t) $ and $U$ 
at $ (x,t) = (0,T) $, defined by
\be{defvanishnumber}
 n=\max\bigl\{i\le m;\ \liminf_{t\to T_-}x_i(t)=0\bigr\}.
\ee

The following two theorems give the complete classification of 
 bubble including the determination of its stability/instability in the GBU case.

\medskip

\begin{thm}
\label{mainThm1}
Let $p>2$, $0<R\le \infty$ and $T\in(0,\infty)$. 

\smallskip
(i) Suppose that a viscosity solution $ u $ of \eqref{equ} with $u_0\in \mathcal{W}$ undergoes GBU at $ (x,t) = (0,T) $.
Let $ n $ be the number of vanishing intersections between $ u(\cdot,t) $ and $U$ 
at $ (x,t) = (0,T) $.
Then there exists a constant $L>0$ such that  
\be{rateell}
\lim_{t\to T^-}  (T-t)^{\frac{n}{p-2}}u_x(0,t)= L
\ee
and 
\be{asympyux1}
u(x,t) = U_{a(t)}(x) + O(x^2) \quad \mbox{ and } \quad 
u_x(x,t) = U'_{a(t)}(x)+O(x),
\ee
with $a(t):= \beta u_x^{1-p}(0,t)\sim \beta L^{1-p} (T-t)^{\frac{p-1}{p-2} n}$ as $t \to T_-$.

\smallskip

(ii) For each integer $n\ge 1$, there exists $u_0\in \mathcal{W}$ such that 
the solution of \eqref{equ} satisfies \eqref{rateell} and \eqref{asympyux1} 
with $T=T^*(u_0)<\infty$ and some $L>0$.  
\end{thm}

\begin{rem} \label{RemThm1}
	 The	 constant $L$ in \eqref{rateell} is not universal. In fact, it can take any positive value as follows: if $u$ is a solution 
	of \eqref{equ} with $R=\infty$ which satisfies \eqref{rateell}, then for any $\alpha>0$, $u_\alpha(x,t)=\alpha^{-k}
	u(\sqrt\alpha\,x,T+\alpha (t-T))$
	solves the same equation and satisfies \eqref{rateell} 
	with $L$ replaced by $\alpha^q L$ where $q=\frac{1}{2(p-1)}-\frac{\ell}{p-2}<0$. 
\end{rem}

 In view of the next statement, we define the stability of the space-time profiles with the continuity of GBU times.

\begin{defn}
\label{Defstab}
Let $u_0\in\mathcal{W}$, $n\ge 1$ be an integer and assume that the solution
$u=u(\cdot,\cdot;u_0)$ of \eqref{equ} undergoes GBU at $(x,T)=(0,T)$ and satisfies 
\eqref{rateell}-\eqref{asympyux1}.
\smallskip

(i) We say that the GBU time is {\it continuous at $ u_0 $} if for each $ \eps > 0 $ there exists $ \delta > 0 $ 
such that 
$$\hbox{$ \| \tu_0 - u_0 \|_{W^{1,\infty}} < \delta \Longrightarrow \tilde u(\cdot,\cdot;\tilde u_0)$
undergoes GBU at $ (x,t) = (0, \tT) $ for some $\tT\in(T-\eps,T+\eps)$.}$$
Otherwise, the GBU time $T$ is said to be discontinuous.
\smallskip

(ii) We say that the GBU space-time profile of $ u $ is {\it stable} if the GBU time $T$ is {\it continuous at $ u_0 $}
and if, moreover, $ \tu, \tT $ in part~(i) satisfy
\eqref{rateell}-\eqref{asympyux1} for some $\tilde L>0$. 
Otherwise, the GBU space-time profile is said to be unstable.
\end{defn}

\begin{thm}
\label{mainThm1stab}
Let $u_0,u,T,n$ be as in Theorem~\ref{mainThm1}. Then:
\smallskip

(i) The GBU time $T$ is continuous at $ u_0 $ if and only if $n$ is odd.
\smallskip

(ii) The GBU space-time profile of $ u $ is stable if and only if $n=1$.
\end{thm}

We next give a complete classification in the case of RBC, which implies that the bubbles do not appear, unlike in GBU case.
Our classification result actually applies to solutions of the following more general RBC problem:
\be{equREC}
\left\{\ 
\begin{aligned}
	u_t&=u_{xx}+|u_x|^p,&&\quad\hbox{in $\Omega=(0,R)\times(0,\tau)$,} \\
	u&=0,&&\quad \hbox{on $\partial\Omega\times(0,\tau)$ in the viscosity sense,} \\
	u&>0,&&\quad \hbox{on $\{0\}\times(0,\tau)$ in the classical sense,} \\
	u(0,\tau)&=0,&&\quad \hbox{in the classical sense} \\
\end{aligned}
\right.
\ee
(cf.~\eqref{defviscBC} for the definition of the boundary conditions in the viscosity sense).

\begin{thm}
	\label{th:RBC} 
	Let $p>2$, $0 < R\le \infty$, $\tau\in(0,\infty)$, 
	 and set $Q=(0,R)\times(0,\tau)$.
	
	\smallskip
	
		(i) Let $u\in C^{2,1}(Q)\cap C_b(\overline Q)$ be a solution of problem \eqref{equREC}.
	Let $ n $ be the number of vanishing intersections between $ u(\cdot,t) $ and $U$ at $ (x,t) = (0,\tau) $. 
	Then there exists a constant $L>0$ such that 
	\be{rateell-REC}
	\lim_{t\to \tau^-} (\tau-t)^{-n} u(0,t)= L
	\ee
		 and
	\be{recovery-profile}
u(x,t) = L(\tau -t)^n  \phi_n  \bigl((\tau -t)^{-1/2} x\bigr) + o ( (\tau -t)^n ) \quad \mbox{ as } t \to \tau_-,
	\ee
uniformly for $(\tau -t)^{-1/2} x\ge 0$ bounded,
where $ \phi_n $ is the eigenfunction of 
	\[\phi_{yy}+\Bigl(\frac{p}{p-1}\frac{1}{y}-\frac{y}{2}\Bigr)\phi_y+k\phi = - \lambda \phi\]
	 with $ \phi_n (0) = 1 $ corresponding to the $ n $th eigenvalue $ \lambda_n := n - k $
(see subsection~\ref{subseceigen}).

	 	\smallskip

(ii) For each integer $n\ge 1$, there exists a solution of \eqref{equREC} 
which satisfies \eqref{rateell-REC} and \eqref{recovery-profile} 
for some $ \tau\in(0,\infty)$ and $L>0$.
\end{thm}

We see that Theorem~\ref{th:RBC}(i) applies to viscosity solutions of \eqref{equ} with $u_0\in\mathcal{W}$ as a special case.

\begin{rem} \label{RemThm1REC}
	The analogue of Remark~\ref{RemThm1} remains true for Theorem~\ref{th:RBC}.
\end{rem}

 Next going back to problem \eqref{equ} with $u_0\in\mathcal{W}$, 
 the stability of space-time profiles with the continuity of RBC times is defined in the same way as in Definition~\ref{Defstab},
 replacing GBU with RBC and \eqref{rateell}-\eqref{asympyux1} with 
  \eqref{rateell-REC}-\eqref{recovery-profile}.
The following result gives the complete classification of 
stability/instability in the RBC case.

\begin{thm}
\label{mainThmRBCstab}
Suppose that a viscosity solution $ u $ of \eqref{equ} with $u_0\in \mathcal{W}$ undergoes RBC at $ (x,t) = (0,\tau)$
and satisfies \eqref{rateell-REC}-\eqref{recovery-profile} for some $n\ge 1$. Then:
\smallskip

(i) The RBC time $\tau$ is continuous at $ u_0 $ if and only if $n$ is odd.
\smallskip

(ii) The RBC space-time profile of $ u $ is stable if and only if $n=1$.
\end{thm}

\begin{rem} \label{Rem_immed}
(i) In Theorem~\ref{mainThm1}, if $n$ is odd (resp., even), then $u$ undergoes immediate LBC (resp., regularization) after GBU at $t=T$.
In Theorem~\ref{th:RBC}, if $n$ is even (resp., odd), then $u$ undergoes immediate LBC (resp., regularization) after RBC at $t=\tau$;
in the case $n$ even, $\tau$ is thus a bouncing time.
This follows from the proof of Theorems~\ref{mainThm1stab} and \ref{mainThmRBCstab}.

\smallskip

(ii) Some restricted stability such as finite codimensional stability has been observed in blow-up phenomena of other equations 
(see, e.g.,~\cite{ColRapSze, Kr, MerRapSze, ColMerRap}).
In these works, codimensional stability is established by using spectral properties of a linearized operator 
in a suitable weighted space.
However, they deal only with some special (type~I or type~II) blowup solutions.
Our method to prove stabillity/instability is completely different, 
based on the study of the vanishing intersections with the singular steady state,
and it covers general solutions of \eqref{equ} in GBU case (which is always of type II)
and RBC case which is quite different from known studies on blowup.
\end{rem}

\subsection{Ideas of proofs of the main theorems}
Theorems~\ref{mainThm1}(ii), \ref{th:RBC}(ii) immediately follow from the corresponding theorems 
on construction of special solutions (Theorems~\ref{prop:special} and \ref{prop:special-LBC}).
Although the method of Merle, Rapha\"el~et~al. works in various equations as mentioned in subsection 1.4, we adopt the method of Herrero and Vel\'azquez with modifications to adapt to the viscous Hamilton-Jacobi equation 
for the following reasons: 
\begin{itemize}
	\item When comparison principle, which is a very strong tool in parabolic equations, is applicable, the method of Herrero and Vel\'azquez seems to be simpler than that of Merle, Rapha\"el~et~al. owing to this tool. 
	\vskip 2pt
	
	\item We deal with the RBC case as well. In this case, 
	the dynamics of special solutions is determined only by the linearization around the singular steady state $ U $ instead of the quasi-stationary state (called a bubble in Merle, Rapha\"el~et~al.). 
		\vskip 2pt
		
		\item  We do not deal with the stability/instability of the special solutions since we make use of the method of Herrero and Vel\'azquez. However, making use of intersection argument, we determine the stability/instability of all solutions and not only special solutions.

\end{itemize}
We will give heuristic arguments on the construction of special solutions in Sections 4, 5. 

The proofs of our classification results Theorems~\ref{mainThm1}(i) and \ref{th:RBC}(i) rely on a sophisticated parabolic comparison method based on braid group theory.
The method originated at \cite{mat_cm07} and \cite{mizo_tams11} applying the notion of parabolic reduction defined by Matano to the three solutions introduced in  \cite{mizo_ma07} to investigate type II blow-up rate of Fujita equation.
In this paper, we deal with viscosity solutions, whose difference from classical solutions essentially appears in the RBC case.
Constructing special solutions with specific behavior in outer region, which is far from bubble and RBC point, corresponding to our purpose in various situations,  
we make use of the special solutions to apply braid group theory and to rule out oscillation of the rates.
Therefore the delicate construction of special solutions is one of essential ingredients also in the complete classification. 
Whereas bubble is a phenomenon of special GBU solutions in inner region, the behavior of the special solutions in outer region also plays an important role to describe the space-time profile of bubble of general solutions in our method. 
We have the same matter in the RBC case.

We first consider the GBU case.
Let $ u $ be a viscosity solution of \eqref{equ} which undergoes GBU with $ n\ (\ge 1) $ vanishing intersections 
$ 0 < x_1 (t) < x_2 (t) < \cdots < x_n (t) $ with $ U $  at $ (x,t) = (0,T) $.
For any $ 0 < D \ll 1 $, there exist $ \delta_0 \in (0,M D^{p/(p-1)}/2], t_0 < T $ such that 
\be{idea-eq:zeros}
x_n < D \quad \mbox{ and } \quad |u(D,t) - U(D)| \ge \delta_0 U(D) \quad \mbox{ in } [t_0, T),
\ee
where $ M = M(u_0)> 0 $ (Lemma~\ref{lemma:zero-curve}).
We show that 
\be{idea-liminfsupGBU}
0<\liminf_{t\to T_-} \, (T-t)^{\frac{n}{p-2}} u_x(0,t)\le
\limsup_{t\to T_-} \, (T-t)^{\frac{n}{p-2}} u_x(0,t)<\infty.
\ee
For $ a > 0 $, define a solution $ u_a $ by
\be{idea-eq:u_a-def}
u_a (x,t) := 
a^k u( a^{-1/2} x, T + a^{-1}(t-T)) \quad \mbox{ in } (0, a^{1/2}R) \times ((1-a)T,T).
\ee

We construct a special solution $ v $ with $ n $ vanishing intersections with $ U $ at $ (x,t) = (0,T) $ 
such that 
\[
\lim_{t \to T_-} (T-t)^{n/(p-2)} v_x (0,t) = C 
\] 
for some $ C > 0 $ and  
  \be{idea-eqndescr1}
  |U(a^{1/2} D) -  v (a^{1/2} D, t)| > \delta_1 U(a^{1/2} D) > |U(a^{1/2} D) -  u_a (a^{1/2} D, t) |,\quad  t\in [t_1, T)
  \ee
  and
  \be{idea-eqndescr2}
  z\bigl( u_a(\cdot,t_1) - U: [0, a^{1/2} D] \bigr) = z\bigl( u_a(t_1) - v(t_1): [0, a^{1/2} D]\bigr)= n
  \ee 
for some $a \gg 1 $, $ t_1 < T $, $ \delta_1 \in (0,1) $ (Theorem \ref{prop:special}, Lemma \ref{lemma:u_a}). 
Assume for contradiction that the first inequality of \eqref{idea-liminfsupGBU} does not hold.
Then there exists $ t_2 \in (t_1, T) $ such that $ v(\cdot,t) - u_a(\cdot,t) $ loses one zero (or odd number of zeros) at $ (x,t) = (0,\hat{t}) $ for some $ \hat{t} < t_2 $ close enough to $ t_2 $, and
\be{idea-eq:t_2}
	v (x, t_2) > u_a (x, t_2) \quad \mbox{ for } 0  < x \ll 1.
\ee
For $ 0 < \lambda < 1 $, let 
\[
\tu_a (x,t) := 
\lambda^k u_a ( \lambda^{-1/2} x, t_1 + \lambda^{-1} (t - t_1) ) 
\quad \mbox{ in } (0, \lambda^{1/2} a^{1/2}R) \times ( t_1, \tT )
\]
with $ \tT : = t_1 + \lambda (T - t_1)  <T$. 
For $ \lambda $ close enough to $ 1 $, \eqref{idea-eqndescr1}-\eqref{idea-eq:t_2} hold true with $ u_a $ replaced by $ \tu_a $.

On the other hand, following \cite{Ghrist-VandenBerg-Vandervorst}, we regard three solutions only with transversal intersections as a positive braid of three strands
(we review basic properties about application of braid group theory to parabolic PDE in Section 6).
Now let $ \tA_n, \tB_n $ be the braids defined in Lemma~\ref{lemma:braid-even-odd}.
Take $ 0 < \rho_1 \ll 1 $. 
For $ 0 < \rho \le \rho_1 $, the situation of $ U, \tu_a(t), v(t) $ in $ [\rho, a^{1/2} D] $ at $ t = t_1 $ is represented by $ \tA_n $.
Since $ \tu_a $ undergoes GBU at $ (x,t) = (0,\tT) $, there exists $ t_3 \in (t_2, \tT) $ such that $ v (\cdot,t) - \tu_a (\cdot, t)$ loses one zero (or odd number of zeros) at $ (x,t) = (0, \tilde{t})$ for some $ \tilde{t} < t_3 $ close to $ t_3 $, and
$ v (x, t_3) < \tu (x, t_3) $ for $ 0  < x \ll 1 $. 
Choosing $ 0 < \rho_2 \le \rho_1 $, the situation of $ U, \tu_a(t), v(t) $ in $ [\rho_2, a^{1/2} D] $ at $ t = t_3 $ is translated into~$\tB_n$. 
Roughly speaking, parabolic reduction in term of braid means a phenomenon of vanishing or collapsing
intersections between two solutions of parabolic PDE.  
The process from $ t = t_1 $ to $ t = t_3 $ implies that $ \tA_n \Rrightarrow \tB_n $.
 But on the other hand we have $ \tA_n \not\Rrightarrow \tB_n $ (Lemma \ref{lemma:braid-even-odd}). 
This contradiction implies the first inequality in \eqref{idea-liminfsupGBU}.

In order to get the upper estimate in \eqref{idea-liminfsupGBU}, we notice that all zeros of $ v_a (\cdot,t) - U$  
(and $ \tv_a(\cdot,t) - U$) locate 
in $ (0, C_1 (T-t)^{1/2} )$ for $ t \in [ t_0, T) $ with some $ C_1 > 0 $ for $ a \gg 1 $ (Theorem~\ref{prop:special}). 
If we choose $ t_0 < T $ such that $ C_1 (T-t_0)^{1/2} < D $, then
it suffices to take the same way as above with $ v, u_a $ and the spatial interval $ [ \rho_2, a^{1/2} D] $ replaced 
by $ u, v_a $ and $ [\rho_2, D] $, respectively, where $ v_a $ is defined in \eqref{idea-eq:u_a-def} with $ u $ replaced by $ v $.
Combining \eqref{idea-liminfsupGBU} with non-oscillation Lemma~\ref{LemNonOsc} implies the assertion on GBU rate of Theorem~\ref{mainThm1}. 
Then the space-time profile easily follows from a general property of solutions of  \eqref{equ}.

As seen above, we construct the special solutions not only to show the existence but also to get a key ingredient in the proof of general results, whereas the known papers aimed only at showing the existence of special type II blow-up solutions with their stability/instability. 
Therefore our construction must take care of the behavior of the special solutions at spatial infinity.

 Next, the ideas to determine the stability/instability of GBU space-time profiles with the continuity/discontinuity of GBU times are as follows.
In case $n$ is odd, then $u$ passes over $U$ for $x>0$ small at $t=T$ and one can deduce that $u$ immediately loses BC after $t=T$.
By suitable continuous dependence arguments it follows that solutions starting close to $u_0$
also undergo GBU at a time close to $T$, hence $T$ is continuous at $ u_0 $.
For $n=1$, continuous dependence and zero number arguments then show that such solutions
have only one vanishing intersection near their GBU time, hence the stability of the profile.
In case $n$ is even, $u$ falls under $U$ for $x>0$ small at $t=T$ and one can deduce that $u$ is immediately regularized after $t=T$. For initial data $\hu_0$ close to and below $u_0$, suitable continuous dependence and comparison arguments
next show that the solution $\hu$ remains classical at $t=T$ and then stays classical for some uniform amount of time,
hence the discontinuity of GBU time at $ u_0 $. Finally, for $n\ge 3$ odd and initial data $\hu_0$ close to and below $u_0$, 
the solution $\hu$ undergoes LBC at some $\tilde T>T$ close to $T$ and $\hu-U$ has at most $n$ zeros
in some neighborhood of $x=0$ for $t\in[T-\eps,\tilde T)$ and some $\eps>0$ small.
Moreover, $n-1$ zeros of $\hu-U$ are ``squeezed'' by those of $u-U$ and thus have to vanish at $t=T$.
Consequently only one vanishing intersection can remain for $\hu$ and thus the profile of $u$ is unstable.

The proof of Theorem~\ref{th:RBC}(i) (and of Theorem~\ref{mainThmRBCstab}) in the RBC case, in which features of viscosity solutions appear, is carried out along the above scenario 
except the following:
\begin{itemize}
	\item Special solutions are quite different from those in bubbling case (Theorem~\ref{prop:special-LBC});
	\vskip 1pt
	\item We need extra care for vanishing intersections at $ x = 0 $  to apply braid group theory 
	since the derivative of viscosity solution at $ x = 0 $ is $ + \infty $ at each time during loss of boundary condition
	 (Proposition~\ref{propZeroNumber}).
\end{itemize}

\section{Auxiliary results: linearized operator and properties of viscosity solutions}\label{section4}

In this section we develop a number of auxiliary tools which are required in the subsequent sections.
They concern various properties of viscosity solutions and of the linearized operator: spectral analysis, semigroup properties, heat kernel,
maximum principles, zero number.

\subsection{Bounds for GBU and LBC solutions}

In this subsection we gather some fundamental estimates for GBU and LBC viscosity solutions
that will be used repeatedly.

\begin{prop}
\label{prop:PS} \
Let $0<R\le\infty$ and let 
$ u $ be a viscosity solution of \eqref{equ} with $u_0\in \mathcal{W}$ undergoing GBU at $ (x,t) = (0,T ) $. 
Then 
\begin{equation}
\label{eq:u_t-M} 
M:=\sup_{(0,R)\times(T/2,T)} |u_t| <\infty
\end{equation}
and there exist $M_0>0$, $x_0\in(0,R/2)$ and $\eta\in(T/2,T)$, such that, for all $t \in [T-\eta,T)$,
\be{eq:BU-rate-lower} 
m(t):=u_x(0,t) = \sup_{x\in (0,R/2)} |u_x(x,t)| \ge M_0 (T-t)^{- \frac{1}{p-2}},
\ee
\be{eq:u_x-lowerupper} 
\Bigl|u_x(x,t) - \left[m^{1-p}(t) + (p-1)x \right]^{ - \frac{1}{p-1} } \Bigr| \le M x, \quad 0 < x < x_0,
\ee
and
\be{eq:u_x-lowerupperT} 
\Bigl|u_x(x,T) -U'(x) \Bigr| \le M x, \quad 0 < x < x_0.
\ee
\end{prop}

We give a proof since the results seem only available in special cases
(see \cite{CG, QSbook, PS3}) and some care is needed in the case of viscosity solutions.

\begin{proof} 
Property \eqref{eq:u_t-M}
follows from the maximum principle applied to regularized problems (see \cite[Section~2]{SZ} for details).

\smallskip

To prove \eqref{eq:BU-rate-lower}, let us set $r=R/2$ if $R<\infty$ and $r=1$ otherwise. 
By \cite[Theorem~3.1]{SZ} and \eqref{boundunifvisc}, we have 
\be{estimSZ1}
M_1:=\sup_{t\ge T/2} |u_x(r,t)|<\infty,
\ee
\be{estimSZ2}
\sup_{x\ge 1,\, t\ge T/2}|u_x(x,t)|<\infty\ \hbox{ if $R=\infty$.}
\ee
Since $u$ undergoes GBU at $(x,t)=(0,T)$, there exists $\eta_0\in(0,T)$ such that
$u_x\in C([0,R)\times[T-\eta_0,T))$ and $u(0,t)=0$ in $[T-\eta_0,T]$.
By the maximum principle applied to $u_x$ (cf.~\cite[Proposition~2.3]{SZ}), $v=\pm u_x$ satisfies
\be{remMPux2}
\sup_{[0,r]\times(T-\eta_0,t)} v 
\le \max\biggl\{\sup_{0<x<r}|u_x(x,T-\eta_0)|,M_1,\sup_{s\in(T-\eta_0,t)} v(0,s)\biggr\},\quad T-\eta_0<t<T.
\ee
Next we claim that
\be{remMPux3}
\hbox{ $u_{xt}(0,t)\ne 0$ for $t<T$ close to $T$.}
\ee
Property \eqref{remMPux3} follows by zero number argument applied to $u_t$ (cf.~\cite[Proposition~6.1]{PS3}),
provided $u_t(x_1,T)\ne 0$ for some $x_1\in(0,R)$.
Assume for contradiction that $u_t(\cdot,T)\equiv 0$. Then $u(\cdot,T)$ is a steady state with $u(0,T)=0$,
i.e.~$u(\cdot,T)\equiv U_b$ in $[0,R)$ for some $b\ge 0$
(where $U_0:=U$; the case $u(\cdot,T)\equiv 0$ is excluded since $u\ge e^{t\Delta}u_0>0$).
If $R=\infty$ this contradicts \eqref{boundunifvisc}. 
If $R<\infty$ this implies $u(R,T)=U_b(R)>0$, hence $u_x(R,T)=-\infty$ by \cite[Lemma~5.4]{PS3}, 
a contradiction with $u(\cdot,T)\equiv U_b$.

Now, since $u_x(0,t)\ge 0$ and $u$ undergoes GBU at $(x,t)=(0,T)$, it follows from \eqref{remMPux2} 
and \eqref{remMPux3} that $u_{xt}(0,t)>0$ for $t<T$ close to $T$
and that 
\be{ux0infty}
u_x(0,t)\to\infty\quad\hbox{as $t\to T$.}
\ee
By \eqref{estimSZ2} and \eqref{remMPux2}, 
we deduce that
$$\inf_{[0,R/2)\times(T-\eta,T)}u_x>-\infty$$
and, for some sufficiently small $\eta\in(0,\eta_0)$,
\be{ux0tsup}
u_x(0,t)=\sup_{0\le x< R/2} u_x(x,t),\quad T-\eta<t<T,
\ee
hence the equality in \eqref{eq:BU-rate-lower}.

\smallskip

Let us now show the inequality in \eqref{eq:BU-rate-lower}, based on 
arguments from the proof of \cite[Theorem~1.2]{CG} (see also \cite[Theorem~40.19]{QSbook}).
First note that, in case $R<\infty$, by taking $\eta>0$ smaller if necessary 
we may assume that
\be{ux0tsup2}
 u_x(0,t)\ge U'(R/2),\quad T-\eta<t<T.
\ee
Fix $t_0\in (T-\eta,T)$ and let
$$x_0:=\sup \bigl\{x\in (0,R); u_x(\cdot,t_0)<U' \hbox{ in $(0,x)$} \bigl\}.$$
 Since $u_x(x,t_0)<U'(x)$ for $x>0$ small, $x_0$ is well defined. 
Set $I=(0,x_0]$ if $x_0<\infty$ and $I=(0,\infty)$ if $x_0=R=\infty$.
By definition, we have
$u_x(x,t_0)<U'(x)$ in $(0,x_0)$, hence $u(x,t_0)<U(x) $ in $I$.
Also using \eqref{boundunifvisc} in case $x_0=R=\infty$, it follows that
\be{S-GBUcompUlambda}
u(x,t_0)\leq U_a(x) \quad\hbox{in $I$,\quad for all $a>0$ small.}
\ee
We claim that $x_0\in (0,R)$, hence
\be{S-GBUcompux}
u_x(x_0,t_0)=U'(x_0).
\ee
Indeed, otherwise $x_0=R$ and \eqref{S-GBUcompUlambda} implies
$u\leq U_a$ in $[0,R)\times [t_0,T)$
for $a>0$ small, by the comparison principle (in the case $R<\infty$, we apply the comparison principle for viscosity solutions).
Therefore,
$u_x(0,t)\leq U'_a(0)$ in $[t_0,T)$, which contradicts~\eqref{ux0infty}.
We next claim that 
\be{S-GBUcompU}
\max_{t\in [t_0,T]} u(x_0,t)\geq U(x_0).
\ee
Suppose the contrary. Then, for all $a>0$ small, we have
$u(x_0,t)\leq U_a(x_0)$ in $[t_0,T)$.
By \eqref{S-GBUcompUlambda} and the comparison principle, 
we deduce that
$u\leq U_a$ in $[0,x_0]\times [t_0,T)$,
leading again to a contradiction.

Now using \eqref{eq:u_t-M}, \eqref{S-GBUcompU} and taking $t_1\in[t_0,T]$ such that 
$\max_{t\in [t_0,T]} u(x_0,t)= u(x_0,t_1)$, we get 
$$M(T-t_0)\geq M(t_1-t_0)\geq\int_{t_0}^{t_1} u_t(x_0,t) \geq U(x_0)-u(x_0,t_0)
=\int_0^{x_0} (U'(x)-u_x(x,t_0))\,dx.$$
On the other hand, we have $U'(x_0)=u_x(x_0,t_0)\le u_x(0,t_0)$,
as a consequence of \eqref{S-GBUcompux} and \eqref{ux0tsup} if $x_0<R/2$, or \eqref{ux0tsup2} otherwise.
Therefore, there exists $x_1\in (0,x_0]$ such that
$U'(x_1)=u_x(0,t_0)$.
Since $U'(x)-u_x(x,t_0)>0$ on $(0,x_0)$ by the definition of $x_0$, 
we obtain
$$
M(T-t_0)
\geq\int_0^{x_1} (U'(x)-u_x(x,t_0))\,dx
=U(x_1)-x_1U'(x_1)=\frac{{U'}^{2-p}(x_1)}{(p-1)(p-2)}
= \frac{(u_x)^{2-p}(0,t_0)}{(p-1)(p-2)},
$$
and \eqref{eq:BU-rate-lower} follows.

\smallskip

Let us finally check \eqref{eq:u_x-lowerupper}.
Using \eqref{eq:u_t-M}, it follows from the proof of \cite[Proposition 40.16]{QSbook} that
\be{eq:BU-rate-upper0}
u_x(x,t) \le \left[m^{1-p}(t) + (p-1)x \right]^{ - \frac{1}{p-1} } + M x,\quad 0 < x < R/2,\ T-\eta<t<T
\ee
and
\be{eq:BU-rate-lower0}
( u_x(x,t) )_+ \ge \left[m^{1-p}(t) + (p-1)x \right]^{ - \frac{1}{p-1} } - M x,\quad 0 <x < R/2,\ T-\eta<t<T,
\ee
hence, by \eqref{eq:BU-rate-lower},
$$
( u_x(x,t) )_+ \ge \left[M_0^{1-p}(T-t)^{\frac{p-1}{p-2}} + (p-1)x \right]^{ - \frac{1}{p-1} } - M x,\quad 0 <x < R/2,\ T-\eta<t<T.
$$
Taking $\eta>0$ smaller if necessary and choosing $x_0\in(0,R/2)$ small enough, we have
$M_0^{1-p}\eta^{\frac{p-1}{p-2}}+ (p-1)x_0<(M x_0)^{1-p}$,
so that the RHS of \eqref{eq:BU-rate-lower0} is positive and 
\eqref{eq:u_x-lowerupper}-\eqref{eq:u_x-lowerupperT} follow from \eqref{eq:BU-rate-upper0} 
and~\eqref{eq:BU-rate-lower0}.
\end{proof}

We next consider general solutions of
\be{equRECut}
\left\{\ 
\begin{aligned}
	u_t&=u_{xx}+|u_x|^p,&&\quad\hbox{in $(0,R)\times(0,T)$,} \\
	u&=0,&&\quad \hbox{on $\{0\}\times(0,T)$ in the viscosity sense,} \\
\end{aligned}
\right.
\ee
without regularity assumptions on $u(\cdot,0)$ (nor conditions at $x=R$).
Unlike the case of viscosity solutions of \eqref{equ} with regular initial data $u_0\in \mathcal{W}$,
the bound \eqref{eq:u_t-M} on $u_t$ does not seem available for general viscosity solutions of \eqref{equRECut}.
We now establish such a bound, which requires more elaborate arguments.

\begin{lem} \label{lemsmoothut}
Let $p>2$, $T>0$, $R\in(0,\infty)$ and $Q=(0,R)\times(0,T)$. 
Assume that $u\in C^{2,1}(Q)\cap C(\overline Q)$ satisfies \eqref{equRECut}.
Then, for each $\sigma\in (0,T)$, there holds
\be{boundsmoothut}
\sup_{(0,R/2)\times(\sigma,T)} |u_t| \le M<\infty,
\ee
 where $M$ depends only on $p,R,\sigma$ and 
$\sup_{\overline Q}|u|<\infty$.
\end{lem}

\begin{proof}
Set $K:=\sup_{\overline Q}|u|<\infty$.
 By the Bernstein type estimate in \cite[Theorem 3.1]{SZ}, we have
\be{boundSZux}
|u_x|\le CK(t^{-1/p}+x^{-1}+(R-x)^{-1})\quad \quad\hbox{ in $(0,R)\times(0,T),$}
\ee
with $C=C(p)>0$ (for further reference we note that \eqref{boundSZux} 
remains true with $(R-x)^{-1}:=0$ in case $R=\infty$ and $K<\infty$).
Assume $R=2$ without loss of generality.

We now use a modification of an argument from \cite{Zhao,Att1}.
Fix $\sigma\in (0,T)$. For $\alpha\in(\frac12,1)$, let 
$$u_\alpha (x,t) := \alpha^{-2k} u(\alpha x, \sigma+\alpha^2(t-\sigma))-A(1-\alpha),
\quad (x,t)\in Q_\sigma:=(0,2) \times (\sigma,T),$$
where the constant $A>0$ will be chosen below.
The function $u_\alpha$ satisfies $u_{\alpha,t}-u_{\alpha,xx}=|u_{\alpha,x}|^p$ in $Q_\sigma$.
On the other hand, for any $x\in(0,1]$, we have
$$u(x,\sigma)-u(\alpha x,\sigma)=(1-\alpha)x \,u_x(\tilde x,\sigma)$$
for some $\tilde x\in (\alpha x,x)\subset(x/2,x)$, hence, by \eqref{boundSZux},
$$|u(\alpha x,\sigma)-u(x,\sigma)|\le CK(1-\alpha)x(\sigma^{-1/p}+\tilde x^{-1})\le CK(1-\alpha)(\sigma^{-1/p}+2),\quad 0<x<1.$$
Using $\alpha^{-2k}-1 \le C_1(1-\alpha)$ for all $\alpha\in(\frac12,1)$ and some $C_1>0$, it follows that
  $$\begin{aligned}
u_\alpha (x,\sigma)-u(x,\sigma) 
&= \alpha^{-2k} u(\alpha x,\sigma)-A(1-\alpha)-u(x,\sigma) \\
&= (\alpha^{-2k}-1) u(\alpha x,\sigma)+u(\alpha x,\sigma)-u(x,\sigma)-A(1-\alpha) \\
&\le \bigl[C_1K+CK(\sigma^{-1/p}+2)-A\bigr](1-\alpha)\le 0,\quad 0<x<1,
\end{aligned}$$
by choosing $A\ge [C_1+C(\sigma^{-1/p}+2)]K$.
 By \eqref{boundSZux} and parabolic regularity estimates
in $[\frac12,\frac32]\times[\sigma/2,T]$, 
we see that there exists $C_2=C_2(K,\sigma,p)>0$ such that,
for all $\alpha\in(\frac12,1)$, 
$$|\alpha^{-2k} u(\alpha, \sigma+\alpha^2(t-\sigma))-u(1,t)|\le C_2(1-\alpha),\quad \sigma<t<T$$
hence, by taking $A\ge C_2$,
$$u_\alpha (1,t)-u(1,t)\le 0,\quad \sigma<t<T.$$
Since $u=-A(1-\alpha)<0$ on $\{0\}\times(0,T)$ in the viscosity sense, it follows from the comparison principle 
for viscosity solutions \cite{BdaLio} that
$$u_\alpha (x,t)-u(x,t)\le 0,\quad (x,t)\in \tilde Q_\sigma:=(0,1) \times (\sigma,T),$$
hence
$$u(x,t)-\alpha^{-2k} u(\alpha x, \sigma+\alpha^2(t-\sigma)) \ge -A(1-\alpha),\quad (x,t)\in \tilde Q_\sigma.$$
For fixed $(x,t)\in \tilde Q_\sigma$, dividing by $1-\alpha$ and letting $\alpha\to 1^-$, we obtain
$$2ku(x,t)+xu_x(x,t)+2(t-\sigma)u_t(x,t)\ge -A.$$
Using \eqref{boundSZux} again, we deduce that
$$u_t(x,t)\ge -(2\sigma)^{-1}\bigl[A+2kK+CK(\sigma^{-1/p}+1)\bigr],\quad 0<x<1,\ 2\sigma<t<T,$$
which yields the lower part of \eqref{boundsmoothut}.
Arguing similarly with $+A$ instead of $-A$ in the definition of $u_\alpha$,
we get the upper part.
\end{proof}

The next lemma asserts that $u_x$ cannot stay bounded when the boundary conditions are lost.
This property was given in \cite{PS3} 
for the initial-boundary value problem \eqref{equ} with $u_0 \in \mathcal{W}$
(based on approximation by truncated problems),
but it does not cover our situation, which requires a different proof.

\begin{lem} \label{lemsmoothut2}
Let $p>2$, $T>0$, $R\in(0,\infty)$ and $Q=(0,R)\times(0,T)$.
Assume that $u\in C^{2,1}(Q)\cap C(\overline Q)$ satisfies \eqref{equRECut}.
If $u(0,t_0)>0$ for some $t_0\in(0,T)$, then
$$\limsup_{(x,t) \to (0,t_0)} |u_x(x,t)|=\infty.$$
\end{lem}

\begin{proof}
Set $m(t)=u(0,t)$.
We claim that for each $\eps>0$, there exist $t_\eps\in(t_0-\eps,t_0+\eps)$ and $L_\eps,M_\eps,\eta_\eps>0$ such that
\be{mmteps1}
m(t)\le m(t_\eps)+L_\eps(t-t_\eps)+M_\eps(t-t_\eps)^2,\qquad t_\eps-\eta_\eps<t<t_\eps+\eta_\eps.
\ee
Indeed, if $m$ is convex on $[t_0-\eps,t_0+\eps]$, then it is well known that $m$ is twice differentiable almost everywhere 
and \eqref{mmteps1} immediately follows by choosing a time $t_\eps\in (t_0-\eps,t_0+\eps)$ where $m$ is twice differentiable.
If $m$ is not convex on $[t_0-\eps,t_0+\eps]$ then there exist $t_1,t_2$ with $t_0-\eps\le t_1<t_2\le t_0+\eps$
and $\bar t\in(t_1,t_2)$ such that $\max_{[t_1,t_2]}g=g(\bar t)>0$, where $g(t)=m(t)-m(t_1)-L(t-t_1)$
and $L=\frac{m(t_2)-m(t_1)}{t_2-t_1}$.
Consequently, $m(t)-m(t_1)-L(t-t_1)\le m(\bar t)-m(t_1)-L(\bar t-t_1)$ for all $t\in[t_1,t_2]$,
which implies \eqref{mmteps1} with $t_\eps=\bar t$, $L_\eps=L$, $M_\eps=0$ and $\eta_\eps=\min(\bar t-t_1,t_2-\bar t)$.

Next assume for contradiction that $|u_x|\le K$ in $(0,\eta)\times(t_0-\eta,t_0+\eta)$ for some $\eta,K>0$. 
Put $A_\eps=(2K)^p+|L_\eps|$ and set
$$\psi(x,t)=m(t_\eps)+L_\eps(t-t_\eps)+M_\eps(t-t_\eps)^2+2Kx-A_\eps x^2.$$
We compute $\bigl[\psi_t-\psi_{xx}- |\psi_x|^p\bigr](0,t_\eps)=L_\eps+2A_\eps-(2K)^p>0$.
On the other hand, using \eqref{mmteps1}, we deduce that
$$u(x,t)\le m(t)+Kx \le m(t_\eps)+L_\eps(t-t_\eps)+M_\eps(t-t_\eps)^2+Kx\le \psi(x,t)$$
in $[0,\min(\eta,A_\eps^{-1}K))\times(t_\eps-\eta_\eps,t_\eps+\eta_\eps)$,
hence $\psi$ is a smooth function which touches $u$ from above at $(x,t)=(0,t_\eps)$.
But since $u(0,t_\eps)>0$ for all $\eps>0$ sufficiently small, this contradicts
the definition \eqref{defviscBC} of the boundary conditions in the viscosity sense.
\end{proof}

Based on Lemmas~\ref{lemsmoothut} and \ref{lemsmoothut2}, we obtain the following estimates for 
 general viscosity solutions, including RBC solutions.

\begin{prop}
\label{prop:PS-LBC} 
Let $p>2$, $\tau>0$, $R\in(0,\infty)$ and $Q=(0,R)\times(0,\tau)$.

\smallskip

(i) Assume that $u\in C^{2,1}(Q)\cap C(\overline Q)$ satisfies \eqref{equRECut} 
and let $\sigma\in(0,\tau)$.
Then, for any $t\in[\sigma,\tau]$ such that $\limsup_{(x,s)\to (0,t)}|u_x(x,s)|=\infty$, we have
\be{eq:u-LBC-0}
| u(x,t) - u(0,t) - U(x) | \le \frac{\tilde M}{2} x^2, \quad 0 < x <  R/2,
\ee
\be{eq:u_x-LBC-0}
| u_x(x,t) - U^\prime(x) | \le \tilde M x, \quad 0 < x <  R/2,
\end{equation}
where $\tilde M>0$ depends only on $p,R,\sigma$ and $\sup_{\overline Q}|u|<\infty$.

\smallskip

(ii) Assume that $u\in C^{2,1}(Q)\cap C(\overline Q)$ is a solution of problem \eqref{equREC},
which undergoes RBC at $ (x,t) = (0,\tau)$. Then we have
\be{eq:u-M-LBC} 
u(0,t) \le M (\tau-t),\quad \tau/2<t<\tau,
\ee
\be{eq:u-LBC}
| u(x,t) - u(0,t) - U(x) | \le \frac{\tilde M}{2} x^2, \quad 0 < x <  R/2,\ \tau/2<t<\tau, 
\ee
\be{eq:u_x-LBC}
| u_x(x,t) - U^\prime(x) | \le \tilde M x, \quad 0 < x <  R/2,\ \tau/2<t<\tau,
\end{equation}
 where $M, \tilde M>0$ depend only on $p,R,\tau$ and $\sup_{\overline Q}|u|<\infty$.

\end{prop}

\begin{proof}
 (i) Using \eqref{boundsmoothut} this follows along the lines of the proof of \cite[Lemma~5.3]{PS3}.

\smallskip

(ii) Integrating \eqref{boundsmoothut}, we get
$u(x,t) - u(x,\tau) \le M(\tau-t)$ for each $t\in(\tau/2,\tau)$ and $x\in(0,R/2)$.
Since $u$ is continuous up to the boundary we may let $x\to 0$ and \eqref{eq:u-M-LBC}  follows from $u(0,\tau)=0$.
 The remaining properties are direct consequences of assertion~(i).
\end{proof}

We next gather some useful continuous dependence properties for problem \eqref{equ}.

\begin{prop}
\label{prop:PS-LBC2} 
Let $R\in(0,\infty]$, $\Omega=(0,R)$, $u_0\in \mathcal{W}$ and let $ u $ be the global viscosity solution of~\eqref{equ}.
Let $t_0>0$. 
\smallskip

(i) We have
\be{contdepCC1}
\hat u(\cdot,t_0)\to u(\cdot,t_0) \ \hbox{in $L^\infty(\Omega)\cap C^1_{loc}(\Omega)$, \ \ as $\|\hat u_0-u_0\|_\infty\to 0$,}
\ee
where $\hat u$ denotes the global viscosity solution of \eqref{equ} with initial data $\hat u_0\in \mathcal{W}$.
\smallskip

(ii) Assume that $u(\cdot,t_0)$ is classical at $x=0$. Then there exist $\eps,C>0$ such that
for any $\hat u_0\in \mathcal{W}$ with $\|\hat u_0-u_0\|_\infty\le\eps$, 
the corresponing solution satisfies 
\be{contdep2}
\hat u_x(x,t_0)\le C \ \hbox{in $(0,R/2)$}\quad\hbox{and}\quad \hat u(0,t_0)=0.
\ee
\end{prop}

\begin{proof}
(i) We know that 
\be{contdepCC1a}
\|\hat u(t)-u(t)\|_\infty \le \|\hat u_0-u_0\|_\infty,\quad 0\le t\le T
\ee
(see e.g. \cite[Theorem~3.1]{PS3}), hence the $L^\infty$ convergence in \eqref{contdepCC1}.
Next, let $\hat u_{0,j}\in \mathcal{W}$ be a sequence such that $\|u_0-\hat u_{0,j}\|_\infty\to 0$.
By \eqref{contdepCC1a} we have in particular 
\be{unifbounduj}
\sup_j \|\hat u_j\|_{L^\infty(\Omega\times(0,\infty))}<\infty
\ee
and we deduce from \eqref{boundSZux} and parabolic estimates 
that $\{\hat u_j(\cdot,t_0)\}$ is precompact in $C^1_{loc}(\Omega)$.
The assertion follows.

\smallskip

(ii) Assume for contradiction that there exists a sequence $\hat u_{0,j}\in \mathcal{W}$ such that
$$\sup_{x\in(0, R/2)}\hat u_{j,x}(x,t_0)\to \infty\quad\hbox{and}\quad\|u_0-\hat u_{0,j}\|_\infty\to 0,$$
 where $\hat u_j:=u(\hat u_{0,j};\cdot,\cdot)$.
Let $R_1=R/2$ if $R<\infty$ and $1$ otherwise.
By \cite[Lemma~5.2]{PS3} there exists a constant $M>0$ such that
\be{SGBUprofileLowerEst}
(\hat u_{j,x}(t_0,x))_+ \ge \bigl[\bigl((\hat u_{j,x}(t_0,y))_++My\bigr)^{1-p}+(p-1)(x-y)\bigr]^{-1/(p-1)}-\, Mx, \quad 0<y<x<R
\ee
(the fact that the constant $M$ can be chosen independent of $j$ follows from \eqref{boundsmoothut}).
On the other hand, by \eqref{boundSZux} and \eqref{unifbounduj}, we have $\lim_{j\to\infty} x_j=0$.
Applying \eqref{SGBUprofileLowerEst} with $y=x_j$, letting $j\to\infty$ and using \eqref{contdepCC1}, we obtain
$$(u_x(t_0,x))_+ \ge ((p-1)x)^{-1/(p-1)}-\, Mx, \quad 0<x<R,$$
hence $\lim_{x\to 0}u_x(t_0,x)=\infty$,
which contradicts the assumption that $u(\cdot,t_0)$ is classical at $x=0$.
This implies the first part of \eqref{contdep2} and the second part then follows from 
 Lemma~\ref{lemsmoothut2} and Proposition~\ref{prop:PS-LBC}(i).
\end{proof}

We end this subsection with a simple lemma that gives a sufficient condition to prevent GBU at the right boundary.

\begin{lem} \label{controlGBU1}
Let $\Omega=(0,1)$ and let $u_0\in\mathcal{W}$ satisfy $\|u_0\|_\infty\le \frac38$ and $u_0(x)=0$ on $[\frac12,1]$.
Then the solution of \eqref{equ} satisfies $u(x,t)\le 1-x$ in $[0,1]$ for all $t>0$.
\end{lem}  

\begin{proof}
Set $\overline u(x,t)=\frac12(1-x^2)$. Then
$$\overline u_t-\overline u_{xx}-|\overline u_x|^p=1-x^p\ge 0.$$
Since $u_0\le \overline u(\cdot,0)$ in $[0,1]$ by our assumption, 
we get $u\le \overline u$ in $[0,1]\times[0,\infty)$
by the comparison principle (for viscosity sub-/super-solutions,
or alternatively by approximating $u$ from below by truncated problems; see e.g.~\cite{PS3}).
The conclusion follows.
\end{proof}

\subsection{Similarity variables and linearized operator} \label{secsimil}

Let us introduce the {\it similarity variables}, which is the fundamental framework for the construction of
special GBU and RBC solutions.
	Namely, for given $0<R\le\infty$ and $T>0$, we set 
	\be{defys}
	y=x/\sqrt{T-t},\qquad s=-\log(T-t)
	\ee
	and
	\be{defw}
	w(y,s)=e^{ks}u(ye^{-s/2},T-e^{-s}). 
	\ee
	By straightforward calculations, if $u$ is a classical solution of $u_t-u_{xx}=|u_x|^p$
	in the cylinder $(0,R)\times(0,T)$, then the corresponding equation for $w$ is:
	\be{eqw}
		w_s=w_{yy}-\frac{y}{2}w_y+kw+|w_y|^p,\quad (y,s)\in D,
			\ee
	where $D=\{(y,s);\  0<y<Re^{s/2},\ s>s_0\}$.
Observe that $U(y)$ (cf.~\eqref{defU}) is also a steady state of \eqref{eqw}.
The following simple proposition 
shows that for any GBU or RBC solution, the corresponding $w$ converges to $U$ in $C^1$ 
except at $ y = 0 $ as $s\to\infty$.

\smallskip

\begin{prop}   \label{PropwcvU} 
Let $p>2$, $0<R\le\infty$. 
\smallskip

(i) Let $u$ be a viscosity solution of \eqref{equ} with $u_0\in \mathcal{W}$ which undergoes GBU at $ (x,t) = (0,T) $ 
for some $ T <\infty $ and let $w$ be defined by \eqref{defw}.
Then
\be{wcvU}
\lim_{s\to\infty} w(y,s)=U(y)\quad\hbox{ in $C_{loc}([0,\infty))\cap C^1_{loc}((0,\infty))$.}
\ee

\smallskip
(ii) Let $0<\tau < \infty$, set $Q=(0,R)\times(0,\tau)$. Let
$u\in C^{2,1}(Q)\cap C(\overline Q)$ be a solution of problem~\eqref{equREC}, which
	undergoes RBC at $(x,t) = (0,\tau)$,  and let $w$ be defined by \eqref{defw} with $T=\tau$.
	Then \eqref{wcvU} holds in $C^1_{loc}([0,\infty))$.
\end{prop}

	\begin{proof}[Proof of Proposition~\ref{PropwcvU}]
(i) Let $m(t)=u_x(0,t)>0$. By \eqref{eq:u_x-lowerupper}, we have
 \be{EstSpaceTime0}
u_x(x,t) = \bigl[m^{1-p}(t)+(p-1)x\bigr]^{-1/(p-1)}+O(x),
\ee
for all $x>0$ small and $t$ close $T$.
Consequently,
 \be{EstSpaceTime1}
 w_y(y,s)=e^{-\beta s/2}u_x(ye^{-s/2},T-e^{-s})= \bigl[e^{s/2}m^{1-p}(T-e^{-s})+(p-1)y\bigr]^{-\beta}+O(ye^{-s/2}).
 \ee
On the other hand, \eqref{eq:BU-rate-lower}  guarantees that $e^{s/2}m^{1-p}(T-e^{-s})\to 0$ as $s\to\infty$.
Combining this with \eqref{EstSpaceTime1} yields the $C^1_{loc}$ part of \eqref{wcvU}.
The $C_{loc}$ part follows by integrating \eqref{EstSpaceTime1} over $(0,y)$ and using that $w(0,s)=0$ for $s$ large.

\smallskip

(ii) By \eqref{eq:u_x-LBC} we have
$$
 w_y(y,s)=e^{-\beta s/2}u_x(ye^{-s/2},T-e^{-s})= e^{-\beta s/2}U'(ye^{-s/2})+O(ye^{-(\beta+1)s/2})
=U'(y)+O(ye^{-(\beta+1)s/2}),
$$
while \eqref{eq:u-LBC} and \eqref{eq:u-M-LBC} yield
$$\begin{aligned}
 w(y,s)
 &=e^{k s/2}u(ye^{-s/2},T-e^{-s})= e^{ks/2}[U(ye^{-s/2})+O(e^{-s})]+O\bigl(y^2e^{(\frac{k}{2}-1)s}\bigr) \\
 &= U(y)+O\bigl((1+y^2)e^{(\frac{k}{2}-1)s}\bigr).
\end{aligned}$$
The conclusion follows.
\end{proof}

In view of Proposition~\ref{PropwcvU} it will be natural to attempt to linearize \eqref{eqw} 
around $U$ (in appropriate ways that will be described later).
By direct calculation, we find that the equation for $v:=w-U$ is:
		\be{eqz}
		v_s=-\mathcal{L}v+F(v_y),
		\ee
		with linear term
		\be{eqzdefL}
		-\mathcal{L}v=v_{yy}+\Bigl(\frac{\alpha}{y}-\frac{y}{2}\Bigr)v_y+kv,\qquad \alpha=\beta+1=\frac{p}{p-1}\in (1,2),
		\ee
		and nonlinear remainder term
		\be{eqzdefF}
		F(v_y)=F(y,v_y)=|v_y+U_y|^p-U_y^p-pU_y^{p-1}v_y.
		\ee

\subsection{Eigenvalues and eigenfunctions of the linearized operator} \label{subseceigen}

We shall denote the set of nonnegative integers by $\N=\{0,1,2,\dots\}$. 
For given $\alpha>0$ and $1\le q<\infty$, we define the Banach spaces 
$$L^q_\rho=L^q_\rho(0,\infty)=\Bigl\{\varphi\in L^q_{loc}(0,\infty)\,;\ 
\|\varphi\|_{L^q_\rho}^q:=\ts\int_0^\infty \rho |\varphi|^q(y)\,dy<\infty\Bigr\},
\quad\hbox{ where } \rho(y)=y^\alpha e^{-y^2/4},$$
$$W^{1,q}_\rho=W^{1,q}_\rho(0,\infty)=\Bigl\{\varphi\in W^{1,q}_{loc}(0,\infty)\,;\ 
\|\varphi\|_{W^{1,q}_\rho}^q:=\ts\int_0^\infty \rho(|\varphi|^q+|\varphi'|^q)(y)\,dy<\infty\Bigr\}.$$
$L^2_\rho$ and $H^1_\rho:=W^{1,2}_\rho$ are Hilbert spaces,
with respective inner products 
$$(\varphi,\psi)=(\varphi,\psi)_{L^2_\rho}=\int_0^\infty \rho \varphi \psi\,dy,\qquad (\varphi,\psi)_{H^1_\rho}=\int_0^\infty \rho (\varphi \psi+\varphi'\psi')\,dy.$$
We shall also simply denote by $\|\cdot\|$ the $L^2_\rho$ norm.  Let $k\in\R$.
For each $\varphi\in H^1_\rho$, we define 
$$\mathcal{L}\varphi=-\varphi''+\Bigl(\frac{y}{2}-\frac{\alpha}{y}\Bigr)\varphi'-k\varphi=-\rho^{-1}(\rho \varphi')'-k\varphi$$
as the element of the dual $(H^1_\rho)'$, given by
$$\langle \mathcal{L}\varphi,\psi\rangle:=(\varphi',\psi')-k(\varphi,\psi)=\int_0^\infty \rho(\varphi'\psi'-k\varphi\psi)\, dy,
\quad\hbox{ for all }  \psi\in H^1_\rho.$$
We then consider $\mathcal{L}$ as an unbounded operator on $L^2_\rho$ with domain $D(\mathcal{L})=\{\varphi\in H^1_\rho: \mathcal{L}\varphi\in L^2_\rho\}$.
We note that for the viscous Hamilton-Jacobi equation in similarity variables on the half-line,
the linearized operator around the singular steady state (see Section~\ref{secsimil}) is given by $\mathcal{L}$ with
$\alpha=p/(p-1)$, where $p>2$, hence $\alpha\in (1,2)$, and $k=\frac{1-\beta}{2}$.

\smallskip
For $\Lambda\in\R$, we say that $\varphi\in H^1_\rho$ is an eigenfunction of $\mathcal{L}$ with eigenvalue $\Lambda$ if
 \be{AGaussian2a} 
(\varphi',\psi')-k(\varphi,\psi)
=\Lambda(\varphi,\psi)\quad\hbox{ for all $\psi\in H^1_\rho$.}
\ee
By standard regularity properties, any eigenfunction belongs to $C^\infty(0,\infty)$ and satisfies 
\be{AGaussian4} 
-\varphi''+\Bigl(\frac{y}{2}-\frac{\alpha}{y}\Bigr)\varphi'-k\varphi=\Lambda \varphi,\quad y>0.
\ee
Conversely, if $\varphi\in C^2(0,\infty)$ is a solution of \eqref{AGaussian4} and 
belongs to $H^1_\rho$, then it is not difficult to check that it is an eigenfunction.
We have the following spectral result concerning the operator~$\mathcal{L}$.

\begin{prop}\label{GaussianPoincare2} Let $\alpha\ge 1$, $k\in\R$.
\vskip 2pt
\noindent
(i) There exists a Hilbert basis of $L^2_\rho$ made of eigenfunctions of $\mathcal{L}$.
\vskip 2pt
\noindent
(ii) The eigenvalues of $\mathcal{L}$ are given by $\lambda_j=j-k$, $j\in \N$.
\vskip 2pt
\noindent
(iii) For each $j\in\N$, the eigenspace $E_j=\hbox{Ker}\,(\mathcal{L}-\lambda_j I)$ is of dimension one.
It is of the form $E_j=\hbox{Span}(\varphi_j)$, where $\varphi_j$ is an even polynomial of degree $2j$.
Moreover, we have $\varphi_j(0)\ne 0$ and we normalize $\varphi_j$ by $\|\varphi_j\|=1$ 
and $\varphi_j(0)>0$. Furthermore, the sign of the leading coefficient of $\varphi_j$ is~$(-1)^j$.
\vskip 2pt
\noindent
(iv) For each $j\in \N^*$, $\varphi_j$ has exactly $j$ positive zeros, and they are all simple.
\end{prop}

\begin{rem} \label{RemSpec}
We shall see in the proof that the coefficients of $\varphi_j(y)=\sum_{i=0}^j b_{j,i} \, y^{2i}$ satisfy the recursion relation
 \be{Pn2} 
b_{j,i}=-\frac{2(i+1)(2i+1+\alpha)}{j-i}b_{j,i+1},\quad 0\le i\le j-1. 
 \ee
\end{rem}

We also have the following useful pointwise estimates for the eigenfunctions.

\begin{prop}\label{EstimEigen} Assume $\alpha\in (1,3)$ and let $\varphi_j$ be given by Proposition~\ref{GaussianPoincare2}.
Then we have
\be{EstimEigen1}
|\varphi_j(y)| \le C(j+1)^{3/2} e^{y^2/8}\quad\hbox{ and }\quad
|\varphi_j'(y)| \le C(j+1)^{5/2}  e^{y^2/8} y,\qquad y\ge 0,\ j\in\N. 
\ee
\end{prop}

The proof of Proposition~\ref{GaussianPoincare2} relies on a series of lemmas.

\begin{lem}\label{LemImbedd}
Let $\alpha>0$ and $a\ge 1$. We have
\be{selfadj2int}
\ds\sup_{y\in(0,a)}\frac{|\psi(y)|}\zeta(y)\le C(a)\|\psi\|_{H^1_\rho},
\quad\hbox{ where } \zeta(y):=
\begin{cases}
y^{\frac{1-\alpha}{2}},&\hbox{ if $\alpha>1$} \\
 \noalign{\vskip 1mm}
(1+|\log y|)^{1/2},&\hbox{ if $\alpha=1$,}
\end{cases}
\ee
for all $\psi\in H^1_\rho$, and
\be{Imbedd-q} 
 W^{1,q}_\rho\hookrightarrow L^\infty_{loc}([0,\infty)),\quad\hbox{ if $\alpha+1<q<\infty$}.
\ee
\end{lem}

\begin{proof}
By H\"older's inequality, for all $y, z$ with $0<y\le a\le z\le a+1$, we have
$$
|\psi(y)|\le |\psi(z)|+\int_y^{a+1} |\psi'|z^{\frac{\alpha}{q}}z^{-\frac{\alpha}{q}}\,dz
\le |\psi(z)|+C(a)
\begin{cases}
\zeta(y)\|\psi'\|,&\hbox{if $\alpha\ge 1$ and $q=2$},\\
\noalign{\vskip 1.5mm}
 \|\psi'\|_{L^q_\rho},&\hbox{ if $\alpha<q-1$.}
\end{cases}
$$
Integrating with respect to $z\in(a,a+1)$, it follows from the first case that, for all $0<y\le a$,
$$|\psi(y)|\le \int_a^{a+1}|\psi(z)|\,dz+C(a)\zeta(y) \|\psi'\|  \le C(a)\|\psi\|+C(a)\zeta(y)\|\psi'\| 
\le C(a)\zeta(y) \|\psi\|_{H^1_\rho},$$
hence \eqref{selfadj2int},
 whereas the second case yields \eqref{Imbedd-q}.
\end{proof}

\begin{lem}\label{GaussianPoincare1}
We have
 \be{AGaussian3} 
 \int_0^\infty y^2\varphi^2\rho\, dy\le  16\|\varphi'\|^2+ 4(\alpha+1)\|\varphi\|^2,
\quad\hbox{ for all }  \varphi\in H^1_\rho.
\ee
\end{lem}

\begin{proof}[Proof of Lemma~\ref{GaussianPoincare1}]
Let $0<\eps<R<\infty$. Using the identity 
$(\varphi^2y)\rho'=(\varphi^2y\rho)'-\rho(\varphi^2y)'=(\varphi^2y\rho)'-\rho \varphi^2-2\rho \varphi y\varphi'$
along with $2y\rho'=2\alpha\rho-y^2\rho$, and integrating on $(\eps,R)$, we have
\bas
\int_\eps^R \varphi^2y^2\rho
&=&2\alpha\int_\eps^R \varphi^2\rho-2\int_\eps^R (\varphi^2y)\rho'
=2(\alpha+1)\int_\eps^R\rho \varphi^2+4\int_\eps^R\rho \varphi y\varphi' -2\bigl[\rho y\varphi^2\bigr]_\eps^R \\
&\le&2(\alpha+1)\int_\eps^R\rho\varphi^2+\frac12\int_\eps^R \varphi^2y^2\rho
+8\int_\eps^R{\varphi'}^2\rho -2\bigl[\rho y\varphi^2\bigr]_\eps^R.
\eas
Since $\rho \varphi^2\in L^1(0,\infty)$ there exists sequences $\eps_j\to 0$ and $R_j\to \infty$, such that 
 $[\rho y\varphi^2](\eps_j)\to 0$ and $[\rho y\varphi^2](R_j)\to 0$.
Taking $\eps=\eps_j$ and $R=R_j$ and letting $j\to\infty$, we obtain~\eqref{AGaussian3}.
\end{proof}

\begin{lem}\label{GaussianPoincare1b}
The imbedding $H^1_\rho\subset L^2_\rho$ is compact.
\end{lem}

\begin{proof}
Let $(f_j)$ be a bounded sequence in $H^1_\rho$.
There exists a subsequence, still denoted $(f_j)$, and $f\in H^1_\rho$
such that $f_j\to f$ weakly in $H^1_\rho$.
By Rellich's theorem, we may assume that $f_j\to f$ strongly in $L^2_{loc}(0,\infty)$ and a.e.~on $(0,\infty)$. 
For each $R>0$, using \eqref{AGaussian3}, we write
\bas
\|f_j-f\|^2
&=& \int_{0<y<R} |f_j-f|^2\rho\,dy+\int_{y>R} |f_j-f|^2\rho\,dy\\
&\le& \int_{0<y<R} |f_j-f|^2\rho\,dy+R^{-2}\int_{y>R} |y|^2|f_j-f|^2\rho\,dy\\
&\le& \int_{0<y<R} |f_j-f|^2\rho\,dy+CR^{-2}\bigl(\|f_j\|^2_{H^1_\rho}+\|f\|^2_{H^1_\rho}\bigr) 
\ \le \int_{0<y<R} |f_j-f|^2\rho\,dy+CR^{-2}.
\eas
Fix $\eps>0$. Choosing $R=R_0(\eps)>0$ large enough, we have
$\|f_j-f\|^2 \le \int_{0<y<R_0(\eps)} |f_j-f|^2\rho\,dy+\eps$ 
for all $j$.
 Moreover, as a consequence of \eqref{selfadj2int}, 
the sequence $\{|f_j-f|^2\rho\}_{j\ge 1}$ is bounded in $L^\infty_{loc}([0,\infty))$.
By dominated convergence, we then have $\|f_j-f\|^2 \le 2\eps$ for all large $j$.
Therefore $f_j\to f$ strongly in $L^2_\rho$ and the lemma  is proved.
\end{proof}

 Our last lemma gives the natural integration by parts formula
(this is where the restriction $\alpha\ge 1$ enters).

\begin{lem}\label{GaussianPoincare3}
Let $\alpha\ge 1$.  For all $\varphi\in D(\mathcal{L})$ and $\psi\in H^1_\rho$, we have
\be{selfadj}
\int_0^\infty (\mathcal{L}\varphi)\psi\rho\,dy=\int_0^\infty  \rho(\varphi'\psi'-k\varphi\psi)\, dy. 
\ee
It follows in particular that $\mathcal{L}$ is symmetric.
\end{lem}

\begin{proof}
 Fix $0<\eps<1<R$. We have
\be{selfadj2}
\int_\eps^R [(\mathcal{L}\varphi)\psi-\rho(\varphi'\psi'-k\varphi\psi)]\rho\,dy
=\int_\eps^R [-(\rho \varphi')'\psi-\rho \varphi'\psi']\,dy
= \bigl[\rho \varphi'\psi\bigr]_\eps^R.
\ee
On the other hand, since $\rho{\varphi'}^2, \rho\psi^2\in L^1(0,\infty)$, 
there exist sequences $\eps_j\to 0^+$ and $R_j\to\infty$ such that
$\eps_j[\rho{\varphi'}^2](\eps_j)\to 0$ and $[\rho({\varphi'}^2+\psi^2)](R_j)\to 0$ as $j\to\infty$. 
Therefore, $\eps_j^{(1+\alpha)/2} \varphi'(\eps_j)\to 0$ and $[\rho\varphi'\psi](R_j)\to 0$, as $j\to\infty$.
By \eqref{selfadj2int}, we deduce that
$$
|\rho \varphi'\psi(\eps_j)|=
\begin{cases}
o\bigl(\eps_j^\alpha\eps_j^{-(1+\alpha)/2}\eps_j^{(1-\alpha)/2}\bigr)=o(1),&\hbox{if $\alpha>1$},\\
\noalign{\vskip 1.5mm}
 o\bigl(\eps_j\eps_j^{-1}|\log\eps_j|^{-1/2}\bigr)=o\bigl(\log\eps_j|^{-1/2}\bigr),&\hbox{if $\alpha=1$}.
\end{cases}
$$
Upon taking $\eps=\eps_j$ and $R=R_j$ in \eqref{selfadj2} and letting $j\to\infty$, this yields \eqref{selfadj}.
\end{proof}

\begin{proof}[Proof of Proposition~\ref{GaussianPoincare2}] We may assume $k=0$ without loss of generality.

\smallskip

(i) It follows from the Lax-Milgram or the Riesz representation theorem that, for all $f\in L^2_\rho$, 
there exists a unique solution $u\in H^1_\rho$ of $\mathcal{L} u+u =f$.
Indeed, this equation is equivalent to 
 \be{AGaussian2} 
 (u, \varphi)_{H^1_\rho}\equiv (u',\varphi')+(u,\varphi)=(f,\varphi)\quad\hbox{ for all $\varphi\in H^1_\rho$.}
\ee
Let $T$ be the solution operator $T:L^2_\rho \to L^2_\rho$, $f\mapsto  u$.
Taking $\varphi=u$ in \eqref{AGaussian2} and using the Cauchy-Schwarz inequality, we obtain $\|Tf\|_{H^1_\rho}\le \|f\|$,
hence $T$ is continuous.
Furthermore, it follows from Lemma~\ref{GaussianPoincare1b} that $T$ is compact.

Since $T$ is self-adjoint owing to Lemma~\ref{GaussianPoincare3}, it then follows from the spectral theorem that there exists a Hilbert basis of $L^2_\rho$ made of eigenfunctions of $T$
and this immediately provides the desired result for $\mathcal{L}$.

\smallskip

(ii)(iii) For $\Lambda\in\R$, we look for a solution of $\mathcal{L}\ph=\Lambda\ph$ under (normalized) polynomial form
$\ph(y)=\sum_{i=0}^m a_i y^i$, $a_m=1,$
with $m\in \N$. Note that such a $\ph$ belongs to $H^1_\rho$. We compute
$$\begin{aligned}
(\Lambda\ph-\mathcal{L}\ph)(y)
&=\sum_{i=0}^m i(i-1)a_i y^{i-2}
+\bigl(\ts\frac{\alpha}{y}-\frac{y}{2}\bigr) \ds\sum_{i=0}^m ia_i y^{i-1}
+\Lambda\sum_{i=0}^m a_i y^i \\
&=\sum_{i=1}^m i(i-1+\alpha)a_i y^{i-2}
+\sum_{i=0}^m \bigl(\Lambda-\ts\frac{i}{2}\bigr)a_i y^i \\
&=\sum_{i=-1}^{m-2} (i+2)(i+1+\alpha)a_{i+2} y^i
+\sum_{i=0}^m \bigl(\Lambda-\ts\frac{i}{2}\bigr)a_i y^i \\
&=\alpha a_1 y^{-1}+\sum_{i=m-1}^m \bigl(\Lambda-\ts\frac{i}{2}\bigr)a_i y^i
+\ds\sum_{i=0}^{m-2} \bigl[(i+2)(i+1+\alpha)a_{i+2}+\bigl(\Lambda-\ts\frac{i}{2}\bigr)a_i\bigr] y^i.
\end{aligned}
$$
The conditions for $\ph$ to be a solution of $\mathcal{L}\ph=\Lambda\ph$ are thus
\be{eq1}
a_1=0, 
\qquad\Lambda-\frac{m}{2}=0,
\qquad \bigl(\Lambda-\ts\frac{m-1}{2}\bigr)a_{m-1} =0
\ee
and
\be{eq2}
(i+2)(i+1+\alpha)a_{i+2}+\bigl(\Lambda-\ts\frac{i}{2}\bigr) a_i,\quad 0\le i\le m-2.
\ee
Condition \eqref{eq1} amounts to
$\Lambda=\frac{m}{2}$, $a_1=a_{m-1}=0$,
which implies in particular $m\ne 1$.
Condition \eqref{eq2} then implies that $a_i=0$ for all odd $i$, hence in particular $m$ must be even.
Now, for each even $m\in\N$, rewriting \eqref{eq2} as
$$a_i=-\frac{2(i+2)(i+1+\alpha)}{m-i}a_{i+2},\quad 0\le i\le m-2,$$
and starting from $i=m-2$ and $a_m=1$, this (uniquely) determines $a_{m-2},\dots,a_0$.
Setting $m=2j$, for each $j\in\N$, we have thus found an eigenfunction associated with the eigenvalue 
$\Lambda_j=j$. It is of the form $\varphi(y)=\sum_{i=0}^j b_{j,i} \, y^{2i}$, 
and \eqref{eq2} yields \eqref{Pn2}, so that in particular $b_{j,i}\ne 0$ for all $i=0,\dots,j$.
In particular, we may uniquely normalize $\varphi$ by $\|\varphi_j\|=1$ and $\varphi_j(0)>0$.
Moreover, the sign of $b_{j,j}$ is~$(-1)^j$, in view of \eqref{Pn2}.

Let us check that the family $\{P_j,\ j\in \N\}$ is total.
This will guarantee that there can be no eigenfunction which is linearly independent of the $P_j$,
hence no other eigenvalue, and will conclude the proof.
Thus assume that $\varphi\in L^2_\rho(0,\infty)$ is orthogonal to all $P_j$.
We want to show that $\varphi\equiv 0$.
For each $j\in\N$, since $y^{2j}$ can be expressed as a linear combination of $P_0,\dots,P_j$, 
it follows that $\int_0^\infty \rho(y)\varphi(y)y^{2j}\,dy=0$.
Now setting $\tilde\rho(x)=|x|^\alpha e^{-x^2/4}$ and $\tilde \varphi(x)=\varphi(|x|)$ for all $x\in \R$,
we have $\int_{\R} \tilde\rho(x)\tilde \varphi(x)x^m\,dx=0$ for all $m\in\N$
(noting that the integrand is an odd function when $m$ is odd).
Since the function $g(x)=|x|^{\alpha/2} e^{-x^2/8}\tilde \varphi(x)$ belongs to $L^2(\R)$ by assumption,
so does the function $h(x)=|x|^\alpha e^{-x^2/6}\tilde \varphi(x)$, and $h$ satisfies
$\int_{\R} h(x)x^m e^{-x^2/12}\,dx=\int_{\R} \tilde\rho(x)\tilde \varphi(x)x^m\,dx=0$ for all $m\in\N$.
As a consequence of \cite[Theorem~5.7.1]{Sz}, we deduce that $h\equiv 0$, hence $\varphi\equiv 0$.

\smallskip

(iv) This is an immediate consequence of general properties of orthogonal polynomials
(cf.~\cite[Theorem~3.3.1]{Sz}) and of the fact that $\varphi_j$ is even and $\varphi_j(0)\ne 0$.
\end{proof}

\begin{proof}[Proof of Proposition~\ref{EstimEigen}]
Since $\varphi_0$ is a constant, it is obviously sufficient to show the result for $j\ge 1$.
Set $\varphi=\varphi_j$.
By \eqref{AGaussian2a} with $\psi=\varphi$, we have $\|\varphi'\|^2=j$. 
For $R\ge 2$ and $0<\eps\le 1$, by Sobolev imbedding, we deduce that
\be{EstimEigen2}
\sup_{[\eps,R]} \varphi^2 \le 2\int_\eps^R [\varphi^2+{\varphi'}^2]\,dy\le 
2e^{R^2/4}\eps^{-\alpha} \int_0^\infty [\varphi^2+{\varphi'}^2]\rho\,dy\le Cje^{R^2/4}\eps^{-\alpha}.
\ee
Since $\varphi(0)>0$, $\varphi'(0)=0$ and 
$\varphi''(0)=-\frac{j}{\alpha+1}\varphi(0)<0$, we have $\varphi'(y)<0$ for $y>0$ small.
Let $y_0=\sqrt{2\alpha}$ and $y_*=\sup\{y\in (0,y_0);\, \varphi'<0 \hbox{ in } (0,y)\}$.
By \eqref{selfadj2int}, we have 
\be{EstimEigen3}
|\varphi(y)|\le Cj^{1/2}y^{(1-\alpha)/2},\quad 0<y\le y_0,
\ee
hence
$$\varphi''=-j \varphi+\Bigl(\frac{y}{2}-\frac{\alpha}{y}\Bigr)\varphi'\ge- j \varphi\ge -Cj^{3/2}y^{(1-\alpha)/2},\quad 0<y\le y_*.$$
Since $\alpha<3$, an integration gives $0\ge\varphi'\ge -Cj^{3/2}y^{(3-\alpha)/2}$ on $(0,y_*]$
and then
$$\varphi(0) \le \varphi(y_*)+Cj^{3/2}y_*^{(5-\alpha)/2}\le \varphi(y_*)+Cj^{3/2}.$$
If $y_*<y_0$, then $\varphi'(y_*)=0$ and $\varphi''(y_*)\ge 0$, hence $\varphi(y_*)\le 0$, 
whereas if $y_*=y_0$, then $\varphi(y_*)\le Cj^{1/2}$ by \eqref{EstimEigen3}.
In both cases, we get $\varphi(0)\le Cj^{3/2}$. Since $[\varphi^2+j^{-1}{\varphi'}^2]'=j^{-1}(\frac{y}{2}-\frac{\alpha}{y}){\varphi'}^2\le 0$
on $(0,y_0]$ and $\varphi'(0)=0$, we thus deduce that
$\sup_{[0,y_0]} |\varphi| = \varphi(0)\le Cj^{3/2}$. This combined with \eqref{EstimEigen2} guarantees the first part of \eqref{EstimEigen1}.
Going back to \eqref{AGaussian4}, we then obtain
$$ y^\alpha e^{-y^2/4}|\varphi'(y)|=j\Bigl|\int_0^y \varphi(z) z^\alpha e^{-z^2/4}dz\Bigr|
\le C j^{5/2} y^{\alpha+1},\quad 0<y<1$$
and, using $\lim_{y\to\infty}y^\alpha e^{-y^2/4}\varphi'(y)=0$  (since $\varphi$ is a polynomial),
$$y^\alpha e^{-y^2/4}|\varphi'(y)|=j\Bigl|\int_y^\infty \varphi(z) z^\alpha e^{-z^2/4}dz\Bigr|
\le C j^{5/2} \int_y^\infty z^\alpha e^{-z^2/8}dz\le Cj^{5/2}  y^{\alpha-1} e^{-y^2/8},\quad  y\ge 1.$$
 Combining the last two inequalities, we get the second part of \eqref{EstimEigen1}.
\end{proof}

\subsection{Semigroup properties for the linearized operator}

In view of the construction of appropriate solutions of the semilinear equation \eqref{eqz} satisfied by $v=w-U$,
we shall need good semigroup properties
for the inhomogeneous problem:
\be{eqinhomog}
\left\{\ 
\begin{aligned}
v_s&=-\mathcal{L}v+f(y,s),&&\quad y>0,\ s_0<s<s_1, \\
v(y,s_0) &=v_0(y),&&\quad  y>0.
\end{aligned}
\right.
\ee
Due to the expected boundary singularities for the solution of \eqref{eqz},
 the corresponding data $f(s)$ will not belong to the basic space $L^2_\rho$
 and it turns out that a good working space for our purposes is provided by $H'$,
 the topological dual space of $H:=H^1_\rho(0,\infty)$.
The first order of matters is thus to collect the relevant properties of the semigroup 
$(e^{-s\mathcal{L}})_{s\ge 0}$ on $H'$.

Assume $\alpha\ge 1$, $k\in\R$ and let $(\varphi_j)_{j\in\N}$ be the Hilbert basis of $L^2_\rho$ made of eigenfunctions of $\mathcal{L}$, obtained in 
Proposition~\ref{GaussianPoincare2},
 and $\lambda_j=j-k$ the corresponding eigenvalues. 
Firstly, the semigroup $(e^{-s\mathcal{L}})_{s\ge 0}$ 
is defined on $L^2_\rho$ in the standard way by 
\be{defSG1}
e^{-s\mathcal{L}}\varphi
=\sum_{j=0}^\infty e^{-\lambda_j s} (\varphi,\varphi_j)\,\varphi_j,\quad \varphi\in L^2_\rho.
\ee
Denoting by $\langle \cdot,\cdot \rangle$ the duality pairing between $H'$ and $H$,
the semigroup $(e^{-s\mathcal{L}})_{s\ge 0}$ is then extended to $H'$ by setting
\be{defSG2}
e^{-s\mathcal{L}}\phi=\sum_{j=0}^\infty e^{-\lambda_j s}\langle \phi,\varphi_j\rangle\,\varphi_j,\quad\phi\in H'
\ee
(note that this is of course consistent since, with the usual abuse of notation, the element of $\phi\in H'$ associated with 
a given $\varphi\in L^2_\rho$ is given by $\langle \phi,\psi\rangle=(\varphi,\psi)$, $\psi\in H$).
The properties of $(e^{-s\mathcal{L}})_{s\ge 0}$ on $H'$ are summarized in the following.

\begin{prop} \label{propSGHprime}
Let $\alpha\ge 1$ and $k\in\R$ and let $e^{-s\mathcal{L}}$ be defined by \eqref{defSG2}.

\smallskip

(i) $(e^{-s\mathcal{L}})_{s\ge 0}$ is a strongly continuous semigroup on $H'$ with $\|e^{-s\mathcal{L}}\|_{L(H')}\le e^{ks}$.

\smallskip

(ii) Let $\phi\in H'$ and set $W(y,s):=[e^{-s\mathcal{L}}\phi](y)$. Then 
\be{regSG1}
W \in C([0,\infty);H')\cap C^1((0,\infty);D(\mathcal{L}))
\ee 
and $W$ is a solution of $W_s+\mathcal{L}W=0$ 
with $W(0)=\phi$.
If moreover $\phi\in X$ with $X=L^2_\rho$ or $X=H$, then $W \in C([0,\infty); H)$.
\end{prop}

Although the result more or less follows from general semigroup theory (see, e.g.,~\cite{Hen81}),
we give a short proof for convenience and self-containedness.

\begin{proof}
Set $\hat\varphi_j=(1+j)^{-1/2}\varphi_j$.
By \eqref{AGaussian2a}, we have 
$$(\hat\varphi_i,\hat\varphi_j)_H=(1+i)(\hat\varphi_i,\hat\varphi_j)=
(1+i)^{1/2}(1+j)^{-1/2}(\varphi_j,\varphi_j)=\delta_{ij},$$
 hence $(\hat\varphi_j)$ is a Hilbert basis of~$H$.
Let $\mathcal{T}_{\hat\varphi_j}\in H'$ be defined by $\langle\mathcal{T}_{\hat\varphi_j},\psi\rangle=(\hat\varphi_j,\psi)_H$
for all $\psi\in H$. Then $\langle\mathcal{T}_{\hat\varphi_j},\psi\rangle=(1+j)(\hat\varphi_j,\psi)$, hence
\be{dualitySG1}
\langle \phi,\varphi_j\rangle\varphi_j=\langle \phi,\hat\varphi_j\rangle(1+j)\hat\varphi_j
=\langle \phi,\hat\varphi_j\rangle\mathcal{T}_{\hat\varphi_j}.
\ee
On the other hand, by the Riesz representation theorem, there exists $f\in H$ such that 
$a_j:=\langle \phi,\hat\varphi_j\rangle=(f,\hat\varphi_j)_H$ for all $j\in\N$, and
$\sum_{j=0}^\infty |a_j|^2=\|f\|^2_{H}=\|\phi\|^2_{H'}$. Using \eqref{dualitySG1} and the Cauchy-Schwarz inequality, it follows that
$$\begin{aligned}
\sum_{j=0}^\infty \|e^{-\lambda_j s}\langle \phi,\varphi_j\rangle\,\varphi_j\|_{H'}
&\le e^{ks} \sup_{\psi\in \overline B_H}\sum_{j=0}^\infty |\langle \phi,\varphi_j\rangle\,(\varphi_j,\psi)|
\le e^{ks}\sup_{\psi\in \overline B_H}\sum_{j=0}^\infty |a_j\,(\hat\varphi_j,\psi)_H| \\
&\le e^{ks}\Bigl(\sum_{j=0}^\infty|a_j|^2\Bigr)^{1/2} \sup_{\psi\in \overline B_H}\Bigl(\sum_{j=0}^\infty|(\hat\varphi_j,\psi)_H|^2\Bigr)^{1/2}
\le e^{ks} \|\phi\|_{H'}, 
\end{aligned}$$
hence \eqref{defSG2} converges in $H'$ for each $s\ge 0$, and $e^{-s\mathcal{L}}\in L(H')$
 with $\|e^{-s\mathcal{L}}\|_{L(H')}\le e^{ks}$.
Writing
$\langle \phi,\psi\rangle
=\langle \phi, \sum_{j=0}^\infty(\psi,\hat\varphi_j)_H\,\hat\varphi_j \rangle
=\sum_{j=0}^\infty \langle \phi, \hat\varphi_j \rangle (\psi,\hat\varphi_j)_H
=\sum_{j=0}^\infty \langle \phi, \varphi_j \rangle \,(\psi,\varphi_j)
=\langle \sum_{j=0}^\infty \langle \phi, \varphi_j \rangle\,\varphi_j,\psi \rangle$
for all $\psi\in H$, we then get $e^{-s\mathcal{L}}=I$ for $s=0$,
and the semigroup property $e^{-s\mathcal{L}}e^{-t\mathcal{L}}=e^{-(s+t)\mathcal{L}}$
follows immediately from \eqref{defSG2} and $(\varphi_j,\varphi_j)=\delta_{ij}$.

\smallskip

Next, for all $s,t\ge 0$, we have
$$\begin{aligned}
\|e^{-s\mathcal{L}}\phi-e^{-t\mathcal{L}}\phi\|_{H'}
&\le\sum_{j=0}^\infty \|(e^{-\lambda_j s}-e^{-\lambda_j t})\langle \phi,\varphi_j\rangle\,\varphi_j\|_{H'}
\le\sup_{\psi\in \overline B_H}\sum_{j=0}^\infty 
|e^{-\lambda_j s}-e^{-\lambda_j t}|\,|a_j\,(\hat\varphi_j,\psi)_H| \\
&\le \Bigl(\sum_{j=0}^\infty|e^{-\lambda_j s}-e^{-\lambda_j t}|^2|a_j|^2\Bigr)^{1/2}\to 0,\quad 
\hbox{ as } t\to s,
\end{aligned}$$
which yields the strong continuity of the semigroup, i.e., the first part of \eqref{regSG1}. 
The remaining properties follow from standard computations
justified by the decay of the exponential factors $e^{-\lambda_js}\le Ce^{-js}$.
\end{proof}

We have the following variation of constants formula for problem \eqref{eqinhomog} with data in $H'$.

\begin{prop} \label{varconstf}
Let $\alpha\ge 1$, $k\in\R$.

\smallskip

(i) Let $v_0\in H'$,  $f\in C_b((s_0,s_1];H')$ and assume that $v\in C([s_0,s_1];H')$ $\cap\ C^1((s_0,s_1];H')\cap C((s_0,s_1];H)$ is a solution of \eqref{eqinhomog}.
Then $v$ is given by 
\be{varconstv}
v(s)=e^{-(s-s_0)\mathcal{L}}v_0+\int_{s_0}^s e^{-(s-\tau)\mathcal{L}}f(\tau)\,d\tau, \quad s_0<s<s_1,
\ee
where the integral is valued in $H'$.

\smallskip

(ii) Let $v_0\in D(\mathcal{L})$,  $f\in C([s_0,s_1];D(\mathcal{L}))$ and let $v$ be given by \eqref{varconstv}.
Then $v\in C([s_0,s_1];D(\mathcal{L}))\cap\ C^1([s_0,s_1];L^2_\rho)$ 
and $v$ is a solution of \eqref{eqinhomog}.
\end{prop}

\begin{proof} (i) Assume $s_0=0$ without loss of generality and fix $s\in(0,s_1)$.
Let $z(\tau)=e^{-(s-\tau)\mathcal{L}}v(\tau)$. 
Then it is not difficult to show that $z\in C([0,s];H')\cap C^1((0, s];H)$ and that
$$\frac{dz}{d\tau}=e^{-(s-\tau)\mathcal{L}}\mathcal{L}v(\tau)+e^{-(s-\tau)\mathcal{L}}\frac{dv}{d\tau}
=e^{-(s-\tau)\mathcal{L}}f(\tau),\quad 0<\tau<s$$
(see, e.g., the proof of Lemma~4.1.1 in \cite{CHbook} for details). 
Integrating for $\tau\in(\eps,s-\eps)$ with $\eps>0$, we get
$e^{-\eps\mathcal{L}}v(s-\eps)-e^{-(s-\eps)\mathcal{L}}v(\eps)=z(s-\eps)-z(\eps)=\int_\eps^{s-\eps}e^{-(s-\tau)\mathcal{L}}f(\tau)\,d\tau$
and the conclusion follows by letting $\eps\to 0$.

\smallskip

(ii) This follows similarly as in the proof of \cite[Proposition~4.1.6]{CHbook}.
\end{proof}

\begin{rem}
(i) We stress that for $\alpha\ge 1$, the semigroup $e^{-s\mathcal{L}}$, as well as the operator $\mathcal{L}$, 
{\it does not require} any boundary conditions at $y=0$.
On the other hand, for any $\phi\in H'$, the function $W(y,s):=[e^{-s\mathcal{L}}\phi](y)$
{\it automatically} satisfies the Neumann boundary conditions $W_y(0,s)=0$ for all $s>0$
(this a consequence of the fact that $\varphi_{j,y}(0)=0$).

\smallskip

(ii) As for solutions of the inhomogeneous problem \eqref{varconstv}, they may or may not satisfy $v_y(0,s)=0$,
depending on the regularity of the right-hand side $f$ near $y=0$. 
For instance the solution $v=w-U$ of \eqref{eqz} 
corresponding to $w$ in Theorem~\ref{prop:special-LBC}, whose term
$f=F(v_y)\in L^\infty(s_0,s_1;L^\infty)$ is regular, does satisfy $v_y(0,s)=0$
(this follows from Theorem~\ref{REC-mainThm2}).
On the contrary, the solution $v=w-U$ of \eqref{eqz} corresponding to $w$ in Theorem~\ref{mainThm2}, whose term 
$f=F(v_y)$ is singular (cf.~Lemma~\ref{LemvyFvy}), does not satisfy $v_y(0,s)=0$,
but actually $v(0,y)=0$.
\end{rem}

The following lemma will allow us to apply Proposition~\ref{varconstf} to our solutions of the nonlinear problem \eqref{eqz},
taking advantage of suitable bounds satisfied by their right-hand side near the boundary, 
which can be expressed via the weighted spaces
$$L^\infty_{m,q}=\Bigl\{\phi\in L^\infty_{loc}(0,\infty);(y^q+y^{-m})^{-1}\phi\in L^\infty(0,\infty)\Bigr\}, \quad
\|\phi\|_{\infty,m,q}=\|(y^q+y^{-m})^{-1}\phi\|_\infty, \quad m,q\ge 0$$
(note that these spaces also allow polynomial growth at infinity, which will be useful in our construction
-- see below).

\begin{lem} \label{CSHprime}
Let $\alpha\ge 1$, $q\ge 0$ and $0\le m<\frac{\alpha+3}{2}$.
Then $L^\infty_{m,q}\hookrightarrow H'$ and, for all $0\le s_0<s_1<\infty$, we have 
\be{embedH}
C([s_0,s_1];L^2_{loc}(0,\infty))\cap L^\infty(s_0,s_1;L^\infty_{m,q}) \subset C([s_0,s_1];H').
\ee
\end{lem}

\begin{proof}
For all $\phi\in L^\infty_{m,q}$ and $\psi\in H$, using \eqref{selfadj2int}
and the Cauchy-Schwarz' inequality, we  obtain
$$\begin{aligned}
|(\phi,\psi)|
&\le \int_0^1 |\phi\psi|y^\alpha\,dy+ \int_1^\infty |\phi\psi|\rho\,dy
\le C\|\phi\|_{\infty,m,q}\Bigl(\int_0^1 |\psi| y^{\alpha-m}\,dy+\int_1^\infty |\psi|y^q\rho\,dy\Bigr)\\
&\le C\|\phi\|_{\infty,m,q}\|\psi\|_H\Bigl\{\int_0^1 y^{\frac{\alpha+1}{2}-m}\,dy+
\Bigl(\int_1^\infty y^{2q}\rho \,dy\Bigr)^{1/2}\Bigr\}\le C\|\phi\|_{\infty,m,q}\|\psi\|_H,
\end{aligned}$$
hence $L^\infty_{m,q}\hookrightarrow H'$.
Let $f\in  C([s_0,s_1];L^2_{loc}(0,\infty))\cap L^\infty(s_0,s_1;L^\infty_{m,q})$. 
We have $|f(y,s)|\le C(y^q+y^{-m})$ for some $C>0$.
For all $s,t \in[s_0,s_1]$, $\psi\in H$ and $0<\eta<1<R$, using \eqref{selfadj2int}, we get
$$\begin{aligned}
|(f(t)-f(s),\psi)|
&\le C \int_0^\eta |\psi|y^{\alpha-m}\,dy+ \int_\eta^R |f(t)-f(s)| |\psi| \rho\, dy+ C\int_R^\infty |\psi|y^q\rho\,dy \\
&\le C\|\psi\|_H \int_0^\eta y^{\frac{\alpha+1}{2}-m}\,dy+ C\|\psi\| \|f(\cdot,s)-f(\cdot,t)\|_{L^2(\eta,R)}
+ C\|\psi\|\Bigl(\int_R^\infty y^{2q}\rho\,dy\Bigr)^{\frac12},
\end{aligned}$$
hence
$$\|f(t)-f(s)\|_{H'}\le C\eta^{\frac{\alpha+3}{2}-m}+ C \|f(\cdot,s)-f(\cdot,t)\|_{L^2(\eta,R)}
+\Bigl(\int_R^\infty y^{2q}\rho\,dy\Bigr)^{\frac12}.$$
For any given $\eps>0$, we may choose $\eta\in(0,1)$ and $R>0$ such that
$$\|f(t)-f(s)\|_{H'}\le \eps+ C\|f(\cdot,s)-f(\cdot,t)\|_{L^2(\eta,R)}.$$
Next using the 
assumption $f\in C([s_0,s_1];L^2_{loc}(0,\infty))$, it follows that there exists $\nu>0$ such that $|t-s|\le \nu
\Longrightarrow \|f(t)-f(s)\|_{H'}\le 2\eps$. This proves \eqref{embedH}.
\end{proof}

We end this subsection with a local well-posedness and comparison result on problem \eqref{equ}
for initial data with (at most linear) growth at infinity.
This will be useful for the construction of special GBU solutions which have an odd number of intersections
with the singular steady state on $(0,\infty)$ (such solutions must obviously grow at space infinity).
To this end we define the space
\be{defW1}
\mathcal{W}_1=\bigl\{\psi\in W_{loc}^{1,\infty}([0,\infty));\ \psi(0)=0,\ \psi_x\in L^\infty(0,\infty)\bigr\},
\ee
equipped with the norm $\|\psi\|_{\mathcal{W}_1} =\|\psi_x\|_\infty$.

\begin{prop} \label{RemWPw}

\smallskip

 (i) Let $R=\infty$ and $u_0\in \mathcal{W}_1$.  
There exists a unique, maximal classical solution 
$u\in C^{2,1}([0,\infty)\times(0,T^*))\cap C([0,\infty)\times[0,T^*)) \cap L^\infty_{loc}([0,T^*); \mathcal{W}_1)$
of problem \eqref{equ}.
If its maximal existence time $T^*$ is finite, then $\lim_{t\to T^*_-}\|u_x(t)\|_\infty=\infty$.
Moreover the solution operator $u_0\mapsto u(\cdot,t)$ is continuous on $\mathcal{W}_1$.

\smallskip

(ii) Let $x_0\ge 0$, $T>0$ and $D_T=(x_0,\infty)\times(0,T)$. Let $\underline v, \overline v\in C^{2,1}(D_T)\cap C([x_0,\infty)\times[0,T))$ satisfy 
$\underline v_x, \overline v_x\in L^\infty(D_T)$ and 
$$\underline v_t-\underline v_{xx}-|\underline v_x|^p\le\overline v_t-\overline v_{xx}-|\overline v_x|^p
\quad\hbox{ in $D_T$.}$$
If $\underline v\le\overline v$ on $([x_0,\infty)\times\{0\})\cup (\{x_0\}\times(0,T))$, then $\underline v\le\overline v$ in $D_T$.
\end{prop} 

\begin{proof}
(i) This follows from a standard fixed point argument (see e.g.~\cite{CHbook, QSbook})
in the space $L^\infty(0,\tau;\mathcal{W}_1)$,
using the heat semigroup estimates 
$\|\partial_xe^{t\Delta}\phi\|_\infty\le \|\phi_x\|_\infty$ for all $\phi\in \mathcal{W}_1$
and $\|\partial_xe^{t\Delta}\phi\|_\infty\le Ct^{-1/2}\|\phi\|_\infty$ for all $\phi\in L^\infty(0,\infty)$.

\smallskip

 (ii) By our assumptions, the function $z=\underline v-\overline v$ satisfies 
$z_t-z_{xx}\le  |\underline v_x|^p- |\overline v_x|^p\le M|z_x|$ and $z\le M(1+x)$ in $D_T$ for some $M>0$.
Set $\psi(x,t)=e^{Kt}(1+x^2)$ with $K=M+2$ and, for any $\eps>0$, let $z_\eps:=z-\eps\psi$.
We have
$$\psi_t-\psi_{xx}-M|\psi_x|=Ke^{Kt}(1+x^2)-2e^{Kt}-2Mxe^{Kt}=e^{Kt}[(M+2)(1+x^2)-2-2Mx]\ge 0,$$
hence $z_{\eps,_t}-z_{\eps,xx}-M|z_{\eps,x}|\le 0$ in $D_T$.
Moreover, we may select $R_\eps>x_0$ such that $z_\eps\le M(1+x)-\eps (1+x^2)\le 0$ in $[R_\eps,\infty)\times(0,T)$,
and we have $z_\eps\le 0$ on the parabolic boundary of $D_{\eps,T}=(x_0,R_\eps)\times(0,T)$.
We then deduce from the standard maximum principle that $z_\eps\le 0$ in $D_{\eps,T}$, hence in $D_T$,
and the conclusion follows by letting $\eps\to 0$.
\end{proof}

\subsection{Heat kernel of the linearized operator}

Let $\alpha\ge 1$, $k\in\R$.
In this section we obtain a convenient explicit formula, along with useful estimates, 
for the kernel associated with the semigroup $(e^{-s\mathcal{L}})_{s\ge 0}$.
To this end we introduce the auxiliary problem
\be{eqZtx} 
\left\{\ 
\begin{aligned}
Z_t  &=Z_{xx}+\ts\frac{\alpha}{x}Z_x,&&\quad x>0,\ t>0 \\
Z(x,0) &=\phi(x),&&\quad  x>0,
\end{aligned}
\right.
\ee
and the solution $Y=Y(z)$ of the ODE
\be{eqV-ODE} 
\left\{\ 
\begin{aligned}
&Y''+\ts\frac{\alpha}{z}Y'=Y,\quad z>0 \\
&Y(0)  =1,\quad  Y'(0)=0.
\end{aligned}
\right.
\ee
We first derive the formula and the properties of the kernel associated with problem \eqref{eqZtx}.

\begin{prop} \label{KernelEstim}
Let $\alpha\ge 1$, $k\in\R$
\smallskip

 (i) Set $\mathcal{P}:=\partial_t -\partial_{xx}-\frac{\alpha}{x}\partial_x$ and,
 for $i\in\{0,1\}$, define the kernels
$$H_i(t,x,\xi):=C_\alpha t^{-\frac{\alpha+1}{2}}\exp\Bigl[-\frac{x^2+\xi^2}{4t}\Bigr] Y^{(i)}\Bigl(\frac{x\xi}{2t}\Bigr),\ \quad x,\xi\ge 0,\ t>0,$$
where $C_\alpha>0$ is a normalization constant.
 We have
\be{PDEH} 
\mathcal{P}H_0=0,\ \quad x,\xi\ge 0,\ t>0,
\ee
and the bounds
\be{boundH1} 
0\le H_i(t,x,\xi)\le C t^{-\frac{\alpha+1}{2}}\Bigl(1\wedge \frac{x\xi}{2t}\Bigr)^i\Bigl(1+\frac{x\xi}{2t}\Bigr)^{-\alpha/2}\exp\Bigl[-\frac{(x-\xi)^2}{4t}\Bigr],
\quad i\in\{0,1\},
\ee
\be{boundH1H_0} 
 H_1\le H_0,
\ee
\be{boundH2} 
 \frac{|\partial_xH_0|}{H_0}\le Ct^{-1/2} \min\Bigl\{1,\frac{x}{\sqrt{t}}\Bigr\}\Bigl(1+\frac{|x-\xi|^2}{t}\Bigr),
\ee
\be{boundH2t} 
 \frac{|\partial_tH_0|}{H_0}\le C t^{-1}\Bigl(1+\frac{|x-\xi|^2}{t}\Bigr),
\ee
as well as
\be{normalizeH} 
\int_0^\infty H_0(t,x;\xi)\xi^\alpha d\xi=1,\quad \int_0^\infty H_1(t,x;\xi)\xi^\alpha d\xi\le 1, \qquad t>0,\ x\ge 0.
\ee

 (ii) Let $0\le m<\frac{\alpha+3}{2}$, $q\ge 0$ and let
$\phi\in C(0,\infty)\cap L^\infty_{m,q}$. Let $Z$ be defined by
\be{defZint} 
Z(x,t)=\int_0^\infty H_0(t,x;\xi)\phi(\xi)\xi^\alpha d\xi,\ \quad x\ge 0,\ t>0.
\ee
Then $Z$ enjoys the regularity properties
\be{ZC21} 
Z\in C^{2,1}([0,\infty)\times(0,\infty)),
\ee
\be{ZC21bis} 
\|Z(t)\|_{\infty,m,q}\le C(1+t^{q/2})\|\phi\|_{\infty,m,q},\quad t>0,
\ee
\be{ZC21ter}  
t^{1/2}\|Z_x(t)\|_{\infty,m,q}+t\|Z_t(t)\|_{\infty,m,q}\le C(1+t^{q/2})\|\phi\|_{\infty,m,q},\quad t>0,
\ee
and it satisfies
\be{Zpde} 
Z_t =Z_{xx}+\ts\frac{\alpha}{x}Z_x\ \hbox{ in $(0,\infty)\times(0,\infty)$},
\ee
\be{cvZphi} 
\lim_{t\to 0} Z(x,t)=\phi(x),\quad\hbox{for each $x>0$}.
\ee
Moreover, \eqref{cvZphi} remains valid for $x=0$ whenever $\phi\in C_b([0,\infty))$.
Furthermore, if also $\phi'\in L^\infty_{m+1,r}$ for some $r\ge 0$, then 
\be{computZx} 
Z_x(x,t)=\int_0^\infty H_1(t,x;\xi)\phi'(\xi)\xi^\alpha d\xi,\quad x>0,\ t>0.
\ee
\end{prop}

As a consequence we obtain the kernel associated with the semigroup $(e^{-s\mathcal{L}})_{s\ge 0}$.

\begin{prop} \label{KernelEstim2}
Assume $\alpha\ge 1$, $k\in\R$, $0\le m<\frac{\alpha+3}{2}$, $q\ge 0$ and $\phi\in C(0,\infty)\cap L^\infty_{m,q}\subset H'$. 
Then $W(\cdot,s):=e^{-s\mathcal{L}}\phi$ is given~by
\be{defWkernel}
W(y,s)=\int_0^\infty G_0(s,y;\xi)\phi(\xi)\xi^\alpha d\xi,
\quad\hbox{\ where\ \ }
G_0(s,y,\xi):=e^{ks}H_0\bigl(1-e^{-s},e^{-s/2}y,\xi\bigr),
\ee
for all $(y,s)\in Q:= [0,\infty)\times(0,\infty)$, and we have $W\in C^{2,1}(Q)$.
Moreover, $W\in C(\overline Q)$ whenever $\phi\in C_b([0,\infty))$.
 If $\phi'\in L^\infty_{m+1,r}$ for some $r\ge 0$, then 
\be{computWx} 
W_y(y, s)=\int_0^\infty G_1(s,y;\xi)\phi'(\xi)\xi^\alpha d\xi,
\quad\hbox{\ where\ \ }
G_1(s,y,\xi):=e^{(k-\frac12)s}H_1\bigl(1-e^{-s},e^{-s/2}y,\xi\bigr).
\ee
\end{prop}

\begin{rem}
A related, though more complicated formula is given in \cite{HV94pre,MizADE} for the kernel $G_0$.
The formula that we obtain is more convenient in order to derive the precise estimates 
of the space derivative of $G_0$ that are crucially needed in our case.
Also we point out that although the formulas in \cite{HV94pre,MizADE} are used there for noninteger values of $\alpha$, 
they are only proved for integer values.
Our proof works for all real values.
We also note that for $\alpha=0$, one has $Y(x)=ch(x)$, 
so that one of course recovers the usual one dimensional heat kernel.
\end{rem}

\begin{proof}[Proof of Proposition~\ref{KernelEstim}](i)
{\bf Step 1.} {\it Proof of \eqref{PDEH}.} 
Write $H:=H_0=C_\alpha \mathcal{K}_1\mathcal{K}_2$ with
$$\mathcal{K}_1(t,x):=t^{-\frac{\alpha+1}{2}}\exp\Bigl[-\frac{x^2}{4t}\Bigr],\quad \mathcal{K}_2(t,x,\xi):=\exp\Bigl[-\frac{\xi^2}{4t}\Bigr]Y\Bigl(\frac{x\xi}{2t}\Bigr).$$
Direct computation yields $\mathcal{P}\mathcal{K}_1=0$ and
\be{PHeq1} 
\mathcal{P}(\mathcal{K}_1\mathcal{K}_2)=\mathcal{K}_1\mathcal{P}\mathcal{K}_2+\mathcal{K}_2\mathcal{P}\mathcal{K}_1-2\mathcal{K}_{1,x}\mathcal{K}_{2,x}=\mathcal{K}_1\mathcal{P}\mathcal{K}_2-2\mathcal{K}_{1,x}\mathcal{K}_{2,x}.
\ee
Denoting $\mathcal{E}=\exp\bigl[-\frac{\xi^2}{4t}\bigr]$ and omitting the variable $\frac{x\xi}{2t}$ in $Y, Y', Y''$ for conciseness, 
 we get
$$\mathcal{K}_{2,t}(t,x,\xi)=\mathcal{E}  \Bigl\{\frac{\xi^2}{4t^2}Y
-\frac{x\xi}{2t^2} Y'\Bigr\},\quad
\mathcal{K}_{2,x}(t,x,\xi)=\mathcal{E}  \frac{\xi}{2t}Y',
\quad \mathcal{K}_{2,xx}(t,x,\xi)=\mathcal{E}  \frac{\xi^2}{4t^2}Y''.$$
Consequently,
$$\begin{aligned}
\mathcal{P}\mathcal{K}_2
&=\mathcal{E}  \Bigl\{
\frac{\xi^2}{4t^2}Y
-\frac{x\xi}{2t^2} Y'
-\frac{\xi^2}{4t^2}Y''
-\frac{\alpha}{x}\frac{\xi}{2t}Y'\Bigr\}\\
&=\mathcal{E}  \Bigl\{
\frac{\xi^2}{4t^2}\alpha\frac{2t}{x\xi}Y'
-\frac{x\xi}{2t^2} Y'
-\frac{\alpha}{x}\frac{\xi}{2t}Y'\Bigr\}
=-\mathcal{E}  \frac{x\xi}{2t^2} Y'
=-\frac{x}{t}\mathcal{K}_{2,x} =2\frac{\mathcal{K}_{1,x}}{\mathcal{K}_1}\mathcal{K}_{2,x}.
\end{aligned}$$
Combining this with \eqref{PHeq1}, 
we obtain $\mathcal{P}(\mathcal{K}_1\mathcal{K}_2) =0$, nhence \eqref{PDEH}. 

\smallskip

{\bf Step 2.} {\it Proof of \eqref{boundH1}-\eqref{boundH1H_0}.} 
By elementary ODE arguments, one easily shows that $Y, Y', Y''>0$ for $z>0$.
Thus \eqref{eqV-ODE} yields $Y''\le Y$ and, multiplying by $Y'$, we get $({Y'}^2-Y^2)'\le 0$, hence $Y'\le Y$
for $z\ge 0$, which gives \eqref{boundH1H_0}.
In view of proving \eqref{boundH1}, we claim that
\be{claimY}
Y(z)\le C(1+z)^{-\alpha/2}e^z,\quad z\ge 0.
\ee
To this end, we set $\psi_a(z)=z^{-a}e^z$ for $a,z>0$. 
We have 
$\psi_a'=(1-az^{-1})\psi_a$ and $\psi_a''=\bigl[(1-az^{-1})^2+az^{-2}\bigr]\psi_a$,
hence
$$\psi_a''+\ts\frac{\alpha}{z}\psi_a'-\psi_a
=\bigl[\bigl(1-\frac{a}{z}\bigr)^2+\frac{a}{z^2}+\frac{\alpha}{z}\bigl(1-\frac{a}{z}\bigr)-1\bigr]\psi_a
=\bigl[\frac{\alpha-2a}{z}+\frac{a(a+1-\alpha)}{z^2}\bigr]\psi_a.$$
Putting $\psi:=\psi_b-M\psi_{b+1}=(1-Mz^{-1})z^{-b}e^z$ with $b=\alpha/2$ and $M=1+\frac{b}{2}|b+1-\alpha|$, we get
$$\begin{aligned}
\psi''+\ts\frac{\alpha}{z}\psi'-\psi
&=\ts\frac{b(b+1-\alpha)}{z^2}z^{-b}e^z-M\bigl[-\frac{2}{z}+\frac{(b+1)(b+2-\alpha)}{z^2}\bigr]z^{-b-1}e^z\\
&=\bigl\{2M+b(b+1-\alpha)-M(b+1)(b+2-\alpha)z^{-1}\bigr\}z^{-b-2}e^z\ge 0,\quad z\ge z_0
\end{aligned}$$
for some $z_0>0$. Taking $z_0$ possibly larger, we may also assume that $\psi,\psi'>0$ for $ z\ge z_0$.
Therefore, choosing $L>0$ large enough, we see that the function $\Psi:=L\psi-Y$ satisfies
$\Psi''+\alpha z^{-1}\Psi'\ge \Psi$ for all $z\ge z_0$, along with $\Psi(z_0),\Psi'(z_0)>0$.
An elementary argument then shows that $\Psi>0$ for all $z\ge z_0$.
From \eqref{claimY}, we immediately deduce \eqref{boundH1} for $i=0$.
Then, since $Y'(0)=0$ and $Y(1)=1$, we have 
$Y'\le Cz$ for $z>0$ small,
hence 
\be{boundYY0} 
0\le Y'\le C(z\wedge 1)Y,\quad z>0,
\ee
due to $Y'\le Y$, and \eqref{boundH1} for $i=1$ follows.

\smallskip
{\bf Step 3.} {\it Proof of \eqref{boundH2}--\eqref{normalizeH}.}
We first claim that
\be{boundYY} 
0\le Y-Y'\le C\min(1,z^{-1})Y,\quad z>0.
\ee
Since \eqref{boundYY} is true for $z\le 4$ due to $Y'\ge 0$, we may assume $z>4$.
First, since $Y'\le Y$ hence $Y''\ge \frac12Y$ for $z$ large,
we get  $C_1e^{z/2}\le Y(z)\le C_2e^z$ by integration, hence $Y(\frac{z}{4})\le \frac{C}{z}Y(z)$.
Then using $(Y-Y')'=Y'-Y''\le Y-Y''=\alpha z^{-1}Y'$, we obtain
$$
Y(z)-Y'(z)\le Y(1)+\alpha\int_1^z \frac{Y'(\tau)}{\tau}d\tau
\le C\Bigl(1+\int_1^{\frac{z}{4}}\frac{Y'(\tau)}{\tau}d\tau+\int_{\frac{z}{4}}^z \frac{Y'(\tau)}{\tau}d\tau\Bigr) \le 
C\Bigl(1+Y(\ts\frac{z}{4})+\ds\frac{Y(z)}{z}\Bigr),
$$
hence \eqref{boundYY}.

We now compute
\be{boundHxinit} 
H_x=C_\alpha t^{-\frac{\alpha+1}{2}}\exp\Bigl[-\frac{x^2+\xi^2}{4t}\Bigr]
\Bigl\{\frac{\xi}{2t}Y'\Bigl(\frac{x\xi}{2t}\Bigr)-\frac{x}{2t}Y\Bigl(\frac{x\xi}{2t}\Bigr)\Bigr\}.
\ee
 By \eqref{boundYY0} we get 
$$
|H_x|\le
\begin{cases}
C t^{-\frac{\alpha+1}{2}}\exp\Bigl[-\ds\frac{x^2+\xi^2}{4t}\Bigr] \ds\frac{x}{t}Y\Bigl(\frac{x\xi}{2t}\Bigr)=\frac{Cx}{t}H,
&\hbox{if $\xi\le 2x$,} \\
 \noalign{\vskip 2mm}
C t^{-\frac{\alpha+1}{2}}\exp\Bigl[-\ds\frac{x^2+\xi^2}{4t}\Bigr]Y\Bigl(\ds\frac{x\xi}{2t}\Bigr)\frac{x}{t}
\Bigl\{1+\frac{\xi^2}{t}\Bigr\}
\le  \frac{Cx}{t}\Bigl\{1+\frac{|\xi-x|^2}{t}\Bigr\}H,
&\hbox{if $\xi>2x$,}
\end{cases}
$$
which yields \eqref{boundH2} for $x\le\sqrt{t}$. 
To prove it for $x\ge\sqrt{t}$, we write:
$$\frac{\xi}{2t}Y'\Bigl(\frac{x\xi}{2t}\Bigr)-\frac{x}{2t}Y\Bigl(\frac{x\xi}{2t}\Bigr)
=\frac{\xi-x}{2t}Y\Bigl(\frac{x\xi}{2t}\Bigr)+\frac{\xi}{2t}\bigl(Y'-Y\bigr)\Bigl(\frac{x\xi}{2t}\Bigr)
=\frac{\xi-x}{2t}Y'\Bigl(\frac{x\xi}{2t}\Bigr)+\frac{x}{2t}\bigl(Y'-Y\bigr)\Bigl(\frac{x\xi}{2t}\Bigr).$$
Using \eqref{boundYY}, it follows that
$$\Bigl|\frac{\xi}{2t}Y'\Bigl(\frac{x\xi}{2t}\Bigr)-\frac{x}{2t}Y\Bigl(\frac{x\xi}{2t}\Bigr)\Bigr|
\le Ct^{-1/2}Y\Bigl(\frac{x\xi}{2t}\Bigr)\Bigl\{\frac{|\xi-x|}{\sqrt{t}}+\frac{x\wedge \xi}{\sqrt{t}}\min\Bigl(1,\frac{t}{x\xi}\Bigr)
\Bigr\}.$$
If $\frac{x\xi}{t}\le 1$, then $\frac{x\wedge \xi}{\sqrt{t}}\le 1$.
If $\frac{x\xi}{t}\ge 1$, then $\frac{x\vee \xi}{\sqrt{t}}\ge 1$, hence 
$\frac{x\wedge \xi}{\sqrt{t}}\min\bigl(1,\frac{t}{x\xi}\bigr)=\frac{\sqrt{t}}{{x\vee \xi}}\le 1$.
In all cases we thus get
$$\Bigl|\frac{\xi}{2t}Y'\Bigl(\frac{x\xi}{2t}\Bigr)-\frac{x}{2t}Y\Bigl(\frac{x\xi}{2t}\Bigr)\Bigr|
\le Ct^{-1/2}Y\Bigl(\frac{x\xi}{2t}\Bigr)\Bigl\{\frac{|\xi-x|}{\sqrt{t}}+1\Bigr\}.$$
This along with \eqref{boundHxinit}  
readily implies \eqref{boundH2} for $x\ge \sqrt{t}$.

 As for \eqref{boundH2t}, it follows from \eqref{boundYY} by writing
$$\begin{aligned}
|H_t|
&=\Bigl|-\ts\frac{\alpha+1}{2t}H+C_\alpha t^{-\frac{\alpha+1}{2}} \ds\exp\Bigl[-\frac{x^2+\xi^2}{4t}\Bigr]
\Bigl\{\frac{x^2+\xi^2}{4t^2}Y\Bigl(\frac{x\xi}{2t}\Bigr)-\frac{x\xi}{2t^2}Y'\Bigl(\frac{x\xi}{2t}\Bigr)\Bigr\}\Bigr|\\
&=\Bigl|-\ts\frac{\alpha+1}{2t}H+C_\alpha t^{-\frac{\alpha+1}{2}}\ds\exp\Bigl[-\frac{x^2+\xi^2}{4t}\Bigr] 
\Bigl\{\frac{(x-\xi)^2}{4t^2}Y\Bigl(\frac{x\xi}{2t}\Bigr)+\frac{x\xi}{2t^2}
\Bigl(Y\Bigl(\frac{x\xi}{2t}\Bigr)-Y'\Bigl(\frac{x\xi}{2t}\Bigr)\Bigr)\Bigr\}\Bigr|\\
&\le
Ct^{-1}\Bigl\{\frac{|x-\xi|^2}{t}+1\Bigr\}H.
\end{aligned}$$

To prove \eqref{normalizeH}, we fix $\xi\ge 0$ and $t_2>t_1>0$ and, using $\mathcal{P}H=0$ and \eqref{boundH2}, we obtain
$$\begin{aligned}\int_0^M H(t_2,x;\xi)x^\alpha dx-\int_0^M H(t_1,x;\xi)x^\alpha dx
&=\int_{t_1}^{t_2}\int_0^M H_t(t,x;\xi)x^\alpha dx \\
=\int_{t_1}^{t_2}\int_0^M \bigl[x^\alpha H_x(t,x;\xi)\bigr]_x dx
&=\int_{t_1}^{t_2} M^\alpha H_x(t,M;\xi)\to 0,\quad M\to\infty.
\end{aligned}$$
Therefore $I(t,\xi):=\int_0^\infty H(t,x;\xi)x^\alpha dx=I(\xi)$ for all $t>0$.
Writing
$$I(t,\xi) \exp\Bigl[\frac{\xi^2}{4t}\Bigr] =C_\alpha \int_0^\infty t^{-\frac{\alpha+1}{2}}\exp\Bigl[-\frac{x^2}{4t}\Bigr] Y\Bigl(\frac{x\xi}{2t}\Bigr)x^\alpha \, dx
=C_\alpha \int_0^\infty e^{-z^2/4} Y\Bigl(\frac{z\xi}{2\sqrt{t}}\Bigr)z^\alpha \,  dz,$$
and recalling that $Y(0)=1$ and $Y'(s)\ge 0$ for $s\ge 0$, it follows from monotone convergence that
$$I(\xi)=\lim_{t\to\infty} I(t,\xi)=C_\alpha \int_0^\infty e^{-z^2/4} z^\alpha \,  dz=1,$$
upon choosing the normalization constant $C_\alpha$. Since $H(t,x;\xi)=H(t,\xi;x)$, we deduce 
the first part of \eqref{normalizeH}, and the second part follows from $Y'\le Y$.

\smallskip

(ii) {\bf Step 4.} {\it Proof of \eqref{ZC21}--\eqref{Zpde}.}
 Let $Z$ be given by \eqref{defZint}. 
For $x\ge 0$ and $t>0$, we note that  the convergence of the integral
is guaranteed by the assumptions $|\phi(\xi)|\le C(\xi^q+\xi^{-m})$ and $\alpha-m>\frac{\alpha-3}{2}\ge -1$.
 Owing to the bounds \eqref{boundH2}, \eqref{boundH2t} on $H_x, H_t$ and $H_{xx}=H_t-\alpha x^{-1}H_x$,
we may differentiate the integral in \eqref{defZint} for $x\ge 0$ and $t>0$, and deduce \eqref{ZC21} and \eqref{Zpde}.
On the other hand, setting $N=\|\phi\|_{\infty,m,q}$, we have
$$\begin{aligned}
|Z(x,t)| &\le N\int_0^\infty H(t,x,\xi)(\xi^q+\xi^{-m}) \xi^\alpha d\xi \\
&\le CN \biggl\{x^q\int_0^{2x} H(t,x,\xi)\xi^\alpha d\xi
+ t^{-\frac{\alpha+1}{2}}\int_{2x}^\infty e^{-\frac{(x-\xi)^2}{4t}}\xi^{\alpha+q}\,d\xi  \\
&\qquad +x^{-m}\int_{x/2}^\infty H(t,x,\xi)\xi^\alpha\,d\xi
+ t^{-\frac{\alpha+1}{2}}\int_0^{x/2} e^{-\frac{(x-\xi)^2}{4t}}\xi^{\alpha-m}\,d\xi\biggr\}.
\end{aligned}$$
Noting that
$$t^{-\frac{\alpha+1}{2}}\int_{2x}^\infty e^{-\frac{(x-\xi)^2}{4t}}\xi^{\alpha+q}\,d\xi\le
t^{-\frac{\alpha+1}{2}}\int_{2x}^\infty e^{-\frac{\xi^2}{16t}}\xi^{\alpha+q}\,d\xi=
t^{\frac{q}{2}}\int_{2xt^{-1/2}}^\infty e^{-\frac{z^2}{16}}z^{\alpha+q}\,dz\le
Ct^{\frac{q}{2}}$$
and 
$$ t^{-\frac{\alpha+1}{2}}\int_0^{x/2} e^{-\frac{(x-\xi)^2}{4t}}\xi^{\alpha-m}\,d\xi
\le  t^{-\frac{\alpha+1}{2}} e^{-\frac{x^2}{16t}}\int_0^{x/2}\xi^{\alpha-m}\,d\xi
=Ct^{-\frac{\alpha+1}{2}} e^{-\frac{x^2}{16t}}x^{\alpha+1-m}\le Cx^{-m}$$
and using \eqref{normalizeH}, we get 
$|Z(x,t)|\le CN (x^q+t^{\frac{q}{2}}+x^{-m})\le CN (1+t^{\frac{q}{2}})(x^q+x^{-m})$, hence \eqref{ZC21bis}.
 By the same argument, along with \eqref{boundH2}-\eqref{boundH2t}, we get \eqref{ZC21ter}.

\smallskip
{\bf Step 5.} {\it Proof of \eqref{cvZphi}-\eqref{computZx}.}
 Fix $x>0$ and $\eps>0$. Choose $\eta\in(0,x/2)$ such that $|\phi(\xi)-\phi(x)|\le\eps$ for $|\xi-x|\le\eta$.
Using \eqref{normalizeH} and $\alpha-m>\frac{\alpha-3}{2}\ge -1$, we get
$$\begin{aligned}
|Z(x,t)-\phi(x)|
&=\Bigl|\int_0^\infty H(t,x,\xi)(\phi(\xi)-\phi(x)) \xi^\alpha\,d\xi\Bigr| \\
&\le \eps\int_{|\xi-x|<\eta}H(t,x,\xi) \xi^\alpha\,d\xi 
+\int_{|\xi-x|>\eta} H(t,x,\xi) \bigl( |\phi(x)|+C\xi^q+C\xi^{-m}\bigr)\xi^\alpha\,d\xi \\
&\le \eps+ Ct^{-\frac{\alpha+1}{2}}e^{-\frac{\eta^2}{8t}}\int_0^\infty e^{-\frac{(x-\xi)^2}{8}} 
\bigl( |\phi(x)|+\xi^q+\xi^{-m}\bigr)\xi^\alpha\,d\xi\le 2\eps
\end{aligned}$$
for $t>0$ small enough, hence \eqref{cvZphi}.
Of course, a similar argument applies for $x=0$ whenever $\phi\in C_b([0,\infty))$.

\smallskip
Finally assume $\phi'\in L^\infty_{m+1,r}$ for some $r\ge 0$, which in particular guarantees the existence of the integral in \eqref{computZx} owing to $\alpha-m>-1$ and \eqref{boundH1} for $i=1$. 
Using integration by parts, we compute
$$\begin{aligned}
\int_0^\infty &\exp\Bigl[-\frac{\xi^2}{4t}\Bigr] Y'\Bigl(\frac{x\xi}{2t}\Bigr)\phi'(\xi)\xi^\alpha d\xi
=-\int_0^\infty \Bigl\{\exp\Bigl[-\frac{\xi^2}{4t}\Bigr] Y'\Bigl(\frac{x\xi}{2t}\Bigr)\xi^\alpha \Bigr\}_\xi\phi(\xi)d\xi \\
&=-\int_0^\infty\exp\Bigl[-\frac{\xi^2}{4t}\Bigr]  \Bigl\{ \Bigl(-\frac{\xi}{2t}+\frac{\alpha}{\xi}\Bigr)
Y'\Bigl(\frac{x\xi}{2t}\Bigr)+\frac{x}{2t}Y''\Bigl(\frac{x\xi}{2t}\Bigr)\Bigr\}\xi^\alpha \phi(\xi)d\xi \\
&=\int_0^\infty\exp\Bigl[-\frac{\xi^2}{4t}\Bigr]  \Bigl\{ \Bigl(\frac{\xi}{2t}-\frac{\alpha}{\xi}\Bigr)
Y'\Bigl(\frac{x\xi}{2t}\Bigr)+\frac{x}{2t}\Bigl(\frac{2\alpha t}{x\xi}Y'\Bigl(\frac{x\xi}{2t}\Bigr)-Y\Bigl(\frac{x\xi}{2t}\Bigr)\Bigr)\Bigr\}\xi^\alpha \phi(\xi)d\xi\\
&=\int_0^\infty\exp\Bigl[-\frac{\xi^2}{4t}\Bigr]  \Bigl\{ \frac{\xi}{2t}
Y'\Bigl(\frac{x\xi}{2t}\Bigr)-\frac{x}{2t}Y\Bigl(\frac{x\xi}{2t}\Bigr)\Bigr\}\xi^\alpha \phi(\xi)d\xi.
\end{aligned}$$
Formula \eqref{computZx} then follows from
$$
Z_x(x,t)
=\int_0^\infty H_x(t,x;\xi)\phi(\xi)\xi^\alpha d\xi\\
=C_\alpha t^{-\frac{\alpha+1}{2}}e^{-\frac{x^2}{4t}} \int_0^\infty e^{-\frac{\xi^2}{4t}} 
\bigl\{\ts\frac{\xi}{2t}Y'\bigl(\frac{x\xi}{2t}\bigr)-\frac{x}{2t}Y\bigl(\frac{x\xi}{2t}\bigr)\bigr\}\phi(\xi)\xi^\alpha d\xi. \qedhere
$$
\end{proof}

\begin{proof}[Proof of Proposition~\ref{KernelEstim2}]
Let $Z(x,t)$ be given by Proposition~\ref{KernelEstim}, extended by continuity to $ \tilde Q:= [0,\infty)^2\setminus\{(0,0)\}$.
For $(y,s)\in \tilde Q$, set $\tilde W(y,s)=e^{ks}Z(ye^{-s/2},1-e^{-s})$
(which is equal to the integral in \eqref{defWkernel} for $s>0$).
Then an immediate computation using \eqref{Zpde} shows that $\tilde W$ 
satisfies $\tilde W_s=\mathcal{L}\tilde W$ in $Q$.

The result will then follow directly from Proposition~\ref{KernelEstim}
provided we show that $W=\tilde W$. To this end we use a duality argument. 
Let $\psi=W-\tilde W$.
 By Lemma~\ref{CSHprime}, \eqref{ZC21}, \eqref{ZC21bis}, \eqref{cvZphi} we have $\tilde W\in C([0,\infty);H')$.
Moreover, using \eqref{ZC21}, \eqref{ZC21ter} and 
$$(\tilde W_s-k\tilde W)(y,s)=e^{(k-1)s}Z_t(ye^{-s/2},1-e^{-s})-\ts\frac12 e^{(k-\frac12)s}yZ_x(ye^{-s/2},1-e^{-s}),$$
 we deduce that
 $\tilde W\in C(0,\infty;H)\cap C^1(0,\infty;L^2_\rho)$.
This along with \eqref{regSG1} guarantees that
$$ \psi\in C([0,\infty);H')\cap C(0,\infty;D(\mathcal{L}))\cap C^1(0,\infty;L^2_\rho).$$ 
Fix any $S>0$ and $\varphi\in H$ and set $\zeta(\cdot,s)=e^{-(S-s)\mathcal{L}}\varphi$ for $s\in(0,S)$.
Then, by Proposition~\ref{propSGHprime}, $\zeta$~solves the adjoint problem $\zeta_s-\mathcal{L}\zeta=0$ with  
$$\zeta\in C([0,S];H)\cap C^1([0,S);D(\mathcal{L}))$$
and $\zeta(S)=\varphi$. 
Setting $g(s):=\langle\psi(s),\zeta(s)\rangle$, we thus have
 $g(s)=(\psi(s),\zeta(s))$ for $s\in (0,S]$, 
$g\in C([0,S])\cap C^1(0,S)$ and we compute
$$g'(s)=(\psi,\zeta_s)+(\psi_s,\zeta)=(\psi,\mathcal{L}\zeta)-(\mathcal{L}\psi,\zeta)=0,\quad 0<s<S,$$
owing to \eqref{selfadj}. It follows that
$(\psi(S),\varphi)=(\psi,\zeta)(S)=\langle\psi,\zeta\rangle(0)=0$.
Since $\varphi$ and $S$ are arbitrary, we deduce that $\psi\equiv 0$ hence $W=\tilde W$.
\end{proof}

The following proposition gives useful smoothing properties of the 
  kernel associated with the semigroup $(e^{-s\mathcal{L}})$.

\begin{prop} \label{Prop-smoothing}
Let $\alpha\ge 1$, $k\in\R$ and set
$$[T(s)\phi](y):=\int_0^\infty G_0(s,y;\xi)\phi(\xi)\xi^\alpha d\xi,\ \quad y\ge 0,\ s>0,$$
where $G_0$ is given by \eqref{defWkernel}.
 
\smallskip

(i) For all $1<q<\infty$ and $\eps>0$, there exists $C>0$ such that, for all $\phi\in L^q_\rho$,
\be{smoothing0}
\|T(s)\phi\|_{L^q_\rho}\le e^{ks} \|\phi\|_{L^q_\rho},\quad s>0
\ee
and
\be{smoothing1}
\|\partial_y T(s)\phi\|_{L^q_\rho}\le C\bigl(1+s^{-\frac12-\eps}\bigr) e^{(k-\frac12)s}  \|\phi\|_{L^q_\rho},\quad s>0.
\ee

(ii) For all $1<q<r<\infty$, there exist $s^*=s^*(m,q)>0$ and $C>0$ such that, for all $\phi\in L^q_\rho$,
\be{smoothing2}
\|T(s)\phi\|_{L^r_\rho}\le Ce^{ks}\|\phi\|_{L^q_\rho},\quad s\ge s^*
\ee
and, if $\phi\in W^{1,q}_\rho$,
\be{smoothing3}
\|\partial_y T(s)\phi\|_{L^r_\rho}\le Ce^{(k-\frac12)s}\|\phi_y\|_{L^q_\rho},\quad s\ge s^*.
\ee
\end{prop}

\begin{proof}
(i) By density we may assume $\phi\in C_b([0,\infty))$. 
Denote $V(s):=[T(s)\phi]$, which coincides with $e^{-s\mathcal{L}}\phi$ by Proposition~\ref{KernelEstim2}.
Fix $\eps>0$ and set $\psi:=(V^2+\eps^2)^{(q-2)/2}V$ (actually $\eps=0$ will do if $q\ge 2$).
Owing to \eqref{ZC21}--\eqref{ZC21ter}, \eqref{cvZphi}  we then have $V\in C^{2,1}([0,\infty)\times(0,\infty))
\cap C^1(0,\infty;L^q_\rho)\cap C([0,\infty;L^q_\rho)$,
$V$ solves $V_s=-\mathcal{L}V$, and moreover $V(s)\in D(\mathcal{L})$ and $\psi(s)\in H$ for each $s>0$.
Multiplying by $\rho\psi$ with 
and using Lemma~\ref{GaussianPoincare3} we get
$$\frac{1}{q}\frac{d}{ds}\int \rho (V^2+\eps^2)^{q/2}=\int \rho V_s\psi=-(\mathcal{L}V,\psi)
=-(V_y,\psi_y)+k(V,\psi).$$
Since $(V_y,\psi_y)=\int\rho ((q-1)V^2+\eps^2)(V^2+\eps^2)^{(q-4)/2}V_y^2\ge 0$, we deduce
$$\frac{d}{ds}\Bigl(e^{-qks}\int \rho (V^2+\eps^2)^{\frac{q}{2}}\Bigr)
\le ke^{-qks}\Bigl((V,\psi)- \int \rho (V^2+\eps^2)^{\frac{q}{2}}\Bigr)
\le |k|e^{-qks}\eps^2\int\rho(V^2+\eps^2)^{\frac{q-2}{2}}=:J_\eps(s).$$
Integrating in time and then observing that $J_\eps(s)\to 0$ as $\eps\to 0$ uniformly for $s\ge 0$ bounded
 (consider the cases $q\ge 2$ and $q<2$ separately),
inequality \eqref{smoothing0} follows.

\smallskip

We next prove \eqref{smoothing1}. Let $Z$ be given by \eqref{defWkernel}.
Using \eqref{boundH2}, for each $\eta\in(0,1)$, we may split $\partial_xZ$ as follows:
$$\begin{aligned}
|\partial_xZ(x,t)|&\le t^{-\frac12} \int \Bigl(1+\frac{|x-\xi|^2}{t}\Bigr)H_0(t,x,\xi)|\phi(\xi)|\xi^\alpha\,d\xi \\
&\le Ct^{-\frac12-\eps} \int_{E_1}H_0(t,x,\xi)|\phi(\xi)|\xi^\alpha\,d\xi 
+C_\eta t^{-(\alpha+2)/2} \int_{E_2} \exp\Bigl[-\frac{(x-\xi)^2}{(4+\eta)t}\Bigr]|\phi(\xi)|\xi^\alpha\,d\xi \\
&\equiv Z_1(x,t)+Z_2(x,t),
\end{aligned}$$
where
$$E_1=\{\xi>0:\,|x-\xi|^2\le t^{1-\eps}\},\quad E_2=\{\xi>0:\,|x-\xi|^2>t^{1-\eps}\}.$$
To estimate $Z_2$, we use H\"older's inequality to write
$$\begin{aligned}
Z_2(x,t)
&=C_\eta t^{-(\alpha+2)/2} \int_{E_2} e^{\xi^2/4q}\exp\Bigl[-\frac{(x-\xi)^2}{(4+\eta)t}\Bigr]e^{-\xi^2/4q}|\phi(\xi)|\xi^\alpha\,d\xi \\
&=C_\eta \|\phi\|_{L^q_\rho}t^{-(\alpha+2)/2} \Bigl(\int_{E_2} \exp\Bigl[\frac{q'\xi^2}{4q}-\frac{q'(x-\xi)^2}{(4+\eta)t}\Bigr]\xi^\alpha\,d\xi\Bigr)^{1/q'}\\
&\le C_\eta \|\phi\|_{L^q_\rho}t^{-(\alpha+2)/2} \exp\Bigl[-\frac{\eta t^{-\eps}}{5}\Bigr]
\Bigl(\int_{E_2} \exp\Bigl[\frac{q'\xi^2}{4q}-\frac{q'(1-\eta)(x-\xi)^2}{(4+\eta)t}\Bigr]\xi^\alpha\,d\xi\Bigr)^{1/q'}.
\end{aligned}$$
Fix $a\in (1,q)$. Setting $z=\xi-x$ and using $(z+x)^2\le (1+a^{-1}t)x^2+(1+at^{-1})z^2$, we get
$$\begin{aligned}
Z_2(x,t)&\le C_\eta \|\phi\|_{L^q_\rho}  \exp\Bigl[-\frac{\eta t^{-\eps}}{5}\Bigr]
\Bigl(\int_{-\infty}^\infty \exp\Bigl[\frac{q'(z+x)^2}{4q}-\frac{q'(1-\eta)z^2}{(4+\eta)t}\Bigr]|z+x|^\alpha\,dz\Bigr)^{1/q'} \\
&\le C\|\phi\|_{L^q_\rho} \exp\Bigl[-\frac{\eta t^{-\eps}}{5}\Bigr]
\exp\Bigl[\frac{(1+a^{-1}t)x^2}{4q}\Bigr]\Bigl(\int_{-\infty}^\infty 
 \exp\Bigl[\Bigl(\frac{t+a}{4qt}-\frac{1-\eta}{(4+\eta)t}\Bigr)q'z^2\Bigr]|z+x|^\alpha\,dz\Bigr)^{1/q'}.
\end{aligned}$$
Now we may choose $\eta>0$ small such that, for all $t\in (0,\eta)$,
$$\frac{t+a}{4qt}-\frac{1-\eta}{(4+\eta)t}=\frac{(4+\eta)(t+a)-4q(1-\eta)}{4q(4+\eta)t}\le -\eta t^{-1},$$
hence
$$Z_2(x,t)\le C\|\phi\|_{L^q_\rho} \exp\bigl(-2ct^{-\eps}\bigr)(1+x^{\alpha/q'})\exp\Bigl[\frac{(1+a^{-1}t)x^2}{4q}\Bigr],
\quad x\ge 0,\ 0<t<\eta,$$
for some $c>0$. Using $1+a^{-1}t-\frac{1}{1-t}=\frac{(a^{-1}-1)t-a^{-1}t^2}{1-t}\le -c_1t$, it follows that
\be{smoothingZ2}
\begin{aligned}
\Bigr(\int &Z_2^q(x,t)\exp\Bigl[-\frac{x^2}{4(1-t)}\Bigr]x^\alpha dx\Bigr)^{1/q}\\
&\le C\|\phi\|_{L^q_\rho} \exp\bigl(-2ct^{-\eps}\bigr)
\Bigr(\int \exp\Bigl[\Bigl(\frac{1+a^{-1}t}{4}-\frac{1}{4(1-t)}\Bigr) x^2\Bigr](1+x^{\alpha q/q'})x^\alpha dx\Bigr)^{1/q} \\
&\le C\|\phi\|_{L^q_\rho} \exp\bigl(-2ct^{-\eps}\bigr)
\Bigr(\int e^{-c_1tx^2/4}(1+x^{\alpha q/q'})x^\alpha dx\Bigr)^{1/q}\le C\|\phi\|_{L^q_\rho} \exp\bigl(-ct^{-\eps}\bigr).
\end{aligned}
\ee

Going back to $V$ and recalling that
\be{defV2}
V(y,s)=e^{ks}Z(ye^{-s/2},1-e^{-s})=\int_0^\infty G_0(s,y;\xi)\phi(\xi)\xi^\alpha d\xi,
\ee
where $G_0(s,y,\xi):=e^{ks}H_0\bigl(1-e^{-s},e^{-s/2}y,\xi\bigr)$, we write
$$e^{(\frac12-k)s} |\partial_yV(y,s)|= |\partial_xZ(ye^{-\frac{s}{2}},1-e^{-s})|=
Z_1(ye^{-\frac{s}{2}},1-e^{-s})+Z_2(ye^{-\frac{s}{2}},1-e^{-s})\equiv V_1(y,s)+V_2(y,s).$$
Since $Z_1(x,t)\le Ct^{-(1+\eps)/2} \tilde Z(x,t)$, where $\tilde Z$ is given by \eqref{defWkernel}  with $\phi$ replaced by $|\phi|$,
it follows from \eqref{smoothing0} that 
$$\|V_1(s)\|_{L^q_\rho}\le Ce^{-s/2}(1-e^{-s})^{-\frac12-\eps}\|\phi\|_{L^q_\rho}
\le Cs^{-\frac12-\eps}\|\phi\|_{L^q_\rho}$$
for $s>0$ small.
As for $V_2$, putting $t=1-e^{-s}$, it follows from \eqref{smoothingZ2} that
$$\begin{aligned}
\|V_2(s)\|_{L^q_\rho}^q
&=\int V_2^q(y,s)e^{-y^2/4}y^\alpha dy
=\int Z_2^q(ye^{-s/2},1-e^{-s})e^{-y^2/4}y^\alpha dy\\
&=(1-t)^{-(\alpha+1)/2} \int Z_2^q(x,t)\exp\Bigl[-\frac{x^2}{4(1-t)}\Bigr]x^\alpha dx\le C\|\phi\|_{L^q_\rho}^q.
\end{aligned}$$
These estimates on $V_1, V_2$ guarantee \eqref{smoothing1} for $s>0$ small.
On the other hand, by \eqref{boundH1H_0}, \eqref{defWkernel}, \eqref{computWx} and \eqref{smoothing0}, we have
$$\|\partial_yV(s)\|_{L^q_\rho}=\|\partial_yT(s-s_1)V(s_1)\|_{L^q_\rho}\le e^{(k-\frac12)(s-s_1)} \|\partial_yV(s_1)\|_{L^q_\rho},\quad s>s_1>0.$$
This guarantees that \eqref{smoothing1} remains true for all $s>0$.

\smallskip

(ii) We adapt the proof of \cite[Lemma~2.1]{HV93}, given there for $\alpha=0$.
Using 
$$\frac{\xi^2}{q}-\frac{(x-\xi)^2}{t}=\frac{x^2}{q-t}-\Bigl(\frac{1}{t}-\frac{1}{q}\Bigr)\Bigl(\xi-\frac{qx}{q-t}\Bigr)^2
\le \frac{x^2}{q-t}-c\Bigl(\xi-\frac{qx}{q-t}\Bigr)^2,\quad 0<t<1,$$
we have, for all $t\in(0,1)$,
$$\begin{aligned}
t^{(\alpha+1)/2}|Z(x,t)|
&\le C  \int e^{\xi^2/4q}\exp\Bigl[-\frac{(x-\xi)^2}{4t}\Bigr]e^{-\xi^2/4q}|\phi(\xi)|\xi^\alpha\,d\xi \\
&\le C \|\phi\|_{L^q_\rho} \Bigl(\int \exp\Bigl[\frac{q'\xi^2}{4q}-\frac{q'(x-\xi)^2}{4t}\Bigr]\xi^\alpha\,d\xi\Bigr)^{1/q'}\\
&\le C \|\phi\|_{L^q_\rho} \exp\Bigl[\frac{x^2}{4(q-t)}\Bigr]
 \Bigl(\int \exp\Bigl[-c\Bigl(\xi-\frac{qx}{q-t}\Bigr)^2\Bigr]\xi^\alpha\,d\xi\Bigr)^{1/q'}\\
&\le C \|\phi\|_{L^q_\rho} \exp\Bigl[\frac{x^2}{4(q-t)}\Bigr]
 \Bigl(\int_{-\infty}^\infty e^{-cz^2}\Bigl|z+\frac{qx}{q-t}\Bigr|^\alpha\,dz\Bigr)^{1/q'}\\
 &\le C \|\phi\|_{L^q_\rho} \exp\Bigl[\frac{x^2}{4(q-t)}\Bigr]
 (1+x^{\alpha/q'}).
\end{aligned}$$
Consequently, setting $t_0=(r-q)/(r-1)$, we have
$$\Bigr(\int |Z|^r(x,t)\exp[-\ts\frac{x^2}{4(1-t)}]x^\alpha dx\Bigr)^{\frac{1}{r}}
\le M(t)\|\phi\|_{L^q_\rho},\quad t_0<t<1, $$
where 
$M(t)=t^{-\frac{\alpha+1}{2}}
\bigr(\ts\int \exp[(\ts\frac{r}{q-t}-\frac{1}{1-t}) \frac{x^2}{4}](1+x^{\alpha r/q'})x^\alpha dx\bigr)^{1/q}<\infty$.
Fix $t^*\in(t_0,1)$ and $s^*:=-\log(1-t^*)$. By \eqref{defV2}, it follows that 
$$\begin{aligned}
\|V(s^*)\|_{L^r_\rho}^r
&=\int |V(y,s^*)|^re^{-y^2/4}y^\alpha dy
=e^{krs^*}\int |Z(ye^{-s^*/2},1-e^{-s^*})|^re^{-y^2/4}y^\alpha dy\\
&\le C\int |Z(x,t^*)|^r\exp\Bigl[-\frac{x^2}{4(1-t^*)}\Bigr]x^\alpha dx\le C\|\phi\|_{L^q_\rho}^q.
\end{aligned}$$
This along with \eqref{smoothing0} (with $q$ replaced by $r$) yields  \eqref{smoothing2}.
The proof of  \eqref{smoothing3} is completely similar, making use of \eqref{boundH1H_0} and \eqref{computZx}.
\end{proof}

\subsection{Maximum principles for the linearized operator} \label{SubsecMP}

Assume 
\be{hypkTR}
\begin{cases}
&\alpha\ge 1,\ \ T>0, \ \ 0<R\le\infty,\ \ Q=(0,R)\times(0,T), \\
 \noalign{\vskip 1mm}
&a_1,a_2\in L^\infty_{loc}(Q),\ \ |a_i|\le 
C_1(x^i+x^{-\gamma_i}),\ \ 0\le \gamma_1<1, 
\ \ \ 0\le \gamma_2<2, \\ 
\end{cases}
\ee
for some $C_1>0$ and
\be{hypkTR2}
\begin{cases}
&w\in C^{2,1}(Q),\quad w\in C(\overline Q), \quad w_x\in L^\infty((r,R)\times(0,T))
\hbox{ for each $r>0$},\\
 \noalign{\vskip 1mm}
& 
x^\alpha w_x(x,t)\to 0,\ \hbox{ as $x\to 0$, uniformly in $t\in(0,T)$},
\end{cases}
\ee
and consider the linear operator with singularities at $x=0$:
\be{hypkTRP}
Pw:=w_t-w_{xx}-\ts\frac{\alpha}{x}w_x-a_1(x,t)w_x-a_2(x,t)w.
\ee
We first have the following maximum principle and strong maximum principle {\it up to $x=0$.}
 We stress that no boundary conditions at $x=0$ are required.

\begin{prop} \label{propMP}
Assume \eqref{hypkTR}, \eqref{hypkTR2}, $Pw\le 0$ in $Q$, $w(x,0)\le 0$ in $[0,R)$, and
$$
\begin{cases}
\hbox{$w(R,t)\le 0$ on $(0,T)$},& \hbox{if $R<\infty$,} \\
 \noalign{\vskip 1mm}
\hbox{ $w\le e^{M(1+x^2)}$ in $Q$ for some $M>0$},& \hbox{if $R=\infty$.}
\end{cases}
$$

\smallskip

(i) Then $w\le 0$ in $\overline Q$.

\smallskip

(ii) Assume in addition $w(x,0)\not\equiv 0$. Then $w<0$ in $[0,R)\times(0,T]$.
\end{prop}

For the proof we need a Hardy-type inequality, which is provided by the following simple lemma.

\begin{lem} \label{LemHardy}
Let $\alpha\ge 1$ and $\gamma<2$.
For each $\eta>0$, there exists $C>0$ depending only on $\eta,\alpha,\gamma$ such that
for all $0\le a<b$, there holds
$$\int_a^b x^{\alpha-\gamma} \phi^2
\le 2\bigl[\ts\frac{x^{\alpha-\gamma+1}}{\alpha-\gamma+1} \phi^2\bigr]_a^b
+\eta \ds\int_a^b x^\alpha {\phi'}^2
+C[1+b^{2(1-\gamma)}] \int_a^b x^\alpha  \phi^2,
\quad \phi\in C^1([a,b]).$$
\end{lem} 

\begin{proof} By integration by parts, we have
$$\int_a^b x^{\alpha-\gamma} \phi^2
\le \bigl[\ts\frac{x^{\alpha-\gamma+1}}{\alpha-\gamma+1} \phi^2\bigr]_a^b
-\ts\frac{2}{\alpha-\gamma+1} \ds\int_a^b x^{\alpha-\gamma+1} \phi\phi' 
\le \bigl[\ts\frac{x^{\alpha-\gamma+1}}{\alpha-\gamma+1} \phi^2\bigr]_a^b
+\eta \ds\int_a^b x^\alpha {\phi'}^2 +C(\eta) \int_a^b x^{\alpha+2-2\gamma}\phi^2.$$
Choosing $x_0\in(0,1)$ such that $C(\eta) x_0^{2-\gamma}<\frac12$, we get
$$\int_a^b x^{\alpha-\gamma} \phi^2\le \bigl[\ts\frac{x^{\alpha-\gamma+1}}{\alpha-\gamma+1} \phi^2\bigr]_a^b
+\eta \ds\int_a^b x^\alpha {\phi'}^2 +\frac12 \int_{(a,b)\cap(0,x_0)} x^{\alpha-\gamma}\phi^2
+C(\eta)[x_0^{2(1-\gamma)}+b^{2(1-\gamma)}] \int_{(a,b)\setminus(0,x_0)}  x^\alpha\phi^2$$
and the conclusion follows.
\end{proof}

\begin{proof}[Proof of Proposition~\ref{propMP}.] 
{\bf Step 1.} {\it Proof of (i) for $R<\infty$.}
Let $\gamma=\max(2\gamma_1,\gamma_2)<2$.
Fix $\eps\in (0,R/2)$. Multiplying by $x^\alpha w_+$, integrating by parts over $(\eps,R-\eps)$
and using $w_{+,x}=\chi_{\{w>0\}}w_x$, Young's inequality and \eqref{hypkTR},
we obtain, for $t\in (0,T)$,
$$\begin{aligned}
\frac{d}{dt} \int_\eps^{R-\eps} \ts\frac{x^\alpha}{2}w_+^2
&= \int_\eps^{R-\eps} x^\alpha w_t w_+ \le \int_\eps^{R-\eps} \bigl((x^\alpha w_x)_x+x^\alpha(a_1w_x+a_2w)\bigr)w_+\\
&= \Bigl[x^\alpha w_xw_+\Bigr]_\eps^{R-\eps}+\int_\eps^{R-\eps} x^\alpha\bigl\{-(w_{+,x})^2+(a_1w_{+,x} +a_2w_+)w_+\bigr\} \\
&\le\Bigl[x^\alpha w_xw_+\Bigr]_\eps^{R-\eps}-\frac12\int_\eps^{R-\eps} x^\alpha (w_{+,x})^2+C_2\int_\eps^{R-\eps} x^{\alpha-\gamma} w_+^2.
\end{aligned}$$
Here and in what follows, $C_i$ denotes positive constants independent of $\eps$.
Next taking $\eta>0$ such that $\eta\,C_2<\frac12$ and applying Lemma~\ref{LemHardy}, we deduce that
$$
\frac{d}{dt} \int_\eps^{R-\eps} \ts\frac{x^\alpha}{2}w_+^2
\le \Bigl[x^\alpha w_xw_++2 C_2\ts\frac{x^{\alpha-\gamma+1}}{\alpha-\gamma+1} w_+^2\Bigr]_\eps^{R-\eps}+C_3\ds\int_\eps^{R-\eps} x^\alpha w_+^2
\le g_\eps(t)+C_3\ds\int_\eps^{R-\eps} x^\alpha w_+^2,
$$
where $g_\eps(t)=C_3\bigl\{|w_x|w_++w_+^2\bigr\}(R-\eps, t)+\eps^\alpha |w_x|w_+(\eps, t)$.
Integrating in time and using $w(\cdot,0)\le 0$, we get
$$
\int_\eps^{R-\eps} x^\alpha w_+^2(t)
\le\Bigl\{2\int_0^t g_\eps(s)\,ds+\int_\eps^{R-\eps} x^\alpha w_+^2(x,0)\Bigr\}e^{-2C_3t}
=2e^{-2C_3t} \int_0^t g_\eps(s)\,ds.$$
Letting $\eps\to 0$ and using $w(R,s)\le 0$ and \eqref{hypkTR2}, we deduce that $w_+(\cdot,t)\equiv 0$.

\smallskip

{\bf Step 2.} {\it Proof of (i) for $R=\infty$.}
Take $\eta\in(0,2-\gamma_2)$ and  $\phi\in C^2([0,\infty))$, $\psi\in C([0,\infty))\cap C^2(0,\infty)$,  $\phi,\psi\ge 1$, such that
$$
 \phi(x)=
\begin{cases}
1,& 0\le x\le 1, \\
 \noalign{\vskip 1mm}
x^2, & x\ge 2,
\end{cases}
\qquad
\psi(x)=
\begin{cases}
3-x^\eta,& 0\le x\le 1, \\
 \noalign{\vskip 1mm}
1, & x\ge 2.
\end{cases}
$$
 By our assumption, there exists $N>0$ such that 
\be{wexpmx2}
\sup_{t\in(0,T)} w(x,t)\le o\bigl(e^{Nx^2}\bigr),\quad x\to\infty.
\ee
Set 
$$E(x,t)=e^{(N+Kt)\phi(x)}, \qquad Z(x,t)=E(x,t)\psi(x),$$
with $K>0$ to be chosen.
We compute
\be{ZtZxx}
 Z_t=K\phi Z,\qquad Z_x=E\psi'+E_x\psi,\qquad Z_{xx}=E\psi''+2E_x\psi'+E_{xx}\psi.
\ee
For $(x,t)\in (0,1]\times(0,T)$, we have
$$\psi''+\ts\frac{\alpha}{x}\psi'=-\eta(\eta-1+\alpha)x^{\eta-2},\quad
|a_1\psi'|+|a_2\psi|\le C[x^{\eta-1-\gamma_1}+x^{-\gamma_2}].$$
Here and below, $C$ denotes a generic positive constant independent of $K$.
Consequently, there exists $x_0\in(0,1)$ (independent of $K$) such that,
for all $(x,t)\in (0,x_0]\times(0,T)$,
$$ PZ=KZ+e^{N+Kt}P\psi\ge \Bigl\{\eta(\eta-1+\alpha)x^{\eta-2}-C\bigl[x^{\eta-1-\gamma_1}
+x^{-\gamma_2}\bigr]\Bigr\}e^{N+Kt}\ge 0.$$
 Let $T_1=\min(T,1/K)$. For $(x,t)\in [2,\infty)\times(0,T_1)$, we have
$$E_{xx}+\ts\frac{\alpha}{x}E_x=\bigl[2(\alpha+1)(N+Kt)+4(N+Kt)^2x^2\bigr]E\le Cx^2E,\qquad
|a_1E_x|+|a_2E|\le Cx^2E,$$ 
hence $PZ=PE\ge (Kx^2-Cx^2)E$.
For $(x,t)\in (x_0,2)\times(0,T_1)$, we have $|E_x|+|E_{xx}|\le CE$ hence,
using \eqref{ZtZxx}, $\psi\in C^2(0,\infty)$ and $\phi,\psi\ge 1$, $PZ\ge (K-C)E$.
Taking $K>0$ large enough, we thus obtain
$PZ\ge 0$ in $(0,\infty)\times(0,T_1)$.

Now, for each fixed $\eps>0$, we set $w_\eps:=w-\eps Z$, which satisfies
$Pw_\eps=Pw-\eps PZ\le 0$ in $(0,\infty)\times(0,T_1)$.
On the other hand, by $w(\cdot,0)\le 0$ and \eqref{wexpmx2}, we have
$w_\eps\le 0$ on $((0,\infty)\times\{0\})\cup([\tilde R,\infty)\times(0,T))$ for $\tilde R=\tilde R(\eps)>0$ sufficiently large,
and $w_\eps$ moreover satisfies the assumptions in \eqref{hypkTR2} with $R$ replaced by $\tilde R$.
We thus deduce from Step~1 that $w_\eps\le 0$ in $(0,\infty)\times(0,T_1]$,
hence $w\le 0$ by letting $\eps\to 0$. 
 Repeating the argument on $[T_1,\min(T,T_1+1/K))$ (in case $T>1/K$) and so on, the conclusion follows.

\smallskip

{\bf Step 3.} {\it Proof of (ii).}
Take $x_0\in (0,R)$ such that $w(x_0,0)<0$.
Since the coefficients $\alpha x^{-1}+a_1$ and $a_2$ bounded for $x$ in compact subsets of $(0,\infty)$, for each $\eps\in (0,x_0)$,
we may apply the strong maximum principle on $(\eps,R)\times (0,T)$, to deduce that
$w<0$ in $(0,R)\times(0,T]$.

\smallskip

It remains to show that $w(0,t)<0$ on $(0,T]$. 
To this end, we use a comparison argument. 
Take $\eta\in(0,2-\gamma_2)$. 
Fixing any $\tau\in (0,T/2)$, we set 
 $\underline w(x,t)=\delta(t+x^\eta-2\tau)$, where $\delta>0$ will be selected below. 
Taking $r\in (0,\min(R/2,\tau^{1/\eta})$ sufficiently small (independent of $\delta$),
a simple computation gives
$$\begin{aligned}
\delta^{-1} P\underline w
&=1-\eta(\eta-1+\alpha)x^{\eta-2}-a_1\eta x^{\eta-1}-a_2(t+x^{\eta}-2\tau)
\\
&\le 1-\eta(\eta-1+\alpha)x^{\eta-2}+
Cx^{\eta-1-\gamma_1}
+Cx^{-\gamma_2}(T+r^{\eta}) \le 0
\end{aligned}$$
in $(0,r)\times(0,T]$. 
On the other hand we have 
 $\underline w(x,\tau)\le \delta(r^{\eta}-\tau)<0\le -w(x,\tau)$ on $[0,r]$.
 Now choosing 
$$\delta:=(T+r^{\eta})^{-1}\inf_{t\in [\tau,T]}(-w(r,t))>0,$$
we get
$\underline w(r,t)\le \delta(T+r^{\eta})\le -w(r,t)$ on $(\tau,T]$.
We then deduce from assertion (i) that
$\underline w+w\le 0$ in $[0,r)\times (\tau,T]$, hence in particular $w(0,t)\le- \delta(t-2\tau)<0$ for all $t\in (2\tau,T]$.
Since $\tau\in (0,T/2)$ was arbitrary, the assertion follows.
\end{proof}

As a consequence of Proposition~\ref{propMP}(ii), we obtain the following strong separation property up to $x=0$ 
for singular viscosity solutions of the viscous Hamilton-Jacobi equation (which turn out to satisfy assumption \eqref{uxUx}).

\begin{prop} \label{propSMP2}
Let $p>2$, $T>0$, $0<R\le\infty$, $Q=(0,R)\times(0,T)$ and $u_1, u_2\in C^{2,1}(Q)\cap C(\overline Q)$ be classical solutions of
$u_t=u_{xx}+|u_x|^p$ in $Q$, such that 
\be{uxUx}
u_x-U_x\in L^\infty(Q) 
\ee
 and $u_1(\cdot,0)\le u_2(\cdot,0)$. Also, suppose that
$$
\begin{cases}
\hbox{$[u_1-u_2](R,t)\le 0$ on $(0,T)$},& \hbox{if $R<\infty$,} \\
 \noalign{\vskip 1mm}
\hbox{$u_1-u_2\le e^{M(1+x^2)}$ in $Q$ for some $M>0$},& \hbox{if $R=\infty$.}
\end{cases}
$$
Then $u_1\le u_2$ in $[0,R)\times(0,T]$.
If moreover $u_1(\cdot,0)\not\equiv u_2(\cdot,0)$,
then $u_1<u_2$ in $[0,R)\times(0,T]$.
\end{prop}

\begin{proof}
Setting $w=u_1-u_2$ and $g(s)=p|s|^{p-2}s$, we get
$$ 
|u_{1,x}|^p-|u_{2,x}|^p
= |u_{2,x}+w_x|^p-|u_{2,x}|^p = g(u_{2,x}+\theta w_x)w_x =  g\bigl(U_x+O(1)\bigr)w_x.
$$
We have $ g\bigl(U_x+O(1)\bigr)=O(1)$ for $1<x<R\ (\le\infty)$ and
$$ g\bigl(U_x+O(1)\bigr)= U_x^{p-1}g\bigl(1+O(x^\beta)\bigr)= \ts\frac{\alpha}{x}(1+O(x^\beta))
=\ts\frac{\alpha}{x}+O(x^{\beta-1}),\quad 0<x\le 1,$$
with $\alpha=p/(p-1)$.
Therefore the equation for $w$ is $Pw=0$ with 
$a_1$ satisfying \eqref{hypkTR} for $\gamma_1=1-\beta$ and $a_2=0$,
and $w$ satisfies \eqref{hypkTR2} owing to \eqref{uxUx}.
The conclusion thus follows from Proposition~\ref{propMP}.
\end{proof}

\subsection{Zero number properties} \label{SubsecMP2}
Denote by $z(\phi\,:\,[0,R])\in \N\cup\{\infty\}$ 
the number of sign-changes of $\phi$ on $[0,R]$
($=0$ if $\phi\ge 0$ or $\phi\le 0$).
First recall the case of classical solutions of the viscous Hamilton-Jacobi equation
up to the boundary
 (including the case of a moving boundary, which will be also needed).

\begin{prop} \label{propZeroNumber0}
Let $p>2$, $t_1>t_0$.  Let $x_0, x_1:[t_0,t_1]\to \R$ be continuous curves such that $x_0(t)<x_1(t)$
and denote $D=\{(x,t);\ t_0<t<t_1,\ x_0(t)<x<x_1(t)\}$.
Let $u_1, u_2\in C^{2,1}(\overline D)$ be classical solutions of $u_t=u_{xx}+|u_x|^p$ in~$\overline D$ and
assume that, for each $i\in\{0,1\}$, either
\be{hyp1propZeroNumber0}
\hbox{ $x_i$ is constant and $[u_1-u_2](x_i(t),t)=0$ for all $t\in[t_0,t_1]$}
\ee
or
$$ \hbox{ $[u_1-u_2](x_i(t),t)\ne 0$ for all $t\in[t_0,t_1]$.}$$
Then the following holds.

\begin{itemize}

\item[(i)] $N(t):=z\bigl([u_1-u_2](\cdot,t)\,:\,[x_0(t),x_1(t)]\bigr)$ is finite and nonincreasing on $(t_0,t_1]$;
\vskip 1pt

\item[(ii)] If $N(t_0)$ is finite, then (i) is valid on $[t_0,t_1]$;
\vskip 1pt

\item[(iii)] $N(t)$ drops at each time $t\in(t_0,t_1)$ for which $[u_1-u_2](\cdot,t)$ has a degenerate zero in $[x_0(t),x_1(t)]$.
\end{itemize}
\end{prop}

\begin{proof}
 The case when $x_1, x_2$ are constant is covered by standard intersection-comparison theory (cf.~\cite{An88, CP96} and see also \cite{PS3, MizSou}).
 The general case is reduced to the former by a simple compactness argument.
\end{proof}

The next proposition will be useful for the proofs of our results on recovery rates,
since it will allow to apply intersection-comparison arguments to the
 RBC viscosity solutions of the viscous Hamilton-Jacobi equation under consideration.

\begin{prop} \label{propZeroNumber}
(i) Assume \eqref{hypkTR}, \eqref{hypkTR2} with $R<\infty$, 
 let $P$ be given by \eqref{hypkTRP} and set $N(t):=z(w(\cdot,t)\,:\, [0,R])$. Assume that $w$ satisfies
$Pw=0$ in $Q$ and $w\ne 0$ on $\{(0,0)\}\cup(\{R\}\times (0,T))$.
Then $N$ is finite and nonincreasing on $(0,T)$. 
Moreover, $N$ drops at each time $t\in(0,T)$ such that $w(\cdot,t)$ has a degenerate zero in $(0,R)$
or $w(0,t)=0$;
namely $N(t)<\ds\lim_{s\to t_-}N(s)$.

\smallskip

(ii) Let $p>2$, $T>0$, $0<R<\infty$, $Q=(0,R)\times(0,T)$ and $u_1, u_2\in C^{2,1}(Q)\cap C(\overline Q)$ be classical solutions of
$u_t=u_{xx}+|u_x|^p$ in $Q$ satisfying \eqref{uxUx}.
Assume that $u_1\ne u_2$ on $\{(0,0)\}\cup(\{R\}\times (0,T))$.
Then $z(u_1(\cdot)-u_2(\cdot)\,:\,[0,R])$ is finite and nonincreasing on $(0,T)$.
Moreover, it drops at each $t\in(0,T)$ such that $[u_1-u_2](\cdot,t)$ has a degenerate zero in $(0,R)$
or $[u_1-u_2](0,t)=0$.
\end{prop}

\begin{rem} \label{RemZeroNumber2}
We observe that a closely related result to Proposition~\ref{propZeroNumber}(i)
was proved in \cite{CP96} in the case when $\alpha$ is an integer,
but the noninteger case is crucially needed in our study (specifically $\alpha=p/(p-1)\in(1,2)$).
On the other hand, the assumption $w(0,0)\ne 0$ might be technical, as it is not needed when $\alpha$ is an integer.
However it makes our proof considerably simpler and the statement is enough for our needs. 
See also Remark~\ref{RemZeroNumber3} for an alternative assumption.
As for Proposition~\ref{propZeroNumber}(ii), a related result was proved in \cite{MizSou} but it is not sufficient here.
\end{rem}

The idea of the proof of Proposition~\ref{propZeroNumber}(i) is to suitably control the possible zeros at $x=0$,
so as to be able to apply the standard zero number theory with bounded coefficients away from $x=0$.
To this end we set $\mathcal{Z}=\{t\in (0,T);\ w(0,t)=0\}$ and first observe that, for any $0\le t_1<t_2\le T$,
\be{Zmonot}
\begin{aligned}
&\hbox{If $\mathcal{Z}\cap(t_1,t_2)=\emptyset$, then $N$ is finite and nonincreasing and $N$ drops at }\\
&\hbox{any time $t\in(t_1,t_2)$ for which $w(\cdot,t)$ has a degenerate zero on $(0,R)$.}
\end{aligned}
\ee
Indeed, for each $t_1<\tilde t_1<\tilde t_2<t_2$, we have $w\ne 0$ on $[0,\eps]\times [\tilde t_1,\tilde t_2]$
for some $\eps>0$. Since the coefficients $\alpha x^{-1}+a_1$ and $a_2$ are bounded away from $x=0$, by
standard zero number theory~\cite{An88}, property \eqref{Zmonot} holds on $(\tilde t_1,\tilde t_2)$, hence on $(t_1,t_2)$.
We next have the following lemma.

\begin{lem} \label{Zisolation}
Under the assumptions of Proposition~\ref{propZeroNumber}(i),
let $t_0\in \mathcal{Z}$ be such that $(t_0-\eps,t_0)\cap \mathcal{Z}=\emptyset$ for some $\eps>0$. Then
\be{isolated1}
\hbox{$t_0$ is an isolated point of $\mathcal{Z}$,}
\ee
\be{isolated2}
\hbox{$N$ is nonincreasing in the neighborhood of $t_0$ and
$N(t_0)<\lim_{t\to t_0^-} N(t)$.}
\ee
\end{lem} 

\begin{proof}
Assume without loss of generality that $w(0,t)>0$ on $(t_0-\eps,t_0)$.
We claim that 
\be{liminfx1}
\liminf_{t\to t_0^-} x_1(t)=0, \quad\hbox{ where $x_1(t)=\sup\hskip 0.5pt\bigl\{x\in(0,R];\ w(\cdot,t)\ge 0$ on $[0,x]\bigr\}$}.
\ee
Indeed, otherwise, we would have $w\ge 0$ on $(0,\delta]\times[t_0-\delta,t_0)$ for some $\delta\in(0,\eps)$,
and by the strong maximum principle in Proposition~\ref{propMP}(ii) this would imply $w(0,t_0)>0$, a contradiction.
By \eqref{Zmonot} and \eqref{liminfx1}, there exist $\eps_1\in(0,\eps)$ and an integer $m\ge 1$ such that
\be{liminfx2}
\hbox{for all $t\in(t_0-\eps_1,t_0)$, $N(t)=m:=\lim_{t\to t_0^-} N(t)$.}
\ee

We next claim that $N(t_0)<m$. Assume the contrary. Then there exist $0<y_0<\dots<y_m<R$ such that 
$w(y_{i-1},t_0)w(y_i,t_0)<0$ for all $i\in\{1,\dots,m\}$. By continuity, for some $\eps_2\in(0,\eps_1)$, we have
$w(y_{i-1},t)w(y_i,t)<0$ for all $i\in\{1,\dots,m\}$ and $t\in [t_0-\eps_2,t_0]$.
But \eqref{liminfx1} implies the existence of $t\in [t_0-\eps_2,t_0)$ and $\tilde y\in (x_1(t),y_0)$ such that $w(\tilde y,t)<0<w(0,t)$.
Consequently $w(\cdot,t)$ has at least $m+1$ sign changes in $[0,R]$, contradicting \eqref{liminfx2}.

Now, since $N(t_0)<\infty$ and $u(R,t_0)\ne 0$, there exist $\sigma\in \{-1,1\}$ and $r\in(0,R)$ such that 
$\sigma w(\cdot,t_0)\ge 0$ on $[0,r]$ and $\sigma w(r,t_0)>0$.
By continuity, we may find $\eps_3>0$ such that $\sigma w(r,t)>0$ for all $t\in [t_0,t_0+\eps_3]$.
We deduce from Proposition~\ref{propMP}(ii) that $\sigma w>0$ on $[0,r]\times(t_0,t_0+\eps_3]$, hence in particular \eqref{isolated1}.
Moreover, we have $N(t)=z(w(\cdot,t),[r,R))$ for all $t\in [t_0,t_0+\eps_3]$.
Since the coefficients $\alpha x^{-1}+a_1$ and $a_2$ are
bounded away from $x=0$, it follows from standard zero number theory~\cite{An88}
that $N(t)$ is nonincreasing on $[t_0,t_0+\eps_3]$.
This along with \eqref{liminfx2} and $N(t_0)<m$ implies \eqref{isolated2}.
\end{proof}

\begin{proof}[Proof of Proposition~\ref{propZeroNumber}]
(i) We claim that $\mathcal{Z}$ is discrete. Assume the contrary and let $\mathcal{Z}_0\ne\emptyset$ denote the
set of accumulation points of $\mathcal{Z}$. Since $\mathcal{Z}_0$ is a closed subset of $(0,T)$
and $w(0,0)\ne 0$, we may define $\bar t=\min \mathcal{Z}_0 \in (0,T)$. 
For each $\tau\in (0,\bar t)$, $(0,\tau)\cap \mathcal{Z}$ is finite. Therefore, as a consequence of 
 \eqref{Zmonot} and Lemma~\ref{Zisolation},
$N$ is finite and nonincreasing on $(0,\tau)$, hence on $(0,\bar t)$, and it drops at each time 
$t\in (0,\bar t)\cap\mathcal{Z}$.
Therefore $(0,\bar t)\cap \mathcal{Z}$ is finite, hence there exists $\eps>0$ such that  $(\bar t-\eps,\bar t)\cap \mathcal{Z}=\emptyset$.
But Lemma~\ref{Zisolation} then implies that $\bar t$ is an isolated point of $\mathcal{Z}$, a contradiction.

\smallskip

Now, since $\mathcal{Z}$ is discrete and $w(0,0)\ne 0$, for each $\tau\in (0,T)$, $(0,\tau)\cap\mathcal{Z}$ is finite.
Therefore, as a consequence of Lemma~\ref{Zisolation},
$N$ is finite and nonincreasing on $(0,\tau)$, hence on $(0,T)$, and $N$ drops at each time $t\in \mathcal{Z}$.
By \eqref{Zmonot}, $N$ also drops at each $t$ such that $w(0,t)\ne 0$ and $w$ has a degenerate zero in $(0,R)$.
The assertion follows.

\smallskip

(ii) Arguing as in the proof of Proposition~\ref{propSMP2}, the conclusion follows from assertion (i).
\end{proof}

\begin{rem} \label{RemZeroNumber3}
(i) Proposition~\ref{propZeroNumber} remains true if instead of $w(0,0)\ne 0$, we assume more generally that
$0\not\equiv w(\cdot,0)_{|[0,r]}\ge 0$ or $\le 0$ for some $r\in (0,R]$.
Indeed, by continuity, we then have $w(x_0,t)>0$ (or $<0$) in $[0,t_0]$ for some $t_0\in (0,T)$ and $x_0\in (0,r]$,
hence $w(0,t)>0$ in $(0,t_0]$ by Proposition~\ref{propMP}(ii).
We may then apply Proposition~\ref{propZeroNumber} on $(\eps,T)$ for each $\eps\in(0,t_0]$ and the conclusion follows.

\smallskip

(ii) Instead of $\alpha x^{-1}$, our proof could handle a more general class of coefficients with singularity at $x=0$.
\end{rem} 

The following proposition shows that the number of intersections
with the singular steady state is constant near a GBU or RBC time.
It was used to define the number of vanishing intersections in subsection~\ref{secmainres}.

\begin{prop} \label{ZeroNumberConst} 	Let $0 < R\le \infty$ and $ T < \infty $.

	\smallskip
	
(i) Let $u$ be a viscosity solution of \eqref{equ} with $u_0\in \mathcal{W}$.
	If $ u $ undergoes GBU at $(x,t) = (0,T) $, then
    there exist $r\in (0,R]$, $t_1<T$ and an integer $m \ge 1$ such that
	\be{claimintersec}
	\hbox{for all $t\in(t_1,T)$, \ $u(\cdot,t)-U$ has exactly $m$ zeros on $(0,r)$} 
	\ee
	and
	\be{claimintersec-r}
	u(r,t) - U(r)\ne 0,\quad t_1\le t\le T.
	\ee
	Moreover, denoting $0<x_1(t)<\dots<x_m(t)$ the zeros of $u(\cdot,t)-U$ on~$(0,r)$, we have
	\be{claimliminfx1}
	\liminf_{t\to T_-}x_1(t)=0.
	\ee

	(ii) Set $Q=(0,R)\times(0, T)$ and
	let $u\in C^{2,1}(Q)\cap C_b(\overline Q)$ be a solution of problem \eqref{equREC}, which 
	undergoes RBC at $(x,t) = (0, T)$. Then the conclusion of assertion (i) remains valid.

    	\smallskip
	
    	 (iii) In assertions (i)-(ii)  the zeros from \eqref{claimintersec}  are nondegenerate and are $C^1$ functions of~$t$.
\end{prop}

\begin{proof}
We first consider the GBU case.	
If $R=\infty$, we take $r>0$ large enough, so that $U(r)>\|u_0\|_\infty\ge u(r,T)$ 
owing to \eqref{boundunifvisc}. 
If $R<\infty$ and $u(R,T)\ne U(R)$, we may take $r<R$ close to $R$ such that $u(r,T)\ne U(r)$.
If $R<\infty$ and $u(R,T)=U(R)$, then $u_x(R,T)=-\infty$ by \cite[Lemma~5.4]{PS3}, 
hence we may select $r<R$ close to $R$ such that $u(r,T)>U(r)$.
In each case, by continuity, 
\be{claimintersec3}
\hbox{there exists $t_1<T$ such that $u(r,t)\ne U(r)$ for all $t\in(t_1,T)$.}
\ee
 On the other hand, since $ u $ undergoes GBU at $(x,t) = (0,T) $, by definition
there exists $t_2\in(t_1,T)$ such that $u$ is $C^1$ up to $x=0$
for $t\in[t_1,T)$ with $u(0,t)=0$. Therefore for each $\hat T\in (t_2,T)$, we may select $r_0\in(0,r)$ such that 
$u<U$ in $(0,r_0]\times[t_1,\hat T]$.
The claim \eqref{claimintersec} then follows from Proposition~\ref{propZeroNumber0} applied with $u_1=u$, $u_2=U$,
 $x_0=r_0$ and $x_1=r$.

As for $m\ge 1$ and \eqref{claimliminfx1}, otherwise, we would have $u<U$ on $(0,\nu]\times [T-\nu,T]$ for some $\nu\in(0,T-T_1)$.
By continuity and the fact that $u(\cdot,T-\nu)\in C^1([0,\nu])$, we would get 
$u\le U_b$ on the parabolic boundary of $Q:=(0,\nu]\times [T-\nu,T)$ for $b$ small, hence $u\le U_b$ in $Q$
by the comparison principle, contradicting $\lim_{t\to T_-}u_x(0,t)=\infty$ (cf.~\eqref{eq:BU-rate-lower}).

In the RBC case, similarly as above, we get \eqref{claimintersec3} for some $r\in(0,R)$.
 Also, for $\eps>0$ small and each $\hat T\in(T-\eps,T)$, we may select $r_0\in(0,R)$ such that 
$u>U$ in $[0,r_0]\times[T-\eta,\hat T]$.
 The claim then follows from Proposition~\ref{propZeroNumber0}.
As for $m\ge 1$ and \eqref{claimliminfx1}, 
otherwise, we have $u\ge U$ on $[0,\eta]\times[\tau-\eta,\tau]$ for some $\eta>0$ small. 
Using \eqref{eq:u_x-LBC} we deduce from Proposition~\ref{propSMP2} that $u(0,\tau)>0$, which is a contradiction.

 Finally assertion (iii) follows from Proposition~\ref{propZeroNumber0} and the 
implicit function theorem. \end{proof}
 
\begin{rem} \label{rem-zerocurves}
Let $u_1,u_2$ be as in Proposition~\ref{propZeroNumber0} under assumption \eqref{hyp1propZeroNumber0}
(resp., Proposition~\ref{propZeroNumber}(ii)). Pick $T_0\in (t_0,t_1)$ and set  $I=[T_0,t_1)$, $\Omega=(x_0,x_1)$
(resp., $T_0\in (0,T)$, $I=[T_0,T)$, $\Omega=(0,R)$). Let $m$ be the number of zeros of $[u_1-u_2](\cdot,T_0)$ in $\Omega$.
It follows from these propositions and the strong maximum principle that the zeros of 
$u_1-u_2$ in $\Omega\times I$ can be represented by
$m$ continuous curves $x_i(t)$ such that:
\smallskip
\begin{itemize}

\item[1.]For $1\le i\le m$,
$x_i$ is defined on a maximal interval $J_i=\{T_0\}$ or $J_i=[T_0,\tau_i)$ with $\tau_i\in(T_0,T]$.

\smallskip

\item[2.]For $1\le i<j\le m$, we have $x_i<x_j$ on their common interval of existence.

\smallskip

\item[3.]if $\tau_i\in(T_0,T)$, then $x_i(\tau_i^-):=\lim_{t\to \tau_i^-}x_i(t)$ exists and $x_i$ ceases to exist either:

(a){\hskip 1mm}by vanishing (i.e.~$x_i(\tau_i^-)=0$), or
\smallskip

(b){\hskip 1mm}by collapsing with some of the other $x_j$ (i.e.~$x_j(\tau_i^-)=x_i(\tau_i^-)$).

\smallskip

\noindent In case (b), let $i_1<\cdots<i_k$ be the indices of all the curves which collapse together at time $\tau_i$ (including~$i$)

$\bullet$ If $k$ is even then all the collapsing curves cease to exist at $\tau_i$, i.e.~$\tau_{i_1}=\cdots=\tau_{i_k}=\tau_i$

$\bullet$ If $k$ is odd 
 then only one curve survives after $\tau_i$, namely~$\tau_{i_1}>\bar t$ and $\tau_{i_2}=\cdots=\tau_{i_k}=\tau_i$
 
(by convention the surviving curve is labeled with the smallest index).

\smallskip

\item[4.]For $1\le i\le m$, $x_i(t)$ is nondegenerate and $C^1$ for $t\in(T_0,\tau_i)$ except at times when some curves collapse with $x_i$.
\end{itemize}
\end{rem} 

\subsection{Existence of solutions with persistent singularities}

In the following proposition, we consider singular initial data in the space
$$\mathcal{W}_s=\bigl\{\psi\in W^{1,\infty}_{loc}(\Omega);\ \psi\ge 0,\ \psi_x-U_x\in L^\infty(\Omega)\bigr\}$$
and we show the existence of a unique global solution $u$ of the viscous Hamilton-Jacobi equation with persistent singularity at $x=0$
(with $u$ regular at $x=R$ in the case of a bounded interval).
We note that this result is of different nature from the well-posedness of problem \eqref{equ} mentioned in introduction,
since the initial data $u_0$ is {\it not} in the class $\mathcal{W}$ and $u$ will not be $C^1$ up to the boundary at any time $t>0$.
This existence result, and the additional regularity properties, will be useful for the construction of the special RBC viscosity solutions in Theorem~\ref{prop:special-LBC}.
By an obvious scaling argument it suffices to consider the cases $R=1$ and $R=\infty$.

\begin{prop} \label{locexistvapprox}
Let $p>2$, $R=1$ or $R=\infty$, set $\Omega=(0,R)$, $Q=\Omega\times(0,\infty)$,
$\hat Q=(\overline \Omega\setminus\{0\})\times(0,\infty)$ and let $u_0\in \mathcal{W}_s$.
If $R=1$, assume in addition that $\sup u_0\le \frac14$ and $u_0(x)=0$ on~$[\frac12,1]$.

\smallskip

(i) (Existence-uniqueness) There exists a unique $u\in C(\overline Q)\cap C^{2,1}(\hat Q)$, 
such that $u$ is a global solution of the problem
\be{equRECexist}
\left\{\ 
\begin{aligned}
u_t&=u_{xx}+|u_x|^p,&&\quad\hbox{in $Q$,} \\
u(R,t)&=0,&&\quad \hbox{in $(0,\infty)$ in the classical sense (if $R=1$),} \\
u(x,0)&=u_0(x),&&\quad \hbox{in $\Omega$} \\
\end{aligned}
\right.
\ee
and
\be{equRECregulC1}
u_x-U_x \in C_b\bigl(\overline\Omega\times(0,\infty)\bigr).
\ee
Moreover we have
\be{equRECsingux}
(u-U)_x(0,t)=0,\quad t>0
\ee
and 
\be{equRECder1}
\begin{cases}
\ \|u(t)\|_\infty\le \max(\ts\frac12,U(1)) \quad\hbox{and}\quad |u_x(1,t)|\le \max(1,U'(1)),\quad t>0,
&\quad\hbox{if $R=1$,}\\
\noalign{\vskip 1mm}
\ \hbox{$|u(x,t)|\le C(1+x+t)$ \quad in $[0,\infty)^2$ for some $C>0,$}
&\quad\hbox{if $R=\infty$.}\\
\end{cases}
\ee

\smallskip

(ii)  (Sign properties and continuous dependence) The solution $u(u_0;x,t)$ obtained in assertion~(i) enjoys the following properties.
For any $t_0>0$,
\be{equRECviscpos}
\hbox{if $u(0,\cdot)\ge 0$ on $[0,t_0]$ then }
\begin{cases}
\hbox{$u\ge 0$ on $[0,R)\times [0,t_0]$}\\
\noalign{\vskip 1mm}
\hbox{and $u(0,\cdot)=0$ on $[0,t_0]$ in the viscosity sense.}
\end{cases}
\ee
\be{equRECppty0}
\hbox{ If $R=\infty$ and $M_0:=\sup_\Omega u_0<\infty$, then $u\le M_0$ in $Q$.}
\ee
\be{equRECppty1}
\hbox{ If $u_0\le U$ in $\Omega$, then $u\le U$ in $Q$.}
\ee
Let $\hat u_0\in \mathcal{W}_s$ and assume $\hat u_0-u_0\in L^\infty(\Omega)$ in case $R=\infty$. Then
\be{equRECppty2}
\|u(u_0;\cdot,t)-u(\hat u_0;\cdot,t)\|_\infty \le \|u_0-\hat u_0\|_\infty,\quad t>0.
\ee

\smallskip

 (iii) (Additional regularity at $t=0$) Assume that, for some $A\in (0,R)$, $u_0-U$ extends to a $C^1$ function $v_0$ on $[0,A]$,
such that $v_0'(0)=0$ and $v_0''(0)$ exists.
Then, for any $x_0\in(0,A)$ and $\eps>0$, there exist $t_0>0$ such that
$$ |u_x(x,t)-u_{0,x}(x)|\le \eps x\quad\hbox{ in $(0,x_0]\times(0,t_0]$.}$$

\end{prop}

The next proposition shows that the hypothesis in \eqref{equRECviscpos} is necessary:
starting from $u_0\ge 0$, positivity need not be preserved
and we may even have $u(0,t)<0$ for all $t>0$.
Of course, if $u_0(0,0)>0$, then $u(0,t)$ remains positive for some time by continuity.
It may then possibly touch $0$ at some $t=\tau$,
and this is precisely the kind of solutions that we construct in Theorem~\ref{prop:special-LBC}.

\begin{prop} \label{locexistvapprox2}
Let $p>2$, $R=1$ and let $u_0, u$ be as in Proposition~\ref{locexistvapprox}.
Assume in addition that $u\le U-bx^2$ in $(0,1)$ for some $b>0$.
Then $u(0,t)<0$ for all $t>0$.
\end{prop}

\begin{rem} 
For $u_0$ as in Proposition~\ref{locexistvapprox2}, by \cite{BdaLio}
there also exists a unique, global {\it nonnegative} viscosity solution $\hat u\in C^{2,1}(Q)\cap C_b(\overline Q)$ of \eqref{equ}.
This solution $\hat u$ is obviously distinct from $u$ given by Propositions~\ref{locexistvapprox}-\ref{locexistvapprox2},
but this does not contradict the uniqueness of viscosity solutions nor the uniqueness part of Proposition~\ref{locexistvapprox}.
Indeed $u$ does not satisfy $u(0,t)=0$ in the viscosity sense for $t>0$,
whereas $\hat u$ is classical up to $x=0$ for $t>0$ small by \cite[Lemma~5.5]{PS3} (and thus does not satisfy
\eqref{equRECregulC1} in Proposition~\ref{locexistvapprox}).
The same remarks apply for $t\ge t_0$ to any solution $u$ in Proposition~\ref{locexistvapprox}
 such that $u(0,t_0)=0$ and $u(0,t)<0$ for $t>t_0$.
\end{rem}

\begin{proof}[Proof of Proposition~\ref{locexistvapprox}]
Due to the singularity of the initial data and of the sought-for solution,
the proof is far from immediate and requires several steps.

\smallskip

{\bf Step 1.} {\it Approximating problem.}
Writing $z=u-U$, the equivalent equation for $z$ is
$$z_t-z_{xx}=|U'+z_x|^p-{U'}^p.$$
We shall solve the corresponding problem for $z$ by an approximation argument. 
Thus for each $a>0$, recalling that the regular steady state $U_a$ is defined in \eqref{defUa},
we consider the regularized problem for the unknown $z_a$:
\be{equRECexistregul}
\left\{\ 
\begin{aligned}
z_{a,t}&=z_{a,xx}+|U_a'+z_{a,x}|^p-|U_a'|^p,&&\quad 0<x<R, \ t>0, \\
z_{a,x}(0,t)&=0,&&\quad t>0, \\
z_{a}(R,t)&=-U(R),&&\quad t>0\quad\hbox{ (if $R=1$)}, \\
z_a(x,0)&=z_{a,0}(x),&&\quad \hbox{in $(0,R)$,} \\
\end{aligned}
\right.
\ee
where $z_{a,0}(x)=u_0(x)-U(x)$ if $R=1$ and $z_{a,0}(x)=\max\bigl(\min(u_0(x)-U(x),a^{-1}),-a^{-1}\bigr)$ if $R=\infty$.
Since $z_{a,0}\in W^{1,\infty}(0,R)$ and $U_a$ is smooth, by standard theory, problem \eqref{equRECexistregul} admits a unique maximal, classical solution
$z_a\in C(\overline\Omega\times [0,\tau_a))\cap C^{2,1}(\overline\Omega\times (0,\tau_a))
\cap L^\infty_{loc}([0,\tau_a);W^{1,\infty}(0,R))$, where $\tau_a$ denote the maximal existence time.
We also have $z_{a,x}\in C([0,\tau_a);L^2_{loc}(\overline\Omega))$.
We proceed to establish uniform a priori estimates for $z_a$ and $z_{a,x}$
that will allow us to pass to the limit as $a\to 0$ (and will in turn guarantee the global existence of the $z_a$).

\smallskip
{\bf Step 2.} {\it A priori estimates.}
First consider the case $R=1$. Set $M_+:=\sup (u_0-U)$ and $M_-:=\inf (u_0-U)$
and note that $M_+\ge M_-=-U(1)$, due to $u_0\ge 0$ and $u_0(1)=0$.
By the maximum principle, we deduce that
\be{boundza}
M_-\le z_a(x,t)\le M_+,\qquad 0\le x<1,\ 0<t<\tau_a.
\ee
We claim that
\be{parabUK}
-U'(1)-1\le z_{a,x}(1,t)\le 0.
\ee
To show this we let $\bar z_a(x,t)= \frac12(1-x^2)+U_a(1)-U(1)-U_a(x)$ and observe that
$$\mathcal{P}_a\bar z_a:=\bar z_{a,t}-\bar z_{a,xx}-|U_a'+\bar z_{a,x}|^p+|U_a'|^p=1+U_a''-x^p+|U_a'|^p=1-x^p\ge 0$$
with $\bar z_a(1,t)= -U(1)$, $\bar z_{a,x}(0,t)= -U_a'(x)<0$. 
For all $a>0$ sufficiently small we have, in $[0,\frac12]$:
$$\bar z_a(x,0)-z_a(x,0)=\bar z_a(x,0)+U(x)-u_0(x)
\ge \ts\frac12(1-x^2)+U_a(1)-U(1)-u_0(x)\ge \ts\frac18+U_a(1)-U(1)\ge 0$$
and, in $[\frac12,1]$:
$$\begin{aligned}
\bar z_a(x,0)-z_a(x,0)
&=\bar z_a(x,0)+U(x)\ge (U(x)-U_a(x))-(U(1)-U_a(1))+\ts\frac12(1-x) \ge 0.
\end{aligned}$$
It follows from the maximum principle that $\bar z_a\ge z_a$ in $\overline\Omega\times [0,\tau_a)$.
Recalling \eqref{boundza} with $M_-=-U(1)$, we get
\be{parabUK0}
0\le  z_a(x,t)-z_a(1,t)=z_a(x,t)+U(1)\le \ts\frac12(1-x^2)+U_a(1)-U_a(x).
\ee
 Dividing by $1-x$ and letting $x\to 1$, we deduce \eqref{parabUK}.

\smallskip
Next consider the case $R=\infty$.
By assumption, there exists $M>0$ such that 
$|u_0(x)-U(x)|+U(x)\le M(1+x)$, $x\in [0,\infty)$.
We claim that, for all $a>0$ sufficiently small,
\be{controlzaRinfty}
|z_a(x,t)| \le M(1+x)+M^pt \quad\hbox{in $(0,\infty)\times(0,\tau_a)$.}
\ee
To show this we let $\zeta_a(x,t)= -U_a+M(1+x)+M^pt$. 
For all $x>0$, we have $-\zeta_a(x,0)\le z_a(x,0)=u_0(x)-U(x)\le \zeta_a(x,0)$.
Also, 
$$\mathcal{P}_a\zeta_a=M^p+U_{a,xx}-M^p+|U_a'|^p=0,\quad x>0$$
and using 
$2|U_a'|^p=2^{1-p}|2U_a'|^p\le 2^{1-p}(|2U_a'-M|+M)^p\le |2U_a'-M|^p+M^p$, we get
$$\mathcal{P}_a(-\zeta_a)
=-M^p-|2U_a'-M|^p+2|U_a'|^p\le 0,\quad x>0.$$
Moreover, $\zeta_{a,x}(0,t)=-d_pa^{-\beta}+M<0$ for all $t>0$ if we choose $a>0$ small.
On the other hand, since $z_a\in L^\infty_{loc}([0,\tau_a);W^{1,\infty}(0,\infty))$, for each $\tau<\tau_a$,
there exists $R_{a,\tau}$ such that $-\zeta_a\le z_a\le\zeta_a$ on $[R_{a,\tau},\infty)\times (0,\tau)$.
We may thus apply the comparison principle to deduce that that $-\zeta_a\le z_a\le\zeta_a$
in $(0,\infty)\times(0,\tau_a)$, hence \eqref{controlzaRinfty}.

\smallskip

Now we turn to estimate $z_{a,x}$ uniformly. 
Let $M_1:=\|u_{0,x}-U'\|_\infty$ if $R=\infty$ and $M_1:=\max\{\|u_{0,x}-U'\|_\infty,1+U'(1)\}$ if $R=1$.
Set $Z_a:=z_{a,x}$, $g(s)=p|s|^{p-2}s$. Using parabolic regularity we may differentiate in $x$ and we obtain
$$Z_{a,t}-Z_{a,xx}=g(U_a'+Z_a)(U_a''+Z_{a,x})-g(U_a')U_a''
=[g(U_a'+Z_a)-g(U_a')]U_a''+g(U_a'+Z_a)\,Z_{a,x}$$
hence
\be{parabPhi1}
\mathcal{L}_aZ_a:=Z_{a,t}-Z_{a,xx}-b_a(x,t)Z_a-\tilde b_a(x,t)Z_{a,x}=0,
\ee
where 
\be{parabPhi2}
b_a(x,t)=g'(U_a'+\theta_a(x,t)Z_a)U_a'',\quad \tilde b_a(x,t)=g(U_a'+Z_a),
\ee
 with $\theta_a(x,t)\in (0,1)$.
Moreover, $Z_a(0,t)=0$ and $|Z_a(1,t)|\le M_1$ (if $R=1$) by \eqref{parabUK}.
Since $b_a\le 0$ owing to $g'\ge 0$ and $U_a''\le 0$, it follows that $\pm M_1$ are super-/sub-solutions for $Z_a$
and the maximum principle (see~e.g.~\cite[Proposition 52.10]{QSbook} which applies also in the case $R\le\infty$) yields
\be{boundzax}
|z_{a,x}|\le M_1,\qquad 0\le x<R,\ 0<t<\tau_a.
\ee

\smallskip
{\bf Step 3.} {\it Convergence of $z_a$ and existence.}
Estimates \eqref{boundza}, \eqref{controlzaRinfty} and \eqref{boundzax} guarantee that $\tau_a=\infty$.
Moreover, using \eqref{equRECexistregul} and parabolic estimates, 
it follows that $z_a$ is relatively compact in $C^{2,1}_{loc}(\hat Q)$.
 Passing to the limit $a\to 0^+$, we obtain a solution $z\in C^{2,1}(\hat Q)$ of
\be{equRECexistregul2}
\left\{\ 
\begin{aligned}
z_t&=z_{xx}+|U'+z_x|^p-|U'|^p,&&\quad 0<x<R, \ t>0, \\
z(R,t)&=-U(R),&&\quad t>0\quad\hbox{ (if $R=1$)}, \\ 
z(x,0)&=u_0(x)-U(x),&&\quad \hbox{in $(0,R)$,} \\
\end{aligned}
\right.
\ee
and 
\eqref{boundzax} guarantees that 
\be{equREClipestim}
z_x\in L^\infty(\Omega\times(0,\infty)). 
\ee
Moreover, for $R=1$, \eqref{parabUK} and \eqref{parabUK0} imply
\be{parabUK2}
-U(1)\le z\le \ts\frac12(1-x^2)-U(x)\hbox{ in $Q$ \quad and}\quad-U'(1)-1\le z_x(1,t)\le 0.
\ee
The claimed solution is then given by $u:=z+U$
(in particular \eqref{equRECder1} follows from
\eqref{controlzaRinfty} and \eqref{parabUK2}), except for the properties $u\in C(\overline Q)$, \eqref{equRECregulC1}
and \eqref{equRECsingux}, that we shall establish below.

\smallskip

Let us check properties \eqref{equRECviscpos}-\eqref{equRECppty2}.
 To prove \eqref{equRECviscpos} we assume $u(0,\cdot)\ge 0$ on $[0,t_0]$ for some $t_0>0$.
If $R=1$, we immediately get $u\ge 0$ in $[0,R)\times [0,t_0]$ by the maximum principle.
If $R=\infty$, we set $u_\eps(x,t)=u(x,t)+\eps(x^2+2t)$, which satisfies $u_{\eps,t}-u_{\eps,xx}=|u_x|^p\ge 0$.
Moreover, $u_\eps(x,t)\ge 0$ in $[R_\eps,\infty)\times [0,t_0]$ for $R_\eps>0$ large
as a consequence of \eqref{equRECder1}. 
Applying the maximum principle to $u_\eps$ in $[0,R_\eps]\times [0,t_0]$, we deduce that
$u_\eps\ge 0$ in $[0,\infty)\times [0,t_0]$, hence $u\ge 0$ by passing to the limit $\eps\to 0$.

On the other hand, for any $t\in (0,t_0]$, since $u_x(0,t)=\infty$, 
the set of smooth functions touching $u$ from above at $(0,t)$ is empty.
It then follows from Definition~\eqref{defviscBC} that 
$u(0,\cdot)=0$ on $[0,t_0]$ in the viscosity sense.

 To prove \eqref{equRECppty0}, let us fix $\eps>0$. For $a>0$, 
we set $\bar z_{a,\eps}(x,t):=M_0-U_a+\eps^p t+\eps x$, which satisfies
$$\partial_t \bar z_{a,\eps}-\partial^2_x\bar z_{a,\eps}-|U_a'+\partial_x\bar z_{a,\eps}|^p+|U_a'|^p=
\eps^p+U_a''-\eps^p+|U_a'|^p=0.$$
For $a>0$ small (depending on $\eps$), we also have  
$$\bar z_{a,\eps}(\cdot,0)\ge u_0-U+\eps x\ge
\max\bigl(\min(u_0(x)-U(x),a^{-1}),-a^{-1}\bigr)=z_a(\cdot,0)$$
and $\partial_x\bar z_{a,\eps}(0,t)=-U_a'(0)+\eps<0$.
Moreover, for each $T>0$, owing to $z_a\in L^\infty_{loc}([0,\infty);L^\infty(0,\infty))$, we have
$\bar z_{a,\eps}\ge M_0-U+\eps x \ge z_a$ in $[R_{a,\eps,T},\infty)\times [0,T)$ for some $R_{a,\eps,T}>0$ large.
It follows from the comparison principle applied in $(0,R_{a,\eps,T})\times (0,T)$ that $z_a\le\bar z_{a,\eps}$
in $(0,\infty)\times (0,T)$.
Letting $a\to 0$, next $\eps\to 0$ and then $T\to\infty$, we obtain $z\le M_0-U$ i.e., \eqref{equRECppty0}.

For \eqref{equRECppty1}, it suffices to note that if $u_0\le U$ in $\Omega$, then $z_a(\cdot,t)\le 0$ in $(0,R)$ in view of 
 \eqref{equRECexistregul} and the maximum principle,
and to let $a\to 0$.

Now, to prove \eqref{equRECppty2}, take $\hat u_0$ as in the statement 
and let $\hat z_a, \hat z$ be the corresponding solutions obtained by the above procedure. 
By the comparison principle applied to problem \eqref{equRECexistregul}
with the comparison functions $z_a\pm  \|\hat u_0-u_0\|_\infty$, we deduce that 
\be{compzahat}
|\hat z_a-z_a|\le \|\hat u_0-u_0\|_\infty.
\ee
Passing to the limit and going back to $u$, we obtain \eqref{equRECppty2}.

\smallskip

{\bf Step 4.} {\it Continuity of $u$ and uniqueness.}
We shall show that $z$, hence $u$, belongs to $C(\overline Q)$.
Fix any $\eps>0$ and pick $\hat u_0\in \mathcal{W}_s\cap C^2(\Omega)$
such that $|\hat u_0-u_0|\le \eps$ in $\Omega$.
 Assume in addition that $\sup \hat u_0\le \frac14$, $\hat u_0=0$ on~$[\frac12,1]$ if $R=1$, 
or $\sup_{x\ge 1/2}|\hat u_{0,xx}|<\infty$ if $R=\infty$.
Let $\hat z_a$ be the corresponding solutions of \eqref{equRECexistregul} obtained in Steps~2-3. 
Fix $\zeta\in C^2(\R)$ such that $\zeta(s)=0$ for $s\le 1$ and $\zeta(s)=1$ for $s\ge 2$,
 and set $\zeta_\eta(x)=\zeta(x/\eta)$. For $\eta\in(0,1/4)$ and $K>0$ to be chosen below, 
we then define the comparison function
$$z_\pm=(1-\zeta_\eta(x))\hat u_0(0)+[\zeta_\eta(\hat u_0-U)](x)\pm(Kt+\eps),\quad x\in\overline\Omega, \ t>0.$$
We first select $\eta\in (0,1/4)$ 
such that for all 
$x\in \Omega$ and $t>0$,
$$\begin{aligned}
z_-(x,t)-\hat z_a(x,0)
&=(1-\zeta_\eta)\hat u_0(0)+\zeta_\eta(\hat u_0-U)+Kt-\hat u_0(x)+U(x)-\eps \\
&=(1-\zeta_\eta)(\hat u_0(0)-\hat u_0(x)+U(x))+Kt-\eps \\
&\le \sup_{x\in(0,2\eta)}|\hat u_0(0)-\hat u_0(x)|+c_p(2\eta)^{1-\beta}-\eps+
Kt \le Kt
\end{aligned}$$
hence, arguing similarly for $z_+$,
\be{compzpm}
z_-(x,t)-Kt\le \hat z_a(x,0)\le z_+(x,t)+Kt.
\ee
On the other hand, simple computations show that
$z_{\pm,x}=z_{\pm,xx}=0$ in $(0,\eta]$ and, for some $C(\eta)>0$ independent of~$a$,
$\bigl|z_{\pm,xx}+|U_a'+z_{\pm,x}|^p-|U_a'|^p\bigr|\le C(\eta)$
in $[\eta, R)$. Consequently, 
choosing $K=C(\eta)$, we obtain
$$z_{\pm,t}-z_{\pm,xx}-|U_a'+z_{\pm,x}|^p+|U_a'|^p
=\pm K-z_{\pm,xx}-|U_a'+z_{\pm,xx}|^p+|U_a'|^p\ge 0 \ (\le 0)
\quad\hbox{ in $Q$.}$$
Since $z_-(x,0)\le \hat z_a(x,0)\le z_+(x,0)$ by \eqref{compzpm},
$z_{\pm,x}(0,t)=0$ and, if $R=1$, $z_+(1,t)\ge -U(1)$ and $z_-(1,t)\le -U(1)$,  the maximum principle implies
$$z_-(x,t)\le \hat z_a(x,t)\le z_+(x,t).$$
Combining this with \eqref{compzahat} and \eqref{compzpm}, we deduce
$$z_a(x,h)\ge \hat z_a(x,h)-\eps\ge z_-(x,h)-\eps=z_+(x,0)-(Kh+3\eps)
\ge  \hat z_a(x,0)-(Kh+3\eps)\ge z_a(x,0)-5\eps$$
for any $h\in (0,h_0)$, with $h_0=h_0(\eps)>0$ sufficiently small.
This allows to compare $z_a$ with $\bar z_a(x,t):=z_a(x,t+h)+5\eps$ through \eqref{equRECexistregul}, to deduce
$z_a(x,t+h)\ge z_a(x,t)-5\eps$ in $\overline Q$ for any $h\in (0,h_0)$.
Arguing similarly from above and letting $a\to 0$, we get $|z(x,t+h)-z(x,t)|\le 5\eps$,
hence the continuity in time of $z$. Combining with \eqref{equREClipestim}, 
we conclude that $z$ and $u\in C(\overline Q)$.
 As a consequence, uniqueness follows from Proposition~\ref{propSMP2}
(using also \eqref{equRECder1} in case $R=\infty$).

\smallskip

{\bf Step 5.} {\it Proof of $C^1$ regularity up the boundary, \eqref{equRECregulC1} and \eqref{equRECsingux}.}
To this end we shall look for a more precise estimate of $Z_a=z_{a,x}$ near $x=0$,
 based on a self-similar comparison function.
Take any $m\in(0,1)$, fix $T> 1$ and let 
$V(y)=A\,[\min(y,y_0)]^m$ for $y\ge 0$,
with $A, y_0>0$ to be chosen. Then putting $\xi=(x+a)t^{-1/2}$ for given $a>0$, we define
$$\bar Z_a(x,t)=
\begin{cases}
V(\xi),& x\ge 0,\ t>0, \\
 \noalign{\vskip 1mm}
Ay_0^m, & x\ge 0,\ t=0,
\end{cases}
$$
and note that $\|\bar Z_a(\cdot,t)-\bar Z_a(\cdot,0)\|_{L^2(\Omega)}\to 0$, as $t\to 0$.
With the notation from \eqref{parabPhi1}, \eqref{parabPhi2}, a simple computation yields, for $x+a\ne y_0t^{1/2}$,
\be{PaZa}
\mathcal{L}_a\bar Z_a=-\ts\frac12 t^{-3/2}(x+a)V'(\xi)-t^{-1}V''(y)-b_aV(\xi)-\tilde b_at^{-1/2}V'(\xi).
\ee
There is a constant $c=c(p)>0$ such that
 $(1+X)^{p-1}\le 1+cX$ and $(1-X)^{p-2}\ge 1-cX>0$ for all $X\in(0,\frac12]$.
Set $\hat M_1=cd_p^{-1}M_1$, with $M_1$ in \eqref{boundzax} and select $x_0\in (0,1)$ such that 
$\hat M_1(2x_0)^{\beta}<\frac12$. 
Fix any $a\in(0,x_0)$ and $(x,t)\in D_T:=(0,x_0)\times(0,T)$. With $\alpha=p/(p-1)$, we then have
$$\begin{aligned}
\tilde b_a(x,t)&\le p\bigl[d_p(x+a)^{-\beta}+M_1\bigr]^{p-1}\le \alpha(x+a)^{-1}\bigl[1+\hat M_1(x+a)^{\beta}\bigr],\\
b_a(x,t)&\le -p(p-1)\bigl[d_p(x+a)^{-\beta}-M_1\bigr]^{p-2}[d_p(x+a)^{-\beta}]^p
\le -\alpha(x+a)^{-2}\bigl[1-\hat M_1(x+a)^{\beta}\bigr]<0.
\end{aligned}$$
If $x+a\ne y_0t^{1/2}$, it follows from \eqref{PaZa} that
\be{PaZa2}
t\mathcal{L}_a\bar Z_a \ge-\ts\frac12 \xi V'(\xi)-V''(\xi)+\alpha\bigl[1-\hat M_1(x+a)^{\beta}\bigr]\xi^{-2}V(\xi)
-\alpha\bigl[1+\hat M_1(x+a)^{\beta}\bigr]\xi^{-1}V'(\xi).
\ee
For $x+a>y_0t^{1/2}$, we have $V'(\xi)=V''(\xi)=0$ and $V(\xi)>0$, hence the right hand side of \eqref{PaZa2} is positive,
whereas for $x+a<y_0t^{1/2}$, we have 
$$\begin{aligned}
A^{-1}t\mathcal{L}_a\bar Z_a 
&\ge -\ts\frac{m}{2}\xi^m+m(1-m)\xi^{m-2}+\alpha[1-m-(m+1)\hat M_1(x+a)^{\beta}]\xi^{m-2}\\
&\ge \bigl\{(\alpha+m)(1-m)-\bigl[\alpha(m+1)\hat M_1(2x_0)^{\beta}+\ts\frac{m}{2}y_0^2\bigr]\bigr\}\xi^{m-2}\ge 0
\end{aligned}$$
by choosing $y_0=x_0T^{-1/2}\le x_0$ and taking $x_0>0$ smaller if necessary, depending only on $M_1,p,m$.
Since $V'(y_0^-)>0=V'(y_0^+)$, it follows that $\bar Z_a$ is a weak solution of $\mathcal{L}_a\bar Z_a\ge 0$ in $D_T$.
Now choose $A=M_1y_0^{-m}$ and note that $\bar Z_a(x_0,t)=V((x_0+a)t^{-1/2})=M_1$ for all $t\in (0,T)$.
Recalling \eqref{boundzax}, we thus have $Z_a\le\bar Z_a$ on the parabolic boundary of $D_T$,
hence $Z_a\le \bar Z_a$ in $D_T$ by the maximum principle.
 Using $-\bar Z_a$ similarly as a subsolution
and passing to the limit $a\to 0$, we deduce that
$$|z_x(x,t)|\le A\,[\min(xt^{-1/2},y_0)]^m,\quad 0<x<x_0,\ 0<t<T.$$
It follows that $z_x$ extends to a function $z_x\in C(\overline\Omega\times(0,\infty))$
with $z_x(0,t)=0$, hence \eqref{equRECsingux}.

Finally, this along with \eqref{equREClipestim} and $u\in C(\overline Q)$ yields 
\eqref{equRECregulC1}. 

\smallskip
 
{\bf Step 6.} {\it Proof of assertion (iii).}
Set $N=v_{0,xx}(0)\in\R$. Owing to our assumption, there exists $x_1\in(0,\eps)$ such that
\be{v0xMeps1}
(N-\eps)x\le v_{0,x}\le (N+\eps)x\quad\hbox{ for all $x\in (0,x_1]$.}
\ee
Let $Z_a$ and $\mathcal{L}_a$ be as in Step 2.
Let $a\in(0,\delta)$, where $\delta\in(0,x_1)$ is chosen sufficiently small 
so that $|K(x+a)|+x^2\le \frac12 U_a'(x)$ for all $a,x\in(0,\delta)$.
For $K,\sigma\in\R$ and $b\in\{0,a\}$, define the comparison function $\psi=\psi_{K,b,\sigma}:=K(x+b)-\sigma x^2$.
We obtain, for $x\in (0,\delta)$, 
$$\begin{aligned}
p^{-1}&\bigl(\mathcal{L}_a \psi-2\sigma\bigr)\\
&= -\bigl(U_a'+K(x+b)-\sigma x^2\bigr)^{p-1}(U_a''+K-2\sigma x)+{U_a'}^{p-1}U_a''\\
&={U_a'}^{2p-1}\Bigl\{\Bigl[1+\ts\frac{K(x+b)-\sigma x^2}{U_a'}\Bigr]^{p-1}-1\Bigr\}
-(K-2\sigma x){U_a'}^{p-1}\Bigl[1+\ts\frac{K(x+b)-\sigma x^2}{U_a'}\Bigr]^{p-1}\\
&=(p-1){U_a'}^{2p-2}\bigl[K(x+b)-\sigma x^2\bigr]+{U_a'}^{2p-3}\,O\bigl((x+b)^2\bigr)
-(K-2\sigma x){U_a'}^{p-1}+{U_a'}^{p-2}\,O(x+b)\\
&=K{U_a'}^{p-1}\bigl[(p-1){U_a'}^{p-1}(x+b)-1\bigr]+\sigma x{U_a'}^{p-1}\bigl[2-(p-1){U_a'}^{p-1}x\bigr]+O\bigl((x+a)^\beta\bigr)\\
&=K{U_a'}^{p-1}\bigl[\ts\frac{x+b}{x+a}-1\bigr]+\sigma x{U_a'}^{p-1}\bigl[1+\ts\frac{a}{x+a}\bigr]+O\bigl((x+a)^\beta\bigr).
\end{aligned}$$
Taking $\delta>0$ smaller if necessary, it follows that, for $x\in (0,\delta)$,
$\mathcal{L}_a \psi\le 0$ if $\sigma=-1$ and either $K\le 0$, $b=a$ or $K\ge 0$, $b=0$, 
and that $\mathcal{L}_a \psi\ge 0$ if $\sigma=1$ and either $K\ge 0$, $b=a$ or $K\le 0$, $b=0$.

We thus choose $\underline\psi=\psi_{N-3\eps,b_1,-1}$ and $\overline\psi=\psi_{N+3\eps,b_2,1}$
with $b_1=a$ if $N-3\eps\le 0$, $b_1=0$ if $N-3\eps>0$, 
$b_2=a$ if $N+3\eps\ge 0$ and $b_2=0$ if $N+3\eps<0$.
Note that, since $\delta<\eps$, we have $\underline\psi(x)\le (N-2\eps)x$ 
and $\overline\psi(x)\ge (N+2\eps)x$ in $[0,\delta]$.

On the other hand, since the coefficients $b_a, \tilde b_a$ of $\mathcal{L}_a$ in 
\eqref{parabPhi2} are bounded on compact subsets of $(0,A]\times[0,1]$, 
uniformly with respect to $a>0$, and $Z_a(0,x)=v_{0,x}$ is continuous on $(0,A)$,
it follows from standard parabolic barrier arguments that 
\be{v0xMeps2}
\hbox{$Z_a$ is continuous on $(0,x_0]\times[0,1]$, uniformly in $a$.}
\ee
Therefore, in view of \eqref{v0xMeps1}, there exists $t_0\in (0,1)$ such that, for all $a\in(0,\delta)$,
$$(N-2\eps)\delta\le Z_a(\delta,t)\le (N+2\eps)\delta \quad\hbox{ for all $t\in (0,t_0]$,}$$
Moreover,  we have $\underline\psi\le v_{0,x}=0\le \overline\psi$ at $x=0$
and, by \eqref{v0xMeps1},
$$\underline\psi\le v_{0,x}\le\overline\psi\quad\hbox{ in $[0,\delta]$.}$$
It follows from the comparison principle that $\underline\psi\le Z_a\le\overline\psi$ in $(0,\delta)\times(0,t_0)$.
Passing to the limit $a\to 0$, it follows that 
$$(N-3\eps)x\le u_x-U' \le(N+3\eps)x \quad\hbox{ in $(0,\delta)\times(0,t_0)$}$$
hence, using \eqref{v0xMeps1} again,
\be{v0xMeps3}
|u_x-u_{0,x}| \le 4\eps x\quad\hbox{ in $(0,\delta)\times(0,t_0)$.}
\ee
Finally, as a consequence of \eqref{v0xMeps2}, $u_x-U'$ hence $u_x$ is continuous on $(0,A)\times[0,1]$,
thus \eqref{v0xMeps3} remains true on in $(0,x_0]\times(0,t_0)$ for a possibly smaller $t_0$,
which implies the assertion.
\end{proof}

\begin{proof}[Proof of Proposition~\ref{locexistvapprox2}]
By Proposition~\ref{locexistvapprox}, we have $u\le U$, hence in particular $u(0,t)\le 0.$ 
Also, since $u-U\in L^\infty(0,\infty;W^{1,\infty}(0,1))$, there exists $L>0$ such that 
\be{compuUL}
u(x,t)-U(x)\ge u(0,t)-Lx\quad\hbox{ in $(0,1)\times(0,\infty)$.}
\ee
On the other hand, by the proof \cite[Lemma~5.5]{PS3}, there exist $\delta,\eta>0$ and $c\in (0,b)$ such that the function 
$$\overline u(x,t)=U_{a(t)}-cx^2,\quad a(t)=\eta t^{1-\beta}$$
satisfies 
$\overline u_t-\overline u_{xx}\ge |\overline u_x|^p$ and $\overline u\ge 0$ in $Q_\delta:=(0,1)\times(0,\delta]$.
Since $u\le \overline u$ on the parabolic boundary of $Q_\delta$, 
it follows from the comparison principle that 
\be{compuUL2}
u(x,t)\le U_{a(t)}-cx^2\le U'(a(t))x\quad\hbox{ in $Q_\delta$.}
\ee
For each $t\in(0,\delta)$, combining \eqref{compuUL}, \eqref{compuUL2} and choosing $x=x(t)>0$ sufficiently small, we get
$$u(0,t)\le u(x,t)-U(x)+Lx\le (U'(a(t))+L)x-U(x)<0.$$
Finally, since $u(\cdot,\delta)\le \overline u(\cdot,0)$, we may repeat
the argument on $(\delta,2\delta]$ and so on. This yields the conclusion.
\end{proof}

\subsection{Variation of constants formula and extension property}

The construction of special GBU and RBC solutions in
Theorems~\ref{prop:special} and \ref{prop:special-LBC} 
will be carried out by working on the equation for $w$  
and on the corresponding variation of constants formula for $v=w-U$,
that we provide in Proposition~\ref{locexistvapproxExt} below.

Before that, in order to handle the case of a bounded interval,
it is convenient to modify the problem so as to keep carrying out the construction on the half-line $(0,\infty)$,
where the eigenfunctions and the kernel of the operator $\mathcal{L}$ are explicitly described.
This can be done via a suitable extension of the solution of the viscous Hamilton-Jacobi equation, given in the following lemma.
This procedure generates additional terms in the equations (for $u, w$ and $v$), 
but in the process of construction, these terms will be suitably bounded and supported far away from the singularity, 
so that the construction will not be significantly affected.

\begin{lem} \label{BernsteinEst}
Let $R,\tau>0$, $s_0=-\log \tau$ and set $Q=(0,R)\times(0,\tau)$.
Fix a cut-off function $\zeta\in C^2([0,\infty))$ such that 
 $\zeta=1$ in $[0,\frac{R}{3}]$ and $\zeta=0$ in $[\frac{R}{2},\infty)$.
Assume that $u\in C^{2,1}(Q)\cap C(\overline Q)$ satisfies 
$u_t-u_{xx}=|u_x|^p$ in $Q$
and define 
\be{defextuwv0}
\begin{aligned}
\tilde u(x,t)&=\zeta(x) u(x,t) 
&&\quad\hbox{ in $[0,\infty)\times[0,\tau)$,} \\
\tilde w(y,s)&=e^{ks}\tilde u(ye^{-s/2},\tau-e^{-s})=\zeta(ye^{-s/2}) w(y,s)
&&\quad\hbox{ in $[0,\infty)\times[s_0,\infty)$,} \\
\tilde v(y,s)&=\tilde w(y,s)-U(y) &&\quad\hbox{ in $[0,\infty)\times[s_0,\infty)$.}
\end{aligned}
\ee
Then $\tilde u$, $\tilde w$, $\tilde v$ respectively satisfy
\be{Extension1}
\begin{aligned}
\tilde u_t-\tilde u_{xx}
&= g(x,t):=|u_x|^p\zeta -2u_x \zeta_x-u\zeta_{xx}, \\
\tilde w_s-\tilde w_{yy}+\frac{y}{2}\tilde w_y-k\tilde w
&=\tilde g(y,s):=e^{(k-1)s}g(ye^{-s/2},\tau-e^{-s}), \\
\tilde v_s+\mathcal{L}\tilde v&=\tilde F(\tilde v_y,s):=\tilde g(y,s)-U_y^p-pU_y^{p-1}\tilde v_y.
\end{aligned}\ee
\end{lem} 

We omit the proof, which follows from immediate calculations.
Based on Proposition~\ref{varconstf}, 
we now give the suitable variation of constants formulas 
in all cases (GBU and RBC cases, half-line and bounded interval).
As before it suffices to consider the cases $R=1$ and $R=\infty$.

\begin{prop} \label{locexistvapproxExt} 
Let $p>2$, $R=1$ or $R=\infty$, $s_0>0$, $\Omega=(0,R)$.
\smallskip

(i) We suppose either:
$$\begin{cases}
u_0\in \mathcal{W},		& \hbox{ if $R=1$,} \\
 \noalign{\vskip 1mm}
u_0\in \mathcal{W}_1,	& \hbox{ if $R=\infty$ with $\mathcal{W}_1$ in \eqref{defW1},} \\
\end{cases}
\leqno\hbox{\rm (GBU case)}$$
or
$$
\hbox{$u_0\in \mathcal{W}_s$,
with $\sup u_0\le \frac14$ and $u_0(x)=0$ on~$[\frac12,1]$ if $R=1$.}
\leqno\hbox{\rm (RBC case)}$$
Next we denote 
$$u=
\begin{cases}
\hbox{the maximal classical solution of \eqref{equ}},& \hbox{ in GBU case,} \\
 \noalign{\vskip 1mm}
\hbox{the global solution of \eqref{equRECexist}},& \hbox{ in RBC case,} \\
\end{cases}$$
$$S=
\begin{cases}
-\log(e^{-s_0}-T),& \hbox{in GBU case with $T:=T^*(u_0)<e^{-s_0}$,} \\
 \noalign{\vskip 1mm}
\infty,& \hbox{otherwise.} \\
\end{cases}$$
Let $\tilde v, \tilde F$ be defined by \eqref{defextuwv0}-\eqref{Extension1}, with $\zeta$ as in Lemma~\ref{BernsteinEst} if $R=1$,
or $\zeta\equiv 1$ if $R=\infty$.
Then $\tilde v$ satisfies
\be{varconstv2tilde}
\tilde v(s)=e^{-(s-s_0)\mathcal{L}}\tilde v(s_0)+\int_{s_0}^s e^{-(s-\tau)\mathcal{L}}\tilde F(\tilde v_y(\sigma))\,d\sigma, \quad s_0<s<S.
\ee
Moreover $(e^{-s\mathcal{L}})_{s\ge 0}$ in \eqref{varconstv2tilde} satisfies formulas 
\eqref{defWkernel} and \eqref{computWx} in Proposition~\ref{KernelEstim2}.

\smallskip
 (ii) The above remains true if
$u\in C^{2,1}(Q)\cap C_b(\overline Q)$ is a solution of the RBC problem \eqref{equREC},
with any $s_0>-\log \tau$ and $S=\infty$.
\end{prop}

\begin{proof} 
First consider the GBU case in assertion (i).
If $R=\infty$ then we have $w\in C^{2,1}([0,\infty)\times(s_0,S))\cap C([0,\infty)\times[s_0,S)) \cap L^\infty_{loc}([s_0,S); \mathcal{W}_1)$, as a consequence of Proposition~\ref{RemWPw}(i). Noting that $\mathcal{W}_1\hookrightarrow L^\infty_{0,1}$ and $U\in H=H^1_\rho(0,\infty)$,
and using Lemma~\ref{CSHprime} with $m=0$ and $q=1$, we deduce that 
\be{regulvwU}
v\in C([s_0,S);H') \cap C((s_0,S);H).
\ee
On the other hand, setting $f(y,s):=F(v_y)=|U_y+v_y|^p-U_y^p-pU_y^{p-1}v_y$, we have
\be{regulvwUf}
|f(y,s)|\le C(|v_y|^p+U_y^p)\le C(|w_y|^p+U_y^p),
\ee
 hence $f\in L^\infty_{loc}([s_0,S);L^\infty_{\alpha,0})$
with $\alpha=\beta+1=p/(p-1)$.
Since $\alpha<2$, we may apply Lemma~\ref{CSHprime} with $m:=\alpha<\frac{\alpha+3}{2}$ and $q=0$
and deduce that 
$f\in C_b((s_0,s_1];H')$ for all $s_1\in (s_0,S)$.
Since $v$ solves \eqref{eqz}, it then follows from Proposition~\ref{varconstf} that $v$ satisfies \eqref{varconstv2tilde}, and Proposition~\ref{KernelEstim2} applies as well.
 
 \smallskip
 
 If $R=1$ then \eqref{regulvwU} is still clearly satisfied with $v$ replaced by $\tilde v$.
 On the other hand, 
since $g\in L^\infty_{loc}([0,T), L^\infty(0,R))$,
we have $\tilde g,\tilde w_y\in L^\infty_{loc}([s_0,S), L^\infty(0,\infty))$
and we conclude similarly as before that \eqref{varconstv2tilde} holds.

\smallskip

Next consider the RBC case. Under the assumptions of (i), as a consequence of \eqref{equRECregulC1} and \eqref{equRECder1} in 
Proposition~\ref{locexistvapprox}, 
  also recalling that $\zeta=0$ in $[\frac{R}{2},\infty)$ if $R=1$, $\tilde v$ satisfies \eqref{regulvwU}. 
 Under the assumptions of (ii), 
it follows from \eqref{eq:u_x-LBC-0}, using also \eqref{boundSZux} and the boundedness of $u$ if $R=\infty$, that
for each $t_0\in(0,\tau)$, 
\be{bounduxU-RBCgeneral}
|u_x- U'| \le C(t_0) \quad\hbox{ in $(0,R/2)\times(t_0,\tau)$,}
\ee
hence $\tilde v$ again satisfies \eqref{regulvwU}.
If $R=\infty$, since $v$ still solves \eqref{eqz} and $f(y,s):=F(v_y)$ satisfies \eqref{regulvwUf},
we can conclude as in the RBC case.
  The case $R=1$ follows from simple modifications, using again
\eqref{equRECregulC1} or \eqref{bounduxU-RBCgeneral} to control the cut-off terms.
\end{proof}

\subsection{A technical lemma}

The following lemma provides estimates on certain Gaussian integrals with parameters.
They will be used repeatedly in the derivation of our key a priori estimate.

\begin{lem} \label{AuxilLem2}
Let $C_1,\alpha>0$. 
\smallskip

 (i) Let $m\in\{0,1\}$. There exist $C,\tilde C>0$ such that
for all $X_0, X_1, X>0$ with $4X_0\le X_1\le 2X$, we have
 \be{AuxilLem2b}
 \int_{X_0}^{X_1} e^{-C_1(X-z)^2} (1+Xz)^{-\alpha/2} z^{-2m}dz
\le\frac{\tilde C X_{1-m}^{1-2m}e^{-CX^2}}{(1+XX_0)^{\alpha/2}}
+\frac{\tilde C X_1^{1-2m}{\bf 1}_{\{X<2X_1\}}}{(1+X_1^2)^{(\alpha+1)/2}}.
\ee

 (ii) Let $m\ge 0$. There exist $C,\tilde C>0$ such that
\be{P2AuxilLem2}
I(X,Z):=\int_{Z}^\infty 
e^{-C_1(X-z)^2}(1+Xz)^{-\alpha/2} z^m dz\le  
C\bigl(1+X^{m-\alpha}{\bf 1}_{\{X\ge 1\}}\bigr), 
\quad X>0,\ Z\ge 0,
\ee
and $I(X,Z)\le Ce^{-\tilde CZ^2}$ if $0<X\le Z/2$.
\end{lem}

\begin{proof}
 (i) We first claim that, for all $Z\in [X_1/2,2X_1]$,
\be{intX0Z}
I_m:=\int_{X_0}^{Z} \frac{e^{-C_1(Z-z)^2}(1+Xz)^{-\alpha/2}dz}{z^{2m}}
\le \frac{X_{1-m}^{1-2m}e^{-CX_1^2}}{(1+XX_0)^{\alpha/2}}
+ \frac{\tilde CX_1^{1-2m}}{(1+XX_1)^{\alpha/2}(1+X_1)}.
\ee
To this end we write
$I_m=\int_{X_0}^{Z/2} +\int_{Z/2}^{Z} \equiv I_m^1+I_m^2$
and first observe that
$$I_m^1\le (1+XX_0)^{-\alpha/2}  e^{-CZ^2} \int_{X_0}^{Z/2} z^{-2m}dz\le
X_{1-m}^{1-2m} e^{-CX_1^2}(1+XX_0)^{-\alpha/2}.$$
Next, since $X_1\le 2X$ and using $\min(Z,1)\le CZ/(Z+1)\le CX_1/(X_1+1)$, we get
$$I_m^2\le \frac{C(1+XZ)^{-\alpha/2}}{Z^{2m}}  \int_{Z/2}^{Z} e^{-C_1(Z-z)^2} dz
\le  \frac{C(1+XZ)^{-\alpha/2}\min(Z,1)}{Z^{2m}}\le \frac{C(1+XX_1)^{-\alpha/2}X_1^{1-2m}}{1+X_1},$$
hence \eqref{intX0Z}.

Now, if $X\ge 2X_1$, using $(X-z)^2\ge (X-X_1)^2+(X_1-z)^2\ge CX^2+(X_1-z)^2$ for $z\le X_1$ 
and applying \eqref{intX0Z} with $Z=X_1$, we obtain (separating the cases $m=0$ and $m=1$)
$$\begin{aligned}
\int_{X_0}^{X_1} 
&e^{-C_1(X-z)^2}(1+Xz)^{-\alpha/2} z^{-2m} dz\le e^{-CX^2}\int_{X_0}^{X_1} e^{-C(X_1-z)^2}(1+Xz)^{-\alpha/2} z^{-2m}  dz\\
&\le Ce^{-CX^2} \left\{\frac{X_{1-m}^{1-2m} e^{-CX_1^2}}{(1+XX_0)^{\alpha/2}}
+\frac{X_1^{1-2m}}{(1+XX_1)^{\alpha/2}(1+X_1)}\right\}
\le \frac{CX_{1-m}^{1-2m} e^{-CX^2}}{(1+XX_0)^{\alpha/2}}.
\end{aligned}$$
If $X<2X_1$, it follows from \eqref{intX0Z}, applied with $Z=X\in [X_1/2,2X_1]$, that
$$\begin{aligned}
\int_{X_0}^{X_1} \frac{e^{-C_1(X-z)^2}z^{-2m} dz}{(1+Xz)^{\alpha/2}} 
&\le  \int_{X_0}^{X} \frac{e^{-C_1(X-z)^2}z^{-2m} dz}{(1+Xz)^{\alpha/2}}
+\int_X^{\max(X,X_1)} \frac{e^{-C_1(X-z)^2}z^{-2m} dz}{(1+Xz)^{\alpha/2}}  \\
&\le \frac{CX_{1-m}^{1-2m}e^{-CX_1^2}}{(1+XX_0)^{\alpha/2}}
+\frac{CX_1^{1-2m}}{(1+X_1^2)^{\alpha/2}(1+X_1)}
+\frac{X_1^{-2m}\min\bigl((X_1-X)_+,1\bigr)}{(1+X_1^2)^{\alpha/2}}\\
&\le \frac{CX_{1-m}^{1-2m}e^{-CX^2}}{(1+XX_0)^{\alpha/2}}
+\frac{CX_1^{1-2m}}{(1+X_1^2)^{(\alpha+1)/2}}.
\end{aligned}$$

 (ii) Set $E_1=[Z,\infty)\setminus[\frac{X}{2},\frac{3X}{2}]$, $E_2=[Z,\infty)\cap[\frac{X}{2},\frac{3X}{2}]$
and write $I=\int_{E_1}+\int_{E_2} \equiv I_1+I_2$. We have
$$I_1\le 
\int_{|z-X|\ge \frac{X}{2}} e^{-C(X-z)^2}|z|^m dz\le 2 
\int_{\frac{X}{2}}^\infty e^{-C\tau^2}(X+\tau)^m d\tau
\le \tilde C e^{-CX^2}$$
and
$$I_2\le C\frac{X^m}{{(1+X)^\alpha}}\int_{X/2}^{3X/2} e^{-C(X-z)^2} dz
=C\frac{X^m}{{(1+X)^\alpha}}  \int_0^{X/2} e^{-C\tau^2} d\tau\le C\frac{X^{m+1}}{{(1+X)^{\alpha+1}}},$$
 which readily gives \eqref{P2AuxilLem2}.
Finally, if $X\le Z/2$, then
$I(X,Z)\le$
$\int_{Z}^\infty e^{-Cz^2}z^m dz \le \tilde C 
e^{-CZ^2}.$
\end{proof}

\goodbreak

\section{Construction of special solutions: GBU case}

This section is devoted to the construction of special solutions in the GBU case.

\subsection{Main results on special GBU solutions}

We consider the GBU case on bounded intervals and the half line. 

\begin{thm}
	\label{prop:special} \ 
	Let $p>2$, $0 < R \le \infty$, $\ell\in\N^*$. For any $\eps\in(0,1)$, 
	there exists $u_0$, with $u_0\in \mathcal{W}_1$ if $R=\infty$ and $\ell$ is odd or $u_0\in \mathcal{W}$ otherwise,
	 such that $T:=T^*(u_0)<\infty$ and the solution $u$ of \eqref{equ} 
	enjoys the following properties, for some constants 
	$C, K > 0$ and $\sigma\in (0,R)$. 
	\smallskip

	\begin{itemize}

	\item[(i)] {\it (GBU rate and bubbling behavior)} 
	There holds 
\be{rateell2}
\lim_{t\to T^-}  (T-t)^{\frac{\ell}{p-2}}u_x(0,t)=C 
\ee
and, setting $a(t):= \beta u_x^{1-p}(0,t)$,
\be{asympyux2}
u(x,t) = U_{a(t)}(x) + O(x^2) \quad \mbox{ and } \quad 
u_x(x,t) = U'_{a(t)}(x)+O(x).
\ee

	\item[(ii)] {\it (intermediate region)} 
	There holds
	\be{special-intermed}
	\bigl|u(x,t)-U(x)+(T-t)^\ell \varphi_\ell \bigl((T-t)^{-1/2} x\bigr)\bigr|
	\le\eps\bigl[(T-t)^\ell +x^{2\ell}\bigr]
	\quad\hbox{in } [K(T-t)^q, \sigma],
	\ee
		where $\varphi_\ell$ is the polynomial of degree $2\ell$ given by Proposition~\ref{GaussianPoincare2} and Remark~\ref{RemSpec} 
		for $\alpha=\beta+1$ and $k=(1-\beta)/2$ (which satisfies $\varphi_\ell(0)>0$) and $q=(p-1)(\ell+p-2)[(p-2)(2p-1)]^{-1}>\frac12$. 
	
	\smallskip 
	
	\item[(iii)] {\it (outer region)}
	If $ R = \infty $, then there exists $ \delta \in (0,1) $ such that 
	\be{outerdelta}
	| u(x,t) - U(x) | \ge \delta U(x) \quad \mbox{ in } [\sigma, \infty) \times (0, T). 
	\ee
	If $ R < \infty $, then $u$ is regular at $x=R$, i.e.~$u\in C^{2,1}((0,R]\times(0,T])$ and $u(R,t)=0$ for all $t\in (0,T]$ in the classical sense.

	\smallskip
	\item[(iv)] {\it (intersections with the singular steady state)} 
	There exists $t_1\in(0,T)$ such that, for each $t\in(t_1,T)$, $u(\cdot,t)-U$ 
	has exactly $\ell$ zeros in $(0,\sigma)$,
	denoted by $0<X_1(t)<\dots<X_\ell(t)$, and they are all nondegenerate.
	Moreover,  denoting by $0<y_1<\dots<y_\ell$ the positive zeros of~$ \varphi_\ell$, we have
	$$\biggl|\frac{X_i (t)}{y_i (T-t)^{1/2}}-1\biggr| \le \eps,\qquad t_1<t<T.$$
		\end{itemize}
		
	\noindent Finally,  for $R=\infty$ and $\ell$ odd, we may take $u_0\in \mathcal{W}$
	 if we do not require property (iii).

\end{thm}

 Theorem~\ref{mainThm1}(ii) is a direct consequence of Theorem~\ref{prop:special}.
More insight about the behavior described in Theorem~\ref{prop:special} 
can be gained by reformulating the problem in similarity variables $(y,s)$ (cf.~\eqref{defys}-\eqref{defw})
which is also the fundamental framework for its proof.
 Recall that when $u$ undergoes GBU, the corresponding $w=w(y,s)$ converges to $U$ 
  in $C^1$ except at $ y = 0 $ (cf.~Proposition~\ref{PropwcvU}(i)).
The construction of the special GBU solutions will be done 
by looking for a specific mechanism of convergence. 
Theorem~\ref{prop:special} will be essentially derived from the following Theorem~\ref{mainThm2} and non-oscillation Lemma~\ref{LemNonOsc}.

\begin{thm} \label{mainThm2}
	Let $p, R, \ell, \varphi_\ell$ be as in Theorem~\ref{prop:special} and set 
	$$\lambda=\ell-k, \qquad 
	\eta=\frac{p-1}{p-2}\lambda,
	\qquad \gamma=\frac{p-1}{(p-2)(2p-1)} \lambda,
	\qquad a_*=\Bigl(\frac{\varphi_\ell(0)}{c_p}\Bigr)^{\frac{1}{1-\beta}}.
	$$
	For any $\eps\in(0,1)$, 
	there exists $u_0\in \mathcal{W}$ 
 	such that  $T:=T^*(u_0)<\infty$
	and the corresponding solution $w$ of \eqref{eqw} satisfies, for all $i\in\{0,1\}$ and all $s>s_0=-\log T$:

\smallskip
	\noindent (inner region)
	\be{Concl2mainThm2a}
	\bigl(1-\eps_i(s)\bigr) U^{(i)}_{a_+}(e^{\eta s}y)\le e^{(\lambda-i\eta)s} \,\partial_y^iw(y,s)
		\le \bigl(1+\eps_i(s)\bigr) U^{(i)}_{a_-}(e^{\eta s}y),\quad y\in [0,K e^{-\gamma s}],
	\ee

\smallskip
	\noindent  (intermediate region)
	\be{Concl2mainThm2b}
	\bigl|\partial_y^i \bigl(w(y,s)-U(y)+e^{-\lambda s}\varphi_\ell(y)\bigr)\bigr| \le \eps e^{-\lambda s} (y^i+y^{2\ell-i}),
	\quad y\in [K e^{-\gamma s}, \sigma e^{s/2}],
	\ee
	with 
	$a_\pm=(1\pm\eps)a_*$, $\eps_0(s)=0$, $\eps_1(s)=C_1e^{-(\eta-\gamma) s}$
	and some constants $\sigma=\sigma(p,\ell)\in(0,R)$, $K,C_1>0$. 
\end{thm}

\goodbreak
	
In more qualitative terms, Theorem~\ref{mainThm2} says that, in similarity variables, the singular region consists of two parts:
a very thin inner layer, where $w$ has a quasi-stationary behavior, given by suitable time rescalings of regular steady states;
and a larger intermediate layer, where $w$ converges exponentially to the singular steady state $U(y)$ along
the $\ell$-th eigenfunction of the linearized operator. See next subsection for a heuristic argument
leading to this two-layer expansion.

	\begin{rem} 
		For the special solutions given by Theorems~\ref{prop:special}-\ref{mainThm2}, 
	    in the range $K(T-t)^{\gamma+\frac12}\le x\le \sigma$, there holds 
		$$u_x(x,t)\sim U'(x)-C(T-t)^{\ell-\frac12}\varphi'_\ell(x/\sqrt{T-t}),\quad K(T-t)^{\gamma+\frac12}\le x\le \sigma,$$
		which gives a sharp description of the convergence of $u_x$ to its final space profile.
		In particular, for $x>0$ small, in view of the properties of $\varphi_\ell$ in Proposition~\ref{GaussianPoincare2}(iii),
		we get the second order term of the final profile
		\be{finalprofile}
		u_x(x,T)\sim U'(x)+(-1)^{\ell+1} Cx^{2\ell-1}.
		\ee
			\end{rem}

	\subsection{A heuristic argument for the GBU rates by formal matched asymptotics} \label{section3}
	
	Before, entering into the rigorous proof of Theorems~\ref{prop:special}-\ref{mainThm2}, 
	it is convenient to sketch a formal argument, which gives some evidence for the description 
	in Theorem~\ref{mainThm2}, as well as a simple way to guess the corresponding GBU rates \eqref{rateell}.
	It will also be a guideline to the rigorous proof of existence of special GBU solutions.
	Working in similarity variables, the idea is to look for an approximate solution $w$ of equation \eqref{eqw}, 
	respectively in an inner and an intermediate region. 
	
	\smallskip
	{\bf Inner approximate solution.} The approximation for $y\sim 0$ is sought for in quasi-stationary form:
	\be{defvin}
	w_{in}(y,s):=U_{a(s)}(y)=U(a(s)+y)-U(a(s)),
	\ee
	where the function $a(s)>0$, with $\lim_{s\to\infty} a(s)=0$, has to be determined.
	Heuristically this is reasonable: since $w_y$ is very large in that region,
	the dominant terms in \eqref{eqw} are expected to be $w_{yy}$ and $|w_y|^p$
	(and $w_{in,yy}$ in \eqref{defvin} satisfies $w_{in,yy}+|w_{in,y}|^p=0$).
	We note that, for $y\gg a(s)$,
	$$\begin{aligned}
		w_{in}(y,s)
		&=c_p\bigl[(a(s)+y)^{1-\beta}-(a(s))^{1-\beta}\bigr] \\
		&=c_py^{1-\beta}\bigl[\bigl(1+\ts\frac{a(s)}{y}\bigr)^{1-\beta}-\bigl(\ts\frac{a(s)}{y}\bigr)^{1-\beta}\bigr] 
		\sim c_py^{1-\beta}\bigl[1-\bigl(\ts\frac{a(s)}{y}\bigr)^{1-\beta}\bigr] \\
	\end{aligned}$$
	hence
	\be{linin}
	w_{in}(y,s)\sim U(y)-c_p(a(s))^{1-\beta},\quad y\gg a(s).
	\ee

	\smallskip
	{\bf Outer approximate solution.} 
	 We already know from Proposition~\ref{PropwcvU}
	 that $w$ stabilizes to $U$ for $y>0$
	 and that the equation for $v=w-U$ is given by 
		$$
		v_s=-\mathcal{L}v+F(v_y),
		$$
		with $\mathcal{L}, F$ defined in \eqref{eqzdefL}-\eqref{eqzdefF}.
		On the other hand, as a consequence of results in Section~\ref{subseceigen},
		the eigenvalues of the linearized operator $\mathcal{L}$ (in suitable functional setting) are 
		simple and given by $\lambda_\ell=\ell-k$, 
		for all $\ell\in\N$.
		The corresponding eigenfunction $\varphi_\ell(y)$
		is an even polynomial of degree $2\ell$, whose coefficients of even order are all nonzero,
		so that we can normalize $\varphi_\ell$ (according to a suitable weighted $L^2$ norm) 
		in such way that $\varphi_\ell(0)>0$.

		The outer approximate solution is thus given by linearization around the singular steady state
		along the $\ell$-th mode of the linearized operator. Namely,
		\be{outerapprox}
		w_{out}(y,s):=U(y)-e^{-\lambda_\ell s}\varphi_\ell(y).
		\ee
		Since the mode must be stable for this linearization to make sense, we keep only nonnegative eigenvalues, i.e. 
		$\lambda_\ell=\ell-k$, with $\ell\ge 1$.
		Recalling $\varphi_\ell(0)>0$, we note that
		\be{linout}
		w_{out}(y,s)\sim U(y)-\varphi_\ell(0) e^{-\lambda_\ell s},\quad y\ll 1.
		\ee
		We note that the outer region approximation \eqref{outerapprox} will be also the basis for the recovery case 
		(Theorems~\ref{prop:special-LBC} and \ref{REC-mainThm2}),
		but that case is actually simpler since no inner region is required.

	\smallskip
	{\bf Matching.} Now assume that the approximate solution $w_{in}$ (resp. $w_{out}$) is valid 
	in a region $0<y\le y_0(s)$ (resp., $y\ge y_0(s)$),
	where the ``free boundary'' $y_0(s)$ is such that
	$a(s)\ll y_0(s)\ll 1$.
	Then matching \eqref{linin} and \eqref{linout} at $y=y_0(s)$, we get 
	$$c_p(a(s))^{1-\beta}= \varphi_\ell(0) e^{-\lambda_\ell s},$$
	i.e.~$a(s)\sim e^{-(1-\beta)^{-1}\lambda_\ell s}$.
	Since $U_{a,y}(0)=U'(a)=((p-1)a)^{-\beta}$,
	this means that
	$w_{in,y}(0,s)\sim e^{(1-\beta)^{-1}\beta\lambda_\ell s}=\exp\bigl(\ts\frac{\ell-k}{p-2} s\bigr)$.
	Going back to \eqref{defw}, this yields
	$$u_x(0,t)=e^{-(k-\frac12)s}w_y(0,s)\sim 
	\exp\bigl[(\ts\frac{\ell-k}{p-2}+\frac12-k) s\bigr]=\exp\bigl(\ts\frac{\ell}{p-2} s\bigr)=(T-t)^{-\frac{\ell}{p-2}}.$$

\subsection{Parameters and initial data.}

The case $R<\infty$ can be reduced to $R=1$ by an obvious scaling argument.
We thus assume either $R=\infty$ or $R=1$. Set 
\be{defaphak}
\alpha=\beta+1=\ts\frac{p}{p-1}\in(1,2),\qquad k=\ts\frac{1-\beta}{2}
\ee
and let $\mathcal{L}$ be the operator described in subsection~\ref{subseceigen},
with eigenfunctions $\varphi_j$ and eigenvalues $\lambda_j=j-k$ for $j\in\N$ (cf.~Proposition~\ref{GaussianPoincare2}).
Recall that $\|\varphi_j\|=1$, and that $\|\cdot\|$ and $(\cdot,\cdot)$ 
respectively stand for the norm and inner product in $L^2_\rho$.

Let now $\ell\in \N^*$ be fixed and set $\lambda:=\lambda_\ell=\ell-k$, $\phi:=\varphi_\ell$.
Define the constants
\be{defgamma}
\eta=(1-\beta)^{-1}\lambda,\qquad \gamma=\frac{\beta}{\beta+2} \eta=\frac{\beta}{(1-\beta)(\beta+2)} \lambda,
\qquad a_*=\biggl(\frac{\phi(0)}{c_p}\biggr)^{\frac{1}{1-\beta}}
\ee
and let the constants $\eps_0, \sigma, M_0\in (0,1)$, all depending only on $p,\ell$, 
be respectively given by Proposition~\ref{PropInOut}, \eqref{controlVW} and \eqref{defM0} below.
We introduce a parameter $\eps\in(0, \eps_0]$ 
and we set
\be{defK}
\nu=M_0\eps, \qquad K:=\nu^{1-p},\qquad \tilde K= K\nu^{\frac{1}{\alpha+1}},
\ee
which satisfy $K>\tilde K>1$.
The initial time $s_0>0$ will be chosen large enough below and we stress that
$$\hbox{$s_0$ will depend only on $p,\ell,\eps$.}$$  
We denote
\be{defyi}
y_0(s)=e^{-\eta s},\quad y_1(s)=Ke^{-\gamma s},\quad y_2(s)=\sigma e^{s/2},\qquad s\ge s_0,
\ee
and observe that $y_0(s)<y_1(s)<1<y_2(s)$ for large $s_0$ (depending only on $p,\ell,\eps$).
At given time $s$, the {\it inner, intermediate and outer regions} are, respectively, $[0,y_1(s)]$, $[y_1(s),y_2(s)]$ 
and $[y_2(s),\infty)$.

 Throughout the rest of this section, 
 we will denote by $C$ a generic positive constant depending only on $p,\ell$.
By Proposition~\ref{GaussianPoincare2}, it is immediate that 
\be{controlDiphi}
|D^i\phi(y)|\le C(y^i+y^{2\ell-i}),\quad y>0,\ i\in\{0,1\},
\ee
where $D=\partial_y$, and that there exist $c,y_*>0$ such that 
\be{controlphiell}
 cy^{2\ell}\le (-1)^\ell\phi(y)\le 2cy^{2\ell},\quad y\ge y_*.
\ee
For $\theta\in (0,1]$ and $s_1\ge s_0$, we then define
$$\begin{aligned}
&\mathcal{A}^\theta_{s_0,s_1}=
\Bigl\{W\in L^\infty(s_0,s_1;\mathcal{W});\ \bigl|D^i(W-U+e^{-\lambda s}\phi)\bigr| \le \theta\eps e^{-\lambda s} (y^i+y^{2\ell-i})\\
&\qquad\qquad\qquad\qquad
\quad\hbox{ for 	all }  s_0\le s\le s_1,\ y_1(s)\le y\le y_2(s),\ i\in\{0,1\}\Bigr\}.
	\end{aligned}$$
We note that, as a difference and an additional difficulty with respect to the case of the semilinear heat equation \cite{HV94pre,MizADE}, 
we need to take into account the space derivative along with the solution itself.
In this connection, we actually need to subdivise the inner region into two parts, namely $[0,y_0(s)]$ and $[y_0(s),y_1(s)]$,
in order to cope with certain differences in the behaviors 
of $w$ and $w_y$.\footnote{The choice of $y_1(s)$ in \eqref{defyi} is suggested by an approximate matching of $w_y$; namely,
this choice guarantees that the expected behaviors of $w_y$ 
in the outer and inner regions become of the same order near the interface.
As for the matching of $w$ alone, it is more flexible, as it is actually achieved whenever $y\gg y_0(s)$
(and, in turn, this fact will require the introduction of the function $y_0(s)$ in the computations below).}

Let us now prepare our initial data. 
Fix a smooth cut-off function $\Theta_1(z)$ such that $\Theta_1=1$ for $z\le 1$,  $\Theta_1=0$ for $z\ge 2$
and $\Theta_1'\le 0$. Set $\Theta(y)=\Theta_1\bigl(\frac{1}{2\sigma} e^{-s_0/2}y\bigr)$. 
For any $d\in\R^\ell$ that satisfies
\be{condd}
\sum_{j=0}^{\ell-1} |d_j|\le \eps e^{-\lambda s_0},
\ee
we define $w_0=w_0(\cdot;d)$ as follows. We set
\be{defw0}
w_0(y):=
\begin{cases}
e^{-\lambda s_0} U_a(e^{\eta s_0}y)& \hbox{in $[0,\tilde Ke^{-\gamma s_0}]$} \\ 
\Theta(y)\Bigl\{U(y)-e^{-\lambda s_0}\phi+\ds\sum_{j=0}^{\ell-1}d_j\varphi_j\Bigr\}
&\hbox{in $(\tilde Ke^{-\gamma s_0},Re^{s_0/2})$,} 
\end{cases}
\ee
 where $a=a(\eps,s_0)$ is given by Lemma~\ref{controlinner} below.
If $\ell$ is odd and $R=\infty$, we also consider the alternative choice: 
\be{defw0odd}
w_0(y):=
\begin{cases}
e^{-\lambda s_0} U_a(e^{\eta s_0}y)& \hbox{in $[0,\tilde Ke^{-\gamma s_0}]$} \\ 
U(y)-e^{-\lambda s_0}\phi+\ds\sum_{j=0}^{\ell-1}d_j\varphi_j&\hbox{in $(\tilde Ke^{-\gamma s_0}, 2\sigma e^{s_0/2}),$} \\
b_0U(y)&\hbox{in $(2\sigma e^{s_0/2},\infty),$} \\
\end{cases}
\ee
 where 
$$b_0=b_0(d,s_0):=
1-\Bigl\{\Bigl[e^{-\lambda s_0}\phi-\ds\sum_{j=0}^{\ell-1}d_j\varphi_j\Bigr]U^{-1}\Bigr\}(2\sigma e^{s_0/2}).$$
 The choice \eqref{defw0odd} 
comes from the need to construct a solution which intersects $U$ exactly $\ell$ times on $(0,\infty)$
 (in which case $u_0$ must be unbounded and cannot belong to $\mathcal{W}$).
The above choices of $a, b_0$ will ensure that $u_0(x):=e^{-ks_0}w_0(xe^{s_0/2})$ satisfies
\be{u0space}
u_0\in
\begin{cases}
\mathcal{W}& \hbox{if $w_0$ is given by \eqref{defw0},} \\  
\noalign{\vskip 1mm}
\mathcal{W}_1&\hbox{if $w_0$ is given by \eqref{defw0odd}, $\ell$ is odd and $R=\infty$,} 
\end{cases}
\ee
where $\mathcal{W}_1$ is defined in \eqref{defW1}.
Let $u$ be the maximal classical solution of \eqref{equ} 
(whose existence follows from Proposition~\ref{RemWPw}(i) in the second case of \eqref{u0space})
and let $w=w(y,s;d)$ be 
the corresponding solution of \eqref{eqw}, defined by . 
\be{correspuw}
w(y,s)=e^{ks}u(ye^{-s/2},e^{-s_0}-e^{-s}),\quad 0\le y<Re^{s/2}.
\ee
The maximal existence time $S=S(w_0(\cdot,d))$ of $w$ is given by $S=\infty$ if $T:=T(u_0)\ge e^{-s_0}$ and $S=-\log(e^{-s_0}-T)$ otherwise.
Note that $w_0$ and $w$ also depend on 
$\eps,s_0$ 
but we shall keep this dependence implicit without risk of confusion.

Next, so as to work with unknown functions defined on the entire half-line even in the case $R=1$, we recall the extentions introduced in Lemma~\ref{BernsteinEst}:
\be{defextuwv}
\begin{aligned}
\tilde u(x,t)&=\zeta(x) u(x,t) &&\quad\hbox{ in $[0,\infty)\times[0,T)$,} \\
\tilde w(y,s)&=e^{ks}\tilde u(ye^{-s/2},e^{-s_0}-e^{-s})=\zeta(ye^{-s/2}) w(y,s) &&\quad\hbox{ in $[0,\infty)\times[s_0,S)$,} \\
\tilde v(y,s)&=\tilde w(y,s)-U(y) &&\quad\hbox{ in $[0,\infty)\times[s_0,S)$.}
\end{aligned}\ee
Here, if $R=1$, $\zeta\in C^2([0,\infty))$ is a fixed cut-off function such that 
 $\zeta=1$ in $[0,\frac{1}{3}]$ and $\zeta=0$ in $[\frac{1}{2},\infty)$,
 whereas if $R=\infty$ we just set $\zeta\equiv 1$.
We then define the key set
$$\mathcal{U}_{s_0,s_1}=\Bigl\{d\in \R^\ell; \hbox{ \eqref{condd} holds, $S(w_0(\cdot,d))>s_1$ and
$\tilde w=\tilde w(y,s;d) \in \mathcal{A}^1_{s_0,s_1}$}\Bigr\}.$$

\begin{lem} \label{controlinner}
 (i) In the second case of \eqref{u0space}, if  
 $\eps\in(0,1]$ and $d$ satisfies \eqref{condd}, then
\be{boundb0}
1<b^*\le b_0\le C,\quad s_0\gg 1,
\ee
with $b^*$ depending only on $p,\ell$.  

\smallskip

 (ii) In all cases, there exists $c_0=c_0(p,\ell)>0$ with the following property.
If $0<\eps\le (1+2c_0)^{-1}$, 
$\tilde Ke^{-\gamma s_0}\le \sqrt{\eps}$ and $d$ satisfies \eqref{condd},
then there exists $a=a(\eps,s_0)\in[(1-c_0\eps)a_*,(1+c_0\eps)a_*]$, 
such that \eqref{u0space} holds.
\end{lem}

\begin{proof} 
 (i) Since $\lambda=\ell+(\beta-1)/2$, it follows from \eqref{controlphiell} with $\ell$ odd that, for $s_0\gg 1$,
$$-e^{-\lambda s_0}\bigl[U^{-1}\phi\bigr](2\sigma e^{\frac{s_0}{2}})
= -c_p^{-1}e^{-\lambda s_0}(2\sigma e^{\frac{s_0}{2}})^{\beta-1} \phi(2\sigma e^{\frac{s_0}{2}})
=-c_p^{-1}(2\sigma)^{\beta-1}e^{-\ell s_0} \phi(2\sigma e^{\frac{s_0}{2}})\in [\bar c,2\bar c],$$
where $\bar c=c_p^{-1}c(2\sigma)^{2\lambda}$.
On the other hand, using $|\varphi_j(y)|\le C(1+y^{2j})$ and \eqref{condd}, we get 
$$e^{\lambda s_0} \Bigl|\phi^{-1} \ds\sum_{j=0}^{\ell-1}d_j\varphi_j\Big| (2\sigma e^{s_0/2})\le C\eps (2\sigma e^{s_0/2})^{-2}
\le Ce^{-s_0}.$$
The last two inequalities guarantee \eqref{boundb0}.

\smallskip

(ii) Note that $w_0(0)=0$.
The continuity of $w_0$ at $y=\sigma e^{s_0/2}$ in 
 the second case of \eqref{u0space} 
follows from the choice of~$b_0$.
To ensure $u_0\in W^{1,\infty}_{loc}([0,R))$ 
it suffices to verify the continuity of 
$w_0$ at $\hat y=\tilde Ke^{-\gamma s_0}$, which is equivalent to
$e^{-\lambda s_0} U_a(e^{\eta s_0}\hat y)=U(\hat y)-e^{-\lambda s_0}\phi(\hat y)+\sum_{j=0}^{\ell-1}d_j\varphi_j(\hat y)$,
i.e.
$$h(a):=U_a(e^{\eta s_0}\hat y)-U(e^{\eta s_0}\hat y)+\phi(\hat y)-e^{\lambda s_0}\ds\sum_{j=0}^{\ell-1}d_j\varphi_j(\hat y)=0.$$
We note that
\be{UaTAF}
0\le U_a(y)-U(y)+c_pa^{1-\beta}\le Cay^{-\beta},\quad a,y>0,
\ee 
which follows from (for some $\theta\in (0,1)$):
$$U_a(y)-U(y)+c_pa^{1-\beta}=c_p(y+a)^{1-\beta}-c_py^{1-\beta}
=c_py^{1-\beta}\ts\bigl[\bigl(1+\frac{a}{y}\bigr)^{1-\beta}-1\bigr]
=Cy^{1-\beta}\bigl(1+\frac{\theta a}{y}\bigr)^{-\beta}\frac{a}{y}.$$
Using \eqref{condd}, \eqref{UaTAF},  $|\phi(\hat y)-c_pa_*^{1-\beta}|\le C\hat y^2$, $|\varphi_j(\hat y)|\le C$,
 $(\eta-\gamma)\beta=2\gamma$ and $\tilde Ke^{-\gamma s_0}\le \sqrt{\eps}$, it follows that 
$$\begin{aligned}
|h(a)+c_p(a^{1-\beta}-a_*^{1-\beta})|
&\le |U_a(e^{\eta s_0}\hat y)-U(e^{\eta s_0}\hat y)+c_pa^{1-\beta}|+ |\phi(\hat y)-a_*^{1-\beta}|+C\eps\\
& \le Ca(\tilde Ke^{(\eta-\gamma) s_0})^{-\beta}+C(\tilde Ke^{-\gamma s_0})^2+C\eps
\le C(\tilde Ke^{-\gamma s_0})^2+C\eps\le C_1(p,\ell)\eps.
\end{aligned}$$
Let $c_0\ge 1$ to be fixed below and assume $\eps\le (1+2c_0)^{-1}$. For $a=a_*(1\pm c_0\eps)$, we have
$$c_p|a^{1-\beta}-a_*^{1-\beta}|=c_pa_*^{1-\beta}|(1\pm c_0\eps)^{1-\beta}-1|\ge C_2(p,\ell)c_0\eps.$$
Taking $c_0=\max(1,2C_1/C_2)$, we thus have
$\pm h\bigl(a_*(1\pm c_0\eps)\bigr)<0$. 
Thus, by continuity of $h$, there must exist $a\in \bigl(a_*(1-c_0\eps),a_*+(1+c_0\eps)\bigr)$
such that $h(a)=0$, hence $w_0\in \mathcal{W}$. 

 When $R=\infty$, we note that $u_0$ is compactly supported in the first case of \eqref{u0space} 
and $u_{0,x}$ is bounded in the second case. 
To show \eqref{u0space} it thus only remains to verify that $w_0\ge 0$. 
With $\sigma_ 0(p,\ell)>0$ sufficiently small, for $s_0\gg 1$ and $\sigma\le\sigma_0$, using \eqref{condd}, we have,
on $(\tilde Ke^{-\gamma s_0},4\sigma e^{s_0/2})$:
\begin{eqnarray}
U^{-1}\Bigl|e^{-\lambda s_0}\phi-\ds\sum_{j=0}^{\ell-1}d_j\varphi_j\Bigr| {\hskip -5mm}
&&\le e^{-\lambda s_0}U^{-1}\Bigl(|\phi|+\eps\max_{1\le j\le \ell-1}|\varphi_j| \Bigr)
 \le Cy^{\beta-1}(1+y^{2\ell})e^{-\lambda s_0} \notag \\
 &&\le C\bigl[e^{((1-\beta)\gamma-\lambda) s_0}+(\sigma e^{s_0/2})^{\beta+2\ell-1}e^{(k-\ell)s_0}\bigr] \notag \\
  \noalign{\vskip 2mm}
&&\le C\bigl[ e^{-2\lambda s_0/(\beta+2)}+\sigma_0^{\beta+2\ell-1}\bigr]
\le \ts\frac12. \label{choicesigma1}
\end{eqnarray}
In view of \eqref{defw0}-\eqref{defw0odd}, using also $b_0>0$ in the second case of \eqref{u0space} 
(cf.~assertion (i)), this implies $w_0\ge 0$ on $[0,Re^{s_0/2})$.
\end{proof}

For the subsequent analysis, it will be convenient to rewrite the initial data in \eqref{defw0}
(resp., \eqref{defw0odd}) as
\be{w0tildephi}
w_0(y)=\tilde w(y,s_0)=U(y)+\sum_{j=0}^{\ell-1}d_j\varphi_j-e^{-\lambda s_0}\hat\phi,
\ee
with
\be{w0tildephi2}
\hat\phi:=
\begin{cases}
 \noalign{\vskip -1mm}
e^{\lambda s_0}\Bigl\{U(y)-e^{-\lambda s_0} U_a(e^{\eta s_0}y)+\ds\sum_{j=0}^{\ell-1}d_j\varphi_j\Bigr\}& 
\hbox{in $[0,\tilde Ke^{-\gamma s_0}]$} \\
 \noalign{\vskip -1mm}
 \phi& \hbox{in $(\tilde Ke^{-\gamma s_0},2\sigma e^{s_0/2}]$} \\
(1-\Theta)e^{\lambda s_0}\Bigl(U(y)+\ds\sum_{j=0}^{\ell-1}d_j\varphi_j\Bigr)+\Theta \phi
& \hbox{in $(2\sigma e^{s_0/2},\infty)$,} 
\end{cases}
\ee
 (resp., 
\be{w0tildephi2odd}
\hat\phi:=
\begin{cases}
 \noalign{\vskip -1mm}
e^{\lambda s_0}\Bigl\{U(y)-e^{-\lambda s_0} U_a(e^{\eta s_0}y)+\ds\sum_{j=0}^{\ell-1}d_j\varphi_j\Bigr\}& 
\hbox{in $[0,\tilde Ke^{-\gamma s_0}]$} \\
 \noalign{\vskip -1mm}
 \phi& \hbox{in $(\tilde Ke^{-\gamma s_0}, 2\sigma e^{s_0/2}]$} \\
e^{\lambda s_0}\Bigl\{(1-b_0)U(y)+\ds\sum_{j=0}^{\ell-1}d_j\varphi_j\Bigr\}
& \hbox{in $(2\sigma e^{s_0/2},\infty)$).} 
\end{cases}
\ee
The following lemma shows that, for $s_0\gg 1$, 
$\hat\phi$ is close to the eigenfunction $\phi$ in weighted norm.

\begin{lem} \label{Lemcvcphi}
 For given $\eps\in(0,1]$, we have
\be{tildephicv}
\|\hat\phi-\phi\|\to 0,
\quad  \hbox{as $s_0\to\infty$, uniformly for $d$ satisfying  \eqref{condd}.} 
\ee
\end{lem} 

\begin{proof}
Since $\hat\phi=\phi$ in $(\tilde Ke^{-\gamma s_0}, 2\sigma e^{s_0/2})$,
we see that $\hat\phi-\phi\to 0$ pointwise in $(0,\infty)$ as $s_0\to\infty$. 
Using $e^{\lambda s_0}U(y)= U(e^{\eta s_0}y)\ge U_a(e^{\eta s_0}y)$, it follows from \eqref{UaTAF} that
$$|\hat\phi|=\Bigl|U(e^{\eta s_0}y)-U_a(e^{\eta s_0}y)+e^{\lambda s_0}\ds\sum_{j=0}^{\ell-1}d_j\varphi_j\Bigr|
\le c_p a^{1-\beta}+C\eps\le C
\quad\hbox{in } (0,\tilde Ke^{-\eta s_0}).$$
On the other hand, using $\lambda=\ell+(\beta-1)/2$ and \eqref{boundb0}, we have 
$$
|\hat\phi-\phi|
\le |\phi|+e^{\lambda s_0}\Bigl(CU(y)+\Bigl|\ds\sum_{j=0}^{\ell-1}d_j\varphi_j\Bigr|\Bigr) 
\le Cy^{2\ell}+C y^{2\lambda+1-\beta}=Cy^{2\ell} \quad\hbox{in } (2\sigma e^{s_0/2},\infty). 
$$ 
Consequently, $|\hat\phi-\phi|\rho^{1/2}\le C (1+y^{2\ell})y^{\alpha/2}e^{-y^2/8}$ in $(0,\infty)$ 
and \eqref{tildephicv} follows by dominated convergence. \end{proof}

In the case $R=1$, the next lemma provides a suitable control of $u$ near $x=1$, which in particular rules out GBU at $x=1$.
This will be useful for the control of $w$ in the outer region.

\begin{lem} \label{Lemcontrolx1}
 Let $c_0$ be given by Lemma~\ref{controlinner}, $0<\eps\le (1+2c_0)^{-1}$ and
assume $R=1$. For $s_0\gg 1$ and any $d$ satisfying \eqref{condd}, we have
$u(x,t)\le 1-x$ in $[0,1]\times [0,T(u_0))$.
\end{lem} 

\begin{proof}
By \eqref{defw0}, we have $w_0=0$ in $[4\sigma e^{s_0/2},e^{s_0/2}]$ and $w_0(y)=e^{-\lambda s_0} U_a(e^{\eta s_0}y)\le U(y)$ in $[0,\tilde Ke^{-\gamma s_0}]$.
With $\sigma_1(p,\ell)\in(0,\frac{1}{8})$ sufficiently small and $\sigma\le\sigma_1$, using \eqref{condd}, we get,
for $s_0\gg 1$ and all $y\in[0,4\sigma e^{s_0/2})$,
\begin{eqnarray}
w_0(y)
{\hskip -2.5mm}&\le U(y)+e^{-\lambda s_0}|\phi(y)|+\ds\sum_{j=0}^{\ell-1}|d_j\varphi_j(y)|
\le c_py^{1-\beta}+Ce^{-\lambda s_0}(1+y^{2\ell}) \notag \\
&\le Ce^{ks_0}[\sigma_1^{1-\beta}+e^{-\ell s_0}+\sigma_1^{2\ell}]<\ts\frac38 e^{ks_0}.
\label{choicesigma1b}
\end{eqnarray}
Since $\sigma<\frac{1}{8}$, we deduce that $u_0(x)=e^{-ks_0}w_0(xe^{s_0/2})$
satisfies the assumptions of Lemma~\ref{controlGBU1} and the conclusion follows.
\end{proof}

\subsection{Control of $w$ in inner and outer regions assuming $\tilde w\in \mathcal{A}^1_{s_0,s_1}$.}

The following proposition shows that the property $\tilde w\in \mathcal{A}^1_{s_0,s_1}$,
which corresponds to a control of the solution in the intermediate region
in terms of the linearized behavior,
automatically yields a control of the solution in the inner region, of quasi-stationary type, 
as well as in the outer region if $R=\infty$.
In particular, this gives the precise behavior of $w_y$ at $y=0$, hence the GBU rate.
It is the analogue of \cite[Proposition~4.1]{MizADE} for the semilinear heat equation,
which simplified some arguments from the original proof of~\cite{HV94pre}.
Here its proof is longer and more technical due to the need to take into account 
the space derivative along with the solution itself.

\goodbreak

\begin{prop} \label{PropInOut}
Let $c_0$ be given by Lemma~\ref{controlinner} and assume $\eps\le \eps_0$ with 
$\eps_0= \eps_0(p,\ell)\in (0,\min(a_*,\frac{1}{1+2c_0}))$ sufficiently small.
\smallskip

 (i) If $s_0\gg 1$ then, for any $d\in \mathcal{U}_{s_0,s_1}$, we have
\be{estim0win}
e^{-\lambda s} U_{a_+}(e^{\eta s}y)\le w(y,s)\le e^{-\lambda s} U_{a_-}(e^{\eta s}y)\quad 
\hbox{ in $Q:=\bigl\{(y,s),\, s\in [s_0,s_1],\ y\in [0,K e^{-\gamma s}]\bigr\}$ }
\ee
and
\be{estim0wyin}
\bigl(1-\eps_1(s)\bigr)e^{(\eta-\lambda) s} U'_{a_+}(e^{\eta s}y)\le w_y(y,s)
\le \bigl(1+\eps_1(s)\bigr) e^{(\eta-\lambda) s} U'_{a_-}(e^{\eta s}y)\ \ \hbox{ in $Q$,}
\ee
where $a_\pm=(1\pm \sqrt\eps)a_*$,
$\eps_1(s)=CK^{\beta+1}e^{-\mu s}$, $\mu=\eta-\gamma>0$. 

\smallskip

 (ii) Let $R=\infty$. Assume either \eqref{defw0} and $\ell$ even, or \eqref{defw0odd} and $\ell$ odd.
If $s_0\gg 1$, then there exists $\delta\in(0,1)$, independent of $d$ and $s_1$, such that $u$ satisfies
\be{estimuouter}
|u(x,t)-U(x)|\ge \delta U(x) \quad \mbox{ in } [\sigma,\infty)\times[0,e^{-s_0}-e^{-s_1}).
\ee
\end{prop}

\begin{proof}
Denote $\mathcal{I}_1:=(0,\tilde Ke^{-\gamma s_0})$,
$\mathcal{I}_2:=(\tilde Ke^{-\gamma s_0},Ke^{-\gamma s_0})$ and recall that, by Lemma~\ref{controlinner}, we have
\be{aastar}
|a-a_*|\le c_0a_*\eps.
\ee

{\bf Step 1.} {\it Proof of \eqref{estim0win}.}
For $\kappa\in(1/2,2)$ and $b\in(a_*/2,2a_*)\setminus\{a\}$ to be specified below, we shall consider the comparison function
\be{defwtilde}
\hat w(y,s)=\kappa e^{-\lambda s} U_b(\xi),\qquad \xi=e^{\eta s}y\in [0,K e^{(\eta-\gamma)s}].
\ee
We note right away the identities
\be{UbxiIdent}
\xi U'_b-(1-\beta)U_b=bd_p\bigl[b^{-\beta}-(b+\xi)^{-\beta}\bigr],\qquad
\xi U''_b+\beta U'_b=b\beta d_p(b+\xi)^{-\beta-1}.
\ee

$\bullet$ We first check that $\hat w$ is a sub-/supersolution in $Q$.
We have, omitting the variable $\xi$ for conciseness,
$$\begin{aligned}
\mathcal{P}\hat w
&:=\hat w_s-\hat w_{yy}+\ts\frac{y}{2}\hat w_y-k\hat w-|\hat w_y|^p \\
&=-\lambda\kappa e^{-\lambda s} U_b + \kappa \eta e^{(\eta-\lambda) s} yU'_b
-\kappa e^{(2\eta-\lambda) s} U''_b+\kappa \ts\frac{y}{2} e^{(\eta-\lambda) s}U'_b
-k\kappa e^{-\lambda s} U_b-\kappa^p e^{p(\eta-\lambda) s}{U'_b}^p  \\
&=-(\lambda+k)\kappa e^{-\lambda s} U_b + \kappa(\eta+\tsfr) e^{-\lambda s} \xi U'_b
-\kappa e^{(2\eta-\lambda) s} U''_b -\kappa^p e^{p(\eta-\lambda) s}{U'_b}^p.
\end{aligned}$$
Using $U''_b+{U'_b}^p=0$, $p(\eta-\lambda)=2\eta-\lambda$,
$\lambda+k=(\eta+\tsfr)(1-\beta)$ and \eqref{UbxiIdent}, 
we get
$$\begin{aligned}
\mathcal{P}\hat w
&=\kappa e^{(2\eta-\lambda) s} \left\{(1 -\kappa^{p-1}) {U'_b}^p
+e^{-2\eta s}\bigl[(\eta+\tsfr)  \xi U'_b-(\lambda+k)U_b\bigr]\right\}\\
&=\kappa e^{(2\eta-\lambda) s}(b+\xi)^{-\beta-1} \bigl\{(1 -\kappa^{p-1}) d_p^p
+b c_p(\lambda+k)\bigl[b^{-\beta}-(b+\xi)^{-\beta}\bigr](b+\xi)^{\beta+1}e^{-2\eta s}\bigr\}.\\
\end{aligned}$$
We deduce that, for $C_1=C_1(p,\ell)>0$ sufficiently large,
\be{condkappa1}
\mathcal{P}\hat w\
\begin{cases}
\,\ge 0& \hbox{if $\kappa\le 1$} \\
\noalign{\vskip 1mm}
\,\le 0& \hbox{if $\kappa-1 \ge C_1K^{\beta+1} e^{[(\eta-\gamma)(\beta+1)-2\eta]s_0}.$}
\end{cases}
\ee
Indeed, for the second case, we note that since $(\eta-\gamma)(\beta+1)-2\eta=-(1+\beta)\gamma-(1-\beta)\eta<0$ and 
$Ke^{(\eta-\gamma)s}\ge K\ge 1$, we then have
$\kappa^{p-1}-1 \ge (p-1)C_1K^{\beta+1} e^{[(\eta-\gamma)(\beta+1)-2\eta]s}\ge CC_1(b+Ke^{(\eta-\gamma) s})^{\beta+1}e^{-2\eta s}
\ge CC_1(b+\xi)^{\beta+1}e^{-2\eta s}$.

\smallskip
$\bullet$ Let us next compare $\hat w$ and $w$ on the parabolic boundary.
Set $\Delta_0=\hat w(\cdot,s_0)-w_0$.
We claim~that for $C_2=C_2(p,\ell)>0$ sufficiently small,
\be{condkappa2}
0\le\frac{\kappa-1}{b-a}\le C_2(\tilde Ke^{(\eta-\gamma)s_0})^{\beta-1}
\Longrightarrow
\hbox{$\Delta_0$ has the same sign as $a-b$ on $\mathcal{I}_1$.}
\ee
Indeed, on $\mathcal{I}_1$,
$\Delta_0(y)=\kappa e^{-\lambda s_0} U_b(e^{\eta s_0}y)-e^{-\lambda s_0} U_a(e^{\eta s_0}y)$ has the sign of
$$h(\xi):=\kappa[(\xi+b)^{1-\beta}-b^{1-\beta}]-[(\xi+a)^{1-\beta}-a^{1-\beta}],
\quad\hbox{ where } \xi:=ye^{\eta s_0}\in (0,\tilde Ke^{(\eta-\gamma)s_0}).$$
Setting $J_\xi(a)=(\xi+a)^{1-\beta}-a^{1-\beta}$ and noting that 
$J'_\xi(a)=(1-\beta)[(\xi+a)^{-\beta}-a^{-\beta}]$,
we get, for some $c$ between $a$ and $b$,
$$\begin{aligned}
\frac{h(\xi)}{b-a}
&= \frac{\kappa-1}{b-a}[(\xi+b)^{1-\beta}-b^{1-\beta}]+\frac{J_\xi(b)-J_\xi(a)}{b-a} \\
&= \frac{\kappa-1}{b-a}[(\xi+b)^{1-\beta}-b^{1-\beta}]+(1-\beta)[(\xi+c)^{-\beta}-c^{-\beta}].
\end{aligned}$$
For $\xi\le c$, since $b,c\in (a_*/2,2a_*)$, we deduce that, for some $\theta\in (0,1)$,
$$\begin{aligned}
\frac{h(\xi)}{b-a}
&=\Bigl\{\frac{\kappa-1}{b-a}(1-\beta)(\theta\xi+b)^{-\beta}-\beta(1-\beta)(\theta\xi+c)^{-\beta-1}\Bigr\}\xi \\
&= (1-\beta)(\theta\xi+c)^{-\beta-1}\xi\Bigl\{\frac{\kappa-1}{b-a}\frac{(\theta\xi+c)^{\beta+1}}{(\theta\xi+b)^\beta}-\beta\Bigr\}
\le (1-\beta)(\theta\xi+c)^{-\beta-1}\xi(CC_2-\beta)\le 0
\end{aligned}$$
due to $\tilde Ke^{(\eta-\gamma)s_0}\ge \tilde K\ge 1$ and the hypothesis in \eqref{condkappa2}, whereas, for $\xi\ge c$, 
$$\frac{h(\xi)}{b-a}\le \frac{\kappa-1}{b-a}\xi^{1-\beta}-(1-\beta)[1-2^{-\beta}]c^{-\beta}
\le \frac{\kappa-1}{b-a}(\tilde Ke^{(\eta-\gamma)s_0})^{1-\beta}-C\le C_2-C\le 0.$$
This proves claim \eqref{condkappa2}.
We next claim that for sufficiently small $\eps_0(p,\ell)>0$ and $C_3(p,\ell)>0$,
\be{condkappa3}
\begin{aligned}
& K e^{-\gamma s_0}\le\eps,\quad |b-a_*|\ge  a_*\sqrt\eps,\quad 
0\le\frac{\kappa-1}{b-a_*}\le C_3(Ke^{(\eta-\gamma)s_0})^{\beta-1}\\
&\qquad\qquad\qquad \Longrightarrow
\hbox{$\Delta_0$ has the same sign as $a_*-b$ on $\mathcal{I}_2.$}
\end{aligned}
\ee
Indeed, on $\mathcal{I}_2$, 
$\Delta_0=\kappa e^{-\lambda s_0} U_b(e^{\eta s_0}y)-U(y)+e^{-\lambda s_0}\phi(y)-\sum_{j=0}^{\ell-1}d_j\varphi_j(y)$ has the sign of
$$h_\kappa(y):=\kappa U_b(e^{\eta s_0}y)-U(e^{\eta s_0}y)+\phi(y)-e^{\lambda s_0}\ds\sum_{j=0}^{\ell-1}d_j\varphi_j(y).$$
Assuming the hypothesis in \eqref{condkappa3} and 
using \eqref{condd},  $\phi'(0)=0$ and $\eta-\gamma= \frac{2}{\beta} \gamma$, we have
$$\begin{aligned}
|h_1(y)+c_p(b^{1-\beta}-a_*^{1-\beta})|
&\le |U_b(e^{\eta s_0}y)-U(e^{\eta s_0}y)+c_pb^{1-\beta}|+|\phi(y)-c_pa_*^{1-\beta}|
+|e^{\lambda s_0}\ds\sum_{j=0}^{\ell-1}d_j\varphi_j(y)|\\
&\le Cb(e^{\eta s_0}y)^{-\beta}+ C(Ke^{-\gamma s_0})^2+C\eps\le C\eps.
\end{aligned}$$
Consequently, for some $c$ between $a_*$ and $b$, we get
$$\begin{aligned}
\frac{h_\kappa(y)}{b-a_*}
&=\frac{\kappa-1}{b-a_*} U_b(e^{\eta s_0}y)+\frac{h_1(y)+c_p(b^{1-\beta}-a_*^{1-\beta})}{b-a_*} -d_pc^{-\beta}\\
&\le\frac{\kappa-1}{b-a_*} (Ke^{(\eta-\gamma)s_0})^{1-\beta}+\frac{C\eps}{b-a_*}-d_p(2a_*)^{-\beta}
\le C_3+C\sqrt\eps-d_p(2a_*)^{-\beta}\le 0.
\end{aligned}$$
We then claim that for sufficiently small $\eps_0(p,\ell)>0$ and $C_4(p,\ell)>0$,
\be{condkappa4}
\begin{aligned}
&Ke^{-\gamma s_0}\le\eps,\quad |b-a_*|\ge  a_*\sqrt\eps,\quad 
0\le\frac{\kappa-1}{b-a_*}\le C_4(Ke^{(\eta-\gamma)s_0})^{\beta-1}\\
&\qquad\ \Longrightarrow
\hbox{$\Delta:=\hat w-w$ has the same sign as $a_*-b$ on $\Gamma:=\{(Ke^{-\gamma s},s);\ s\in[s_0,s_1]\}$.}
\end{aligned}
\ee
Indeed, $\Delta$ has the sign of
$h_\kappa(Ke^{-\gamma s})$ on $\Gamma$, where $h_\kappa(y):=\kappa U_b(e^{\eta s}y)-e^{\lambda s}w(y,s)$.
Since $\hat w\in \mathcal{A}^1_{s_0,s_1}$, we have in particular
$|w(y,s)-U(y)+e^{-\lambda s}\phi(y)|\le \eps e^{-\lambda s}$
hence $|e^{\lambda s}w(y,s)-U(e^{\eta s}y)+\phi(y)|\le \eps$ on $\Gamma$.
Assuming the hypothesis in \eqref{condkappa4} and 
using \eqref{UaTAF},  $\phi'(0)=0$  and $\eta-\gamma= 
\frac{2}{\beta} \gamma$, we have
$$\begin{aligned}
|h_1(y)+c_p(b^{1-\beta}-a_*^{1-\beta})|
&\le |U_b(e^{\eta s}y)-U(e^{\eta s}y)+c_pb^{1-\beta}|+|U(e^{\eta s}y)-e^{\lambda s}w(y,s)-\phi(y)| \\
&\qquad +|\phi(y)-c_pa_*^{1-\beta}| \le Cb(e^{\eta s}y)^{-\beta}+ \eps+C (Ke^{-\gamma s})^2
\le C\eps\ \hbox{ on $\Gamma$.}
\end{aligned}$$
Consequently, for some $c$ between $a_*$ and $b$, we get
$$\begin{aligned}
\frac{h_\kappa(y)}{b-a_*}
&=\frac{\kappa-1}{b-a_*} U_b(e^{\eta s}y)+\frac{h_1(y)+c_p(b^{1-\beta}-a_*^{1-\beta})}{b-a_*} -d_pc^{-\beta}\\
&\le\frac{\kappa-1}{b-a_*} (Ke^{(\eta-\gamma)s})^{1-\beta}+\frac{ C\eps}{b-a_*}- d_p(2a_*)^{-\beta}
\le C_4+C\sqrt\eps- d_p(2a_*)^{-\beta}\le 0 \ \ \hbox{ on $\Gamma$,}
\end{aligned}$$
which proves claim \eqref{condkappa4}.
Also we have $\hat w=w=0$ at $y=0$.

\medskip

Now, we choose $b=a_\pm=(1\pm\sqrt\eps)a_*$, 
hence $\tsfr a_*\sqrt\eps \le |b-a|\le 2a_*\sqrt\eps$ by \eqref{aastar}
for $\eps\le\eps_0(p,\ell)$ sufficiently small.
We then choose $\kappa=\kappa_\pm=1\pm\tsfr\tilde Ca_* \sqrt\eps(Ke^{(\eta-\gamma)s_0})^{\beta-1}$,
where $\tilde C=\min(C_2,C_3,C_4)$, and consider the corresponding $\hat w_\pm$ given by \eqref{defwtilde}.
Taking $s_0$ large enough, we see that 
the assumptions in \eqref{condkappa2}--\eqref{condkappa4} are satisfied and that, moreover, 
$$\kappa_+-1=\tsfr\tilde Ca_* \sqrt\eps(Ke^{(\eta-\gamma)s_0})^{\beta-1} \ge CK^{\beta+1} e^{[(\eta-\gamma)(\beta+1)-2\eta]s_0},$$
so that the condition in \eqref{condkappa1} is also satisfied.
We deduce that $\hat w_+$ is a subsolution and $w_-$ is a supersolution in $Q$.
It follows from the comparison principle that 
$e^{-\lambda s} U_{a_+}(e^{\eta s}y)\le \hat w_+\le w\le \hat w_-\le e^{-\lambda s} U_{a_-}(e^{\eta s}y)$ in~$Q$.

\smallskip

{\bf Step 2.} {\it Proof of \eqref{estim0wyin}.}
Consider the operator
\be{defP1z}
\mathcal{P}_1z:=z_s-z_{yy}+\frac{y}{2}z_y+(\tsfr-k)z-p|z|^{p-2}zz_y
\ee
and observe that $\mathcal{P}_1w_y=(\mathcal{P}w)_y=0$.
For $m\in\{1,2\}$, 
$b\in(a_*/2,2a_*)\setminus\{a\}$, $\mu\in(0,\eta+\gamma]$ to be specified below and $A>0$ such that 
\be{condAmu}
Ae^{-\mu s_0}\le\tsfr,
\ee
we shall consider the comparison function
$$\hat z=\hat z_m(y,s)=
\kappa(s) e^{(\eta-\lambda)s} U'_b(\xi),\quad s\ge s_0,
\quad\hbox{ where $\xi=e^{\eta s}y$ and $\kappa(s)=1+(-1)^mAe^{-\mu s}\ge \tsfr$}.$$

$\bullet$ We first check that $\hat z$ is a sub-/supersolution of $\mathcal{P}_1\hat z=0$ in $Q$.
 We compute, omitting the variable $\xi$ for conciseness,
$$\begin{aligned}
\mathcal{P}_1\hat z
&=\kappa'(s) e^{(\eta-\lambda)s} U'_b +(\eta-\lambda)\kappa e^{(\eta-\lambda)s} U'_b
+\eta\kappa y e^{(2\eta-\lambda) s}U''_b-\kappa e^{(3\eta-\lambda) s} U'''_b \\
&\qquad +\kappa \ts\frac{y}{2} e^{(2\eta-\lambda) s}U''_b
+(\tsfr-k)\kappa e^{(\eta-\lambda)s} U'_b-p\kappa^p e^{[\eta+p(\eta-\lambda)]s}{U'_b}^{p-1}U''_b\\
&=(-1)^{m+1}A\mu e^{(-\mu+\eta-\lambda)s} U'_b+(-\lambda-k+\eta+\tsfr)\kappa e^{(\eta-\lambda)s} U'_b \\
&\qquad\qquad + \kappa(\eta+\tsfr) e^{(\eta-\lambda)s} \xi U''_b
-\kappa e^{(3\eta-\lambda) s} U'''_b -p\kappa^p e^{[\eta+p(\eta-\lambda)]s}{U'_b}^{p-1}U''_b.
\end{aligned}$$
Using $U'''_b=-p{U'_b}^{p-1}U''_b=p{U'_b}^{2p-1}$, $p(\eta-\lambda)=2\eta-\lambda$,
$\lambda+k=(\eta+\tsfr)(1-\beta)$ and \eqref{UbxiIdent}, we~get
$$\begin{aligned}
\mathcal{P}_1\hat z
&=\kappa e^{(3\eta-\lambda) s} \Bigl\{
(-1)^{m+1}A\kappa^{-1}\mu e^{-(\mu+2\eta)s} U'_b+p(\kappa^{p-1}-1) {U'_b}^{2p-1} \\
&\qquad\qquad \qquad\qquad\qquad\qquad\qquad
+e^{-2\eta s}\bigl[(\eta+\tsfr)  \xi U''_b+(\eta+\tsfr-\lambda-k)U'_b\bigr] \Bigr\}\\
&=\kappa {U'_b}^{2p-1}e^{(3\eta-\lambda) s} \Bigl\{(-1)^{m+1}A\kappa^{-1}\mu e^{-(\mu+2\eta)s} {U'_b}^{2(1-p)}+ \\
&\qquad\qquad \qquad\qquad\qquad\qquad\qquad
+p(\kappa^{p-1}-1) +(\eta+\tsfr)  e^{-2\eta s}\bigl[\xi U''_b+\beta U'_b\bigr]{U'_b}^{1-2p}\Bigr\}\\
&=\kappa {U'_b}^{2p-1}e^{(3\eta-\lambda) s} \Bigl\{(-1)^{m+1}A\kappa^{-1}\mu e^{-(\mu+2\eta)s}\beta^2 (b+\xi)^2 \\
&\qquad\qquad \qquad\qquad\qquad\qquad\qquad
+Cbe^{-2\eta s}(b+\xi)+p\Bigl[\bigl(1+(-1)^mAe^{-\mu s}\bigr)^{p-1}-1\Bigr] \Bigr\}.
\end{aligned}$$
Since $b+\xi\le b+Ke^{(\eta-\gamma) s}  \le CKe^{(\eta-\gamma)s}$ owing to $Ke^{(\eta-\gamma)s}\ge K\ge 1$, 
we deduce that 
$$\begin{aligned}
(-1)^m\mathcal{P}_1\hat z_m
&\ge\kappa  {U'_b}^{2p-1} e^{(3\eta-\lambda) s} \left\{A e^{-\mu s}(\bar C-C\mu e^{-2\eta s} (b+\xi)^2)
+(-1)^mC e^{-2\eta s}(b+\xi)\right\}\\
&\ge\kappa {U'_b}^{2p-1} e^{(3\eta-\lambda) s} \left\{A e^{-\mu s}(\bar C-C\mu K^2 e^{-2\gamma s})
-(1+(-1)^{m-1}) CKe^{-(\eta+\gamma)s}\right\},\\
\end{aligned}$$
 with $\bar C=\bar C(p)>0$. Therefore, there exists $C_5=C_5(p,\ell)>0$ such that, assuming
\be{condKmu}
C_5\mu K^2 e^{-2\gamma s_0}\le 1,
\ee
we have
\be{condkappaD1}
\begin{cases}
\ \mathcal{P}_1\hat z_2\ge 0& \hbox{} \\
\noalign{\vskip 1mm}
\ \mathcal{P}_1\hat z_1\le 0& \hbox{if $\mu\le\eta+\gamma$ and $A\ge C_5Ke^{(\mu-\gamma-\eta)s_0}$.}
\end{cases}
\ee

\smallskip
$\bullet$ Let us next compare $\hat z$ and $w_y$ on the parabolic boundary.
Set $\Delta^m_{1,0}=\hat z_m(\cdot,s_0)-w'_0$.
Since $U'_a(\xi)=d_p(a+\xi)^{-\beta}$ and $w'_0(y)=e^{(\eta-\lambda) s_0} U'_a(e^{\eta s_0}y)$ on $\mathcal{I}_1$,
we obviously have 
\be{condkappaD1b}
\hbox{$b\ge a$ and $m=1$ (resp., $b\le a$ and $m=2$)\ $\Longrightarrow\ \Delta^m _{1,0}\le 0$ (resp., $\ge 0$)  on 
$\mathcal{I}_1$.}
\ee

On $\mathcal{I}_2$, we have
$w'_0(y)=U'(y)-e^{-\lambda s_0}\phi'(y)+\sum_{j=0}^{\ell-1}d_j\phi'_j(y)$.
In view of \eqref{condd}, $\phi''(0)<0$ and $\phi'_j(0)=0$ for any $j\in\N$, we have
$w'_0(y)\ge U'(y)+e^{-\lambda s_0}(c_1-C \eps)y\ge U'(y)=e^{(\eta-\lambda)s} U'(e^{\eta s}y)$
 on $\mathcal{I}_2$ for sufficiently small $\eps_0(p,\ell)>0$,
hence
\be{condkappaD1c}
\Delta^1_{1,0}\le 0 \quad\hbox{in $\mathcal{I}_2$}.
\ee
We claim that, for $C_6=C_6(p,\ell)>0$ sufficiently large,
\be{condkappaD2}
\bigl(\mu\le \eta-\gamma,\quad A\ge C_6K^{\beta+1}e^{(\mu+\gamma-\eta)s_0}\bigr)
\Longrightarrow \Delta^2_{1,0}\ge 0 \quad\hbox{in $\mathcal{I}_2$}.\ee
Indeed $\Delta^2_{1,0}$ has the sign of 
$\hat h(y):=\kappa U'_b(e^{\eta s_0}y)-U'(e^{\eta s_0}y)+e^{-\eta s_0}\phi'(y)-e^{(\lambda-\eta) s_0}\sum_{j=0}^{\ell-1}d_j\phi'_j(y)$
in~$\mathcal{I}_2$, and
\be{hkappa1}
\begin{aligned}
\hat h(y)
&\ge\kappa d_p(b+e^{\eta s_0}y)^{-\beta}-d_p(e^{\eta s_0}y)^{-\beta}-Ce^{-\eta s_0}y\\
&\ge(\kappa-1)d_p(b+e^{\eta s_0}y)^{-\beta}
+d_p(e^{\eta s_0}y)^{-\beta}[(1+be^{-\eta s_0}y^{-1})^{-\beta}-1]-Ce^{-\eta s_0}y\\
&\ge Ae^{-\mu s_0}d_p(b+e^{\eta s_0}y)^{-\beta}-Ce^{-(\beta+1)\eta s_0}y^{-\beta-1}-Ce^{-\eta s_0}y\\
&\ge Ae^{-\mu s_0}d_p(b+e^{\eta s_0}y)^{-\beta}-Ce^{-\eta s_0}y,
\end{aligned}
\ee
where we used \eqref{condd} and $e^{-\eta s_0}y\ge e^{-(\beta+1)\eta s_0}y^{-\beta-1}$ owing to 
$y\ge e^{-\gamma s_0}$. 
Since the assumption in \eqref{condkappaD2} implies
$ Ae^{-\mu s_0} \ge C_6 K^{\beta+1}e^{(\gamma-\eta) s_0} 
=C_6e^{(\beta-1)\eta s_0}(Ke^{-\gamma s_0})^{\beta+1}
\ge CC_6(b+e^{\eta s_0}y)^{\beta}e^{-\eta s_0}y$ due to $Ke^{-\gamma s_0}\ge y$ and $e^{\eta s_0}y\ge 1$, we get
$\hat h(y)\ge 0$, and claim \eqref{condkappaD2} follows.

\smallskip
We then consider the sign of $\Delta^m_1:=\hat z_m-w_y$ on $\Gamma$.
We see that $\Delta^m_1$ has the same sign as
$h_\kappa(Ke^{-\gamma s})$ on $\Gamma$, where $\hat h(y):=\kappa U'_b(e^{\eta s}y)-e^{(\lambda-\eta) s}w_y(y,s)$.
Since $\hat w\in \mathcal{A}^1_{s_0,s_1}$, we have in particular
\be{wyinA}
|e^{(\lambda-\eta) s}w_y(y,s)-U'(e^{\eta s}y)+e^{-\eta s}\phi'(y)|\le \eps e^{-\eta s}y
\quad\hbox{on $\Gamma$.}
\ee
Also, for some $\bar c_\ell,\bar y_\ell>0$, we have $\phi'(y)\le -\bar c_\ell y$ for $y\le\bar y_\ell$.
For $m=1$, assuming $\eps\le \bar c_\ell$ and 
\be{condKmu1}
K e^{-\gamma s_0}\le \bar y_\ell,
\ee
we thus have
$\hat h(y) \le U'(e^{\eta s}y)-e^{(\lambda-\eta) s}w_y(y,s)\le (\phi'(y)+\eps) e^{-\eta s}y\le 0$ on $\Gamma$, hence
\be{condkappaD3}
\Delta^1_1\le 0\ \hbox{on $\Gamma$}.
\ee
Next, for $m=2$, using \eqref{wyinA} and the same calculation as in \eqref{hkappa1}, we get
$$\begin{aligned}
\hat h(y)
\ge\kappa d_p(b+e^{\eta s}y)^{-\beta}-d_p(e^{\eta s}y)^{-\beta}-Ce^{-\eta s}y
\ge Ae^{-\mu s}d_p(b+e^{\eta s}y)^{-\beta}-Ce^{-\eta s}y.
\end{aligned}$$
Since $K^{\beta+1}e^{(\gamma-\eta) s}=
K^{\beta+1}e^{-[(1-\beta)\eta+(\beta+1)\gamma] s}=e^{(\beta-1)\eta s}(Ke^{-\gamma s})^{\beta+1}
\ge C(b+e^{\eta s}y)^{\beta}e^{-\eta s}y$ on $\Gamma$,
we deduce that, for $C_7=C_7(p,\ell)>0$ sufficiently large,
 \be{condkappaD4}
\bigl(\mu\le \eta-\gamma,\quad A\ge C_7K^{\beta+1}e^{(\mu+\gamma-\eta)s_0}\bigr)
\Longrightarrow \Delta^2_1\ge 0 \quad\hbox{on $\Gamma$}.\ee

Now choose $b=a_+=(1+\sqrt\eps)a_*$ for $\hat z_1$, $b=a_-=(1-\sqrt\eps)a_*$ for $\hat z_2$, 
and $\mu=\eta-\gamma$, $A= \hat CK^{\beta+1}$, where $\hat C=\max(C_5,C_6,C_7)$.
Then the assumptions in \eqref{condkappaD1}, \eqref{condkappaD1b}, \eqref{condkappaD2}, \eqref{condkappaD4}
are fulfilled, and conditions \eqref{condAmu}, \eqref{condKmu}, \eqref{condKmu1} are satisfied for $s_0\gg 1$.
Moreover, as a consequence of \eqref{estim0win} and $\hat w=w=0$ at $y=0$, we have
$$ U'(a_+)e^{(\eta-\lambda) s} \le w_y(0,s)\le U'(a_-)e^{(\eta-\lambda) s},\quad s_0\le s\le s_1,$$
hence $\hat z_1\le w_y\le \hat z_2$ at $y=0$.
Thus $\hat z_1$ is a subsolution and $\hat z_2$ a supersolution and
we infer from the comparison principle that $ \hat z_1\le w_y\le \hat z_2$ in~$Q$.
Inequality \eqref{estim0wyin} follows. 

\smallskip
{\bf Step 3.} {\it Proof of \eqref{estimuouter}.} 
Here we assume $R=\infty$.
Since $\hat w\in \mathcal{A}^1_{s_0,s_1}$, taking $\eps_0(p,\ell)>0$ smaller if necessary, we deduce from
\eqref{controlphiell} that
$$
(-1)^{\ell+1} \bigl[w(\sigma e^{s/2},s)-U(\sigma e^{s/2})\bigr] \ge  \ts\frac12 e^{-\lambda s}\phi(\sigma e^{s/2}) 
\ge \frac{c}{2}\sigma^{2\ell} e^{-\lambda s}e^{\ell s}= \frac{c}{2}\sigma^{2\ell} e^{k s},
\quad s_0<s<s_1.
$$
Going back to $u$ through \eqref{correspuw}, we get
\be{controlsignell}
(-1)^{\ell+1} \bigl[u(\sigma,t)-U(\sigma)\bigr] \ge \ts\frac{c}{2}\sigma^{2\ell},
\quad 0<t<e^{-s_0}-e^{-s_1}.
\ee

If $\ell$ is odd, then we have $u(x,0)\ge b^*U(x)$ for $x\ge 2\sigma$ due to \eqref{defw0odd}, and
$$\Bigl[-e^{-\lambda s_0}\phi(y)+\ds\sum_{j=0}^{\ell-1}d_j\varphi_j(y)\Bigr]U^{-1}(y) \ge
c_p^{-1}y^{\beta-1}[cy^{2\ell}-\ell\eps C(1+y^{2\ell-2})]e^{-\lambda s_0}\ge \ts\frac{c}{2c_p},
\quad \sigma e^{\frac{s_0}{2}} \le y\le 2\sigma e^{\frac{s_0}{2}}$$ 
for $s_0\gg 1$, hence $u(x,0)\ge (1+\ts\frac{c}{2c_p})U(x)$ for $\sigma\le x\le 2\sigma$.
Setting $\delta=\min(b^*-1, \ts\frac{c}{2c_p},\ts\frac{c}{2}\sigma^{2\ell})>0$ 
and noting that $(1+\delta)U$ is a subsolution of the PDE in \eqref{equ},
it follows from the comparison principle (cf.~ Proposition~\ref{RemWPw}(ii))
that $u\ge (1+\delta)U$ in $[\sigma,\infty)\times(0,e^{-s_0}-e^{-s_1})$.

Next assume that $\ell$ is even.
Since $|\varphi_j(y)|\le C(1+y^{2j})$, we deduce from \eqref{controlphiell} and $\eps<1$ that
$$\Bigl[-e^{-\lambda s_0}\phi(y)+\ds\sum_{j=0}^{\ell-1}d_j\varphi_j(y)\Bigr]U^{-1}(y) \le 
c_p^{-1}y^{\beta-1}[\ell\eps C(1+y^{2\ell-2})-cy^{2\ell}]e^{-\lambda s_0}\le -\ts\frac{c}{2c_p},
\quad y\ge \sigma e^{s_0/2}$$ 
for $s_0\gg 1$. Consequently, by \eqref{defw0}, since $0\le\Theta\le 1$, we have 
$u(x,0)\le (1-\ts\frac{c}{2c_p})U(x)$ for $x\ge\sigma$
if $s_0\gg 1$.
This along with \eqref{controlsignell} guarantees the existence of $\delta\in(0,1)$ 
(independent of $d,s_1$) such 
that $u\le (1-\delta)U$ on $(\{\sigma\}\times (0,e^{-s_0}-e^{-s_1}))\cup([\sigma,\infty)\times\{0\})$.
Noting that $(1-\delta)U$ is a supersolution of the PDE in \eqref{equ},
it follows from the comparison principle that $u\le (1-\delta)U$ in $[\sigma,\infty)\times(0,e^{-s_0}-e^{-s_1})$.
This completes the proof of \eqref{estimuouter}.
\end{proof}

\subsection{Topological argument and proof of Theorems~\ref{prop:special}, \ref{mainThm2}}

Define the map
\be{defmapP}
P(d;s_0,s_1)=\bigl(p_0(d;s_0,s_1),\dots,p_{\ell-1}(d;s_0,s_1)\bigr),\qquad\hbox{where } p_j=(\tilde v(s_1),\varphi_j).
\ee
In view of the topological argument used in the proof of Theorem~\ref{prop:special},
the crucial ingredient will be the following key a priori estimate.
It shows that $\mathcal{A}^1_{s_0,s_1}$ constitutes a trapping region,
from which the solution cannot escape at $s=s_1$ if $P(d;s_0,s_1) = 0$.

\begin{prop} \label{mainAPE}
For any $\eps\in(0, \eps_0]$, there exists $\bar s_0>0$ such that, 
if $s_1\ge s_0\ge \bar s_0$ and $d\in \mathcal{U}_{s_0,s_1}$ satisfy $P(d;s_0,s_1) = 0$, then 
$\tilde w(\cdot,\cdot;d) \in \mathcal{A}^{1/2}_{s_0,s_1}$ 
 and, moreover, 
$\ds\sum_{j=0}^{\ell-1} |d_j|\le \ts\frac{\eps}{2} e^{-\lambda s_0}$.
\end{prop}

Since the proof is very long and technical, in order not to interrupt the main line of arguments,
we postpone it to subsection~\ref{section6}.

\begin{prop} \label{TopolArg1}
Let $\eps,s_0$ be as in Proposition~\ref{mainAPE}.
If $\mathcal{U}_{s_0,s_1}\neq\emptyset$ with some $s_1 > s_0$, then 
$${\rm deg}(P(\cdot;s_0,s_1)) =1,$$
where ${\rm deg}(P(\cdot;s_0,s_1))$, denotes the degree of $P(\cdot;s_0,s_1)$ with respect to $0$ in $\mathcal{U}_{s_0,s_1}$.
\end{prop}

With Proposition~\ref{mainAPE} at hand, the proof is completely similar to \cite{HV94pre,MizADE}. We include it for completeness.

\begin{proof}
By \eqref{w0tildephi},
we have $p_j(d;s_0,s_0) = d_j-e^{-\lambda s_0}(\varphi_j,\hat\phi)$ for $j=0,\dots,\ell-1$.
Owing to Lemma~\ref{Lemcvcphi}, for $s_0\gg 1$, we have, for all $d\in\partial U_{s_0,s_0}$ and $\tau\in [0,1]$,
$$\sum_{j=0}^{\ell-1} |d_j+\tau (p_j(d;s_0,s_0)-d_j)|\ge \sum_{j=0}^{\ell-1} |d_j|
-e^{-\lambda s_0} \sum_{j=0}^{\ell-1} |(\varphi_j,\hat\phi)|
=e^{-\lambda s_0}\Bigl(\eps-\sum_{j=0}^{\ell-1} |(\varphi_j,\hat\phi)|\Bigr)>0.$$
Letting $I$ be the identity mapping in $\R^\ell$, it follows that
$$I(d)+\tau (P(d;s_0,s_0)-I(d))\ne 0 \quad\hbox{ on $\partial \mathcal{U}_{s_0,s_0}$ for $\tau\in [0,1]$}$$
hence, by the homotopy invariance of the degree, 
$$deg (P(\cdot;s_0,s_0),0,\mathcal{U}_{s_0,s_0})= deg (I,0,\mathcal{U}_{s_0,s_0}) = 1.$$
By Proposition~\ref{mainAPE}, for any $s_1 > s_0$, there is no $d\in\partial U_{s_0,s_1}$ such that $P(d;s_0,s_1)=0$. 
 Therefore, the homotopy invariance of the degree implies the conclusion.
\end{proof}

\begin{prop} \label{TopolArg2}
Let $\eps,s_0$ be as in Proposition~\ref{mainAPE}.
Then $\mathcal{U}_{s_0,s_1}\neq\emptyset$ for all $s_1>s_0$.
\end{prop}

\begin{proof}
Put 
$s_* = \sup\{s\ge s_0;\ \mathcal{U}_{s_0,s}\ne\emptyset\}$.
 Since $u_0$ is $C^1$ on $(x_0, 2\sigma)$ with $x_0=\tilde Ke^{(\gamma+\frac12)s_0}$, it follows
from standard parabolic theory that $u, u_x$ are continuous in 
$(x_0, 2\sigma)\times[0,t_0]$
for $t_0>0$ small. Consequently, taking for instance $d=0$, we see that $0\in \mathcal{U}_{s_0,s}$ for $s>s_0$ close to $s_0$, 
hence $s_* > s_0$.  
Assume that $s_*<\infty$. Taking a sequence $(s_n)$ with $s_n\to s_*$ as $n\to\infty$,
for each $n$ there is $d_n\in \mathcal{U}_{s_0,s_n}$ such that $P(d_n;s_0,s_n)=0$ by Proposition~\ref{TopolArg1}. 
 Since $(d_n)$ is bounded, we may assume without loss of generality that $d_n\to d_*$ as $n\to\infty$.
  Then, by continuous dependence with respect to initial data, we obtain $d_*\in \mathcal{U}_{s_0,s_*}$ and $P(d_*;s_0,s_*)= 0$. From Proposition~\ref{mainAPE},
  we have $\tilde w=\tilde w(y,s;d_*) \in \mathcal{A}^{1/2}_{s_0,s_*}$. 
  By continuity, we get $\tilde w \in \mathcal{A}^1_{s_0,s_*+\delta}$ 
  for some $\delta > 0$. This contradicts the definition of $s_*$, which completes the proof.
\end{proof}

\begin{proof}[Proof of Theorem~\ref{mainThm2}.] 
Let $w_0$ be given by \eqref{defw0} or \eqref{defw0odd}\footnote{Recalling \eqref{u0space}, we note that the case \eqref{defw0} will be sufficient for Theorem~\ref{mainThm2}; however the case 
\eqref{defw0odd} will be used in the proof of Theorem~\ref{prop:special}.}
 and $s_0$ be as in Proposition~\ref{mainAPE}.
Take a sequence $\{s_n\} \subset (s_0,\infty)$ with $s_n\to\infty$ as $n\to\infty$. From Proposition~\ref{TopolArg2}, 
for each $n$ there exists $d_n\in \mathcal{U}_{s_0,s_n}$, hence $\tilde w(y, s; d_n)\in \mathcal{A}^1_{s_0,s_n}$. 
Since $\{d_n\}$ is bounded, we may assume that $d_n\to \bar d$ as $n\to\infty$ for some $\bar d$.
By continuous dependence, it holds that $\tilde w(y, s; \bar d)\in \mathcal{A}^1_{s_0,\infty}$,
hence $w$ exists globally and 
\eqref{Concl2mainThm2b} is satisfied. 
Property~\eqref{Concl2mainThm2a} then follows from Proposition~\ref{PropInOut}.
Theorem~\ref{mainThm2} is proved. 

For later use, we note that \eqref{Concl2mainThm2a} yields
\be{firstrate}
{\underline a\hskip 0.6pt}_\eps e^{\frac{\lambda}{p-2}s} \le w_y(0,s)\le \overline a_\eps e^{\frac{\lambda}{p-2}s},
\quad\hbox{for all $s>s_0$,} 
\ee
where 
${\underline a\hskip 0.6pt}_\eps=d_p((1+\eps)a_*)^{-\beta}$, 
$\overline a_\eps=d_p((1-\eps)a_*)^{-\beta}$ and $d_p=\beta^\beta$.
\end{proof}

The following lemma plays a crucial role to rule out oscillation of the coefficients of the GBU rates.
It is a consequence of Theorem~\ref{mainThm2} and of an intersection-comparison argument.

\begin{lem}  \label{LemNonOsc} 
	Let $0<R\le\infty$.
	If a viscosity solution $\hat u$ 
	of \eqref{equ} with $\hat u_0\in \mathcal{W}$ or $\hat u_0\in \mathcal{W}_1$
	undergoes GBU at $(x,t) =(0, T) $ with $T<\infty$, then
	for any integer $\ell\ge 1$, we have
	\be{NonOsc1}
	0\le \liminf_{t\to T_-} \, (T-t)^{\frac{\ell}{p-2}} \hat u_x(0,t)= 
	\limsup_{t\to T_-} \, (T-t)^{\frac{\ell}{p-2}} \hat u_x(0,t) \le\infty.
	\ee
\end{lem}

\begin{proof}
	Assume for contradiction that \eqref{NonOsc1} fails and pick $0<\hat L_1<\hat L_2<\infty$ such that
	\be{NonOsc1neg}
	\liminf_{t\to T_-} \, (T-t)^{\frac{\ell}{p-2}} \hat u_x(0,t)
	< \hat L_1 <\hat L_2 < \limsup_{t\to T_-} \, (T-t)^{\frac{\ell}{p-2}} \hat u_x(0,t).
	\ee
	Taking $\eps\in(0, 1)$ 
	small enough so that $\overline a_\eps/{\underline a\hskip 0.6pt}_\eps<\hat L_2/\hat L_1$
	(with $\overline a_\eps,{\underline a\hskip 0.6pt}_\eps$ defined after \eqref{firstrate}), we set 
	\be{NonOsc2}
	N=\bigl(\ts\frac{{\underline a\hskip 0.6pt}_\eps\hat L_2}{\overline a_\eps\hat L_1}\bigr)^{1/2}>1,
	\quad m={\frac{1}{2(p-1)}-\frac{\ell}{p-2}}<0, ,
	\quad \tau=\bigl(\frac{\hat L_1 N}{{\underline a\hskip 0.6pt}_\eps}\bigr)^{1/m}.
	\ee
	Let $x_0=R/2$ if $R<\infty$, $x_0=1$ if $R=\infty$, and set $\bar R=x_0\sqrt{\tau}$. 
	Let $w$ be the global solution of \eqref{eqw} given by Theorem~\ref{mainThm2} applied for
	$\Omega=(0,\bar R)$ and the above choice of $\eps$.
	Set 
	$$u(x,t)=(T-t)^k w\bigl(x(T-t)^{-1/2},-\log(T-t)\bigr),\quad (x,t)\in[0,\bar R]\times[T-e^{-s_0},T).$$
	By \eqref{firstrate}, we have
	\be{NonOsc3}
	{\underline a\hskip 0.6pt}_\eps\le L_1:=\liminf_{t\to T} \, (T-t)^{\frac{\ell}{p-2}} u_x(0,t)\le L_2:=
	\limsup_{t\to T} \, (T-t)^{\frac{\ell}{p-2}} u_x(0,t)\le \overline a_\eps.
	\ee
	Set $\bar t=T+\tau(t-T)$, $\bar u(x,t)=\tau^{-k}u(\sqrt\tau\,x,\bar t)$, 
	and $Q:=(0,x_0)\times (t_0,T)$ with $t_0=T-\tau^{-1}e^{-s_0}$.
	The function $\bar u\in C^{2,1}(Q)\cap C(\overline Q)$ is then a classical solution of \eqref{equ} in $Q$,
	which satisfies the equality $(T-t)^{\frac{\ell}{p-2}} \bar u_x(0,t)=\tau^m (T-\bar t)^{\frac{\ell}{p-2}} u_x(0,\bar t)$ hence, 
	by \eqref{NonOsc2}, \eqref{NonOsc3},
	\be{NonOsc4}
	\hat L_1<\liminf_{t\to T} \, (T-t)^{\frac{\ell}{p-2}} \bar u_x(0,t)=\tau^m L_1
	<\tau^m L_2=\limsup_{t\to T} \, (T-t)^{\frac{\ell}{p-2}} \bar u_x(0,t)<\hat L_2.
	\ee
	It follows from \eqref{NonOsc1neg} and
	 \eqref{NonOsc4} that there exists an increasing sequence $t_i\to T$ such that $\bar u_x(0,t_i)=\hat u_x(0,t_i)$,
	hence $[\bar u-\hat u](\cdot,t_i)$ has a degenerate zero at $x=0$.
	This leads to a contradiction with the intersection-comparison principle (Proposition~\ref{propZeroNumber0}).
The latter can be applied because $\bar u(x_0,t)=\tau^{-k}u(\bar R,t)=0<\hat u(x_0,t)$. 
\end{proof}

\begin{proof}[Proof of Theorem~\ref{prop:special}.]
Pick $0<a<1<b$ and $c>0$ such that $|\varphi_\ell(y)|\ge cy^{2\ell}$ for all $y\in[b,\infty)$ and $\varphi_\ell(y)\ge c$ for all $y\in[0,a]$.
For $\eps\in(0,\min(c/2,1))$, let $w=w_\eps(y,s)$ be the special solution of \eqref{eqw} given by 
 the proof of Theorem~\ref{mainThm2} (either with \eqref{defw0} or \eqref{defw0odd}),
and let $u=u(x,t)$ be the solution of \eqref{equ} obtained from $w$ through the transformation \eqref{defw}, setting $T:=e^{-s_0}$. 

Let us first prove assertion (iv).
Take $\bar s>s_0$ such that $[a,b]\subset [K e^{-\gamma \bar s},\sigma e^{\bar s/2}]$.
By \eqref{Concl2mainThm2a}-\eqref{Concl2mainThm2b}, for all $s\ge \bar s$, the function $v:=w-U$ satisfies
$v(y,s)\le e^{-\lambda s} U_{a_-}(e^{\eta s}y)-U(y)<e^{-\lambda s} U(e^{\eta s}y)-U(y)=0$
for all $y\in (0,K e^{-\gamma s}]$ and 
$v(y,s)\ne 0$ for all $y\in[K e^{-\gamma s},a]\cup[b,\sigma e^{s/2})$. 
Moreover, 
\eqref{Concl2mainThm2b} guarantees that $\|v(\cdot,s)-\varphi_\ell\|_{C^1([a,b])}\to 0$ as $\eps\to 0$,
uniformly for $s\ge \bar s.$ 
Since the $\ell$ positive zeros $y_1<\dots<y_\ell$ 
of $\varphi_\ell$ are simple and located in $(a,b)$, 
we deduce that, for all $\eps$ sufficiently small, $v(\cdot,s)-\varphi_\ell$ has exactly $\ell$ zeros in $(0,\sigma e^{s/2})$,
which are simple and converge to $y_1,\dots,y_\ell$ as $\eps\to 0$.
Taking $\eps=\eps(\delta)$ sufficiently small (without loss of generality) and restating in terms of $u(x,t)$, this yields assertion~(iv). 

 Property \eqref{rateell2} in assertion (i) follows from \eqref{firstrate} and Lemma~\ref{LemNonOsc}.
The space-time profile of $u_x$ in \eqref{asympyux2} is 
a consequence of \eqref{eq:u_x-lowerupper} and we get that of $u$ by integration.
Assertion (ii) follows from \eqref{Concl2mainThm2b} with $i=0$.

To check assertion~(iii), let us first consider the case $R=\infty$.
When $\ell$ is even we take $w_0$ given by \eqref{defw0}, hence $u_0\in\mathcal{W}$.
When $\ell$ is odd we take $w_0$ given by \eqref{defw0odd}, hence $u_0\in\mathcal{W}_1$
(note that this is the only case where we need $u_0\not\in \mathcal{W}$; cf.~the last part of the theorem).
Then \eqref{outerdelta} follows from \eqref{estimuouter}.
Finally, when $R=1$, assertion~(iii) is guaranteed by Lemma~\ref{Lemcontrolx1}.
\end{proof}

\subsection{Proof Proposition~\ref{mainAPE}: the key a priori estimate}\label{section6}

This section is devoted to the proof of Proposition~\ref{mainAPE} and we make the following conventions 
throughout:

\vskip 2pt

$\bullet${\hskip 1.5mm}$C$ will denote a generic positive constant depending only on $p,\ell$;

\vskip 1pt

$\bullet${\hskip 1.5mm}the required largeness of $s_0\gg 1$ will depend on the parameter $\eps$, but not on $d$. 

\vskip 1pt

\noindent 
If $R=1$, we shall consider the extension $\tilde F$ defined in Lemma~\ref{BernsteinEst},
whereas if $R=\infty$ we just set $\tilde F=F$.
We shall make use of the variation of constants formula for $\tilde v$, given by \eqref{varconstv2tilde}
in Proposition~\ref{locexistvapproxExt},
where the initial data is given by
\be{defv0}
\tilde v(y,s_0)=\tilde w(y,s_0)-U(y)=w_0(y)-U(y)=\sum_{j=0}^{\ell-1}d_j\varphi_j-e^{-\lambda s_0}\hat\phi.
\ee

\subsubsection{First estimates.}

The following lemma gives pointwise estimates of the nonlinear term in the different regions.

\begin{lem} \label{LemvyFvy}
If $s_1\ge s_0\gg 1$ then, for any $d\in \mathcal{U}_{s_0,s_1}$ and $s\in[s_0,s_1]$, we have
\be{controlvybasic}
|\tilde v_y(y,s)|\le
\begin{cases}
Cy^{-\beta}& \hbox{ for $0< y\le y_0(s)$} \\
Ce^{-\eta s}y^{-\beta-1} +CK^{\beta+1}e^{-\mu s}y^{-\beta}
& \hbox{ for $y_0(s)\le y\le y_1(s)$} \\
Ce^{-\lambda s}(y+y^{2\ell-1})& \hbox{ for $y_1(s)\le y\le y_2(s)$} \\
Ce^{-\beta s/2}& \hbox{ for $y\ge y_2(s)$}
\end{cases}
\ee
where $\mu=\eta-\gamma$, and 
\be{controlFvybasic}
0\le \tilde F(\tilde v_y(y,s))\le
\begin{cases}
Cy^{-\beta-1}& \hbox{ for $0< y\le y_0(s)$} \\
C e^{-2\eta s}y^{-\beta-3} +CK^{2(\beta+1)}e^{-2\mu s}y^{-\beta-1}
& \hbox{ for $y_0(s)\le y\le y_1(s)$} \\
C e^{-2\lambda s}y^{\beta+1} \bigl(1+y^{4(\ell-1)}\bigr) & \hbox{ for $y_1(s)\le y\le y_2(s)$} \\
Ce^{-(\beta+1)s/2}& \hbox{ for $y\ge y_2(s)$.}  
\end{cases}
\ee
 \end{lem} 

\begin{proof}
By \eqref{estim0wyin} and recalling $\eta-\lambda=\beta\eta$, for $y\in (0,y_1(s)]$, we have
$$ (1-\eps_1(s))e^{\beta\eta s} U'_{a_+}(e^{\eta s}y)-U'(y)\le v_y(y,s)  
\le (1+\eps_1(s))e^{\beta\eta s} U'_{a_-}(e^{\eta s}y)-U'(y),$$
with $\eps_1$ given in Proposition~\ref{PropInOut}.
Since $0\le e^{\beta\eta s} U'_a(e^{\eta s}y)\le e^{\beta\eta s}U'(e^{\eta s}y)=U'(y)$, this particular gives the first case in \eqref{controlvybasic}.
Next, using $0\le 1-(1+h)^{-\beta} 
\le\beta h$, we have, for $a\in\{a_+,a_-\}$,
$$\begin{aligned}
|e^{\beta\eta s} U'_a(e^{\eta s}y)-U'(y)|
&= d_p\bigl|  e^{\beta\eta s} (a+ye^{\eta s})^{-\beta}-y^{-\beta}\bigr| 
= d_p\bigl| (y+ae^{-\eta s})^{-\beta}-y^{-\beta} \bigr| \\
&= d_p\bigl|y^{-\beta} (1+ay^{-1}e^{-\eta s})^{-\beta}-1\bigr|
\le Ce^{-\eta s}y^{-\beta-1}.
\end{aligned}$$
Since $0\le \eps_1(s) e^{\beta\eta s} U'_a(e^{\eta s}y)\le CK^{\beta+1}e^{-\mu s}y^{-\beta}$, 
this ensures the second case in \eqref{controlvybasic}.
As for the third case, it follows from $\tilde w\in \mathcal{A}^1_{s_0,s_1}$.

To verify the fourth case, we shall apply the maximum principle to the equation 
\be{eqnzux}
z_t-z_{xx}=p|z|^{p-2}zz_x
\ee
satisfied by $z=u_x$.
Since $\tilde w\in \mathcal{A}^1_{s_0,s_1}$, at $y=y_2(s)=\sigma e^{s/2}$, we have
\be{estimwyy2}
|w_y|\le Cy^{-\beta}+Ce^{-\lambda s}y^{2\ell-1}
\le Ce^{-\beta s/2}+Ce^{(k-\ell+\frac{2\ell-1}{2})s}=Ce^{-\beta s/2}.
\ee
Also, by \eqref{condd}--\eqref{defw0odd} and \eqref{boundb0}, for $y\in [\sigma e^{s_0/2},\infty)$, we have
\begin{eqnarray}
 \noalign{\vskip -2mm}
 \hskip -6mm |w_y(y,s_0)|
\hskip -6mm &&\le \ts\frac{1}{2\sigma}e^{-\frac{s_0}{2}}|\Theta_1'\bigl(\ts\frac{1}{2\sigma}e^{-\frac{s_0}{2}}y\bigr)|\Bigl|U-e^{-\lambda s_0}\phi
+\ds\sum_{j=0}^{\ell-1}d_j\varphi_j\Bigr|  \notag \\
 \noalign{\vskip -2mm}
\hskip -6mm &&\qquad\qquad\qquad +\Theta_1\bigl(\ts\frac{1}{2\sigma}e^{-\frac{s_0}{2}}y\bigr)\Bigl|U'-e^{-\lambda s_0}\phi'+\ds\sum_{j=0}^{\ell-1}d_j\phi'_j\Bigr| 
+ C e^{-\frac{\beta s_0}{2}} \label{estimwys0} \\
 \noalign{\vskip -1mm}
&&\le C\Bigl\{\ts\frac{1}{2\sigma}e^{ -\frac{s_0}{2}}\bigl(y^{1-\beta}+e^{-\lambda s_0}y^{2\ell}\bigr)
+y^{-\beta}+e^{-\lambda s_0}y^{2\ell-1}\Bigr\} \chi_{\{y\le 4\sigma e^{\frac{s_0}{2}}\}} + Ce^{-\frac{\beta s_0}{2}} \le Ce^{-\frac{\beta s_0}{2}}. \notag
\end{eqnarray}
Moreover, in the case $R=1$, Lemma~\ref{Lemcontrolx1} guarantees  
\be{bounduux1a}
\|u(t)\|_\infty\le 1,\quad 0<t<T
\ee
and
$|u_x(1,t)|\le 1,\ 0<t<T$.
Expressing \eqref{estimwyy2}, \eqref{estimwys0} in terms of $u_x$ (cf.~\eqref{defw}),
we see that $|u_x|\le C$ on the parabolic boundary of 
$(\sigma,R)\times (0,T-e^{-s_1})$.
Also recall that the solution of problem \eqref{equ} with \eqref{u0space} satisfies  
$u_x\in L^\infty((0,R)\times (0,\tau))$ for each $\tau\in (0,T)$.
In view of  this, we may apply the maximum principle (see~e.g.~\cite[Proposition 52.4]{QSbook})
 to equation \eqref{eqnzux} to deduce that 
\be{bounduux2}
|u_x|\le C\quad\hbox{ in $[\sigma,R]\times [0,T-e^{-s_1}]$.}
\ee
Since 
$$\tilde w_y(y,s)=e^{-\beta s/2}\bigl[\zeta'(ye^{-s/2})u(ye^{-s/2},T-e^{-s})+\zeta(ye^{-s/2})u_x(ye^{-s/2},T-e^{-s})\bigr]$$
by \eqref{defextuwv},  the fourth case in \eqref{controlvybasic} follows from \eqref{bounduux2},
 using also \eqref{bounduux1a} in case $R=1$.

On the other hand, for some $\bar\theta\in(0,1)$, we have
\be{Fvy2}
0\le F(v_y)=|U_y+v_y|^p-U_y^p-pU_y^{p-1}v_y=\frac{p(p-1)}{2}|U_y+\theta v_y|^{p-2}(v_y)^2.
\ee
Note that
$e^{-\lambda s}(y+y^{2\ell-1})\le Ce^{-\lambda s}e^{(2\ell-1)s/2}=Ce^{-\beta s/2}\le Cy^{-\beta}$ for $y\in[y_1(s),y_2(s)]$.
Assuming $s_0$ large enough so that $K^{\beta+1}e^{-\mu s_0}\le 1$, we see from \eqref{controlvybasic} that $|U_y+\theta v_y|\le Cy^{-\beta}$ for $y\le y_2(s)$.
and that $|U_y+\theta v_y|\le Ce^{-s/2}$ for $y\ge y_2(s)$.
This along with \eqref{controlvybasic}, \eqref{Fvy2} and $(p-2)\beta=1-\beta$, readily yields 
the first three cases in \eqref{controlFvybasic}, as well as the fourth case when $R=\infty$.

To check the fourth case when $R=1$,
note from \eqref{Extension1} that
\be{Ftildeg1}
\tilde F(\tilde v_y,s)=e^{(k-1)s}g(ye^{-s/2},T-e^{-s})-U_y^p-pU_y^{p-1}\tilde v_y,\quad y\ge y_2(s),
\ee
where $g(x,t):=|u_x|^p\zeta -2u_x \zeta_x-u\zeta_{xx}$. Since $|g(x,t)|\le C$
in $[\sigma,R]\times [0,T-e^{-s_1}]$ in view of \eqref{bounduux2} and \eqref{boundunifvisc}, we deduce that
$$|\tilde F(\tilde v_y,s)|\le Ce^{(k-1)s}+CU_y^p+C\tilde v_y^p\le Ce^{-(\beta+1)s/2},\quad y\ge y_2(s).
\qedhere$$
\end{proof}

We next estimate the Fourier coefficients of the nonlinear term with
respect to the eigenfunctions of the linearized operator.

\begin{lem}  \label{LemFvytau}
 Set $\bar\eta=\min(\eta,2\lambda)$. If $s_1\ge s_0\gg 1$ then, for any $d\in \mathcal{U}_{s_0,s_1}$, we have
\be{Fvytau}
|(\tilde F(\tilde v_y(\tau)),\varphi_j)|\le C(j+1)^{3/2}e^{-\bar\eta\tau},\quad j\in\N,\ \tau\in[s_0,s_1].
\ee
\end{lem} 

\begin{proof} 
We write
$$|(\tilde F(\tilde v_y(\tau)),\varphi_j)|\le 
\int_0^\infty |\tilde F(\tilde v_y(\tau))| |\varphi_j| \rho\, dy
=\int_0^{e^{-\eta\tau}}+\int_{e^{-\eta\tau}}^{Ke^{-\gamma\tau}}+\int_{Ke^{-\gamma\tau}}^{\sigma e^{\tau/2}}
+\int_{\sigma e^{\tau/2}}^\infty \equiv \sum_{n=1}^4 J_n.$$
Using \eqref{controlFvybasic}, \eqref{defK}, \eqref{EstimEigen1}, the Cauchy-Schwarz' inequality
 and $\|\varphi_j\|=1$, we obtain
$$\begin{aligned}
J_1&\le C \int_0^{e^{-\eta\tau}} y^{-(\beta+1)}y^\alpha |\varphi_j|dy
= C \int_0^{e^{-\eta\tau}} |\varphi_j|dy\le C(j+1)^{3/2} e^{-\eta\tau}, \\
J_2&\le C\int_{e^{-\eta\tau}}^{Ke^{-\gamma \tau}} 
\bigl(e^{-2\eta\tau} y^{-\beta-3} +K^{\beta+1}e^{-2\mu\tau} y^{-\beta-1}\bigr) |\varphi_j| y^\alpha dy\\
&\le C(j+1)^{3/2} \int_{e^{-\eta\tau}}^{Ke^{-\gamma \tau}}  \bigl(e^{-2\eta\tau}y^{-2}+K^{\beta+1}e^{-2\mu\tau}\bigr) dy 
\le C(j+1)^{3/2} \bigl(e^{-\eta\tau}+K^{\beta+2}e^{-(2\mu+\gamma)\tau}\bigr) \\
&\le C(j+1)^{3/2}e^{-\eta\tau} \bigl(1+\nu^{1-2p}e^{-\mu\tau}\bigr)\le C(j+1)^{3/2}e^{-\eta\tau},\\
J_3&\le C e^{-2\lambda\tau} \int_{Ke^{-\gamma\tau}}^{\sigma e^{\tau/2}} 
y^{\beta+1} \bigl(1+y^{4(\ell-1)}\bigr) |\varphi_j|\rho\le 
C e^{-2\lambda\tau} \|\varphi_j\| =C e^{-2\lambda\tau}, \\ 
J_4&\le Ce^{-(\beta+1)\tau/2}\int_{\sigma e^{\tau/2}}^\infty |\varphi_j|\rho\le 
Ce^{-(\beta+1)\tau/2} \|\varphi_j\|\Bigl(\int_{\sigma e^{\tau/2}}^\infty e^{-y^2/4}\,dy\Bigr)^{1/2}
\le Ce^{-C e^\tau}, 
\end{aligned}$$
hence \eqref{Fvytau}.
\end{proof}

We next show that, if the projections of $\tilde v$ on the lower eigenmodes vanish at the final time $s_1$,
then the corresponding coefficients $d_i$ of the initial data have to be small when $s_0\gg 1$.

\begin{lem}  \label{Lemdsmallnu}
If $s_1\ge s_0\gg 1$ then, for any $d\in \mathcal{U}_{s_0,s_1}$ such that $P(d;s_0,s_1)=0$, we have 
\be{dsmall}
|d_j|\le \nu e^{-\lambda s_0},\quad j\in\{0,\dots,\ell-1\}.
\ee
\end{lem} 

\begin{proof}
Taking inner product of \eqref{varconstv2tilde} with $\varphi_j$ and using \eqref{defSG2}, we get
$$(\tilde v(s),\varphi_j)=e^{-\lambda_j(s-s_0)}(\tilde v(s_0),\varphi_j)+\int_{s_0}^s e^{-\lambda_j(s-\tau)}(\tilde F(\tilde v_y(\tau)),\varphi_j)\,d\tau,\quad j\in\N.$$
For each $j\in\{0,\dots,\ell-1\}$, using $(\tilde v(s_1),\varphi_j)=0$ and multiplying with $e^{-\lambda_j(s-s_1)}$, we get, for all $s\in [s_0,s_1]$,
\be{v0phij}
e^{-\lambda_j(s-s_0)}(\tilde v(s_0),\varphi_j)=-\int_{s_0}^{s_1} e^{-\lambda_j(s-\tau)}(\tilde F(\tilde v_y(\tau)),\varphi_j)\,d\tau,\quad j\in\{0,\dots,\ell-1\}.
\ee
Let $j\in\{0,\dots,\ell-1\}$. It follows from \eqref{v0phij} with $s=s_0$ and \eqref{defv0} that
$$|d_j|\le e^{-\lambda s_0}|(\hat\phi,\varphi_j)|+\int_{s_0}^{s_1} e^{-\lambda_j(s_0-\tau)}|(\tilde F(\tilde v_y(\tau)),\varphi_j)|\,d\tau.$$
Then using $(\hat\phi,\varphi_j)=(\hat\phi-\phi,\varphi_j)$, \eqref{tildephicv}, \eqref{Fvytau}, 
$\bar\eta>\lambda>\lambda_j$, we deduce that,
for $s_0\gg 1$,
$$
|d_j|
\le \tsfr\nu e^{-\lambda s_0}+Ce^{-\lambda_j s_0}\ds\int_{s_0}^\infty e^{(\lambda_j-\bar\eta)\tau}\,d\tau
=  \tsfr\nu e^{-\lambda s_0}+Ce^{-\bar\eta s_0}\le \nu e^{-\lambda s_0}.
\qedhere
$$
\end{proof}

To proceed with estimating $\tilde v$, using \eqref{varconstv2tilde}, we split $\tilde v$ as
\be{defS123}
\tilde v=S_1+S_2+S_3,\ \hbox{ where }\ 
\begin{cases}
S_1(y,s)&=-e^{-\lambda s}(\hat\phi,\phi)\phi,\\
\noalign{\vskip 1mm}
S_2(y,s)&=\ds\sum_{j=0}^{\ell-1}d_je^{-\lambda_j(s-s_0)}\varphi_j-\sum_{j\ne\ell} e^{-\lambda s_0}e^{-\lambda_j(s-s_0)}(\hat\phi,\varphi_j)\varphi_j, \\
S_3(y,s)&=\ds\int_{s_0}^s e^{-(s-\tau)\mathcal{L}}\tilde F(\tilde v_y(\tau))\,d\tau.
\end{cases}
\ee
The following pointwise bounds for the initial data of $S_2$ will be needed below.

\begin{lem}  \label{LemestimS2s0}
If $s_1\ge s_0\gg 1$ then, for any $d\in \mathcal{U}_{s_0,s_1}$ such that $P(d;s_0,s_1)=0$ we have
\be{estimS2}
|S_2(y,s_0)|\le
\begin{cases}
Ce^{-\lambda s_0}& \hbox{ in $D_1:=(0,\tilde K e^{-\gamma s_0}]$} \\
C\nu e^{-\lambda s_0}(1+y^{2\ell})& \hbox{ in 
$D_2:=[\tilde K e^{-\gamma s_0},2\sigma e^{s_0/2}]$} \\
C y^{2\ell}& \hbox{ in $D_3:=[2\sigma e^{s_0/2},\infty)$, }
\end{cases}
\ee
and
\be{estimDS2}
|\partial_yS_2(y,s_0)|\le
\begin{cases}
Cy^{-\beta}& \hbox{ in $D_{1,1}:=(0,e^{-\eta s_0}]$} \\
C \tilde K^{\beta+2} y^{-\beta-1}  e^{-\eta s_0}& \hbox{ in $D_{1,2}:=[e^{-\eta s_0},\tilde K e^{-\gamma s_0}]$} \\
C\nu e^{-\lambda s_0}(y+y^{2\ell-1})& \hbox{ in $D_2$} \\
C y^{2\ell-1}& \hbox{ in $D_3$.}
\end{cases}
\ee
\end{lem} 

\begin{proof}
From \eqref{defS123} and \eqref{w0tildephi} we have
\be{splitS2b}
S_2(y,s_0)
=\tilde v_0+e^{-\lambda s_0}(\hat\phi,\phi)\phi 
=\sum_{j=0}^{\ell-1}d_j\varphi_j+e^{-\lambda s_0}\bigl((\hat\phi,\phi)\phi-\hat\phi\bigr).
\ee
Let us first consider the range $D_1$.
Observe that the function $h(z):=a^{1-\beta}+z^{1-\beta}-(a+z)^{1-\beta}$ satisfies 
$0\le h(z)\le a^{1-\beta}$ for $z\ge 0$ (due to $h'(z)\ge 0$, $h(0)=0$). 
Therefore, by \eqref{defw0}, we have, for all $y\in D_1$, 
$$
|v_0(y)|
= \bigl|e^{-\lambda s_0} U_a(ye^{\eta s_0})-U(y)\bigr| = \Bigl|c_p e^{-\lambda s_0} \bigl[(a+ye^{\eta s_0})^{1-\beta}-a^{1-\beta}-(ye^{\eta s_0})^{1-\beta}\bigr]
\Bigr|\le Ce^{-\lambda s_0}.
$$
Also, for all $y\in D_1$,
$$\begin{aligned}
|v_{0,y}(y)|
&= \bigl|e^{-\lambda s_0}e^{\eta s_0} U'_a(ye^{\eta s_0})-U'(y)\bigr| 
= d_p\bigl|  e^{(\eta-\lambda)s_0} (a+ye^{\eta s_0})^{-\beta}-y^{-\beta}\bigr| \\
&= d_p\bigl| (y+ae^{-\eta s_0})^{-\beta}-y^{-\beta} \bigr|
= d_p y^{-\beta}\bigl|(1+ay^{-1}e^{-\eta s_0})^{-\beta}-1\bigr|\le
 C\min(y^{-\beta}, y^{-\beta-1}  e^{-\eta s_0}).
\end{aligned}$$
Consequently,
$$|S_2(y,s_0)|\le Ce^{-\lambda s_0},\quad
|\partial_yS_2(y,s_0)|\le  C\min(y^{-\beta}, y^{-\beta-1}  e^{-\eta s_0}) + Cye^{-\lambda s_0},
\ \quad y\in D_1.$$
Moreover, we have 
$ye^{-\lambda s_0}\le \tilde K^{\beta+2} y^{-\beta-1}  e^{-\eta s_0}$ in $D_1$,
due to $y^{\beta+2}  \le \tilde K^{\beta+2} e^{-\beta\eta s_0}= \tilde K^{\beta+2} e^{-(\eta-\lambda) s_0}$.
This implies \eqref{estimS2}-\eqref{estimDS2} in the range $D_1$.

Next, by \eqref{splitS2b} and \eqref{w0tildephi}--\eqref{w0tildephi2odd}, we have
$S_2(y,s_0)=\sum_{j=0}^{\ell-1}d_j\varphi_j+e^{-\lambda s_0}\bigl((\hat\phi,\phi)-1\bigr)\phi$ in $D_2$.
Consequently, \eqref{estimS2}-\eqref{estimDS2} in $D_2$
follows from \eqref{controlDiphi} and Lemmas~\ref{Lemcvcphi}, \ref{Lemdsmallnu}.
Finally, from \eqref{defw0} (resp., \eqref{defw0odd}), we have
$$
S_2(y,s_0)=(\Theta(y)-1)U(y)+\Theta(y)\ds\sum_{j=0}^{\ell-1}d_j\varphi_j
+e^{-\lambda s_0}\bigl((\hat\phi,\phi)-\Theta(y)\bigr)\phi
\quad\hbox{in $D_3$}
$$
(resp., $S_2(y,s_0)=(b_0-1)U(y)+e^{-\lambda s_0}(\hat\phi,\phi)\phi$).
This along with \eqref{controlDiphi} guarantees \eqref{estimS2}-\eqref{estimDS2} 
in the range $D_3$. \end{proof}

\goodbreak

\subsubsection{Short time estimate of $w$ and $w_y$.}

In this subsection, we shall prove:

\begin{prop} \label{Propwshort}
If $s_1\ge s_0\gg 1$ then, for any $d\in \mathcal{U}_{s_0,s_1}$ such that $P(d;s_0,s_1)=0$, we have
\be{estimM2}
|D^i(w-U+e^{-\lambda s}\phi)|\le  M_2\nu e^{-\lambda s}(y^i+y^{2\ell-i}),\quad s\in (s_0,s_0+1],
 y\in [y_1(s),y_2(s)],
\ee
for $i\in\{0,1\}$, with $M_2=M_2(p,\ell)>0$.
\end{prop} 

The proof of Proposition~\ref{Propwshort} requires the short time estimation of $S_2$ and $S_3$.
This is done in the following Lemmas~\ref{LemS2short} and \ref{LemS3short}, respectively.

\begin{lem} \label{LemS2short}
If $s_1\ge s_0\gg 1$ then, for any $d\in \mathcal{U}_{s_0,s_1}$ such that $P(d;s_0,s_1)=0$, we have
$$|D^iS_2(y,s)|\le  C\nu e^{-\lambda s}(y^i+y^{2\ell-i}),\quad s\in (s_0,s_0+1],
\ \ y\in [y_1(s),y_2(s)].$$
\end{lem} 

\begin{proof}
Take $i\in\{0,1\}$, $s\in (s_0,s_0+1]$ and $y\in (y_1(s),y_2(s))$.
Since $S_2(\cdot,s)=e^{-(s-s_0)\mathcal{L}}S_2(\cdot,s_0)$, by Proposition~\ref{KernelEstim2},
\eqref{boundH1} and
using $1-e^{s_0-s}\ge \frac12(s-s_0)$, 
it follows that
$$\begin{aligned}
|D^iS_2(y,s)|
&\le \int_0^\infty G_i(s-s_0,y,\xi)|D^iS_2(\xi,s_0)|\xi^\alpha d\xi \\
&= \int_0^{e^{-\eta s_0}}+\int_{e^{-\eta s_0}}^{\tilde K e^{-\gamma s_0}}+ \int_{\tilde K e^{-\gamma s_0}}^{2\sigma e^{s_0/2}}+ \int_{2\sigma e^{s_0/2}}^\infty\equiv \sum_{n=1}^4 S^i_{2,n},
\end{aligned}$$
where $|D^iS_2(\xi,s_0)$ will be estimated by Lemma~\ref{LemestimS2s0} and
\be{defGi}
G_i(s-s_0,y,\xi)\le C(s-s_0)^{-\frac{\alpha+1}{2}}\Bigl(1\wedge \frac{y\xi}{s-s_0}\Bigr)^i\Bigl(1+\frac{y\xi}{s-s_0}\Bigr)^{-\frac{\alpha}{2}}
\exp\Bigl[-\frac{C(e^{(s_0-s)/2}y-\xi)^2}{s-s_0}\Bigr].
\ee

\smallskip
$\bullet$ {\bf Estimate of $S^0_{2,1}$ and $S^0_{2,2}$.}
Let 
\be{defM1}
M_1(p)=(\ts\frac12 e^{-1/2})^{\alpha+1}.
\ee
Since $\nu\le M_0\le M_1$ (recall \eqref{defK}), we have 
$K/\tilde K=\nu^{-1/(\alpha+1)}\ge 2e^{1/2}$, hence
$e^{(s_0-s)/2}y\ge e^{-\frac12}K e^{-\gamma s_0}\ge 
2\xi$
for all $\xi\in(0,\tilde K e^{-\gamma s_0})$. Consequently,
$$\begin{aligned}
S_{2,1}^0+S_{2,2}^0
&\le Ce^{-\lambda s_0} (s-s_0)^{-\frac{\alpha+1}{2}}
\int_0^{\tilde K e^{-\gamma s_0}}
\exp\Bigl[-\frac{C\bigl(e^{(s_0-s)/2}y-\xi\bigr)^2}{s-s_0}\Bigr]\xi^\alpha d\xi \\
&\le Ce^{-\lambda s_0}(s-s_0)^{-\frac{\alpha+1}{2}} \exp\Bigl[-\frac{Cy^2}{s-s_0}\Bigr]
(\tilde K e^{-\gamma s_0})^{\alpha+1} \\
&\le Ce^{-\lambda s_0} \Bigl(\frac{y}{\sqrt{s-s_0}}\Bigr)^{\alpha+1} \exp\Bigl[-\frac{Cy^2}{s-s_0}\Bigr]
\Bigl(\frac{\tilde K e^{-\gamma s_0}}{y}\Bigr)^{\alpha+1}
\le C\Bigl(\frac{\tilde K}{K}\Bigr)^{\alpha+1} e^{-\lambda s_0} = C\nu e^{-\lambda s_0}.
\end{aligned}$$

\smallskip
$\bullet$ {\bf Estimate of $S^1_{2,1}$ and $S^1_{2,2}$.}
Recalling $\alpha=\beta+1$, we have
$$\begin{aligned}
S_{2,1}^1
&\le C(s-s_0)^{-\frac{\alpha+1}{2}}\int_0^{e^{-\eta s_0}} \frac{y\xi}{s-s_0}
\exp\Bigl[-\frac{C\bigl(e^{(s_0-s)/2}y-\xi\bigr)^2}{s-s_0}\Bigr]\xi^\alpha \xi^{-\beta}d\xi \\
&\le Cy^{-\alpha-2}\Bigl(\frac{y}{\sqrt{s-s_0}}\Bigr)^{\alpha+3}\exp\Bigl[-\frac{Cy^2}{s-s_0}\Bigr]\int_0^{e^{-\eta s_0}} \xi^2 d\xi 
\le Cy^{-\alpha-2}e^{-3\eta s_0},
\end{aligned}$$
$$\begin{aligned}
S_{2,2}^1
&\le C\tilde K^{\beta+2} e^{-\eta s_0} (s-s_0)^{-\frac{\alpha+1}{2}}
\int_{e^{-\eta s_0}}^{\tilde K e^{-\gamma s_0}}
\frac{y\xi}{s-s_0}\exp\Bigl[-\frac{C\bigl(e^{(s_0-s)/2}y-\xi\bigr)^2}{s-s_0}\Bigr]
\xi^{-\beta-1}\xi^\alpha d\xi \\
&\le  C\tilde K^{\beta+2} y e^{-\eta s_0} (s-s_0)^{-\frac{\alpha+3}{2}} \exp\Bigl[-\frac{Cy^2}{s-s_0}\Bigr]
\int_0^{\tilde K e^{-\gamma s_0}}
\xi d\xi\\
&=C\tilde K^{\beta+2} y^{-\alpha-2}e^{-\eta s_0} \Bigl(\frac{y}{\sqrt{s-s_0}}\Bigr)^{\alpha+3} \exp\Bigl[-\frac{Cy^2}{s-s_0}\Bigr]
(\tilde K e^{-\gamma s_0})^2 
\le C\tilde K^{\beta+4}y^{-\alpha-2}e^{-(\eta+2\gamma) s_0}.
\end{aligned}$$
Since $y^{\alpha+3}e^{2\gamma s_0}  \ge (Ke^{-\gamma s_0})^{\beta+4}e^{2\gamma s_0} 
=K^{\beta+4}  e^{-\beta\eta s_0}=K^{\beta+4}  e^{(\lambda-\eta)s_0}$
and $\eta>\gamma$, it follows from \eqref{defK} that
$$S_{2,1}^1+S_{2,2}^1
\le Cy^{-\alpha-2}\tilde K^{\beta+4}e^{-(\eta+2\gamma) s_0}\le C\Bigl(\frac{\tilde K}{K}\Bigr)^{\beta+4}ye^{-\lambda s_0}
\le C\nu y e^{-\lambda s_0}.$$

\smallskip
$\bullet$ {\bf Estimate of $S^i_{2,3}$.}
Set $E_y=(\tilde K e^{-\gamma s_0}, 2\sigma e^{s_0/2})$,
$E^1_y=E_y\cap (y/2,2y)$,
$E^2_y=E_y\setminus (y/2,2y)$
and
$$S_{2,3}^i= \int_{E_y^1} G_i(s-s_0,y,\xi)|D^iS_2(\xi,s_0)|\xi^\alpha d\xi+\int_{E_y^2} G_i(s-s_0,y,\xi)|D^iS_2(\xi,s_0)|\xi^\alpha d\xi
\equiv S_{2,3}^{i,1}+S_{2,3}^{i,2}.$$
Putting $\hat y=\frac{e^{(s_0-s)/2}y}{\sqrt{s-s_0}}$ 
and using the change of variables $\xi=z\sqrt{s-s_0}$, we get  
$$\begin{aligned}
S_{2,3}^{i,1}
&\le C\nu e^{-\lambda s_0} (s-s_0)^{-\frac{\alpha+1}{2}} \int_{E_y^1}
 \Bigl(1+\frac{y\xi}{s-s_0}\Bigr)^{-\alpha/2}\exp\Bigl[-C\Bigl(\hat y-\frac{\xi}{\sqrt{s-s_0}}\Bigr)^2\Bigr] (\xi^i+\xi^{2\ell-i})\xi^\alpha d\xi \\
&\le C\nu e^{-\lambda s_0}(s-s_0)^{-\frac{\alpha}{2}}  (y^i+y^{2\ell-i})   
 \Bigl(1+\frac{y^2}{s-s_0}\Bigr)^{-\alpha/2} y^\alpha\int_{y/2}^{2y}
\exp\Bigl[-C\Bigl(\hat y-\frac{\xi}{\sqrt{s-s_0}}\Bigr)^2\Bigr]  \frac{d\xi}{\sqrt{s-s_0}}\\
&\le C\nu e^{-\lambda s_0}  (y^i+y^{2\ell-i})   
\Bigl(\frac{s-s_0+y^2}{y^2}\Bigr)^{-\alpha/2} 
 \int_0^\infty e^{-C(\hat y-z)^2}  dz\le C\nu e^{-\lambda s_0}(y^i+y^{2\ell-i}).
\end{aligned}$$
Next, using $s\in (s_0,s_0+1]$ and $|\xi-e^{(s_0-s)/2}y|\ge C\xi$ for $\xi\in {E_y^2}$, we obtain
$$\begin{aligned}
S_{2,3}^{i,2}
&\le C\nu e^{-\lambda s_0}(s-s_0)^{-\frac{\alpha+1}{2}}
\int_{E_y^2} \Bigl(\frac{y\xi}{s-s_0}\Bigr)^i\exp\Bigl[-\frac{\tilde C\xi^2}{s-s_0}\Bigr]
(\xi^i+\xi^{2\ell-i})\xi^\alpha d\xi\\
&=C\nu e^{-\lambda s_0}(s-s_0)^{-\frac{\alpha+1}{2}}y^i
\int_{E_y^2} \Bigl(\frac{\xi^2}{s-s_0}\Bigr)^i\exp\Bigl[-\frac{\tilde C\xi^2}{s-s_0}\Bigr]
(1+\xi^{2(\ell-i)})\xi^\alpha d\xi\\
&\le C\nu e^{-\lambda s_0} y^i  
\int_0^\infty e^{-\tilde Cz^2}z^{2 i}(1+z^{2(\ell- i)})z^\alpha dz
\le C\nu e^{-\lambda s_0} y^i.
\end{aligned}$$

\smallskip
$\bullet$ {\bf Estimate of $S^i_{2,4}$.}
We have $\xi-e^{(s_0-s)/2}y\ge \xi/2$ 
for any $y\le \sigma e^{s_0/2}$ and $\xi\ge 2\sigma e^{s_0/2}$, hence
$$\begin{aligned}
S^i_{2,4}
&\le C(s-s_0)^{-\frac{\alpha+1}{2}} 
\int_{2\sigma e^{s_0/2}}^\infty
 \Bigl(\frac{y\xi}{s-s_0}\Bigr)^i
 \exp\Bigl[-\frac{C\xi^2}{s-s_0}\Bigr] \xi^{2\ell-i+\alpha} d\xi \\
 & \le C(s-s_0)^{-\frac{\alpha+1+2i}{2}}  y^i \int_{2\sigma e^{s_0/2}}^\infty
 \exp\Bigl[-\frac{C\xi^2}{s-s_0}\Bigr] \xi^{2\ell+\alpha} d\xi \\
& =C(s-s_0)^{\ell-i-\frac12}y^i \int_{2\sigma e^{s_0/2}/\sqrt{s-s_0}}^\infty e^{-Cz^2} z^{2\ell+\alpha} dz
\le C  y^i\exp\bigl[-Ce^{s_0}\bigr]
\le  C\nu e^{-\lambda s_0} y^i.
\end{aligned}$$
Gathering the above estimates, the lemma follows. 
\end{proof}

We now turn to $S_3$. For later purpose, we shall actually prove a slightly more general estimate,
replacing $\int_{s_0}^s$ in $S_3$ by $\int_{\bar s}^s$ with any $\bar s\in \bigl[\max(s_0,s-1),s\bigr]$,
and considering any $s_1\ge s_0$ (not necessarily $s_1\le s_0+1$).
Indeed, the key point in the following argument is that the integration is made on a time interval of 
bounded length (say, at most $1$).
 
\begin{lem}  \label{LemS3short}
Assume $s_1\ge s_0\gg 1$ and let $d\in \mathcal{U}_{s_0,s_1}$ be such that $P(d;s_0,s_1)=0$.
Then, for any $s\in (s_0,s_1]$, $\bar s\in \bigl[\max(s_0,s-1),s\bigr]$ and $i\in\{0,1\}$, we have
\be{controlS3short0}
|D^i\hat S_3(y,s)|:=\Bigl|\int_{\bar s}^s D^ie^{-(s-\tau)\mathcal{L}}\tilde F(\tilde v_y(\tau))\,d\tau\Bigr| 
\le  C\nu e^{-\lambda s}(y^i+y^{2\ell-i}), \qquad y\in [y_1(s),y_2(s)].
\ee
In particular, we have 
\be{controlS3short}
|D^iS_3(y,s)|\le  \nu e^{-\lambda s}(y^i+y^{2\ell-i}), \qquad s_0\le s\le s_0+1\le s_1,\ y\in [y_1(s),y_2(s)], \ i\in\{0,1\}.
\ee
\end{lem}

\begin{proof}
Take $i\in\{0,1\}$, $s\in (\bar s,s_1)$ and $y\in (y_1(s),y_2(s))$.
Also, for $\tau\in [s_0,s)$, we shall use the notation
$$X(\tau)=e^{(\tau-s)/2}(s-\tau)^{-1/2}y,\qquad X_j(\tau)=(s-\tau)^{-1/2}y_j(\tau),\ \ j\in\{0,1,2\}$$
and we observe that $X(\tau)\in(e^{(\tau-s)/2}X_1,e^{(\tau-s)/2}X_2)\subset (e^{-1/2}X_1,X_2)$.
Using Proposition~\ref{KernelEstim2}, \eqref{boundH1} with $i=0$, \eqref{boundH2} and $1-e^{-t}\ge Ct$ for $t\in (0,1]$, we have
$$\begin{aligned}
|D^i\hat S_3(y,s)|
&\le C \int_{\bar s}^s\int_0^\infty (s-\tau)^{-\frac{\alpha+1+i}{2}}
\Bigl(1\wedge \frac{e^{\frac{\tau-s}{2}}y}{\sqrt{s-\tau}}\Bigr)^i\\ 
&\qquad\qquad \times\Bigl(1+\frac{e^{\frac{\tau-s}{2}}y\xi}{s-\tau}\Bigr)^{-\frac{\alpha}{2}}\exp\Bigl[-C\frac{\bigl(e^{\frac{\tau-s}{2}}y-\xi\bigr)^2}{s-\tau}\Bigr] |\tilde F(\tilde v_y(\xi,\tau))| \xi^\alpha d\xi d\tau. 
\end{aligned}$$
By the change of variables $z=(s-\tau)^{-1/2}\xi$, it follows that
\begin{eqnarray}
|D^i\hat S_3(y,s)|
\hskip -6mm &&\le\int_{\bar s}^s\int_0^\infty 
 (s-\tau)^{-\frac{i}{2}} e^{-C(X(\tau)-z)^2}\bigl(1\wedge X(\tau)\bigr)^i 
 \bigl(1+X(\tau)z\bigr)^{-\frac{\alpha}{2}} \bigl|\tilde F\bigl(\tilde v_y(z\sqrt{s-\tau},\tau)\bigr)\bigr| 
 z^\alpha dz d\tau \notag\\
&&=\int_{\bar s}^s\int_0^{X_0(\tau)} + \int_{\bar s}^s\int_{X_0(\tau)}^{X_1(\tau)}  + \int_{\bar s}^s\int_{X_1(\tau)}^{X_2(\tau)} 
+ \int_{\bar s}^s\int_{X_2(\tau)}^\infty \equiv \sum_{n=1}^4 S^i_{3,n}(y,s). \label{splitS3}
\end{eqnarray}

\smallskip
$\bullet$ {\bf Estimate of $S^i_{3,1}$.}
 Using \eqref{controlFvybasic} and $X(\tau)\ge 2X_0(\tau)$, we get
$$\begin{aligned}
S_{3,1}^i
&\le C\int_{\bar s}^s\int_0^{X_0(\tau)} (s-\tau)^{-\frac{\alpha+i}{2}} 
e^{-C(X(\tau)-z)^2}  
dz d\tau\\
&\le C\int_{\bar s}^s (s-\tau)^{-\frac{\alpha+i}{2}} X_0(\tau)e^{-CX^2(\tau)} d\tau
\le Ce^{-\eta s}\int_{\bar s}^s (s-\tau)^{-\frac{\alpha+1+i}{2}} e^{-Cy^2/(s-\tau)}d\tau.
\end{aligned}$$
On the other hand, by \eqref{defK}, since $y^{\beta+2i}\ge (Ke^{-\gamma s})^{\beta+2i}\ge K^\beta e^{-(\beta+2)\gamma s}= K^\beta e^{(\lambda-\eta)s}$, we have
\be{condbetaeta1}
y^{-\beta-i}\le \nu e^{(\eta-\lambda)s}y^i.
\ee
Using the change of variables $\zeta=y^2/(s-\tau)$, hence $s-\tau=y^2/\zeta$,  
$d\tau=y^2\zeta^{-2}d\zeta$ and $(s-\tau)^{-\frac{\alpha+1+i}{2}}=y^{-(\alpha+1+i)}\zeta^{\frac{\alpha+1+i}{2}}$, 
along with $\alpha>1$ and \eqref{condbetaeta1}, we obtain,  
$$S_{3,1}^i
\le C e^{-\eta s}y^{1-\alpha-i}\int_0^\infty\zeta^{\frac{\alpha+i-3}{2}} 
e^{-C\zeta}  d\zeta\le C e^{-\eta s}y^{-\beta-i}  
\le C\nu y^ie^{-\lambda s}.$$

\smallskip
$\bullet$ {\bf Estimate of $S^i_{3,2}$.}  
By \eqref{controlFvybasic} and Lemma~\ref{AuxilLem2}(i), we have
$$
\begin{aligned}
S_{3,2}^i
&\le Ce^{-2\eta s} \int_{\bar s}^s  \int_{X_0(\tau)}^{X_1(\tau)}
(s-\tau)^{-\frac{\alpha+2+i}{2}} e^{-C(X(\tau)-z)^2}(1+X(\tau)z)^{-\alpha/2} z^{-2} dz d\tau\\
&\qquad + \bar Ce^{-2\mu s} \int_{\bar s}^s  \int_{X_0(\tau)}^{X_1(\tau)}
(s-\tau)^{-\frac{\alpha+i}{2}} e^{-C(X(\tau)-z)^2}(1+X(\tau)z)^{-\alpha/2} dz d\tau\\
&\le Ce^{-2\eta s} \int_{\bar s}^s \frac{(s-\tau)^{-\frac{\alpha+2+i}{2}}e^{-C\frac{y^2}{s-\tau}}d\tau}{X_0(\tau)}
+Ce^{-2\eta s} \int_{\bar s}^s \frac{(s-\tau)^{-\frac{\alpha+2+i}{2}}d\tau}{(1+X_1^2(\tau))^{\frac{\alpha+1}{2}}X_1(\tau)}\\
&\qquad \bar Ce^{-2\mu s}  \int_{\bar s}^s X_1(\tau)(s-\tau)^{-\frac{\alpha+i}{2}}e^{-C\frac{y^2}{s-\tau}}d\tau
+\bar Ce^{-2\mu s}  \int_{\bar s}^s \frac{{X_1(\tau)}(s-\tau)^{-\frac{\alpha+i}{2}}d\tau}{(1+X_1^2(\tau))^{\frac{\alpha+1}{2}}}
\equiv \sum_{j=1}^4 S_{3,2}^{i,j},
\end{aligned}
$$
where $\bar C=CK^{2(\beta+1)}$.
Using $y_0(\tau)\ge e^{-\eta s}$, $y_1(\tau)\le CKe^{-\gamma s}$ for $\tau\in (\bar s,s)$,
 and the change of variables $t=s-\tau$, we~have
$$\begin{aligned}
S_{3,2}^{i,1}
&\le Ce^{-2\eta s}\int_0^{s-\bar s} \frac{t^{-\frac{\alpha+1+i}{2}}e^{-Cy^2/t}}{y_0(\tau)}\, dt
\le Ce^{-\eta s} y^{-\alpha-1-i}
\int_0^{s-\bar s} (y^2/t)^{\frac{\alpha+1+i}{2}}e^{-Cy^2/t}\, dt\\
S_{3,2}^{i,3}
&\le \bar Ce^{-2\mu s}\int_0^{s-\bar s} y_1(\tau)t^{-\frac{\alpha+1+i}{2}}e^{-Cy^2/t}\, dt
\le \bar CKe^{-(2\mu+\gamma) s} y^{-\alpha-1-i}
\int_0^{s-\bar s} (y^2/t)^{\frac{\alpha+1+i}{2}}e^{-Cy^2/t}\, dt.
\end{aligned}$$
Recalling $\mu=\eta-\gamma$ and 
using the change of variables $u=\frac{y^2}{t}$, hence $t=\frac{y^2}{u}$, $dt=-\frac{y^2}{u^2}du$, we~obtain
$$\begin{aligned}
S_{3,2}^{i,1}+S_{3,2}^{i,3}
&\le C(1+e^{-\mu s}K^{2\beta+3}) e^{-\eta s} y^{-\alpha-1-i}
\int_0^{s-\bar s} (y^2/t)^{\frac{\alpha+1+i}{2}}e^{-Cy^2/t}\, dt \\
&\le C(1+e^{-\mu s_0}K^{2\beta+3})e^{-\eta s} y^{1-\alpha-i}\int_0^\infty u^{\frac{\alpha+i-3}{2}}e^{-Cu}\, du
\le Ce^{-\eta s} y^{-\beta-i}\le C\nu y^ie^{-\lambda s},
\end{aligned}
$$
due to $\alpha>1$, 
 \eqref{condbetaeta1} and $s_0\gg 1$.
Next, by the inequality $Ke^{-\gamma s}\le y_1(\tau)\le CKe^{-\gamma s}$ for $\tau\in (\bar s,s)$
and the change of variables $t=s-\tau$, we have
$$\begin{aligned}
S_{3,2}^{i,2}
&\le Ce^{-2\eta s} \int_0^{s-\bar s} 
\frac{t^{-\frac{\alpha+i+1}{2}}}{(1+t^{-1}y_1^2(\tau))^{(\alpha+1)/2}y_1(\tau)}\, dt 
\le CK^{-1}e^{(\gamma-2\eta)s} \int_0^{s-\bar s} 
\frac{t^{-\frac{\alpha+i+1}{2}}}{[1+t^{-1}e^{-2\gamma s}]^{(\alpha+1)/2}}\, dt \\
S_{3,2}^{i,4}
&\le \bar Ce^{-2\mu s} \int_0^{s-\bar s} 
\frac{t^{-\frac{\alpha+i+1}{2}}y_1(\tau)}{(1+t^{-1}y_1^2(\tau))^{(\alpha+1)/2}}\, dt 
\le \bar CK e^{-(2\mu+\gamma)s} \int_0^{s-\bar s} 
\frac{t^{-\frac{\alpha+i+1}{2}}}{[1+t^{-1}e^{-2\gamma s}]^{(\alpha+1)/2}}\, dt.
\end{aligned}$$
Using $\mu=\eta-\gamma$ and the change of variables $z=te^{2\gamma s}$, 
we get
$$\begin{aligned}
S_{3,2}^{i,2}+S_{3,2}^{i,4}
&\le \bar CK e^{(\gamma-2\eta)s}\int_0^{s-\bar s}  
\frac{t^{-\frac{\alpha+i+1}{2}}\, dt}{[1+(te^{2\gamma s})^{-1}]^{(\alpha+1)/2}}
\le \bar CKe^{(\gamma(\alpha+i)-2\eta)s}  \int_0^{s-\bar s}  
\frac{(te^{2\gamma s})^{-\frac{\alpha+i+1}{2}}\, e^{2\gamma s} dt}{[1+(te^{2\gamma s})^{-1}]^{(\alpha+1)/2}}\\
&\le \bar C K e^{(\gamma(\alpha+i)-2\eta)s} \int_0^{s-\bar s}  
(te^{2\gamma s})^{-i/2}[1+te^{2\gamma s}]^{-(\alpha+1)/2}\, e^{2\gamma s} dt\\
&\le \bar C K e^{(\gamma(\alpha+i)-2\eta)s}\int_0^\infty z^{-i/2}[1+z]^{-(\alpha+1)/2}\, dz
\le \bar C K e^{(\gamma(\alpha+i)-2\eta)s} \le \nu y^ie^{-\lambda s},
\end{aligned}
$$
owing to $\bar C K e^{(\gamma(\alpha+i)-2\eta)s}e^{\lambda s}y^{-i}\le 
CK^{2\beta+3}e^{(\gamma(\alpha+2)-2\eta+(1-\beta)\eta)s} 
\le CK^{2\beta+3}e^{-2(p-1)\gamma s_0}$ and $s_0\gg 1$.

\smallskip
$\bullet$ {\bf Estimate of $S^i_{3,3}$.}
By \eqref{controlFvybasic} and Lemma~\ref{AuxilLem2}(ii), we have
$$\begin{aligned}
S_{3,3}^i
&\le  C e^{-2\lambda s}\int_{\bar s}^s\int_{X_1(\tau)}^{X_2(\tau)} (s-\tau)^{-\frac{i}{2}} e^{-C(X(\tau)-z)^2}\bigl(1\wedge X(\tau)\bigr)^i \\
 & \qquad\qquad\qquad\qquad\qquad \bigl(1+X(\tau)z\bigr)^{-\frac{\alpha}{2}}
 \bigl(z\sqrt{s-\tau}\bigr)^\alpha \bigl(1+(z\sqrt{s-\tau})^{4(\ell-1)}\bigr) z^\alpha dz d\tau \\
 \noalign{\vskip 2mm}
 &\le C e^{-2\lambda s} \bigl[S_{3,3}^{i,2\alpha}+S_{3,3}^{i,2\alpha+4(\ell-1)}\bigr],
 \end{aligned}$$
where, for $m>\alpha$,
$$\begin{aligned}
S_{3,3}^{i,m}
&:=\int_{\bar s}^s (s-\tau)^{\frac{m-\alpha-i}{2}}  \bigl(1\wedge X(\tau)\bigr)^i\int_0^\infty e^{-C(X(\tau)-z)^2} 
 (1+X(\tau)z)^{-\alpha/2} z^m dz d\tau \\
 & \le C\int_{\bar s}^s (s-\tau)^{\frac{m-\alpha-i}{2}}  \bigl(1\wedge X(\tau)\bigr)^i 
 \bigl(1+X^{m-\alpha}(\tau)\bigr) d\tau \\
&\le C \int_{\bar s}^s (s-\tau)^{\frac{m-\alpha-i}{2}}\bigl((s-\tau)^{-\frac{i}{2}}y^{i}+(s-\tau)^{\frac{\alpha-m}{2}}y^{m-\alpha}\bigr)d\tau \\
&\le C\int_{\bar s}^s \bigl((s-\tau)^{\frac{m-\alpha}{2}-i}y^{i}+(s-\tau)^{-\frac{i}{2}}y^{m-\alpha}\bigr)d\tau 
\le C(y^i+y^{m-\alpha}).
\end{aligned}$$
Consequently,
$S_{3,3}^i\le  C e^{-2\lambda s}(y^i+y^{\alpha+4(\ell-1)})$.
 If $s_0\gg 1$, since
$y^{\alpha+2\ell-3}\le e^{(\alpha+2\ell-3)s/2}\le \nu e^{\lambda s}$ for $y\le y_2(s)$ owing to 
$\frac12(2\ell+\alpha-3) < \lambda=\frac12(2\ell+\alpha-2)$, 
we have $Ce^{-2\lambda s}y^{\alpha+4(\ell-1)}\le\nu e^{-\lambda s}y^{2\ell-1}$, hence we conclude that
\be{estimS3-4}
S_{3,3}^i\le \nu e^{-\lambda s}(y^i+y^{2\ell-i}).
\ee

\smallskip
$\bullet$ {\bf Estimate of $S^i_{3,4}$.}
By \eqref{controlFvybasic}, we have
$$S_{3,4}^i\le 
Ce^{-\alpha s/2}\int_{\bar s}^s(s-\tau)^{-i/2}  \bigl(1\wedge X(\tau) \bigr)^i
\int_{X_2(\tau)}^\infty e^{-C(X(\tau)-z)^2}(1+X(\tau)z)^{-\alpha/2}  z^\alpha dz d\tau.$$
If $y\le y_2(s)/2$, then 
$X(\tau)=e^{(\tau-s)/2}(s-\tau)^{-1/2}y\le e^{(\tau-s)/2}(s-\tau)^{-1/2}y_2(s)/2=X_2(\tau)/2$ for all $\tau\in (\bar s,s)$.
Therefore, applying Lemma~\ref{AuxilLem2}(ii) and using 
 $X_2(\tau)=\sigma(s-\tau)^{-1/2}e^{\tau/2}\ge C(s-\tau)^{-1/2}e^{s/2}$, we get
$$ S_{3,4}^i
\le Ce^{-\alpha s/2}y^i \int_{\bar s}^s(s-\tau)^{-i}\exp[-C(s-\tau)^{-1}e^s] d\tau
\le Cy^i\exp[-Ce^s]\le C\nu y^ie^{-\lambda s}.
$$  
On the other hand, if $y_2(s)/2\le y\le y_2(s)$, then
$$  
S_{3,4}^i\le Ce^{-\alpha s/2}   \int_{\bar s}^s(s-\tau)^{-i/2}d\tau
\le Ce^{-\alpha s/2}   \le \nu e^{-\lambda s} y^{2\ell-1}
$$  
 for $s_0\gg 1$, since
 $\nu y^{2\ell-1}\ge C\nu e^{(\ell-\frac{1}{2})s}\ge C e^{(\lambda-\frac{\alpha}{2})s}$,
owing to $2\ell-1>2\ell-2  
=2\lambda-\alpha$.
Consequently,
\be{estimS3-4b}
S_{3,4}^i\le \nu e^{-\lambda s}(y^i+y^{2\ell-i}).
\ee
Finally, inequality \eqref{controlS3short0} follows by combining the above estimates.
\end{proof}

\begin{proof}[Proof of Proposition~\ref{Propwshort}]
This is a direct consequence of \eqref{controlDiphi}, \eqref{defS123} and Lemmas~\ref{Lemcvcphi}, \ref{LemS2short}, \ref{LemS3short}.
\end{proof}

\subsubsection{Long time estimate of $w$ and $w_y$.}

In this subsection, we shall prove:

\begin{prop} \label{Propwlong}
If $s_0\gg 1$ and $s_1\ge s_0+1$ then, for any $d\in \mathcal{U}_{s_0,s_1}$ such that $P(d;s_0,s_1)=0$, we have
\be{estimM3}
|D^i(\tilde w-U+e^{-\lambda s}\phi)|\le  M_3\nu e^{-\lambda s}(y^i+y^{2\ell-i}),\quad s\in (s_0+1,s_1],
\  y\in [y_1(s),y_2(s)],
\ee
for $i\in\{0,1\}$, with $M_3=M_3(p,\ell)>0$.
\end{prop}

For this purpose it will be convenient to use a different decomposition of $\tilde v$.
Namely, if $\tilde w\in\mathcal{A}^1_{s_0,s_1}$ satisfies $P(d;s_0,s_1)=0$, then
\be{defI0123}
\tilde v=I_0+I_1+I_2+I_3,\ \hbox{ where }\ 
\begin{cases}
I_0(y,s)&=-(\hat\phi,\phi)e^{-\lambda s}\phi(y), \\
\noalign{\vskip 1mm}
 I_1(y,s)&=-\ds\sum_{j=0}^{\ell-1} \varphi_j(y)\int_s^{s_1} e^{-\lambda_j(s-\tau)}(\tilde F(\tilde v_y(\tau)),\varphi_j)\,d\tau,\\
\noalign{\vskip 1mm}
 I_2(y,s)&=-\ds\sum_{j=\ell+1}^\infty e^{-\lambda s_0}(\hat\phi,\varphi_j)e^{-\lambda_j(s-s_0)}\varphi_j(y),\\
\noalign{\vskip 1mm}
 I_3(y,s)&=\ds\sum_{j=\ell}^\infty \varphi_j(y)\int_{s_0}^s e^{-\lambda_j(s-\tau)}(\tilde F(\tilde v_y(\tau)),\varphi_j)\,d\tau.
\end{cases}
\ee
Equality \eqref{defI0123} follows from 
\be{defI0123b}
\tilde v(s)=\sum_{i=0}^\infty 
 e^{-\lambda_j(s-s_0)}(\tilde v(s_0),\varphi_j)\varphi_j(y)
+\sum_{j=0}^\infty \varphi_j(y)\int_{s_0}^s e^{-\lambda_j(s-\tau)}(\tilde F(\tilde v_y(\tau)),\varphi_j)\,d\tau,
\ee
along with \eqref{v0phij}
and $\tilde v(s_0)=\sum_{j=0}^{\ell-1}d_j\varphi_j-e^{-\lambda s_0}\hat\phi$.
The outline of proof of Proposition~\ref{Propwlong} is then as follows:
\vskip 1pt

$\bullet$ Estimation of $I_1$

$\bullet$ Estimation of $I_2$ and $I_3$ in the intermediate region for {\it bounded} $y$, 
i.e.~$y\in [y_1(s), R_1]$ (with suitably chosen large $R_1$)

$\bullet$ Estimation of $w$ and $w_y$ at the right boundary of the intermediate region $y=y_2(s)$

$\bullet$ Estimation of $w$ and $w_y$ in the remaining part of the intermediate region
i.e.~$y\in [R_1,y_2(s)]$.

\smallskip

We start with $I_1$, which is easy to estimate globally.

\begin{lem} \label{LemI1long} 
If $s_0\gg 1$ and $s_1\ge s_0+1$ then, for any $d\in \mathcal{U}_{s_0,s_1}$, we have
\be{controlDI1}
|D^iI_1(y,s)|\le  C\nu e^{-\lambda s} (y^i+y^{2\ell-i}),\quad\hbox{ for all $s\in [s_0+1,s_1]$, $y>0$ and $i\in\{0,1\}$.}
\ee
\end{lem} 

\begin{proof}
For later purpose, we actually show a slightly stronger estimate, namely:
\be{controlDI1b}
\begin{aligned}
&\sum_{j=0}^\ell |D^i\varphi_j(y)|\int_{s-1}^{s_1} e^{-\lambda_j(s-\tau)}|(\tilde F(\tilde v_y(\tau)),\varphi_j)|\,d\tau
\le C(y^i+y^{2\ell-i})\sum_{j=0}^\ell
e^{-\lambda_j s} \int_{s-1}^\infty e^{(\lambda_j-\bar\eta)\tau}\,d\tau \\
&\quad\le C(y^i+y^{2\ell-i}) e^{-\bar\eta s},\quad
s\in [s_0+1,s_1],\ y>0,
\end{aligned}
\ee
for $s_0\gg 1$, which is a consequence of inequality \eqref{Fvytau}
and $\bar\eta>\lambda_\ell$. 
Estimate \eqref{controlDI1} is then an immediate 
consequence of \eqref{controlDI1b}.
\end{proof}

We shall now estimate $I_2$ and $I_3$ in the intermediate region for bounded $y$. 
More precisely, we choose $R_1>1, C_1>0$ (depending only on $p,\ell$) such that 
\be{choiceR}
(-1)^\ell\phi(y)\ge C_1y^{2\ell},\quad (-1)^\ell\phi'(y)\ge C_1y^{2\ell-1},\qquad y\ge R_1,
\ee
which is obviously possible in view of Proposition~\ref{GaussianPoincare2}.

\begin{lem}  \label{LemI2long}
If $s_0\gg 1$ and $s_1\ge s_0+1$ then, for any $d\in \mathcal{U}_{s_0,s_1}$, we have, for $m\in\{2,3\}$,
\be{controlDIm}
|D^iI_m(y,s)|\le  C\nu e^{-\lambda s} y^i,\quad\hbox{ for all $s\in [s_0+1,s_1]$, $y\in [y_1(s),R_1]$ and $i\in\{0,1\}$}.
\ee
Consequently,
\be{controlDIm2}
|D^i(v+e^{-\lambda s}\phi)\le  M_4\nu e^{-\lambda s} y^i,\quad\hbox{ for all $s\in [s_0+1,s_1]$,
$y\in [y_1(s),R_1]$ and $i\in\{0,1\}$},
\ee
 with $M_4=M_4(p,\ell)>0$.
\end{lem}

\begin{proof}
{\bf Step 1.} {\it Proof of \eqref{controlDIm} for $m=2$.} 
For $s\ge s_0+1$ and $j\ge\ell$, we have
$$e^{-\lambda s_0}e^{-\lambda_j(s-s_0)}=e^{-\lambda s}e^{-(\lambda_j-\lambda)(s-s_0)}
=e^{-\lambda s}e^{-(j-\ell)(s-s_0)}\le e^{-\lambda s}e^{\ell-j}.$$
Also, by Lemma~\ref{Lemcvcphi} and since $\|\phi\|=1$, we have $\sum_{j\ne\ell} |(\hat\phi,\varphi_j)|^2
=\|\hat\phi\|^2- |(\hat\phi,\phi)|^2\to \|\phi\|^2-\|\phi\|^4=0$, as $s_0\to\infty$, uniformly for 
$d$ satisfying \eqref{condd}.
Therefore, taking $s_0\gg 1$, we have $\sup_{j\ge\ell+1} |(\hat\phi,\varphi_j)|\le \nu$
and, by \eqref{EstimEigen1}, we get, for all $y\in [y_1(s),R_1]$,
$$\begin{aligned}
|D^iI_2(y,s)|
&=\Bigl|\sum_{j=\ell+1}^\infty  e^{-\lambda s_0}(\hat\phi,\varphi_j) e^{-\lambda_j(s-s_0)} D^i\varphi_j(y)\Bigr|
\le Ce^{-\lambda s} \sum_{j=\ell+1}^\infty  e^{-j}|(\hat\phi,\varphi_j)||D^i\varphi_j(y)|\\
&\le Ce^{R_1^2/8}\nu e^{-\lambda s}y^i \sum_{j=\ell+1}^\infty  j^{5/2}e^{-j}\le 
C\nu e^{-\lambda s}y^i.
\end{aligned}$$

{\bf Step 2.} {\it Proof of \eqref{controlDIm} for $m=3$.} We have
$$\begin{aligned}
D^iI_3(y,s)
&=\int_{s_0}^s e^{-\lambda_j(s-\tau)}\sum_{j=\ell}^\infty  (\tilde F(\tilde v_y(\tau)),\varphi_j)D^i\varphi_j(y)  \,d\tau
=\int_{s_0}^{s-1}+\int_{s-1}^s \equiv D^iI_{3,1}+D^iI_{3,2}
\end{aligned}$$
and we further split $D^iI_{3,2}(y,s)=D^iI_{3,2}^1-D^iI_{3,2}^2$, with
$$D^iI_{3,2}^1=\int_{s-1}^s D^i e^{-(s-\tau)\mathcal L}(\tilde F(\tilde v_y(\tau)) \,d\tau,\quad 
D^iI_{3,2}^2=\sum_{j=0}^{\ell-1} D^i\varphi_j(y) \int_{s-1}^s e^{-\lambda_j(s-\tau)}(\tilde F(\tilde v_y(\tau)),\varphi_j) \,d\tau.$$
By \eqref{controlS3short0} in Lemma~\ref{LemS3short} and \eqref{controlDI1b}, we have
$$|D^iI_{3, 1}^1(y,s)|+|D^iI_{3, 2}^2(y,s)|\le  C\nu e^{-\lambda s} y^i,\quad\hbox{ for all $s\in [s_0+1,s_1]$ and $y\in [y_1(s),R_1]$}.$$
In view of estimating $D^iI_{3,1}$, we note that, for all $\tau\in (s_0,s-1)$ and $j\ge\ell$, we have $e^{-\lambda_j(s-\tau)}$ $=e^{-\lambda(s-\tau)}e^{(\lambda-\lambda_j)(s-\tau)}\le e^{-\lambda(s-\tau)}e^{\lambda-\lambda_j} =e^{-\lambda(s-\tau)}e^{\ell-j}$. 
Combining this with \eqref{EstimEigen1} and \eqref{Fvytau} yields
$$\begin{aligned}
|D^iI_{3,1}|
&\le \sum_{j=\ell}^\infty |D^i\varphi_j(y)| \int_{s_0}^{s-1} e^{-\lambda_j(s-\tau)} |(\tilde F(\tilde v_y(\tau)),\varphi_j)| \,d\tau \\
&\le Ce^{R_1^2/8}y^i \sum_{j=\ell}^\infty j^{5/2} e^{-j} \int_{s_0}^{s-1} j^{3/2}e^{-\lambda(s-\tau)} 
e^{-\bar\eta\tau}\,d\tau\\
&\le Ce^{R_1^2/8}y^i e^{-\lambda s} \Bigl(\sum_{j=\ell}^\infty j^4 e^{-j}\Bigr) \int_{s_0}^{s-1} 
e^{(\lambda-\bar\eta)\tau}\,d\tau 
\le Ce^{R_1^2/8}y^i e^{-\lambda s} e^{(\lambda-\bar\eta)s_0} \le 
\nu y^i e^{-\lambda s}.
 \end{aligned}$$
The conclusion follows.

\smallskip

{\bf Step 3.} {\it Proof of \eqref{controlDIm2}.} 
This is a direct consequence of \eqref{defI0123}, \eqref{controlDI1}, \eqref{controlDIm}
and Lemma~\ref{Lemcvcphi}. 
\end{proof}

We next estimate $v$ and $v_y$ at the outer boundary of the intermediate region.
For this purpose, we go back to the decomposition \eqref{defS123}.

\begin{lem}  \label{LemS2long}
Set $y_2=y_2(s)=\sigma e^{s/2}$. 
If $s_0\gg 1$ and $s_1\ge s_0+1$ then, for any $d\in \mathcal{U}_{s_0,s_1}$ such that $P(d;s_0,s_1)=0$, we have, for $m\in\{2,3\}$,
\be{controlDSm}
|D^iS_m(y_2,s)|\le C \nu e^{-\lambda s} y_2^{2\ell-i},\quad\hbox{ for all $s\in [s_0+1,s_1]$ and $i\in\{0,1\}$}.
\ee
Consequently,
\be{controlDSm2}
\bigl|D^i\bigl(v+e^{-\lambda s}\phi\bigr)(y_2,s)\bigr| 
\le  M_5\nu e^{-\lambda s} y_2^{2\ell-i},\quad\hbox{ for all $s\in [s_0+1,s_1]$ and $i\in\{0,1\}$}.
\ee
 with $M_5=M_5(p,\ell)>0$.
\end{lem} 

\begin{proof}
{\bf Step 1.} {\it Proof of \eqref{controlDSm} for $m=2$.} 
Let $i\in\{0,1\}$ and $s\in [s_0+1,s_1]$. Since $S_2(\cdot,s)=e^{-(s-s_0)\mathcal{L}}S_2(\cdot,s_0)$, by Proposition~\ref{KernelEstim2} 
and \eqref{boundH1} it follows that
\be{DiS2}
|D^iS_2(y_2,s)|\le \int_0^\infty G_i(s-s_0,y_2,\xi)|D^iS_2(\xi,s_0)|\xi^\alpha d\xi
= \int_0^{\tilde K e^{-\gamma s_0}}+ \int_{\tilde K e^{-\gamma s_0}}^{2\sigma e^{s_0/2}}+ \int_{2\sigma e^{s_0/2}}^\infty\equiv \sum_{n=1}^3 S^i_{2,n},
\ee
where
\begin{eqnarray}
G_i(s-s_0,y_2,\xi)
\hskip -6mm &&\le C e^{(k-\frac{i}{2})(s-s_0)}(1-e^{s_0-s})^{-\frac{\alpha+1}{2}} 
\Bigl(1+\frac{e^{(s_0-s)/2}y_2\xi}{1-e^{s_0-s}}\Bigr)^{-\frac{\alpha}{2}}
\exp\Bigl[-C\frac{\bigl(e^{(s_0-s)/2}y_2-\xi\bigr)^2}{1-e^{s_0-s}}\Bigr] \notag \\
&&\le C e^{-\lambda s}y_2^{2\ell-i} e^{(\frac{i}{2}-k)s_0}
\bigl(1+\sigma e^{s_0/2}\xi\bigr)^{-\frac{\alpha}{2}}\exp\Bigl[-C\bigl(\sigma e^{s_0/2}-\xi\bigr)^2\Bigr],
\label{Giss0}
\end{eqnarray}
owing to $1-e^{s_0-s}\ge C>0$ for $s\ge s_0+1$ and
\be{lambday2}
e^{-\lambda s}y_2^{2\ell-i}=\sigma^{2\ell-i} e^{(\ell-\frac{i}{2}-\lambda)s}=\sigma^{2\ell-i} e^{(k-\frac{i}{2})s}.
\ee
$\bullet$ {\bf Estimate of $S^i_{2,1}$.} 
Using Lemma~\ref{LemestimS2s0} and 
$\sigma e^{s_0/2}-\xi\ge \frac{\sigma}{2}e^{s_0/2}$ for $\xi\in [0, \tilde K e^{-\gamma s_0}]$, 
we get for $s_0\gg 1$,
$$S_{2,1}^0\le C e^{-\lambda s_0}e^{-\lambda s}y_2^{2\ell} e^{-ks_0}
\int_0^{\tilde K e^{-\gamma s_0}} e^{-Ce^{s_0}} 
\xi^\alpha d\xi \le 
Ce^{-\lambda s}y_2^{2\ell} e^{-Ce^{s_0}}
\le C\nu e^{-\lambda s}y_2^{2\ell},$$
$$S_{2,1}^1\le C \tilde K^{\beta+1}e^{-\lambda s}y_2^{2\ell-1} e^{(\frac{1}{2}-k)s_0}
\int_0^{\tilde K e^{-\gamma s_0}} e^{-Ce^{s_0}} \xi d\xi \le 
Ce^{-\lambda s}y_2^{2\ell-1}e^{-Ce^{s_0}}
\le C\nu e^{-\lambda s}y_2^{2\ell-1}.$$
$\bullet$ {\bf Estimate of $S^i_{2,2}$.} 
Set $\tilde E=(\tilde K e^{-\gamma s_0}, 2\sigma e^{s_0/2})$, $\tilde I= (\frac{\sigma}{2} e^{s_0/2},2\sigma e^{s_0/2})$,
$\tilde E^1=\tilde E\setminus \tilde I$, $\tilde E^2=\tilde E\cap \tilde I$. 
Using Lemma~\ref{LemestimS2s0} and $|\sigma e^{s_0/2}-\xi|\ge C\xi$ for $\xi\in \tilde E^1$, we obtain
$$\begin{aligned}
&\int_{\tilde K e^{-\gamma s_0}}^{2\sigma e^{s_0/2}}
\bigl(1+\sigma e^{s_0/2}\xi\bigr)^{-\frac{\alpha}{2}}\exp\Bigl[-C\bigl(\sigma e^{s_0/2}-\xi\bigr)^2\Bigr]
(\xi^i+\xi^{2\ell-i})\xi^\alpha d\xi \\
&\qquad \le\int_{\tilde E^1} e^{-C\xi^2}(\xi^i+\xi^{2\ell-i})\xi^{\alpha}d\xi+
e^{(\ell-\frac{i}{2})s_0} \int_{\tilde E^2} \exp\Bigl[-C\bigl(\sigma e^{s_0/2}-\xi\bigr)^2\Bigr]d\xi
\le Ce^{(\ell-\frac{i}{2})s_0}.
\end{aligned}$$
Since $\lambda=\ell-k$, it follows that
\be{estS22i}
S_{2,2}^i\le C\nu e^{-\lambda s_0}e^{-\lambda s}y_2^{2\ell-i} e^{(\frac{i}{2}-k)s_0} e^{(\ell-\frac{i}{2})s_0}
=C\nu e^{-\lambda s}y_2^{2\ell-i}.
\ee
$\bullet$ {\bf Estimate of $S^i_{2,3}$.} 
Using Lemma~\ref{LemestimS2s0} and $(\xi-\sigma e^{s_0/2})^2\ge C(\xi^2+e^{s_0})$ for 
$\xi\ge 2\sigma e^{s_0/2}$, we get
\begin{eqnarray}
S_{2,3}^i
\hskip -6mm &&\le  C e^{-\lambda s}y_2^{2\ell-i} e^{(\frac{i}{2}-k)s_0}
\int_{2\sigma e^{s_0/2}}^\infty\exp\Bigl[-C\bigl(\sigma e^{s_0/2}-\xi\bigr)^2\Bigr] \xi^{2\ell-i+\alpha} d\xi  \notag \\
&&\le C e^{-\lambda s}y_2^{2\ell-i} e^{(\frac{i}{2}-k)s_0} e^{-Ce^{s_0}} 
\int_0^\infty e^{-C\xi^2}\xi^{2\ell-i+\alpha} d\xi
\le  \nu e^{-\lambda s}y_2^{2\ell-i}. \label{estS23i} 
\end{eqnarray}
Gathering the above estimates gives the desired conclusion.

\medskip

{\bf Step 2.} {\it Proof of \eqref{controlDSm} for $m=3$.} 
We shall use the splitting
$$\begin{aligned}
D^iS_3(y,s)
&=\int_{s_0}^s D^ie^{-(s-\tau)\mathcal{L}}\tilde F(\tilde v_y(\tau)) \,d\tau
=\int_{s_0}^{s-1}+\int_{s-1}^s \equiv D^iS_{3,1}+D^iS_{3,2}.
\end{aligned}$$
By \eqref{controlS3short0} in Lemma~\ref{LemS3short}, we have
$$|D^iS_{3,2}(y,s)|\le C\nu e^{-\lambda s} y^{2\ell-i},\quad\hbox{ for $y=y_2(s)$ and $s\in [s_0+1,s_1]$}.$$
Let us thus estimate $D^iS_{3,1}$.
By Proposition~\ref{KernelEstim2}, \eqref{boundH1} with $i=0$ and \eqref{boundH2}, we have
\be{DiS3}
\begin{aligned}
|D^iS_{3,1}(y_2,s)|
&\le \int_{s_0}^{s-1}\int_0^\infty |D^i_yG(s-\tau,y_2,\xi)| |\tilde F(\tilde v_y(\xi,\tau))| \xi^\alpha d\xi d\tau \\
&= \int_{s_0}^{s-1}\int_0^{K e^{-\gamma\tau}}+ 
\int_{s_0}^{s-1}\int_{K e^{-\gamma\tau}}^{\sigma e^{\tau/2}}+ 
\int_{s_0}^{s-1}\int_{\sigma e^{\tau/2}}^\infty\equiv \sum_{n=1}^3 S^{i,n}_{3,1},
\end{aligned}
\ee
where
\begin{eqnarray}
|D^i_yG(s-\tau,y_2,\xi)|
\hskip -6mm 
&&\le C e^{(k-\frac{i}{2})(s-\tau)}(1-e^{\tau-s})^{-\frac{\alpha+1+i}{2}}
\Bigl(1+\frac{e^{\frac{\tau-s}{2}}y_2\xi}{1-e^{\tau-s}}\Bigr)^{-\frac{\alpha}{2}}
\exp\Bigl[-C\frac{\bigl(e^{\frac{\tau-s}{2}}y_2-\xi\bigr)^2}{1-e^{\tau-s}}\Bigr] \notag \\
&&\le C e^{-\lambda s}y_2^{2\ell-i} e^{\beta\tau/2} \bigl(1+\sigma e^{\tau/2}\xi\bigr)^{-\frac{\alpha}{2}}
\exp\Bigl[-C\bigl(\sigma e^{\tau/2}-\xi\bigr)^2\Bigr], \label{DiyG2}
\end{eqnarray}
owing to $1-e^{\tau-s}\ge C>0$ for $\tau\le s-1$, $\frac{1}{2}-k=\beta/2$ and \eqref{lambday2}.
\smallskip

$\bullet$ {\bf Estimate of $S^{i,1}_{3,1}$.} By \eqref{controlFvybasic}, we have
$|\tilde F(\tilde v_y(\xi,\tau))| \le C\xi^{-\alpha}$ for all $\xi\in (0,Ke^{-\gamma\tau})$ and $\tau\in(s_0,s_1)$. Consequently,
$$\begin{aligned}S^{i,1}_{3,1}
\le Ce^{-\lambda s}y_2^{2\ell-i}\int_{s_0}^{s-1} e^{\frac{\beta\tau}{2}} 
\int_0^{K e^{-\gamma\tau}} e^{-Ce^\tau} d\xi d\tau 
\le CK e^{-\lambda s}y_2^{2\ell-i} \int_{s_0}^\infty e^{(\frac{\beta}{2}-\gamma)\tau} e^{-Ce^\tau}d\tau 
\le \nu e^{-\lambda s}y_2^{2\ell-i}.
\end{aligned}$$

$\bullet$ {\bf Estimate of $S^{i,2}_{3,1}$.} 
Since $\exp[-C\bigl(\sigma e^{\tau/2}-\xi\bigr)^2]\le e^{-Ce^\tau}e^{-C\xi^2}$
 for $\xi\in [K e^{-\gamma\tau},\frac12\sigma e^{\tau/2}]$, we have
$$\begin{aligned}
&\int_{K e^{-\gamma\tau}}^{\sigma e^{\tau/2}} 
 \bigl(1+\sigma e^{\frac{\tau}{2}}\xi\bigr)^{-\frac{\alpha}{2}}
\exp\Bigl[-C\bigl(\sigma e^{\frac{\tau}{2}}-\xi\bigr)^2\Bigr]\bigl(1+\xi^{4(\ell-1)}\bigr) \xi^{2\alpha} d\xi \\
&\le
e^{-Ce^\tau} \int_{K e^{-\gamma\tau}}^{\frac{\sigma}{2}e^{\tau/2}} 
e^{-C\xi^2}\bigl(1+\xi^{4(\ell-1)}\bigr) \xi^{2\alpha} d\xi+
e^{(\frac{\alpha}{2}+2(\ell-1))\tau}\int_{\frac{\sigma}{2}e^{\tau/2}}^{\sigma e^{\tau/2}} 
e^{-C\bigl(\sigma e^{\frac{\tau}{2}}-\xi\bigr)^2}d\xi 
\le Ce^{(\frac{\alpha}{2}+2(\ell-1))\tau}.
 \end{aligned}$$
Using \eqref{controlFvybasic}, $k=(1-\beta)/2$, $-2\lambda+\alpha-\frac12+2(\ell-1) 
=2k+\beta+\frac12-2 
=-\frac12$ and $s_0\gg 1$, we deduce that
\begin{eqnarray}
S^{i,2}_{3,1}
\hskip -6mm&&\le  C e^{-3\lambda s}y_2^{2\ell-i}
\int_{s_0}^{s-1}  e^{\frac{\beta\tau}{2}} \int_{K e^{-\gamma\tau}}^{\sigma e^{\tau/2}} 
 \bigl(1+\sigma e^{\frac{\tau}{2}}\xi\bigr)^{-\frac{\alpha}{2}}
e^{-C\bigl(\sigma e^{\frac{\tau}{2}}-\xi\bigr)^2}\bigl(1+\xi^{4(\ell-1)}\bigr) \xi^{2\alpha} d\xi d\tau \notag \\
&&\le  C e^{-3\lambda s}y_2^{2\ell-i}
\int_{s_0}^{s-1}  e^{(\alpha-\frac12+2(\ell-1))\tau}d\tau
\le  Cy_2^{2\ell-i} e^{(-3\lambda+\alpha-\frac12+2(\ell-1))s}
\le  \nu e^{-\lambda s}y_2^{2\ell-i}. \label{Si231}
\end{eqnarray}

$\bullet$ {\bf Estimate of $S^{i,3}_{3,1}$.} 
Observing that
$$\begin{aligned}
&\int_{\sigma e^{\tau/2}}^\infty \bigl(1+\sigma e^{\tau/2}\xi\bigr)^{-\frac{\alpha}{2}}
\exp\Bigl[-C\bigl(\sigma e^{\tau/2}-\xi\bigr)^2\Bigr] \xi^\alpha d\xi
\le Ce^{-\alpha\tau/4}\int_{\sigma e^{\tau/2}}^\infty \xi^{\frac{\alpha}{2}}
\exp\Bigl[-C\bigl(\sigma e^{\tau/2}-\xi\bigr)^2\Bigr] d\xi\\
&\qquad =e^{-\alpha\tau/4}\int_0^\infty (\sigma e^{\tau/2}+z)^{\frac{\alpha}{2}} e^{-Cz^2} dz
\le C\int_0^\infty e^{-Cz^2} dz+Ce^{-\alpha\tau/4}\int_0^\infty z^{\frac{\alpha}{2}} e^{-Cz^2} dz\le C,
 \end{aligned}$$
 it follows that
\begin{eqnarray}
S^{i,3}_{3,1}
\hskip -6mm&&\le Ce^{-\lambda s}y_2^{2\ell-i}
\int_{s_0}^{s-1}  e^{-\tau/2}  \int_{\sigma e^{\tau/2}}^\infty
\bigl(1+\sigma e^{\tau/2}\xi\bigr)^{-\frac{\alpha}{2}}
\exp\Bigl[-C\bigl(\sigma e^{\tau/2}-\xi\bigr)^2\Bigr] \xi^\alpha d\xi d\tau \notag \\
&&\le Ce^{-\lambda s}y_2^{2\ell-i}\int_{s_0}^{s-1}  e^{-\tau/2} d\tau 
\le Ce^{-s_0/2}e^{-\lambda s}y_2^{2\ell-i}. \label{Si331}
\end{eqnarray}
Combining the above estimates yields the conclusion.
 
 \medskip

{\bf Step 3.} {\it Proof of \eqref{controlDSm2}.} 
This is a direct consequence of \eqref{defS123}, \eqref{controlDSm}
and Lemma~\ref{Lemcvcphi}. 
\end{proof}

Finally, we shall estimate $w$ and $w_y$ in the remaining part $y\ge R_1$ of the intermediate region.
To this end, we shall use a comparison argument, combined with the already obtained estimates.

\begin{lem}  \label{LemComplong}
 Set
\be{DefDEstimComplong}
D:=\bigl\{(y,s);\ s\in [s_0+1,s_1],\ y\in [R_1,\sigma e^{s/2}]\bigr\}.
\ee
If $s_0\gg 1$ and $s_1\ge s_0+1$ then, for any $d\in \mathcal{U}_{s_0,s_1}$ such that $P(d;s_0,s_1)=0$, we have
\be{EstimComplong}
|D^i(v+e^{-\lambda s}\phi)|\le  C\nu e^{-\lambda s} y^{2\ell-i}\quad\hbox{ in $D$ for $i\in\{0,1\}$}.
\ee
\end{lem} 

\begin{proof} 
 Let $B=C_1^{-1}\bigl(1+\max(2M_2, M_4, M_5)\bigr)$, where
$C_1$ is the constant in \eqref{choiceR} and
 $M_2, M_4, M_5$ are respectively from \eqref{estimM2},  \eqref{controlDIm2}, \eqref{controlDSm2}.
Also set $\Lambda=\lambda+\frac12$ and recall the notation
$$-\mathcal{L}v=v_{yy}+\Bigl(\frac{\alpha}{y}-\frac{y}{2}\Bigr)v_y+kv,
\quad F(v_y)=|v_y+U_y|^p-U_y^p-pU_y^{p-1}v_y.$$

{\bf Step 1.} {\it Case $i=0$.} We define
$$\begin{aligned}
\underline v(y,s)&=\bigl(-1-(-1)^\ell B\nu\bigr)\,v_1 \\
\bar v(y,s)&=\bigl(-1+(-1)^\ell B\nu\bigr)\,v_1-\nu{\hskip 1pt}v_2,
\end{aligned}
\quad\hbox{where } \ v_1(y,s)=e^{-\lambda s}\phi(y),\ \ v_2(y,s)=e^{-\Lambda s}y^{2\ell}. 
$$
By 
 \eqref{estimM2}, \eqref{choiceR}, \eqref{controlDIm2} and \eqref{controlDSm2}, taking $s_0\gg 1$, we see that
$$\underline v(y,s)\le -e^{-\lambda s}\phi(y)- BC_1\nu e^{-\lambda s}y^{2\ell}\le v(y,s)
\le -e^{-\lambda s}\phi(y)+(BC_1-1)\nu e^{-\lambda s} y^{2\ell}\le \bar v(y,s)
\quad\hbox{ on $\partial_PD$.}$$
We obviously have $\partial_s\underline v+\mathcal{L}\underline v=0\le F(v_y)=\partial_s v+\mathcal{L}v$ in $D$.
We claim that $\bar v$ satisfies
\be{EstimComplong1}
\bar v_s+\mathcal{L}\bar v\ge F(v_y)\quad\hbox{in $D$.}
\ee
Using $\Lambda=\ell-k+\frac12$, we compute
\be{ineqtildev1}
\begin{aligned}
\partial_s v_2+\mathcal{L}v_2
&=-\Lambda e^{-\Lambda s}y^{2\ell}-e^{-\Lambda s}\bigl\{2\ell(2\ell-1)y^{2\ell-2}+2\ell\alpha y^{2\ell-2}+(k-\ell)y^{2\ell}\bigr\}\\
&=-e^{-\Lambda s}\bigl\{(\Lambda+k-\ell) y^{2\ell}+2\ell(2\ell-1)y^{2\ell-2}+2\ell\alpha y^{2\ell-2}\bigr\}\le 
-\tsfr e^{-\Lambda s}y^{2\ell}.
\end{aligned}
\ee
By Lemma~\ref{LemvyFvy}, we have $0\le F(v_y)\le C 
e^{-2\lambda s}y^{4\ell+\beta-3}$ in $D$.
Since $\partial_s v_1+\mathcal{L}v_1=0$, taking $s_0\gg 1$, 
we then have
\be{compbarv}
\bar v_s+\mathcal{L}\bar v-F(v_y)
\ge \ts\frac{\nu}{2} e^{-\Lambda s}y^{2\ell}-C 
e^{-2\lambda s}y^{4\ell+\beta-3}
\ge e^{-\Lambda s}y^{2\ell}\bigl\{\ts\frac{\nu}{2} -Ce^{(\Lambda-2\lambda)s}y^{2\ell+\beta-3}\bigr\}\ge 0,
\ee
since $e^{(\Lambda-2\lambda)s}y^{2\ell+\beta-3}\le e^{(\Lambda-2\lambda+\ell+\frac{\beta-3}{2})s}=e^{-s/2}$ in $D$,
hence \eqref{EstimComplong1}.
It follows from the maximum principle that $\underline v\le v\le\bar v$ in $D$, which guarantees \eqref{EstimComplong} for $i=0$.

\medskip

{\bf Step 2.} {\it Case $i=1$.}
For this case it will be more convenient to consider the operator
\be{defP1}
\mathcal{P}_1z:=\mathcal{L}_1 z-\mathcal{N}z,\quad\hbox{ where } 
\mathcal{L}_1z=z_s-z_{yy}+\ts\frac{y}{2}z_y+\ts\frac{\beta}{2}z,\quad \mathcal{N}z=p|z|^{p-2}zz_y.
\ee
We have $\mathcal{P}_1w_y=\bigl(w_s-w_{yy}+\ts\frac{y}{2}w_y-kw-|w_y|^p\bigr)_y=0$. 
Set 
$$V=U',\quad\psi=\phi', \quad b=1+(-1)^{\ell+m} B \nu$$
 and let $m\in\{0,1\}$.
We define
\be{defzm}
z_m(y,s)=V+W 
\quad\hbox{where} 
\ W=-bW_0+(-1)^m\nu\,W_1,
\ \ \ W_0=e^{-\lambda s}\psi,\ \ \ W_1=e^{-\Lambda s}y^{2\ell-1}.
\ee
We shall show that $z_0$ is a subsolution of $\mathcal{P}_1z=0$ in $D$, and $z_1$ a supersolution.

By \eqref{estimM2},  
\eqref{choiceR}, \eqref{controlDIm2} and \eqref{controlDSm2}, 
 taking $s_0\gg 1$, we see that
$$z_0\le V-e^{-\lambda s}\psi(y)-(BC_1-1)\nu e^{-\lambda s} y^{2\ell-1}\le w_y
\le V-e^{-\lambda s}\psi(y)+(BC_1-1)\nu e^{-\lambda s} y^{2\ell-1} \le z_1
\ \hbox{on $\partial_PD$.}$$
Also, for $\sigma=\sigma_2(p,\ell)\in (0,\min(\sigma_0,\sigma_1)]$ sufficiently small 
(where $\sigma_0,\sigma_1$ were respectively given by \eqref{choicesigma1} and \eqref{choicesigma1b}), we have
\be{controlVW}
\bigl|\ts\frac{W}{V}\bigl|+\bigl|\frac{W_y}{V_y}\bigl|\le Cy^{\beta+2\ell-1}e^{(k-\ell)s} 
\le C\sigma^{\beta+2\ell-1}e^{(k-\ell+\frac{\beta+2\ell-1}{2})s} 
=C\sigma_2^{\beta+2\ell-1}\le \frac12\ \ \hbox{ in $D$}.
\ee
Since $U_{yy}+U_y^p=0$ and $yU_y=(1-\beta)U$, we have
$\mathcal{L}_1V=-V_{yy}+\ts\frac{y}{2}V_y+\ts\frac{\beta}{2}V=pV^{p-1}V_y=\mathcal{N}V$,
hence
\be{computVW1}
\mathcal{P}_1 z_m=\mathcal{L}_1V+\mathcal{L}_1W-\mathcal{N}(V+W)=
\mathcal{N}V-\mathcal{N}(V+W)-b\mathcal{L}_1W_0+(-1)^m \nu\mathcal{L}_1W_1.
\ee
Differentiating the equality 
$0=[e^{-\lambda s}\phi]_s+\mathcal{L}[e^{-\lambda s}\phi]=
e^{-\lambda s}\bigl\{-\lambda\phi-\phi_{yy}-\bigl(\frac{\alpha}{y}-\frac{y}{2}\bigr)\phi_y-k\phi\bigr\}$,
and using $k=(1-\beta)/2$, we get
\be{computVW2}
\mathcal{L}_1W_0=e^{-\lambda s}\bigl\{-\lambda\psi-\psi_{yy}+\ts\frac{y}{2}\psi_y+\frac{\beta}{2}\psi\bigr\}
=e^{-\lambda s}\frac{\alpha}{y^2}\bigl(y\psi_y-\psi\bigr).
\ee
Moreover, using $\Lambda=\ell-k+\frac12$, we have
\be{computVW3}
\mathcal{L}_1W_1=e^{-\Lambda s}\bigl\{-\Lambda y^{2\ell-1}-(2\ell-1)(2\ell-2)y^{2\ell-3}+
(\ell-\tsfr)y^{2\ell-1}+\ts\frac{\beta}{2}y^{2\ell-1}\bigr\}\le -\frac12 e^{-\Lambda s}y^{2\ell-1}.
\ee
On the other hand, by elementary computation, $(1+X)^{p-1}(1+Y)=1+(p-1)X+Y+O(X^2+Y^2)$ for $|X|,|Y|\le 1/2$. 
Using \eqref{controlVW}, $V_y=-\beta y^{-1}V$ and $(p-1)V^{p-2}V_y=-\beta y^{-2}$, we then obtain
$$\begin{aligned}
\mathcal{N}(V+W)-\mathcal{N}V
&=p(V+W)^{p-1}(V_y+W_y)-pV^{p-1}V_y
=pV^{p-1}V_y\bigl\{\bigl(1+\ts\frac{W}{V}\bigr)^{p-1}\bigl(1+\ts\frac{W_y}{V_y}\bigr)-1\bigr\}\\
&=pV^{p-1}V_y\Bigl\{(p-1)\ts\frac{W}{V}+\frac{W_y}{V_y}
+O\bigl(\ts\frac{W^2}{V^2}+\ts\frac{W_y^2}{V_y^2}\bigr)\Bigr\}\\
&=p(p-1)V^{p-2}V_y\Bigl\{W-yW_y+O\bigl(\ts\frac{W^2+y^2W_y^2}{V}\bigr)\Bigr\}
=\frac{\alpha(yW_y-W)}{y^2} +O\bigl(\ts\frac{W^2+y^2W_y^2}{y^{2-\beta}}\bigr).
\end{aligned}$$
Combining this with \eqref{computVW1}--\eqref{computVW3}, we deduce that 
$$\begin{aligned}
(-1)^m\mathcal{P}_1z_m
&\le (-1)^m\ts\frac{\alpha}{y^2}(W-yW_y) 
+(-1)^mbe^{-\lambda s}\frac{\alpha}{y^2}\bigl(\psi-y\psi_y\bigr)
+C\ts\frac{W^2+y^2W_y^2}{y^{2-\beta}}-\ts\frac{\nu}{2} e^{-\Lambda s} y^{2\ell-1}\\
&\le \ts\frac{\alpha\nu}{y^2}\bigl(W_1-y\partial_y W_1\bigl)
-\ts\frac{\nu}{2} e^{-\Lambda s} y^{2\ell-1} +Ce^{-2\lambda s}y^{\beta+4\ell-4}\\
&\le -\ts\frac{\nu}{2} e^{-\Lambda s} y^{2\ell-1} +Ce^{-2\lambda s}y^{\beta+4\ell-4}
= e^{-\Lambda s} y^{2\ell-1}\bigl[-\ts\frac{\nu}{2} +Ce^{(\Lambda-2\lambda)s}y^{\beta+2\ell-3}\bigr]\le 0
\end{aligned}$$
in $D$ for $s_0\gg 1$ (where the last inequality follows similarly as in \eqref{compbarv}). By the comparison principle, it follows that
 $z_0\le w_y\le z_1$ in $D$, which guarantees \eqref{EstimComplong} for $i=1$.
\end{proof}

\begin{rem} \label{remzm}
The proof of Lemma~\ref{LemComplong} more generally shows the following.
For $s_1\ge s_0+1$, let $D$ defined by \eqref{DefDEstimComplong}, where $R_1$ satisfies 
 \eqref{choiceR} and $\sigma\in(0,\sigma_2]$, where $\sigma_2(p,\ell)$ is given by \eqref{controlVW}.
Let $u\in C^{2,1}(\overline D)$ be a solution of $u_t-u_{xx}=|u_x|^p$ in $D$ and let $v, F$ be defined from $u$ as above.
Assume that, for some $c\in\{-1,1\}$ and $c_1,c_2>0$, $v$ satisfies 
$F(v_y)\le c_1e^{-2\lambda s}y^{4\ell+\beta-3}$ in $D$ and
\be{controlLemComplong}
|D^i(v+ce^{-\lambda s}\phi)|\le  c_2 e^{-\lambda s} y^{2\ell-i},\quad i\in\{0,1\}
\ee
 on the parabolic boundary of $D$.
 It $s_0\gg 1$ (depending on $p,\ell,c_1,c_2$), then \eqref{controlLemComplong} remains true in $D$ 
 with $c_2$ replaced by $C(p,\ell)c_2$.
\end{rem}

\begin{proof}[Proof of Proposition~\ref{Propwlong}]
This is a direct consequence of Lemmas~\ref{LemI2long} and \ref{LemComplong}.
\end{proof}

\subsection{Completion of proof of Proposition~\ref{mainAPE}}
Let 
\be{defM0}
M_0=M_0(p,\ell)=\ts\frac{1}{2} \min\bigl\{M_1,M_2^{-1},M_3^{-1},\ell^{-1}\bigr\},
\ee 
where $M_1, M_2, M_3$ are respectively given by \eqref{defM1}, \eqref{estimM2} and \eqref{estimM3}.
Since $\eps\in(0, \eps_0]$ and $\nu=M_0\eps$, by Propositions~\ref{Propwshort} and \ref{Propwlong},
if $s_1\ge s_0\gg 1$ then, for any $d\in \mathcal{U}_{s_0,s_1}$ such that $P(d;s_0,s_1)=0$, we have
$$|D^i(w-U+e^{-\lambda s}\phi)|\le  \ts\frac{\eps}{2} e^{-\lambda s}(y^i+y^{2\ell-i}),\quad s\in [s_0,s_1],\ K e^{-\gamma s}\le y\le \sigma e^{s/2},\ i\in\{0,1\},$$
hence $\tilde w\in \mathcal{A}^{1/2}_{s_0,s_1}$. 
 Moreover, we have $\sum_{j=0}^{\ell-1} |d_j|\le \ts\frac{\eps}{2} e^{-\lambda s_0}$
by Lemma~\ref{Lemdsmallnu}. The proof is complete.

\section{Construction of special solutions: RBC case}  \label{REC-section5}

\subsection{Main results on special RBC solutions} 

In this section, we construct special solutions in the RBC case modifying the proof of Theorems~\ref{prop:special}, \ref{mainThm2}.
	
\begin{thm}
	\label{prop:special-LBC} 
	Let $p>2$, $0 < R \le \infty$, $Q=(0,\infty)\times (0,\tau)$, $\ell\in \N^*$ and let $\varphi_\ell$ be as in Theorem~\ref{prop:special}.
	For any $\eps\in(0,\varphi_\ell(0))$, there exist $\tau>0$ and a nonnegative 
	solution $u\in C(\overline Q)\cap C^{2,1}(Q)$ of 
	\eqref{equREC} with the following properties,
	for some constant $\sigma\in(0,R)$. 
	\smallskip
	
	\begin{itemize}

	\item[(i)]  {\it (space-time behavior)} 
	There holds
	\be{REC-spacetimeglobal}
	\bigl|u(x,t)-U(x)-(T-t)^\ell \varphi_\ell \bigl((T-t)^{-1/2} x\bigr)\bigr|
	\le\eps\bigl[(T-t)^\ell +x^{2\ell}\bigr]
	\quad\hbox{in } [0, \sigma].
	\ee
		\item[(ii)] {\it (outer region)}
	If $ R = \infty $, then there exists $ \delta \in (0,1) $ such that 
	\be{REC-outerdelta}
	| u(x,t) - U(x) | \ge \delta U(x) \quad \mbox{ in } [\sigma, \infty ) \times (0, \tau). 
	\ee
	If $ R < \infty $, then $u$ is regular at $x=R$, i.e.~$u\in C^{2,1}((0,R]\times(0,\tau])$ 
	and $u(R,t)=0$ for all $t\in (0,\tau]$ in the classical sense.

	\smallskip
	\item[(iii)]  {\it (intersections with the singular steady state)} 
	The solution $u$ satisfies assertion (iv) of Theorem~\ref{prop:special} with $T $ replaced by $\tau$.
	
		\end{itemize}
		\smallskip 
		
	\noindent Finally, for $R=\infty$, we may take $u\in C_b(\overline Q)$
	 if either $\ell$ odd or if we do not require property (ii).

\end{thm}

\begin{rem} \label{RemNonoscRBC}
 Although \eqref{REC-spacetimeglobal} at $x=0$ only gives
$C_1(\tau-t)^{\ell} \le u(0,t)\le C_2(\tau-t)^{\ell}$ for some constants $C_1,C_2>0$,
$u$ actually satisfies 
\be{REC-limitC}
\lim_{t\to \tau^-} (\tau-t)^{-\ell} u(0,t)=C
\ee
 for some $C>0$, as a consequence of Theorem \ref{th:RBC}(i),
that we will prove in subsection~\ref{SecProofBraid2}. Let us point out that a non-oscillation lemma similar to Lemma~\ref{LemNonOsc} holds also in the RBC case, 
and~\eqref{REC-limitC} could be deduced 
from such lemma. However, the space-time profile \eqref{recovery-profile} (which implies \eqref{REC-limitC})
will be established in subsection~\ref{SecProofBraid2} for general RBC solutions by using dynamical systems methods.
Thus the non-oscillation lemma is not needed here
(unlike in the GBU case were dynamical systems methods do not seem easily applicable
due to the existence of a boundary layer or inner region).
\end{rem} 

Theorem~\ref{prop:special-LBC} is obtained as a consequence of the following existence result for the corresponding problem in similarity variables.

\begin{thm} \label{REC-mainThm2}
Let $p>2$, $0<R\le\infty$, $\ell\in\N^*$, $\lambda=\ell-k$ and let $\phi=\varphi_\ell$ be as in Theorem~\ref{mainThm2}. 
Set $D=\{(y,s);\  0<y<Re^{s/2},\ s>s_0\}$, $\Sigma=\{(y,s);\  y=Re^{s/2},\ s>s_0\}$. For any $\eps\in(0,1)$, there exist $s_0,\sigma>0$ and a nonnegative classical solution 
$w\in C(\overline D)\cap C^{2,1}(D\cup\Sigma)$ of 
\be{REC-eqw2}
\left\{\ 
\begin{aligned}
w_s&=w_{yy}-\frac{y}{2}w_y+kw+|w_y|^p,&&\quad\hbox{ in $D$,} \\
w&=0,&&\quad \hbox{ on $\Sigma$ in the classical sense (if $R<\infty$),} \\
\end{aligned}
\right.
\ee
such that, for all $s>s_0$:
\be{REC-Concl2mainThm2}
\begin{aligned}
\ \bigl|w(y,s)-U(y)-e^{-\lambda s}\phi(y)\bigr| &\le \eps e^{-\lambda s} (1+y^{2\ell}),
&&y\in [0, \sigma e^{s/2}],\\
\ \bigl|w_y(y,s)-U'(y)-e^{-\lambda s}\phi'(y))\bigr| &\le \eps e^{-\lambda s} (y+y^{2\ell-1}),
&&y\in (0, \sigma e^{s/2}].
\end{aligned}
\ee
\end{thm}

The proof of Theorem~\ref{REC-mainThm2} is similar to, but simpler than, the proof of Theorem~\ref{mainThm2}. 
Since $w$ is now sought to be positive at $y=0$, it is natural to consider an approximate solution
of the form 
$$w\sim U+e^{-\lambda s}\phi(y).$$
Therefore, we do not need any inner region with quasi-stationary behavior
but more simply look for an eigenfunction expansion of $v=w-U$ which holds up to the boundary $y=0$
(combined as before with an outer region to reconnect with the regular part of the solution).
Since we are thus looking for a solution with persistent singularities, we shall work with the initial boundary value problem~\eqref{equRECexist} (recast in similarity variables).

 \subsection{Initial data and topological argument}

Again it suffices to consider the cases $R=\infty$ and $R=1$. 
 We keep the notation in the paragraph containing \eqref{defaphak}.
 Let the constants $\sigma\in (0,\frac18)$ and $M_0\in (0,1)$,
depending only on $p,\ell$, be respectively given by Lemma~\ref{controlinner-RBC} and~\eqref{defM0-REC}.
We introduce a parameter $\eps\in\bigl(0,\min\{1,\ts\frac12\phi(0)\}\bigr)$ and set $\nu=M_0\eps$.
The initial time $s_0>0$ will be chosen large enough below and will depend only on $p,\ell,\eps$. 
We denote $y_2(s)=\sigma e^{s/2}$. 
Finally, we fix a smooth cut-off function $\Theta_1(z)$ such that $\Theta_1=1$ for $z\le 1$,  $\Theta_1=0$ for $z\ge 2$
and $\Theta_1'\le 0$. Set $\Theta(y)=\Theta_1\bigl(\ts\frac{1}{2\sigma}e^{-s_0/2}y\bigr)$.

For any $d\in\R^\ell$ that satisfies \eqref{condd}, we 
define $v_0=v_0(\cdot,d)$ as follows:
\be{REC-defw0}
v_0(y):=
\Theta(y)\Bigl\{e^{-\lambda s_0}\phi+\ds\sum_{j=0}^{\ell-1}d_j\varphi_j\Bigr\} +(\Theta(y)-1)U(y).
\ee
If $\ell$ is even and $R=\infty$, we also consider the alternative choice:
\be{REC-defw0even}
v_0(y):=
\begin{cases}
e^{-\lambda s_0}\phi+\ds\sum_{j=0}^{\ell-1}d_j\varphi_j&\hbox{in $[0, 2\sigma e^{s_0/2}],$} \\
b_1U(y)&\hbox{in $(2\sigma e^{s_0/2},\infty),$} \\
\end{cases}
\ee
where $b_1=b_1(d,s_0):=\Bigl\{\Bigl[e^{-\lambda s_0}\phi+\ds\sum_{j=0}^{\ell-1}d_j\varphi_j\Bigr]U^{-1}\Bigr\}(2\sigma e^{s_0/2})$
(which ensures the continuity of $v_0$).
 The choice \eqref{REC-defw0even} comes from
the need to construct a solution which intersects $U$ exactly $\ell$ times on $(0,\infty)$ 
 (in which case $u_0$ must be unbounded).
Observe also that, instead of a minus sign in front of the term $e^{-\lambda s_0}\phi$ in 
\eqref{defw0}-\eqref{defw0odd},
we now have a positive sign.
We then denote $w_0(y)=U(y)+v_0(y)$ and $u_0(x)=e^{-ks_0}w_0(xe^{s_0/2})$.

\begin{lem} \label{controlinner-RBC}
Let $\sigma_2$ be given by \eqref{controlVW}.
For $\sigma=\sigma(p,\ell)\in\bigl(0,\sigma_2]$ sufficiently small, 
we have the following properties for $s_0\gg 1$.
\smallskip
\begin{itemize}
\item[(i)]$u_0\ge 0$ in $[0,R)$.
\smallskip

\item[(ii)]  Under assumption \eqref{REC-defw0even}, for $\eps\in(0,1]$ and $d$ satisfying
 \eqref{condd}, we have $b_*\le b_1\le C$, for some $b_*(p,\ell)>0$.
\smallskip

\item[(iii)] If $R=1$, then $\sup u_0\le\frac14$ 
and $u_0(x)=0$ for $x\ge \frac12$.
\end{itemize}
\end{lem}  

\begin{proof}
(i) Take $y_*>0$ such that
$\phi(y)>\phi(0)/2$ on $[0,y_*]$. Assumptions \eqref{REC-defw0}-\eqref{REC-defw0even} and \eqref{condd} guarantee that $w_0\ge 0$ on $[0,y_*]$.
On the other hand, by \eqref{controlDiphi}, for $\sigma\le\sigma_1$ and $y\in [y_*,2\sigma e^{s_0/2}]$, we have
$$U^{-1}\Bigl|e^{-\lambda s_0}\phi-\ds\sum_{j=0}^{\ell-1}d_j\varphi_j\Bigr|
\le e^{-\lambda s_0}U^{-1}\Bigl(|\phi|+\eps\max_{1\le j\le \ell-1}|\varphi_j| \Bigr)
\le Cy^{\beta-1+2\ell} e^{(k-\ell)s_0}
\le C\sigma_1^{\beta+2\ell-1}\le \ts\frac12.$$
Consequently $w_0\ge 0$ on $[y_*,Re^{s_0/2}]$.
\smallskip

 (ii) It is similar to that of Lemma~\ref{controlinner}(i).
\smallskip

(iii) This easily follows from the support properties of $\Theta$ and the fact that $U(0)=0$, by taking 
$\sigma=\sigma(p,\ell)$ sufficiently small.
\end{proof}

Since, by Lemma~\ref{controlinner-RBC}, $u_0$ satisfies all the assumptions of Proposition~\ref{locexistvapprox},
this guarantees the existence of a global solution $u$ of \eqref{equRECexist}-\eqref{equRECregulC1}.
Let $w=w(y,s;d)$ be the corresponding solution of \eqref{eqw} defined by 
$$ 
w(y,s)=e^{ks}u(ye^{-s/2},e^{-s_0}-e^{-s}),\quad 0\le y<Re^{s/2}.
$$ 
So as to work with unknown functions defined on the entire half-line, we recall the extentions introduced in Lemma~\ref{BernsteinEst}:
$$ 
\begin{aligned}
\tilde u(x,t)&=\zeta(x) u(x,t) &&\quad\hbox{ in $[0,\infty)\times[0,\infty)$,} \\
\tilde w(y,s)&=e^{ks}\tilde u(ye^{-s/2},e^{-s_0}-e^{-s})=\zeta(ye^{-s/2}) w(y,s) &&\quad\hbox{ in $[0,\infty)\times[s_0,\infty)$,} \\
\tilde v(y,s)&=\tilde w(y,s)-U(y) &&\quad\hbox{ in $[0,\infty)\times[s_0,\infty)$,}
\end{aligned}
$$ 
where for $R=1$, $\zeta\in C^2([0,\infty))$ is a fixed cut-off function such that 
 $\zeta=1$ in $[0,\frac{1}{3}]$ and $\zeta=0$ in $[\frac{1}{2},\infty)$,
 whereas for $R=\infty$ we just set $\zeta=1$.
 Note that due to the support properties of $\Theta, \zeta$, we have $\tilde v(\cdot,s_0)\equiv v_0$.
For $\theta\in (0,1]$ and $s_1\ge s_0$, we define
$$\begin{aligned}
&\mathcal{A}^\theta_{s_0,s_1}=
\Bigl\{V\in L^\infty(s_0,s_1;W^{1,\infty}([0,\infty)));\ \bigl|D^i(V-e^{-\lambda s}\phi)\bigr| \le \theta\eps e^{-\lambda s} (y^i+y^{2\ell-i})\\
&\qquad\qquad\qquad\qquad
\quad\hbox{ for 	all }  s_0\le s\le s_1,\ 0\le y\le y_2(s),\ i\in\{0,1\}\Bigr\}
	\end{aligned}$$
and
$$\mathcal{U}_{s_0,s_1}=\Bigl\{d\in \R^\ell; \hbox{ \eqref{condd} holds 
and $\tilde v=\tilde v(y,s;d) \in \mathcal{A}^1_{s_0,s_1}$}\Bigr\}.$$
Note that we can rewrite the initial data
 in \eqref{REC-defw0} (resp., \eqref{REC-defw0even}) as
\be{REC-w0tildephi}
v_0(y)=\sum_{j=0}^{\ell-1}d_j\varphi_j+e^{-\lambda s_0}\hat\phi,
\ee
with
\be{REC-w0tildephi2}
\hat\phi:=
\begin{cases}
 \noalign{\vskip -1mm}
 \phi& \hbox{in $[0,2\sigma e^{s_0/2}]$} \\
 (\Theta-1)e^{\lambda s_0}\Bigl(U(y)+\ds\sum_{j=0}^{\ell-1}d_j\varphi_j\Bigr)+\Theta \phi
& \hbox{in $(2\sigma e^{s_0/2},\infty)$}\\
 \noalign{\vskip -1mm}
\end{cases}
\ee
 (resp.,
\be{REC-w0tildephi2odd}
\hat\phi:=
\begin{cases}
 \noalign{\vskip -1mm}
 \phi& \hbox{in $[0,2\sigma e^{s_0/2}]$} \\
 e^{\lambda s_0}\Bigl\{b_1U(y)-\ds\sum_{j=0}^{\ell-1}d_j\varphi_j\Bigr\}
& \hbox{in $(2\sigma e^{s_0/2},\infty)$).}
\end{cases}
\ee

By similar arguments as in the proof of Lemma~\ref{Lemcvcphi} and \eqref{estimuouter}, we obtain that
\be{REC-tildephicv}
\|\hat\phi-\phi\|\to 0,
\quad  \hbox{as $s_0\to\infty$, uniformly for 
$d$ satisfying \eqref{condd}}
\ee
and that for $R=\infty$ and $s_0\gg 1$,
 if either \eqref{REC-defw0} holds and $\ell$ is odd, or \eqref{REC-defw0even} holds and $\ell$ is even,
 then, for any $d\in \mathcal{U}_{s_0,s_1}$,
\be{REC-estimuouter}
|u(x,t)-U(x)|\ge \delta U(x) \quad \mbox{ in } [\sigma,\infty)\times[0,e^{-s_0}-e^{-s_1}),
\ee
with $\delta\in(0,1)$ independent of $d$ and $s_1$.

Let the map $P(d;s_0,s_1)$ be defined in \eqref{defmapP}.
Analogous to Proposition~\ref{mainAPE}, we shall establish the following key a priori estimate, which is the main ingredient of the topological argument used in the proof of Theorem~\ref{REC-mainThm2}.

\begin{prop} \label{REC-mainAPE}
There exists $\bar s_0>0$ 
such that
if $s_1\ge s_0\ge \bar s_0$ and $d\in \mathcal{U}_{s_0,s_1}$ satisfy $P(d;s_0,s_1) = 0$, then 
$\tilde v(\cdot,\cdot;d) \in \mathcal{A}^{1/2}_{s_0,s_1}$
 and, moreover, 
$\ds\sum_{j=0}^{\ell-1} |d_j|\le \ts\frac{\eps}{2} e^{-\lambda s_0}$.
\end{prop}

We postpone its proof to the next subsection.
As a consequence of Proposition~\ref{REC-mainAPE}, we have the following two propositions, whose proofs are completely similar to 
those of Propositions~\ref{TopolArg1} and \ref{TopolArg2} and are thus omitted.
 We just mention that the property $s_*>s_0$ in the proof of Proposition~\ref{REC-TopolArg2}
is now a consequence of the continuity result in Proposition~\ref{locexistvapprox}(iii).
In the proof of Proposition~\ref{REC-TopolArg2} (and of Theorem~\ref{REC-mainThm2}), we use the 
continuous dependence property from \eqref{equRECppty2}
(which is enough to pass to the limit also in the estimate of $w_y$, by using finite differences).

\begin{prop} \label{REC-TopolArg1}
Let $s_0$ be as in Proposition~\ref{REC-mainAPE}.
If $\mathcal{U}_{s_0,s_1}\neq\emptyset$ with some $s_1\ge s_0$, then 
$${\rm deg}(P(\cdot;s_0,s_1)) =1,$$
where ${\rm deg}(P(\cdot;s_0,s_1))$, denotes the degree of $P(\cdot;s_0,s_1)$ with respect to $0$ in $\mathcal{U}_{s_0,s_1}$.
\end{prop}

\begin{prop} \label{REC-TopolArg2}
Let $s_0$ be as in Proposition~\ref{REC-mainAPE}.
Then $\mathcal{U}_{s_0,s_1}\neq\emptyset$ for all $s_1\ge s_0$.
\end{prop}

\begin{proof}[Proof of Theorem~\ref{REC-mainThm2}]
 Let $v_0$ be given by \eqref{REC-defw0} or \eqref{REC-defw0even}\footnote{We note that the case \eqref{REC-defw0} will be sufficient for Theorem~\ref{REC-mainThm2}; however the case 
\eqref{REC-defw0even} will be used in the proof of Theorem~\ref{prop:special-LBC}.}
 and $s_0$ be as in Proposition~\ref{REC-mainAPE}.
Take a sequence $\{s_n\} \subset (s_0,\infty)$ with $s_n\to \infty$ as $n\to\infty$. From Proposition~\ref{REC-TopolArg2}, 
for each $n$ there exists $d_n\in \mathcal{U}_{s_0,s_n}$, hence $\tilde v(y, s; d_n)\in \mathcal{A}^1_{s_0,s_n}$. 
Since $\{d_n\}$ is bounded, we may assume that $d_n\to \bar d$ as $n\to\infty$ for some $\bar d$. 
We deduce $\tilde v(y, s; \bar d)\in \mathcal{A}^1_{s_0,\infty}$ by continuous dependence.
The corresponding solution $w$ of \eqref{REC-eqw2} thus satisfies~\eqref{REC-Concl2mainThm2}.

Set $\tau=e^{-s_0/2}$. Since $u_0\ge 0$ and $u(0,t)>0$ for all $t\in(0,\tau)$ by \eqref{REC-Concl2mainThm2}
 and $\eps<\phi(0)$, we have $u>0$ in $(0,R)\times [0,\tau)$ by \eqref{equRECviscpos}, hence $w>0$.
Moreover, if $R=1$, we have $w=0$ on $\Sigma$ in the classical sense by \eqref{equRECexist}.
The theorem is proved.
\end{proof}

\begin{proof}[Proof of Theorem~\ref{prop:special-LBC}]
Let $w$ be given by Theorem~\ref{REC-mainThm2} and let
$u(x,t)=(\tau-t)^k w\bigl(x(\tau-t)^{-1/2},$ $-\log(\tau-t)\bigr)$ with $\tau=e^{-s_0}$ be the corresponding solution of 
\eqref{equRECexist}.
Since $w\ge 0$, it follows from \eqref{equRECviscpos} that
$u$ satisfies the boundary conditions in \eqref{equREC}$_2$ at $x=0$ in the viscosity sense.

 Assertion (i) is an immediate consequence of \eqref{REC-Concl2mainThm2}.
To check assertion~(ii), let us first consider the case $R=\infty$.
If $\ell$ is odd (resp., even), we take $v_0$ given by \eqref{REC-defw0} (resp., \eqref{REC-defw0even}).
Then \eqref{REC-outerdelta} follows from \eqref{REC-estimuouter}.
When $R=1$, assertion~(ii) is guaranteed by Proposition~\ref{locexistvapprox}(i).
 As for assertion (iii), it follows similarly as in the proof of Theorem~\ref{prop:special}. 

 Finally, to show the last part of the theorem, we note that, when $R=\infty$, $u_0$ is bounded whenever $v_0$
is given by \eqref{REC-defw0} (which is possible if either $\ell$ is odd or property (ii) is not required).
We then deduce from \eqref{equRECppty0} that $u\in C_b(\overline Q)$.
\end{proof}

\subsection{Proof of Proposition~\ref{REC-mainAPE}.}\label{REC-section6}

It is similar to, but easier than, the proof of Proposition~\ref{mainAPE}.
To avoid lengthy repetitions, we will therefore often refer to the latter
and only indicate the necessary changes.
In this subsection:

\vskip 2pt

$\bullet$ $C$ will denote a generic positive constant depending only on $p,\ell$; 

\vskip 1pt

$\bullet$ the required largeness of $s_0\gg 1$ will depend on the parameter $\eps$, but not on $d$.

\vskip 2pt

\noindent We shall make use of the variation of constants formula for $\tilde v$, given by \eqref{varconstv2tilde}
in Proposition~\ref{locexistvapproxExt}, with initial data $v_0$ in \eqref{REC-defw0}.

\smallskip

Proposition~\ref{REC-mainAPE} will be an immediate consequence of the following short-time and long-time estimates.

\begin{prop} \label{REC-Propwshort}
If $s_1\ge s_0\gg 1$ then, for any $d\in \mathcal{U}_{s_0,s_1}$ such that $P(d;s_0,s_1)=0$, we have
\be{est-REC-Propwshort}
|D^i(v-e^{-\lambda s}\phi)|\le  M_1\nu e^{-\lambda s}(y^i+y^{2\ell-i}),\quad s\in (s_0,s_0+1],\ 0\le y\le \sigma e^{s/2},\ i\in\{0,1\},
\ee
with $M_1=M_1(p,\ell)>0$. 
\end{prop} 

\begin{prop} \label{REC-Propwlong}
If $s_1\ge s_0\gg 1$ then, for any $d\in \mathcal{U}_{s_0,s_1}$ such that $P(d;s_0,s_1)=0$, we have
\be{est-REC-Propwlong}
|D^i(v-e^{-\lambda s}\phi)|\le  M_2\nu e^{-\lambda s}(y^i+y^{2\ell-i}),\quad s\in (s_0+1,s_1],\ 
0\le y\le \sigma e^{s_0/2},\ i\in\{0,1\},
\ee
with $M_2=M_2(p,\ell)>0$. 
\end{prop}

In view of their proofs, we first collect a number of preliminary estimates.

\begin{lem} \label{REC-LemvyFvy}
If $s_1\ge s_0\gg 1$ then, for any $d\in \mathcal{U}_{s_0,s_1}$ and $s\in[s_0,s_1]$, we have
\be{REC-controlvybasic}
|\tilde v_y(y,s)|\le
\begin{cases}
Ce^{-\lambda s}(y+y^{2\ell-1})& \hbox{ for $0\le y\le y_2(s)$} \\
Ce^{-\beta s/2}& \hbox{ for $y\ge y_2(s)$},
\end{cases}
\ee
\be{REC-controlFvybasic}
0\le \tilde F(\tilde v_y(y,s))\le
\begin{cases}
Ce^{-2\lambda s}y^{\beta+1} \bigl(1+y^{4(\ell-1)}\bigr) & \hbox{ for $0\le y\le y_2(s)$} \\
Ce^{-(\beta+1)s/2}& \hbox{ for $y\ge y_2(s)$,}  
\end{cases}
\ee
and
\be{REC-Fvytau}
|(\tilde F(\tilde v_y(\tau)),\varphi_j)|\le C e^{-2\lambda\tau},\quad j\in\N,\ \tau\in[s_0,s_1].
\ee
If, moreover, $P(d;s_0,s_1)=0$, then
\be{REC-dsmall}
|d_j|\le \nu e^{-\lambda s_0},\quad j\in\{0,\dots,\ell-1\}.
\ee

 \end{lem} 

\begin{proof}
The first case of \eqref{REC-controlvybasic} follows from $\tilde v\in \mathcal{A}^1_{s_0,s_1}$.
The proof ot the second case of \eqref{REC-controlvybasic},
completely similar to that of the fourth case of \eqref{controlvybasic}, is a consequence of \eqref{equRECsingux} and of the estimate 
\be{bounduux2-REC}
|u_x|\le C\quad\hbox{ in $Q:=(\sigma,R)\times (0,T-e^{-s_1})$.}
\ee
As before, \eqref{bounduux2-REC} is obtained by
appying the maximum principle to the equation $z_t-z_{xx}=p|z|^{p-2}zz_x$ satisfied by $z=u_x$ in $Q$
and using \eqref{equRECsingux}, \eqref{equRECder1}
(the application of the maximum principle is licit also for $R=\infty$ since $u_x\in L^\infty(Q)$ owing to \eqref{equRECregulC1}).

\smallskip

To prove \eqref{REC-controlFvybasic}, noting that
$e^{-\lambda s}(y+y^{2\ell-1})\le Ce^{-\lambda s}e^{(2\ell-1)s/2}=Ce^{-\beta s/2}\le Cy^{-\beta}$ for $y\in[0,y_2(s)]$,
we deduce from \eqref{REC-controlvybasic} that $U_y+|v_y|\le Cy^{-\beta}$ for $y\le y_2(s)$,
and $U_y+|v_y|\le Ce^{-\beta s/2}$ for $y\ge y_2(s)$.
The first case of \eqref{REC-controlFvybasic}, and the second case for $R=\infty$, then readily follows from \eqref{REC-controlvybasic} and \eqref{Fvy2}. As for the second case for $R=1$, it follows from \eqref{Ftildeg1} and \eqref{bounduux2-REC}.
\smallskip

To show \eqref{REC-Fvytau}, using \eqref{REC-controlFvybasic}, the Cauchy-Schwarz' inequality and $\|\varphi_j\|=1$, we write
$$\begin{aligned}
|(F(v_y(\tau)),\varphi_j)|
&\le Ce^{-2\lambda\tau} \int_0^{\sigma e^{\tau/2}} [y^{\beta+1} \bigl(1+y^{4(\ell-1)}\bigr)]|\varphi_j|\rho
+Ce^{-(\beta+1)\tau/2}\int_{\sigma e^{\tau/2}}^\infty |\varphi_j|\rho \\
&\le Ce^{-2\lambda\tau} 
+Ce^{-(\beta+1)\tau/2} \Bigl(\int_{\sigma e^{\tau/2}}^\infty \rho\,dy\Bigr)^{1/2}
\le Ce^{-2\lambda\tau}+Ce^{-C e^\tau} \le Ce^{-2\lambda\tau}. 
\end{aligned}$$
Finally, based on \eqref{REC-Fvytau}, the proof of \eqref{REC-dsmall} is completely similar to that of Lemma~\ref{Lemdsmallnu}.
\end{proof}

In view of the proof of Proposition~\ref{REC-Propwshort}, similar to \eqref{defS123}, we split $\tilde v$ as
\be{REC-defS123}
\tilde v=S_1+S_2+S_3,\ \hbox{ where }\ 
\begin{cases}
S_1(y,s)&=e^{-\lambda s}(\hat\phi,\phi)\phi,\\
S_2(y,s)&=\ds\sum_{j=0}^{\ell-1}d_je^{-\lambda_j(s-s_0)}\varphi_j+\sum_{j\ne\ell} 
e^{-\lambda s_0}e^{-\lambda_j(s-s_0)}(\hat\phi,\varphi_j)\varphi_j, \\
S_3(y,s)&=\ds\int_{s_0}^s e^{-(s-\tau)\mathcal{L}}\tilde F(\tilde v_y(\tau))\,d\tau.
\end{cases}
\ee
We record the following pointwise bounds for the initial data of $S_2$.

\begin{lem}  \label{REC-LemestimS2s0}
If $s_1\ge s_0\gg 1$ then, for any $d\in \mathcal{U}_{s_0,s_1}$ such that $P(d;s_0,s_1)=0$, we have
\be{REC-estimS2}
|S_2(y,s_0)|\le
\begin{cases}
C\nu e^{-\lambda s_0}(1+y^{2\ell})& \hbox{ in 
$D_1:=[0,2\sigma e^{s_0/2}]$,} \\
C y^{2\ell}& \hbox{ in $D_2:=[2\sigma e^{s_0/2},\infty)$, }
\end{cases}
\ee
and
\be{REC-estimDS2}
|\partial_yS_2(y,s_0)|\le
\begin{cases}
C\nu e^{-\lambda s_0}(y+y^{2\ell-1})& \hbox{ in $D_1$,} \\
C y^{2\ell-1}& \hbox{ in $D_2$.}
\end{cases}
\ee
\end{lem} 

\begin{proof}
From \eqref{REC-defS123} we have
\be{REC-splitS2b}
S_2(y,s_0)=v_0-e^{-\lambda s_0}(\hat\phi,\phi)\phi.
\ee
By \eqref{REC-w0tildephi}-\eqref{REC-w0tildephi2odd}, we get in particular
$S_2(y,s_0)=\sum_{j=0}^{\ell-1}d_j\varphi_j+e^{-\lambda s_0}\bigl(1-(\hat\phi,\phi)\bigr)\phi$ in $D_1$,
so that \eqref{REC-estimS2}-\eqref{REC-estimDS2} in $D_1$
follow from \eqref{controlDiphi}, \eqref{REC-tildephicv} and \eqref{REC-dsmall}.
As for \eqref{REC-estimS2}-\eqref{REC-estimDS2} in the range $D_2$, they easily follow from
 \eqref{controlDiphi}, \eqref{REC-defw0}, \eqref{REC-defw0even}, \eqref{REC-dsmall} and \eqref{REC-splitS2b}. \end{proof}

\goodbreak

Recalling \eqref{REC-tildephicv} and \eqref{REC-defS123},
Proposition~\ref{REC-Propwshort} is a direct consequence of the following two lemmas,
which respectively estimate $S_2$ and $S_3$.

\begin{lem} \label{REC-LemS2short}
If $s_1\ge s_0\gg 1$ then, for any $d\in \mathcal{U}_{s_0,s_1}$ such that $P(d;s_0,s_1)=0$, we have
$$|D^iS_2(y,s)|\le  C\nu e^{-\lambda s}(y^i+y^{2\ell-i}),\quad s\in (s_0,s_0+1],\ 0\le y\le y_2(s),\ i\in\{0,1\}.$$
\end{lem} 

\begin{proof}
Since the proof is very similar to that of Lemma~\ref{LemS2short}, we only indicate the necessary changes.
For $i\in\{0,1\}$, we write
$$|D^iS_2(y,s)|\le \int_0^\infty G_i(s-s_0,y,\xi)|D^iS_2(\xi,s_0)|\xi^\alpha d\xi
= \int_0^{2\sigma e^{s_0/2}}+ \int_{2\sigma e^{s_0/2}}^\infty\equiv S^i_{2,3}+S^i_{2,4},$$
where $G_i$ is given by \eqref{defGi}, and we now use the splitting $S^i_{2,3}=\int_{E^1_y}+\int_{E^2_y}$ 
where $E^1_y=(0, 2\sigma e^{s_0/2})\cap (y/2,2y)$ and $E^2_y=(0, 2\sigma e^{s_0/2})\setminus (y/2,2y)$.
Arguing exactly as in the proof of Lemma~\ref{LemS2short}, we then obtain
$S^i_{2,3}\le C\nu e^{-\lambda s_0}(y^i+y^{2\ell-i})$, as well as
$S^i_{2,4}\le C\nu e^{-\lambda s_0}y^i$.
\end{proof}

Turning to $S_3$, we obtain the following.
 
\begin{lem}  \label{REC-LemS3short}
Assume $s_1\ge s_0\gg 1$ and let $d\in \mathcal{U}_{s_0,s_1}$ be such that $P(d;s_0,s_1)=0$.
Then, for any $s\in (s_0,s_1]$, $\bar s\in \bigl[\max(s_0,s-1),s\bigr]$ and $i\in\{0,1\}$, we have
\be{REC-controlS3short0}
|D^i\hat S_3(y,s)|:=\Bigl|\int_{\bar s}^s D^ie^{-(s-\tau)\mathcal{L}}\tilde F(\tilde v_y(\tau))\,d\tau\Bigr| 
\le  C\nu e^{-\lambda s}(y^i+y^{2\ell-i}), \qquad y\in [0,y_2(s)].
\ee
In particular, we have 
\be{REC-controlS3short}
|D^iS_3(y,s)|\le  \nu e^{-\lambda s}(y^i+y^{2\ell-i}), \qquad s_0\le s\le s_0+1\le s_1,\ y\in [0,y_2(s)], \ i\in\{0,1\}.
\ee
\end{lem}

\begin{proof}
Again, the proof is very similar to that of Lemma~\ref{LemS3short} and we only indicate the necessary changes.
Take $s\in (\bar s,s_1)$ and $y\in (0,y_2(s))$.
For $\tau\in [s_0,s)$, we denote
$$X(\tau)=e^{(\tau-s)/2}(s-\tau)^{-1/2}y,\qquad X_2(\tau)=(s-\tau)^{-1/2}y_2(\tau)$$
and observe that $X(\tau)\in(0,X_2)$.
Arguing as for \eqref{splitS3}, we now have
$$
\begin{aligned}
|D^i\hat S_3(y,s)|
&\le\int_{\bar s}^s\int_0^\infty 
 (s-\tau)^{-\frac{i}{2}} e^{-C(X(\tau)-z)^2}\bigl(1\wedge X(\tau)\bigr)^i
 \bigl(1+X(\tau)z\bigr)^{-\frac{\alpha}{2}} F\bigl(v_y(z\sqrt{s-\tau},\tau)\bigr) z^\alpha dz d\tau \\
&=\int_{\bar s}^s\int_0^{X_2(\tau)} 
+ \int_{\bar s}^s\int_{X_2(\tau)}^\infty \equiv S^i_{3,3}(y,s)+S^i_{3,4}(y,s).
\end{aligned}
$$
Using estimate \eqref{REC-controlFvybasic}, the proofs of \eqref{estimS3-4} and \eqref{estimS3-4b} then directly give
$$S_{3,3}^i+S_{3,4}^i \le \nu e^{-\lambda s}(y^i+y^{2\ell-i}). \qedhere$$
\end{proof}

We next turn to the proof of Proposition~\ref{REC-Propwlong}.
For this purpose, we shall use the decomposition $\tilde v=-I_0+I_1-I_2+I_3$
where $I_j$ are defined in \eqref{defI0123} with $\hat\phi$ given by 
\eqref{REC-w0tildephi2}-\eqref{REC-w0tildephi2odd}.
This decomposition is valid whenever $\tilde v\in\mathcal{A}^1_{s_0,s_1}$ satisfies $P(d;s_1,s_2)=0$,
as a consequence of \eqref{v0phij}, \eqref{defI0123b}
and $v_0=\sum_{j=0}^{\ell-1}d_j\varphi_j+e^{-\lambda s_0}\hat\phi$.
The outline of proof of Proposition~\ref{REC-Propwlong} is then as follows:
\vskip 2pt

$\bullet$ Estimation of $I_1$

$\bullet$ Estimation of $I_2$ and $I_3$ for $y\in [0,R_1]$ (with suitably chosen large $R_1$)

$\bullet$ Estimation of $v$ and $v_y$ at $y=y_2(s)$

$\bullet$ Estimation of $v$ and $v_y$ in the remaining part $y\in [R_1,y_2(s)]$.

\smallskip

We start with $I_1$, which is easy to estimate globally.

\begin{lem} \label{REC-LemI1long} 
If $s_0\gg 1$ and $s_1\ge s_0+1$ then, for any $d\in \mathcal{U}_{s_0,s_1}$, we have
\be{REC-controlDI1}
|D^iI_1(y,s)|\le  C\nu e^{-\lambda s} (y^i+y^{2\ell-i}),\quad\hbox{ for all $s\in [s_0+1,s_1]$, $y>0$ and $i\in\{0,1\}$.}
\ee
\end{lem}

\begin{proof}
For later purpose, we actually show a slightly stronger estimate, namely:
\be{REC-controlDI1b}
\begin{aligned}
&\sum_{j=0}^\ell |D^i\varphi_j(y)|\int_{s-1}^{s_1} e^{-\lambda_j(s-\tau)}|(\tilde F(\tilde v_y(\tau)),\varphi_j)|\,d\tau
\le C(y^i+y^{2\ell-i})\sum_{j=0}^\ell
e^{-\lambda_j s} \int_{s-1}^\infty e^{(\lambda_j-2\lambda)\tau}\,d\tau\\
&\quad\le C(y^i+y^{2\ell-i}) e^{-2\lambda s},\quad
s\in [s_0+1,s_1],\ y>0,
\end{aligned}
\ee
for $s_0\gg 1$, which is a consequence of inequality \eqref{REC-Fvytau}.
Estimate \eqref{REC-controlDI1} is then an immediate consequence of \eqref{REC-controlDI1b}.
\end{proof}

We now estimate $I_2$ and $I_3$ for $y\in [0,R_1]$, where $R_1>0$ (depending only on $p,\ell$) 
is chosen to satisfy \eqref{choiceR}.
By using \eqref{REC-Fvytau}, \eqref{REC-controlS3short0}, \eqref{REC-controlDI1b}
and arguing along the lines of the proof of Lemma~\ref{LemI2long}, we obtain:

\begin{lem}  \label{REC-LemI2long}
If $s_0\gg 1$ and $s_1\ge s_0+1$ then, for any $d\in \mathcal{U}_{s_0,s_1}$, we have, for $m\in\{2,3\}$,
\be{REC-controlDIm}
|D^iI_m(y,s)|\le  C\nu e^{-\lambda s} y^i,\quad\hbox{ for all $s\in [s_0+1,s_1]$, $y\in [0,R_1]$ and $i\in\{0,1\}$}.
\ee
Consequently,
\be{REC-controlDIm2}
|D^i(v-e^{-\lambda s}\phi)\le  C\nu e^{-\lambda s} y^i,\quad\hbox{ for all $s\in [s_0+1,s_1]$,
$y\in [0,R_1]$ and $i\in\{0,1\}$}.
\ee
\end{lem}

We next estimate $v$ and $v_y$ at $y=y_2(s)$.
For this purpose, we go back to the decomposition~\eqref{REC-defS123}.

\begin{lem}  \label{REC-LemS2long}
Set $y_2=y_2(s)=\sigma e^{s/2}$. 
If $s_0\gg 1$ and $s_1\ge s_0+1$ then, for any $d\in \mathcal{U}_{s_0,s_1}$ such that $P(d;s_0,s_1)=0$, we have, for $m\in\{2,3\}$,
\be{REC-controlDSm}
|D^iS_m(y_2,s)|\le C \nu e^{-\lambda s} y_2^{2\ell-i},\quad\hbox{ for all $s\in [s_0+1,s_1]$ and $i\in\{0,1\}$}.
\ee
Consequently,
\be{REC-controlDSm2}
\bigl|D^i\bigl(v-e^{-\lambda s}\phi\bigr)(y_2,s)\bigr| 
\le  C\nu e^{-\lambda s} y_2^{2\ell-i},\quad\hbox{ for all $s\in [s_0+1,s_1]$ and $i\in\{0,1\}$}.
\ee
\end{lem} 

\begin{proof}
Let $i\in\{0,1\}$ and $s\in [s_0+1,s_1]$. 
First consider the case $m=2$.
Similar to \eqref{DiS2}, we have
$$|D^iS_2(y_2,s)|\le \int_0^\infty G_i(s-s_0,y_2,\xi)|D^iS_2(\xi,s_0)|\xi^\alpha d\xi
= \int_0^{2\sigma e^{s_0/2}}+ \int_{2\sigma e^{s_0/2}}^\infty\equiv S^i_{2,2}+S^i_{2,3},$$
where $G_i$ satisfies \eqref{Giss0}.
Property \eqref{REC-controlDSm} for $m=2$ then follows by estimating $S^i_{2,2}$ and $S^i_{2,3}$ along the lines of proof of 
\eqref{estS22i}-\eqref{estS23i},
replacing the sets $\tilde E^1$, $\tilde E^2$ with 
$\tilde E^1=(0, 2\sigma e^{s/2})\setminus (\frac{\sigma}{2} e^{s_0/2},2\sigma e^{s_0/2})$, 
$\tilde E^2=(0, 2\sigma e^{s/2})\cap(\frac{\sigma}{2} e^{s_0/2},2\sigma e^{s_0/2})$
and using Lemma~\ref{REC-LemestimS2s0} instead of Lemma~\ref{LemestimS2s0}.

\smallskip
We next consider the case $m=3$. We use the splitting
$$\begin{aligned}
D^iS_3(y,s)
&=\int_{s_0}^s D^ie^{-(s-\tau)\mathcal{L}}\tilde F(\tilde v_y(\tau)) \,d\tau
=\int_{s_0}^{s-1}+\int_{s-1}^s \equiv D^iS_{3,1}+D^iS_{3,2}.
\end{aligned}$$
By \eqref{REC-controlS3short0} in Lemma~\ref{REC-LemS3short}, we have
\be{REC-controlDSm3}
|D^iS_{3,2}(y,s)|\le C\nu e^{-\lambda s} y^{2\ell-i},\quad\hbox{ for $y=y_2(s)$ and $s\in [s_0+1,s_1]$}.
\ee
Let us thus estimate $D^iS_{3,1}$.
Similar to \eqref{DiS3}, we have
$$\begin{aligned}
|D^iS_{3,1}(y_2,s)|
&\le \int_{s_0}^{s-1}\int_0^\infty |D^i_yG(s-\tau,y_2,\xi)|\tilde F(\tilde v_y(\xi,\tau)) \xi^\alpha d\xi d\tau \\
&= \int_{s_0}^{s-1}\int_0^{\sigma e^{\tau/2}}+ 
\int_{s_0}^{s-1}\int_{\sigma e^{\tau/2}}^\infty\equiv  S^{i,2}_{3,1}+S^{i,3}_{3,1},
\end{aligned}$$
where $D^i_yG$ satisfies \eqref{DiyG2}.
Arguing along the lines of proof of \eqref{Si231}-\eqref{Si331}, 
using \eqref{REC-controlFvybasic} instead of \eqref{controlFvybasic},
we obtain $S^{i,2}_{3,1}+S^{i,2}_{3,1} \le \nu e^{-\lambda s}y_2^{2\ell-i}$.
This along with \eqref{REC-controlDSm3} yields \eqref{REC-controlDSm} for $m=3$.

\medskip
Property \eqref{REC-controlDSm2} is then a direct consequence of \eqref{REC-defS123}, \eqref{REC-controlDSm} and \eqref{REC-tildephicv}.
\end{proof}

Finally, $v$ and $v_y$ are estimated in the remaining part $y\in [R_1,y_2(s)]$
by means of the following lemma.

\begin{lem}  \label{REC-LemComplong}
If $s_0\gg 1$ and $s_1\ge s_0+1$ then, for any $d\in \mathcal{U}_{s_0,s_1}$ such that $P(d;s_0,s_1)=0$, we have
\be{REC-EstimComplong}
|D^i(v-e^{-\lambda s}\phi)|\le  C\nu e^{-\lambda s} y^{2\ell-i},\quad\hbox{ for all $s\in [s_0+1,s_1]$, $y\in [R_1,y_2(s)]$
and $i\in\{0,1\}$}.
\ee
\end{lem}

\begin{proof} This follows by combining Remark~\ref{remzm}, applied with $c=-1$, with
Lemma~\ref{REC-LemvyFvy}, \eqref{est-REC-Propwshort} for $s=s_0+1$, \eqref{REC-controlDIm2} for $y=R_1$
and \eqref{REC-controlDSm2}.
\end{proof}

\begin{proof}[Proof of Proposition~\ref{REC-Propwlong}]
Let 
\be{defM0-REC}
M_0=M_0(p,\ell)=\ts\frac12\min\{M_1^{-1},M_2^{-1},\ell^{-1}\},
\ee 
where $M_1, M_2$ are respectively given by \eqref{est-REC-Propwshort} and \eqref{est-REC-Propwlong}.
Since $\nu=M_0\eps$, the proposition is a direct consequence of \eqref{REC-dsmall} and 
Propositions~\ref{REC-LemI2long} and \ref{REC-LemComplong}.
\end{proof}

\section{Application of Braid group to PDE} \label{SecBraid} 

In this section we gather results on braid group with simple proofs, which play crucial role in the proofs of 
Theorems~\ref{mainThm1}(i) and \ref{th:RBC}(i), 
for readers' convenience (see \cite{mat_cm07}, \cite{mizo_tams11}).
We first recall fundamental properties on braid group $ {\bf G} $ 
of three strands.
Denote by $ X, Y $ the generators of $ {\bf  G} $ as in the following figure

\vspace{0.2cm}

 \hskip 5cm
\includegraphics[width=6cm, height=1.8cm]{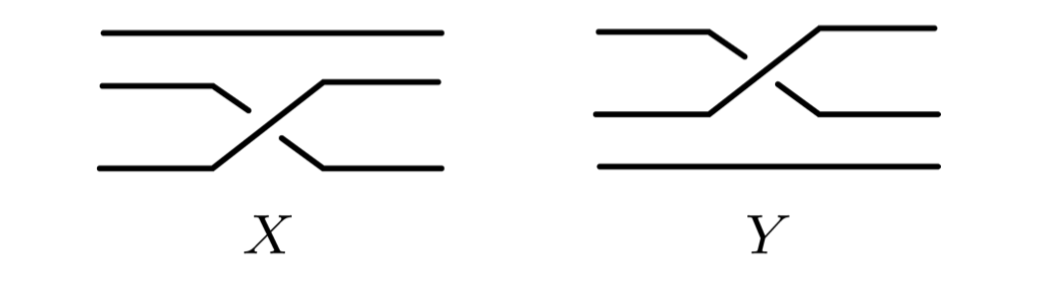}

\noindent and by $ I $ the trivial braid of $ {\bf G} $.

 \hskip 4.1cm
\includegraphics[width=8.5cm, height=2cm]{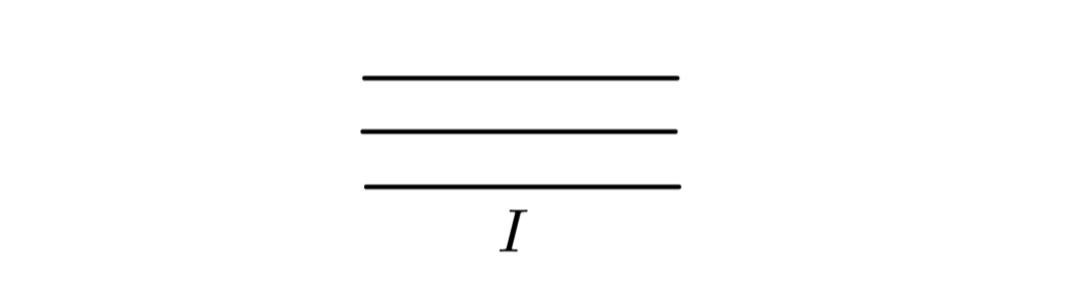}

\noindent Artin's formula 
\begin{equation}
\label{eq:Artin}
X Y X = Y X Y
\end{equation}
is well-known.
Owing to \eqref{eq:Artin}, there holds
\begin{equation}
\label{eq:basic}
\begin{array}{l}
X Y^2 X Y^2 = Y X^2 Y X^2 = X Y X^2 Y X = Y X Y^2 X Y  \\
\;\;\; = X^2 Y X^2 Y = Y^2 X Y^2 X = X Y X Y X Y = Y X Y X Y X,
\end{array}
\end{equation}
and for each positive integer $ k $,
\begin{equation}
\label{eq:XYX-2k}
Y^{2k} X Y = X Y X^{2k} \quad \mbox{ and } \quad X^{2k} Y X = Y X Y^{2k}.
\end{equation}

If $ A \in {\bf G} $ contains neither $ X^{-1} $ nor $ Y^{-1} $, then $ A $ is called a positive braid. 
Denote by $ {\bf G}^+ $ the semigroup of positive braids in $ {\bf G} $.
Let us explain the motivation to deal with positive braids in this paper following \cite{Ghrist-VandenBerg-Vandervorst},  \cite{Ghrist-Vandervorst}.
Let $ v_1, v_2, v_3 $ be solutions of a parabolic equation
\begin{equation}
	\label{eq:v} 
		v_t = \alpha (x) v_{xx} + \beta (x) v_x + f (x, v, v_x) \quad \mbox{ in } (a, b) \times (T_1,T_2)
\end{equation}
which do not intersect for $ (x,t) \in \{a,b\} \times [T_1, T_2] $.
Here $ \alpha, \beta, f $ are smooth and $ \alpha $ is positive for $ x \in [a, b] $ and $ t \in [ T_1, T_2 ] $.

\vspace{0.2cm}
\includegraphics[width=16cm, height=6.5cm]{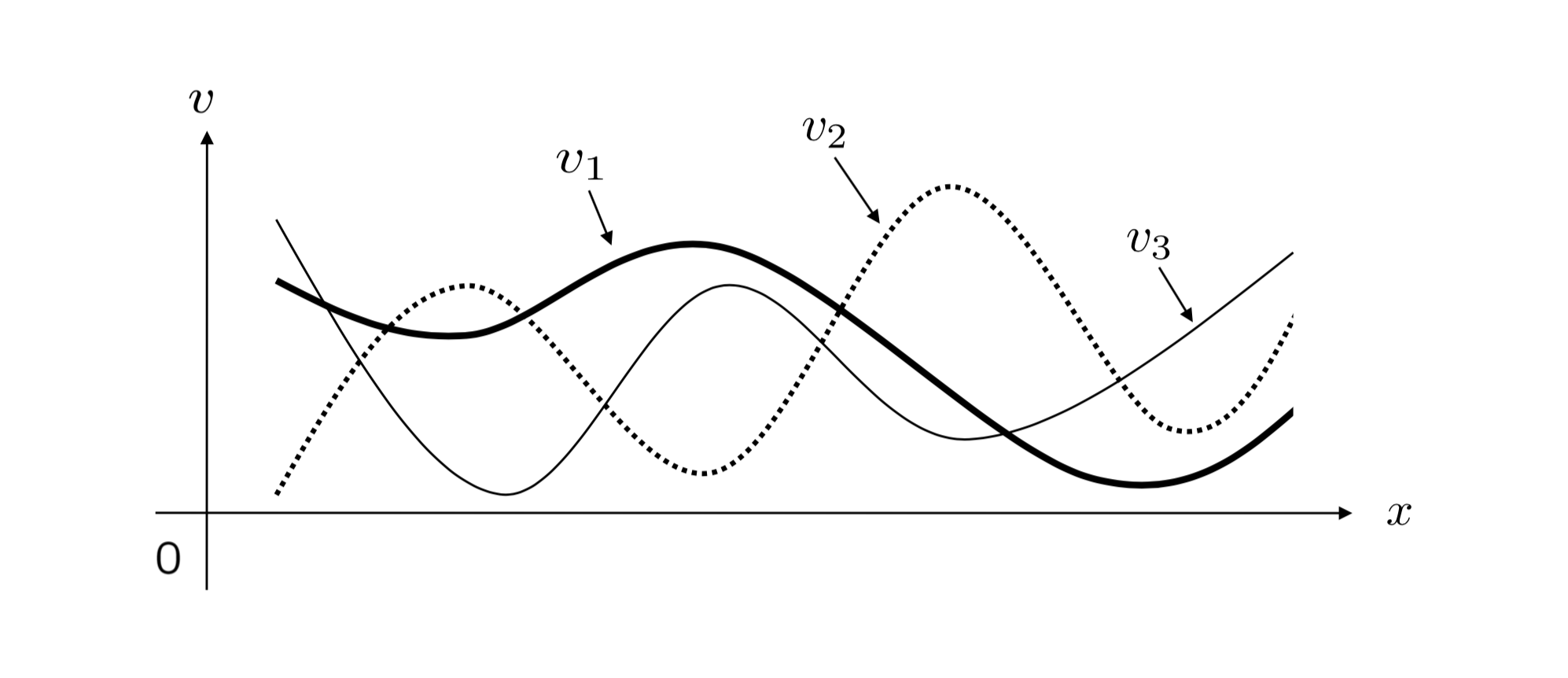}
 \vskip 1mm
\centerline{Fig.~1}

\vskip 5mm

\noindent
Fix $ t \in (T_1, T_2)$ arbitrarily. When
\be{transv-int}
\begin{aligned}
&\hbox{$ v_i (\cdot,t) $ and $ v_j(\cdot,t) $  with $ i \neq j $ ($ i, j = 1, 2, 3 $)} \\
&\hbox{transversally intersect at each of their intersections}
\end{aligned}
\ee
 (see Fig.~1), let us consider these three solutions $ v_1(\cdot,t), v_2(\cdot,t), v_3(\cdot,t) $ in the space 
$ \{ (\partial_x v, x, v) : x \in [a, b]\} $ (see Fig.~2).

\vspace{0.2cm}
\includegraphics[width=16cm, height=6.5cm]{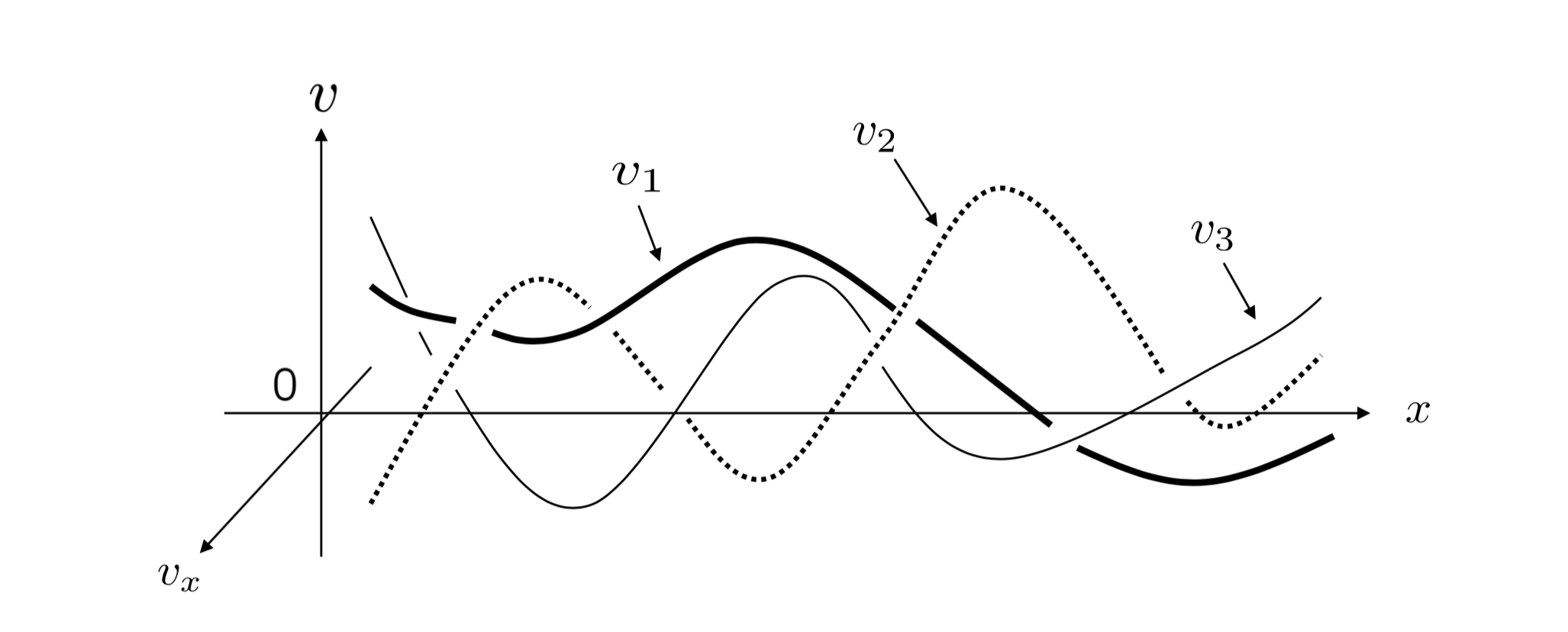}
 \vskip 1mm
\centerline{ Fig.~2}

\vskip 5mm
\vspace{0.3cm} 

\noindent
Then $G(t):=(v_1(\cdot,t), v_2(\cdot,t), v_3(\cdot,t)) $
can be regarded as an element of $ {\bf G}^+ $ like the following figure.

\vspace{0.1cm}
\hskip 1.5cm
\includegraphics[width=12cm, height=4.5cm]{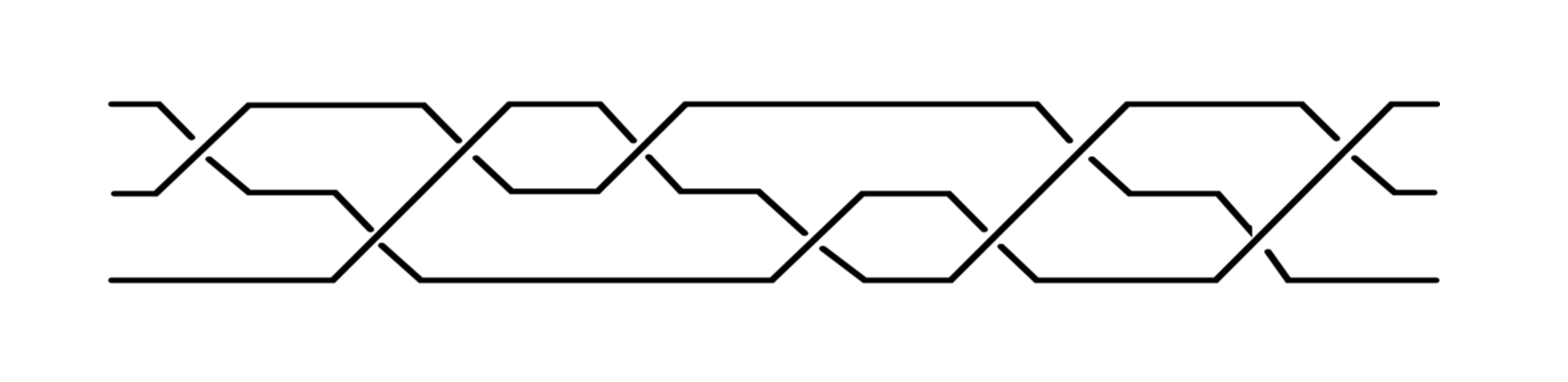}
\vspace{-0.1cm}

For $ A, B \in {\bf G} $, $ A $ is topologically equivalent to $ B $ if and only if $ A $ is modified to $ B $ by applying \eqref{eq:Artin} at most finitely many times. 
The following lemma is easily shown.

\goodbreak

\begin{lem}
\label{lemma:basic}
Let $ X, Y $ be the generators of $ {\bf G} $. 
\begin{description}
\item [{\rm (i)}] For $ A \in {\bf G}^+ $,  there holds 
\[
X^2 Y X^2 Y A = A X^2 Y X^2 Y. 
\]
\item [{\rm (ii)}] For any positive integer $ k $, there holds
\[
(Y X^2 Y)^k X^{2k} = (Y X^2 Y X^2)^k = X^{2k} (Y X^2 Y)^k.
\]
\end{description}
\end{lem}

\begin{proof}
(i) It suffices to show
\[
X^2 Y X^2 Y X = X X^2 Y X^2 Y \quad \mbox{ and } \quad X^2 Y X^2 Y Y = Y X^2 Y X^2 Y. 
\]
Owing to \eqref{eq:Artin}, \eqref{eq:basic}, we have
\begin{eqnarray*}
& & X^2 Y X^2 Y X = X^2 YX XYX = X^2 YX YXY = X^2 XYX XY \\
& & \quad = X^3 YX^2 Y = X X^2 Y X^2 Y
\end{eqnarray*}
and 
\begin{eqnarray*}
& & X^2 Y X^2 Y Y = X XYX XY^2 = X YXY XY^2 = YXY XYXY \\
& & \quad = YX XYX XY = YX^2 YX^2 Y.
\end{eqnarray*}
(ii) For $ k = 1 $ the two equalities are immediate from \eqref{eq:basic}.
Suppose that the first equality holds for $ k $. 
Owing to \eqref{eq:basic} and (i) above, we have 
\begin{eqnarray*}
& & (Y X^2 Y X^2)^{k+1} = Y X^2 Y X^2 (Y X^2 Y X^2)^k = Y X^2 Y X^2 (Y X^2 Y)^k X^{2k} \\
& & \quad = (Y X^2 Y)^k Y X^2 Y X^2 X^{2k} = (Y X^2 Y)^{k+1} X^{2(k+1)}.
\end{eqnarray*}
The induction implies the first equality.
The second equality is similarly shown. 
\end{proof}

In application of braid group theory to parabolic equation, we need a notion corresponding to vanishing intersections between solutions of a parabolic equation, which is parabolic reduction defined by Matano.
Let $ A, B \in {\bf G}^+ $.
It is said that $ B $ is a simple parabolic reduction of $ A $ if there exist $ C, D \in {\bf G}^+ $ such that 
\[
 A = C X^2 D, \; B = C D \quad \mbox{ or } \quad A = C Y^2 D, \; B = C D,
\]
which is denoted by $ A \Rrightarrow_1 B $.
If $ A $ is modified to $ B $ through finitely many simple parabolic reductions, that is, 
there exist $ A_1, A_2, \cdots, A_k \in {\bf G}^+ $ with some positive integer $ k $ such that 
$ A \Rrightarrow_1 A_1 \Rrightarrow_1 A_2 \Rrightarrow_1 \cdots \Rrightarrow_1 A_k \Rrightarrow_1 B $, 
then $ B $ is called a parabolic reduction of $ A $, and it is denoted by $ A \Rrightarrow B $.  
We then have the following reduction principle for parabolic evolution:
\be{reduction-principle}
\hbox{If $t_1<t_2$ and \eqref{transv-int} holds for $t\in\{t_1,t_2\}$, then $G(t_1) \Rrightarrow G(t_2)$.}
\ee

The following result, which plays an essential role in calculations on parabolic reduction, was shown independently in Proposition 5.6 of 
\cite{mat_cm07} and Lemma 3.1 of \cite{mizo_tams11}.
We note that the assertions of Lemma~\ref{lemma:crucial} are not trivial since one cannot multiply $ H^{-1} $.

\begin{lem}
\label{lemma:crucial} \
Let $ A, B, H \in {\bf G}^+ $.
If $ H A \Rrightarrow H B $, then $ A \Rrightarrow B $.
If $ A H \Rrightarrow B H $, then $ A \Rrightarrow B $.
\end{lem}

The following result was proved independently in Lemma 5.11 of \cite{mat_cm07} and Lemma 3.4 of \cite{mizo_tams11}.
We make use of it in determination of all GBU rates in Theorem~\ref{mainThm1}(i) 
 (except for the asymptotic equality of the coefficients in the upper and lower estimates).

\begin{lem}
\label{lemma:braid-even-odd} \
For positive integer $ k $, let 
\[
 \tA_{2k} = (X Y^2 X)^k Y^{2k}, \; \tA_{2k+1} = (X Y^2 X)^k X Y X^{2k+1},
\]
\[ 
\tB_{2k} = X^2 Y^{2k} X Y^{2k} X, \; \tB_{2k+1} = X^2 Y^{2k+1} X^{2k+1} Y.
\]
Then $ \tA_{2k} \not\Rrightarrow \tB_{2k} $ and $ \tA_{2k+1} \not\Rrightarrow \tB_{2k+1} $.
\end{lem}

We need the following for Theorem~\ref{th:RBC}(i) in the RBC case.

\begin{lem}
\label{lemma:braid-even-odd-2} \
For positive integer $ k $, let 
\[
 \hA_{2k} = (Y X^2 Y)^k X^{2k}, \; \hA_{2k+1} = (Y X^2 Y)^k Y X Y^{2k+1},
\] 
\[
\hB_{2k} = Y^2 X^{2k} Y X^{2k} Y, \; \hB_{2k+1} = Y^2 X^{2k+1} Y^{2k+1} X.
\]
Then $ \hA_{2k} \not\Rrightarrow \hB_{2k} $ and $ \hA_{2k+1} \not\Rrightarrow \hB_{2k+1} $.
\end{lem}

Lemma \ref{lemma:braid-even-odd-2} can be obtained similarly to the proof of Lemma \ref{lemma:braid-even-odd} due to \cite{mat_cm07} 
and \cite{mizo_tams11}. 
For readers' convenience, we give a simple proof here based only on algebraic computations like \cite{mat_cm07}, 
though the method in \cite{mizo_tams11} works well also in this case.
We have an auxiliary result to prove Lemma~\ref{lemma:braid-even-odd-2}.

\begin{lem}
\label{lemma:main-braid} \
Let $ A_{2k} = (Y X^2 Y)^k X^{2k+1} $ and $ B_{2k} = Y^2 X^{2k+1} Y^{2k} $ for a positive integer $ k $.
Then $ A_{2k} \not\Rrightarrow B_{2k} $.
\end{lem}

{\bf Proof.} \
Assume for contradiction that $ A_{2k} \Rrightarrow B_{2k} $.
Then we have 
\[
(Y^2 X Y^2 X)^k X = (Y X^2 Y X^2)^k X = A_{2k} \Rrightarrow Y^2 X X^{2k} Y^{2k} 
\]
by \eqref{eq:basic} and Lemma \ref{lemma:basic} (ii).
Owing to this and Lemma \ref{lemma:crucial}, we derive
\[
(Y^2 X Y^2 X)^{k-1} Y^2 X X = Y^2 X (Y^2 X Y^2 X)^{k-1} X \Rrightarrow X^{2k} Y^{2k}
\]
and hence
\[
X^{2(k-1)} (Y X^2 Y)^{k-1} Y^2 X^2 = (Y X^2 Y X^2)^{k-1} Y^2 X^2 = (Y^2 X Y^2 X)^{k-1} Y^2 X^2 \Rrightarrow X^{2k} Y^{2k}.
\]
by \eqref{eq:basic} and Lemma \ref{lemma:basic} (ii).
Owing to Lemma \ref{lemma:crucial}, we get
\[
Y X^2 (Y^2 X^2)^{k-2} Y^3 X^2 = (Y X^2 Y)^{k-1} Y^2 X^2 \Rrightarrow X^2 Y^{2k}.
\]
Multiplying both sides by $ Y X^2 $ from right yields 
\begin{eqnarray*}
& & X^2 Y X^2 Y Y X^2 (Y^2 X^2)^{k-2} Y^2 = Y X^2 (Y^2 X^2)^{k-2} Y^2 X^2 Y X^2 Y \\
& & \quad = Y X^2 (Y^2 X^2)^{k-2} Y^2 Y X^2 Y X^2 = Y X^2 (Y^2 X^2)^{k-2} Y^3 X^2 Y X^2 \\
& & \quad \Rrightarrow X^2 Y^{2k} Y X^2 = X^2 Y  Y^{2k} X^2
\end{eqnarray*}
by \eqref{eq:basic} and Lemma \ref{lemma:basic}(i).
From Lemma \ref{lemma:crucial}, we obtain $ X^2 Y^2 X^2 (Y^2 X^2)^{k-2} Y^2 \Rrightarrow Y^{2k} X^2 $.
The left-hand side consists of product of $ X^2 $ and $ Y^2 $ and hence its simple parabolic reduction is carried out by dropping 
$ X^2 $ or $ Y^2 $ with their orders kept.
 Both sides contain the same number of $ Y^2 $s, which implies that the parabolic reduction does not lose any $ Y^2 $.
Therefore the parabolic reduction is impossible since  last term on the left-hand side is not $ X^2 $.
The contradiction completes the proof.
\hfill $ \Box $
\vspace{0.5cm}

{\bf Proof of Lemma \ref{lemma:braid-even-odd-2}.} 
Assume for contradiction that $ \hA_{2k} \Rrightarrow \hB_{2k} $.
Multiplying this by $ X $ from right and using \eqref{eq:XYX-2k} yields
\[
(Y X^2 Y)^k X^{2k+1} \Rrightarrow Y^2 X^{2k} Y X^{2k} Y X = Y^2 X^{2k} Y^2 X Y^{2k} 
\Rrightarrow Y^2 X^{2k+1} Y^{2k},
\]
i.e., $ A_{2k} \Rrightarrow B_{2k} $.
Since $ A_{2k} \not\Rrightarrow B_{2k} $ by Lemma \ref{lemma:main-braid}, 
the contradiction implies the first assertion.

It follows from \eqref{eq:Artin}, \eqref{eq:XYX-2k} that
\begin{eqnarray*}
& & \hA_{2k+1} = (Y X^2 Y)^k Y X Y^{2k} Y = (Y X^2 Y)^k X^{2k} Y X Y = (Y X^2 Y)^k X^{2k} X Y X \\
& & \quad  = (Y X^2 Y)^k X^{2k+1} Y X.
\end{eqnarray*}
Assume for contradiction that $ \hA_{2k+1} \Rrightarrow \hB_{2k+1} $.
Then we have $$ (Y X^2 Y)^k X^{2k+1} Y X \Rrightarrow Y^2 X^{2k+1} Y^{2k+1} X.$$
Owing to Lemma \ref{lemma:crucial}, we have $ (Y X^2 Y)^k X^{2k+1} \Rrightarrow Y^2 X^{2k+1} Y^{2k} $, i.e., $ A_{2k} \Rrightarrow B_{2k} $,
Since $ A_{2k} \not\Rrightarrow B_{2k} $ by Lemma \ref{lemma:main-braid}, the contradiction implies the second assertion.
\hfill $ \Box $

\section{Complete classification: proof of Theorems~\ref{mainThm1}{\rm(i)}, 
 \ref{mainThm1stab}, \ref{th:RBC}{\rm(i)}
 and \ref{mainThmRBCstab}} \label{SecProofBraid1}

In \cite{mizo_ma07}, to investigate all type II blowup rates in Fujita equation, 
suitable three solutions of the equation transformed in backward self-similar variables, 
 analogous to \eqref{eqw}, were introduced and partial result was obtained there.
In \cite{mizo_tams11}, all type II blowup rates except the coefficients were determined applying braid group theory 
to the three solutions together with behavior of solutions in the transformed form.
In \cite{mat_cm07}, the three solutions due to \cite{mizo_ma07}, \cite{mizo_tams11} were 
converted to the corresponding solutions of the original equation and then the braid group theory was applied to them.
Although properties of special solutions are clearer in transformed equation, we use the three solutions converted to the original equation 
\eqref{equ} since analytic evaluation of solutions to apply the braid group theory is simpler in the original equation than in 
the transformed one in the viscous Hamilton-Jacobi equation. 
Moreover, we give simpler proof 
without several steps in \cite{mat_cm07} (and \cite{GMW}), \cite{mizo_tams11} 
owing to some properties peculiar to the viscous Hamilton-Jacobi equation.

\subsection{Proof of Theorems \ref{mainThm1}(i) and \ref{mainThm1stab}} 
Let $0<R\le\infty$ and let 
$ u $ be a viscosity solution of \eqref{equ} with $u_0\in \mathcal{W}$ undergoing GBU at $ (x,t) = (0,T ) $.
The proof of Theorem \ref{mainThm1}(i) will use the following rescaled version of $u$:
\begin{equation}
\label{eq:u_a-def}
u_a (x,t) := 
a^k u( a^{-1/2} x, T + a^{-1}(t-T)) \quad \mbox{ in } (0, a^{1/2}R) \times ((1-a)T, T).
\end{equation}
 The following lemma guarantees that the rescaled solution $u_a$ is suitably close to the singular steady state $U$
for large $a$.

\begin{lem} 
\label{lemma:u_a} \
There exist constants $\eta\in(0,T)$, $D_0\in(0,1)$, $C, C_1>0$ and $a_0>1$ depending on $u$ such that,
for each $m\in(0,1/4]$ and $D\in(0,D_0]$, 
$u_a$ enjoys the following properties. For all $a\ge a_0$,
\be{eq:u_a-upper2} 
 - C\Bigl\{1+a^{ -2q }(T-t) \Bigr\} a^{- m} 
 \le u_a(x,t) - U(x) \le Ca^{-m} \qquad \mbox{in } (0, a^q)\times(T -a \eta, T), 
\ee
\be{eq:u_a-upper3} 
-  CD^{ \frac{p}{p-1} }  \Bigl\{1+a^{ - 2q }(T-t) \Bigr\} U(x) 
\le u_a(x,t) - U(x) \le CD^{ \frac{p}{p-1} } U(x) \quad \mbox{ in } 
(a^q, a^{1/2} D)\times(T -a \eta, T),
\ee
where $q= \frac{1}{2} \{ \frac{p}{2(p-1)} - m \}\in (0,1/2)$.
 For all ${\eps\in(0,1)}$ and $t_1<T$,  
\be{uU0x}
\sup_{(x,t)\in [\eps, a^{1/2}D]\times[t_1,T)} |u_{a,x}(x,t)-U_x(x)|\to 0,\quad\hbox{as } a\to\infty,
\ee
and there exists $ a_1 = a_1 (t_1, \eps, D) > 1 $ such that, for all $ a \ge a_1 $,  
\be{eq:u_a-lower4} 
(u_a)_x (x,t) \ge  C_1\eps^{-\beta} \qquad \mbox{in } [0, \eps]\times[t_1, T).
\ee
\end{lem}

\begin{proof}
 Let $\eta, x_0, M$ be given by Proposition~\ref{prop:PS} and let $0<D\le D_0:=x_0$.
We first derive the upper estimates of $ u_a $.
It follows from the upper part of \eqref{eq:u_x-lowerupper} that 
$$u_x (x,t) \le U^\prime (x) + M x \quad \mbox{ in } (0,D) \times [T-\eta, T),$$
hence
\be{eq:u_a-x-upper0} 
(u_a)_x (x,t) \le U^\prime (x) + M a^{ - \frac{p}{2(p-1)} } x \quad \mbox{ in } (0,a^{1/2} D) \times [T-a\eta, T).
\ee
By integration, we get
\be{eq:u_a-x-upper} 
u_a (x,t) \le U(x) + \frac{M}{2} a^{ - \frac{p}{2(p-1)} } x^2 \quad \mbox{ in } (0, a^{1/2} D) \times [T-a\eta, T),
\end{equation}
hence in particular the upper part of \eqref{eq:u_a-upper2}.
On the other hand, \eqref{eq:u_a-x-upper} also yields the upper part of \eqref{eq:u_a-upper3} 
since
$$x^2=x^{\frac{p}{p-1}}\ts\frac{U(x)}{c_p}\le a^{\frac{p}{2(p-1)}}D^{\frac{p}{p-1}} \ts\frac{U(x)}{c_p}
\quad\hbox{ for $x\le a^{1/2}D$}.$$

We next derive the lower estimates of $ u_a $. 
It follows from \eqref{eq:BU-rate-lower} and the lower part of \eqref{eq:u_x-lowerupper} that 
\begin{equation}
\label{eq:u_x-lower2} 
u_x(x,t)  \ge \left\{ M_0^{1-p} (T-t)^{\frac{p-1}{p-2} }+ (p-1)x \right\}^{ - \frac{1}{p-1} } - M x 
\quad \mbox{ in } (0,D) \times (T-\eta,T).
\end{equation}
Integrating \eqref{eq:u_x-lower2} in $ (0,x) $ yields
\begin{eqnarray*}
u(x,t) & \ge & \int_0^x \left[ \{ M_0^{1-p} (T-t)^{ \frac{p-1}{p-2} } + (p-1)\xi \}^{- \frac{1}{p-1} } - M \xi \right] d \xi \\
& = & \frac{1}{p-2} \{ M_0^{1-p} (T-t)^{ \frac{p-1}{p-2} } + (p-1) x \}^{ \frac{p-2}{p-1} } 
- \frac{1}{p-2} \{ M_0^{1-p} (T-t)^{ \frac{p-1}{p-2} } \}^{ \frac{p-2}{p-1} } - \frac{M}{2} x^2 \\
& \ge & U(x) - C (T-t+x^2)  \qquad \mbox{in }  (0, D) \times [T-\eta, T).
\end{eqnarray*} 
Here and in what follows,  $ C, C_1 $ are constants  varying from line to line, which depend only on $ u $.
This implies that
\begin{equation}
\label{eq:u_a-lower} 
u_a(x,t) \ge U(x) - Ca^{ - \frac{p}{2(p-1)} } \bigl( (T-t)+x^2 \bigr)\quad \mbox{ in } 
(0, a^{1/2} D] \times [T-a\eta, T),
\end{equation}
hence in particular the lower part of \eqref{eq:u_a-upper2}.
When $ a^{ \frac{1}{2} \{ \frac{p}{2(p-1)} - m \} } < x < a^{1/2} D $, hence
$x^{-2} \le a^{m- \frac{p}{2(p-1)}}$, \eqref{eq:u_a-lower} yields
$$\begin{aligned}
u_a (x,t) - U(x) 
& \ge - CU(x) a^{ - \frac{p}{2(p-1)}} \bigl( x^{-2}(T-t) + 1\bigr)  x^{2 - \frac{p-2}{p-1} }\\
& \ge - CU(x) \left\{ a^{ m- \frac{p}{2(p-1)}  }  (T-t) + 1\right\}
a^{ - \frac{p}{2(p-1)} } x^{\frac{p}{p-1} } 
 \ge - CU(x) \left\{ a^{-2q}  (T-t) + 1\right\} D^{\frac{p}{p-1} },
\end{aligned}$$
that is, the lower part of \eqref{eq:u_a-upper3}.

Let us finally show \eqref{uU0x} and \eqref{eq:u_a-lower4}. 
 Fix any $\eps\in(0,1)$ and $t_1<T$.
The upper part of \eqref{uU0x} follows from \eqref{eq:u_a-x-upper0}.
From \eqref{eq:u_x-lower2}, we deduce that, for  all $(x,t)\in (a^{-1/2} \eps,D) \times (T-\eta, T)$, 
$$\begin{aligned}
	U'(x)-u_x(x,t)
	&\le \left\{(p-1)x \right\}^{ -\frac{1}{p-1} }-\left\{ M_0^{1-p} (T-t)^{\frac{p-1}{p-2} }+ (p-1)x \right\}^{ - \frac{1}{p-1} }
	+Mx\\
	&\le  C_p(a^{-1/2} \eps)^{ - \frac{p}{p-1} }M_0^{1-p} (T-t)^{\frac{p-1}{p-2} }+Mx
\end{aligned}$$
 with $C_p=(p-1)^{ - \frac{p}{p-1} -1}$, where we used the mean value theorem.
Since $U'(x)\equiv a^{-\frac{1}{2(p-1)}} U'( a^{-\frac12} x)$, it follows that,
for all $(x,t)\in (\eps,a^{1/2}D) \times (T-a\eta, T)$, 
$$\begin{aligned}
	U'(x)-u_{a,x}(x,t)
	&\le a^{-\frac{1}{2(p-1)}}
	\left\{  C_p(a^{-1/2} \eps)^{ - \frac{p}{p-1} } M_0^{1-p} (a^{-1}(T-t))^{\frac{p-1}{p-2} }+Ma^{-1/2}x \right\}\\
	&\le 
	\left\{  C_p \eps^{ - \frac{p}{p-1} }M_0^{1-p}a^{-\frac{p}{2(p-2)} }(T-t)^{\frac{p-1}{p-2} }
	+Ma^{ - \frac{p}{2(p-1)} }x \right\},
\end{aligned}$$
which ensures the lower part of \eqref{uU0x}.

By \eqref{eq:u_x-lower2}, for all $(x,t)\in [ 0,a^{1/2} D) \times [T-a\eta,T)$, we have
$$(u_a)_x(x,t) \ge a^{ - \frac{1}{2(p-1)} } \left[ \left\{ C(a^{-1} (T-t))^{ \frac{p-1}{p-2} } + (p-1)a^{-1/2} x \right\}^{ - \frac{1}{p-1} } 
- M a^{-1/2} x \right].$$ 
We may thus choose $a_1 = a_1 (t_1, \eps ,D) >1 $ such that for all $ a \ge a_1$,
we have $T-t_1<a\eta$, $\eps<a^{1/2} D$ 
and for all $(x,t)\in [0, \eps] \times [t_1, T) $,
\begin{eqnarray*}
(u_a)_x (x,t) 
& \ge & a^{ - \frac{1}{2(p-1)} } \Bigl[ \Bigl\{ Ca^{ - \frac{p-1}{p-2} } 
 (T-t_1)^{ \frac{p-1}{p-2} } + (p-1)a^{-1/2} \eps \Bigr\}^{ - \frac{1}{p-1} }- M a^{-1/2} \eps \Bigr] \\
& \ge & C_1 a^{ - \frac{1}{2(p-1)} } \bigl( a^{-1/2}  \eps \bigr)^{- \frac{1}{p-1}}= C_1  \eps^{- \frac{1}{p-1}}.
\qedhere
\end{eqnarray*}
\end{proof}

 The next lemma rules out any possibility of oscillations of the vanishing intersections.

\begin{lem}
	\label{lemma:zero-curve}
Let $n$ be the number of vanishing intersections between $ u(\cdot, t) $ and $U$ 
at $ (x,t) = (0,T) $
(cf.~\eqref{defvanishnumber}). 
 Denoting these intersections by $ 0 < x_1(t) < x_2 (t) < \cdots < x_n (t) $, we have
\be{limitxn}
 \lim_{t \to T_-} x_n (t) = 0.
 \ee
 \end{lem}

\begin{proof}
	By Proposition~\ref{ZeroNumberConst},  there exist $r\in (0,R]$, $t_0<T$ and an integer $m\ge 1$ such that
	\be{mzeros1}
	\hbox{for all $t\in(t_0,T)$, \ $u(\cdot,t)-U$ has exactly $m$ (nondegenerate) zeros on $(0,r)$}
	\ee
	 and
	\be{mzeros2}
	  u(r,t) - U(r)\ne 0,\quad t_0\le t\le T.
	\ee
	Assume for contradiction that \eqref{limitxn} not hold, i.e.,  $R_0:=\limsup_{t\to T} x_n(t)>0$.
    Then $ u(x,T) = U(x) $ for all $ x \in (0, R_0] $ and $R_0<r$.

	Take an integer $ q > m$.
	Owing to Theorem~\ref{prop:special} (after a time shift), there exist $t_1\in(t_0,T)$ and a solution $ v $ of \eqref{equ} on $(0,\infty)\times(t_1,T)$ undergoing GBU at $ t = T $, for which 
	\be{mzeros3}
	\hbox{for all $t\in(t_1,T)$, \ $v(\cdot,t)-U$ has exactly $q$ (nondegenerate) zeros on $(0,\infty),$}
	\ee
		\be{mzeros4}
		X_q(t) \le C (T-t)^{1/2},\quad t_1<t<T
		\ee
    with some $ C > 0 $, where $ 0 < X_1(t) < X_2(t) < \cdots < X_q(t) $ are the zeros of $ v(\cdot, t) - U$.	
	For $ a > 1 $, let
	\[
	v_a (x,t) := 
	a^k v( a^{-1/2} x, T + a^{-1}(t-T)) \quad \mbox{ in } (0, a^{1/2} R) \times ((1-a)T, T).
	\]
	 By \eqref{mzeros1}-\eqref{mzeros2} and Lemma~\ref{lemma:u_a} applied to $v$, choosing $t_2\in(t_1,T)$ close enough to $T$ and then $ a \gg 1 $, we have
	\be{v_a-u0}
	z( v_a(\cdot,t_2) - u(\cdot, t_2): [0,r] ) = m
	\quad\hbox{ and }\quad v_a(r,t) - u(r,t)\ne 0,\quad t_2\le t\le T.	
	\ee
			For $ \lambda < 1 $, let
	\[
	\tu (x,t) := 
	\lambda^k u ( \lambda^{-1/2} x, t_2 + \lambda^{-1} (t - t_2) ) 
	\quad \mbox{ in } (0, \lambda^{1/2} R) \times ( t_2, \tT )
	\]
	with $ \tT : = t_2 + \lambda(T- t_2) <T$.
	By \eqref{v_a-u0}, we may find $ \lambda < 1 $ close enough to $ 1 $ so that
	\[
	z( v_a (\cdot, t_2) - \tu(\cdot, t_2): [0,\lambda^{1/2} r] ) = m 
	 \quad\hbox{ and }\quad v_a(r,t) - \tu(r,t)\ne 0,\quad t_2\le t\le \tT
	\]
	 hence, by Proposition~\ref{propZeroNumber0},
	\be{v_a-u}
	z( v_a(\cdot,t) - \tu(\cdot, t): [0,r] )  \le m, \quad t \in [t_2, T).
	\ee
	On the other hand, using \eqref{mzeros3}-\eqref{mzeros4} and the fact that $ \tu(x,\tT) = U(x) $ for $ x \in (0, \lambda^{1/2}R_0] $,
	 we easily see that, for $ t\in [t_2, \tT)$ close enough to $ \tT $, 
	\[
	z( v_a (\cdot,t) - \tu(\cdot, t): [0,\lambda^{1/2} r] ) = q.
	\]
	By the choice $ q> m $, we reach a contradiction.
\end{proof}	

\begin{proof}[Proof of Theorem \ref{mainThm1}{\rm(i)}] 
We shall show that
\be{liminfsupGBU}
0<\liminf_{t\to T} \, (T-t)^{\frac{n}{p-2}} u_x(0,t)\le
\limsup_{t\to T} \, (T-t)^{\frac{n}{p-2}} u_x(0,t)<\infty.
\ee
This combined with non-oscillation Lemma~\ref{LemNonOsc} implies \eqref{rateell}.
The space-time profile of $u_x$ in \eqref{asympyux1} is then a consequence of \eqref{eq:u_x-lowerupper}
and we get that of $u$ by integration.

\smallskip

 {\bf Step 1.} {\it Preparations based on PDE.} 
 By Lemma~\ref{lemma:zero-curve}, we have $ \lim_{t \to T_-} x_n (t) = 0 $.  
Therefore, for any $ 0 < D \ll 1 $ there exist $ t_0 = t_0(D)<T$ and 
$\delta _0= \delta_0 (D)>0$ 
such that
\begin{equation}
\label{eq:zeros} 
  z\bigl( u(\cdot,t) - U:  (0, D] \bigr) = n
 \quad \mbox{ and } \quad |u(D,t) - U(D)| \ge \delta_0 U(D), \quad t\in [t_0, T). 
\end{equation}
Let $ v $ be a special solution given in Theorem \ref{prop:special} with $ \ell=n $ and $\Omega=(0,\infty)$.
By a time shift we may assume that $v$ undergoes GBU at $(x,t)=(0,T)$.

Owing to Theorem \ref{prop:special}, Lemma \ref{lemma:u_a} and \eqref{eq:zeros},  
 we may choose $0<D\ll 1$ and then 
$a \gg 1 $, $ t_1 < T $, $ \delta_1 \in (0,1) $,  such that 
\be{eqndescr1}
|U(a^{1/2} D) -  v (a^{1/2} D, t)| > \delta_1 U(a^{1/2} D) > |U(a^{1/2} D) -  u_a (a^{1/2} D, t) | >0, \quad t\in [t_1, T), 
\ee
\be{eqndescr2b}
 z\bigl( u_a(\cdot,t) - U:  (0, a^{1/2} D] \bigr) = z\bigl( v( \cdot, t)-U:  (0, a^{1/2} D]\bigr)= n, \quad t\in [t_1, T),
\ee
\be{eqndescr2}
z\bigl( u_a( \cdot,t_1) - v( \cdot, t_1):  (0, a^{1/2} D]\bigr)= n,
\ee
\be{eqndescr3}
  v (x,  t_1) < u_a (x, t_1), \quad 0  < x \ll 1.
\ee

 {\bf Step 2.} {\it Choice of suitable solutions and braid interpretation.}
We now prove the first inequality in \eqref{liminfsupGBU}. 
Assume for contradiction that the first inequality in \eqref{liminfsupGBU} does not hold.
Then there exists $ t_2 \in (t_1,  T) $ 
 such that $ v(t) - u_a(t) $ loses one zero (or odd number of zeros) at $ x = 0 $   
and some $ \hat t_2< t_2 $ close enough $ t_2 $, and 
\begin{equation}
\label{eq:t_2}
 v (x, t_2) > u_a (x, t_2), \quad 0  < x \ll 1. 
\end{equation}
For $ 0 < \lambda < 1 $, let 
\[
\tu_a (x,t) := 
 \lambda^k u_a ( \lambda^{-1/2} x, t_1 + \lambda^{-1} (t - t_1) ) 
\quad \mbox{ in } (0, \lambda^{1/2} a^{1/2} R) \times ( t_1, \tT )
\]
with $ \tT : = t_1 + \lambda(T- t_1) $. 
For $ 0 < \lambda < 1 $ close enough to $ 1 $, \eqref{eqndescr1}-\eqref{eq:t_2} 
 hold true with $ u_a $ replaced by $ \tu_a $. 
 We shall denote by $\tilde x_1(t)<\dots<\tilde x_n(t)$ and $X_1(t)<\dots<X_n(t)$,
the zeros of $ \tu_a(\cdot, t)-U $ and 
of $ v(\cdot, t)-U $ in $ (0, a^{1/2} D] $, respectively.

Take $ 0 < \rho_1 \ll 1 $. 
For $ 0 < \rho \le \rho_1 $, the situation of $ U, \tu_a(\cdot, t_1), v(\cdot, t_1) $ in $ [\rho, a^{1/2} D] $ is represented by $ \tA_n $, namely, topologically equivalent to $ \tA_n $. In order to explain more, lifting the solutions at $ t = t_1 $ to three dimensional space as stated in 
Section~\ref{SecBraid} , the situation of $ U, \tu_a(\cdot, t_1), v(\cdot, t_1) $ 
 in $ [\rho, a^{1/2} D] $ means that $ \tu_a(\cdot, t_1) $ coils $ U$ 
 closely enough and they are inside a big coil $ v(\cdot, t_1) $ (see Fig.~3 and recall Lemma \ref{lemma:u_a}).
	Shifting the big coil $ v (\cdot, t_1) $ to the left and the 
	 small coil $\tu_a(\cdot, t_1) $ to the right, which are topologically equivalent deformations, the resulting braid is $ \tA_n $ (see Fig.~4). 

 \vskip 5mm
 \hskip 1.3cm
\includegraphics[width=13cm, height=4cm]{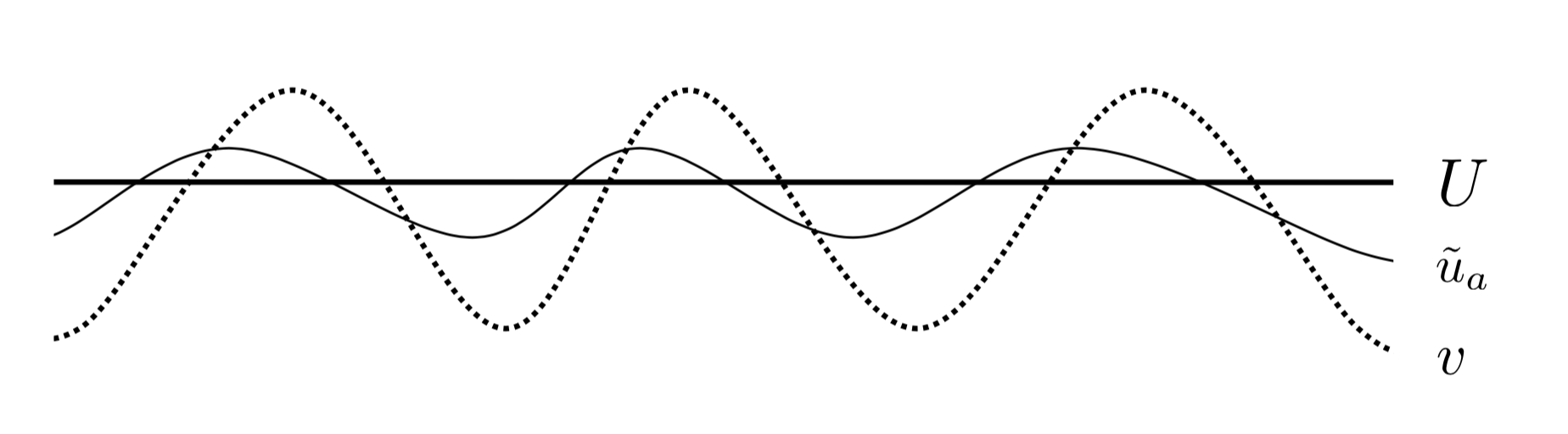}
 \vskip 1mm
   \centerline{Fig.~3}
 
  \vskip 5mm
  
    \hskip 1.5cm
\includegraphics[width=13cm, height=3.5cm]{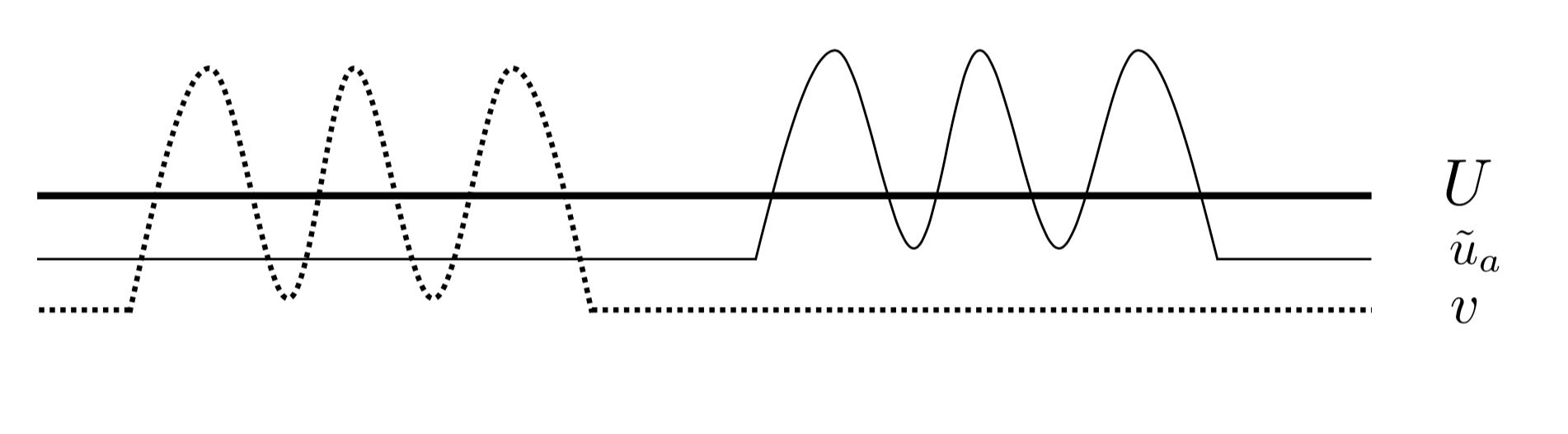}

 \vskip 1mm
  \centerline{Fig.~4}
 \vskip 5mm

 {\bf Step 3.} {\it Time evolution of the solutions in terms of braid.}
Since $ \tu_a $ undergoes GBU at  $ (x,t) = (0,\tT) $ with $ \tT < T $ 
 and $\lim_{t\to \tilde T}\tilde x_n(t)=0$,
there exists $ t_3 \in (t_2, \tT) $ 
such that $ v (\cdot, \hat t) - \tu_a (\cdot,\hat t) $ loses one zero (or odd number of zeros) at $ x = 0 $ for some 
$\hat t<t_3$ close to $ t_3 $, and  
$$ v (x, t_3) 
< \tu_a (x, t_3),\quad 0  < x \ll 1$$
and
\be{eqndescr4}
 \tilde x_n(t_3)<X_1(t_3).
\ee
Choose $ 0 < \rho_2 \le \rho_1 $.
The situation of $ U, \tu_a(\cdot,t_3), v (\cdot, t_3)$ in $ [\rho_2, a^{1/2} D] $ is translated into $\tB_n$  (see Fig.~6). 
Indeed, by \eqref{eqndescr1} and \eqref{eqndescr4}, $ \tu_a(\cdot, t_3) $ and $ v(\cdot, t_3) $ have at least two intersections 
 $x'\in(\tilde x_n(t_3),X_1(t_3))$ and $x'' \in(X_n(t_3),a^{1/2} D)$.
There may be more intersections between $ \tu_a(\cdot, t_3) $ and $ v(\cdot, t_3) $, whose existence are unknown (see Fig.~5).
We delete such uncertain intersections by parabolic reduction.
The first elements $ X^2 $ in $\tB_n$ correspond to vanishing intersections between $\tu_a (\cdot, t) $ and $ v(\cdot,t) $ at 
 $ x = 0 $ just before $ t = t_2, t_3 $.
 They are invisible in real figure at $ t = t_3 $ and regarded as hidden\footnote{In appendix we give an alternative proof of Step~3 
for readers who are not familiar with application of braid group theory, where explicit intersections are considered instead of ``hidden'' intersections}
 beyond $ x = \rho_2 $.
They give no effect to other part, but reflect the hypothesis for contradiction. 
The resulting braid is $ \tB_n $ (see Fig.~6).

 \vskip 5mm
 \hskip 1.5cm
\includegraphics[width=13cm, height=3.5cm]{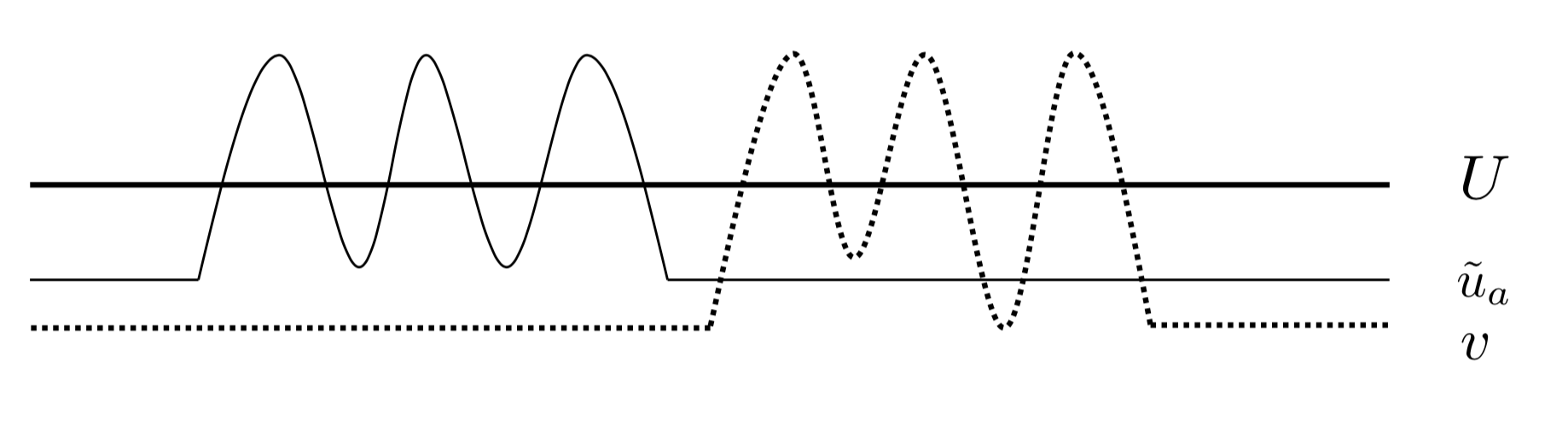}
   \vskip 1mm
  
  \centerline{Fig.~5}
  
   \vskip 5mm
    \hskip 1.5cm
\includegraphics[width=13cm, height=3.5cm]{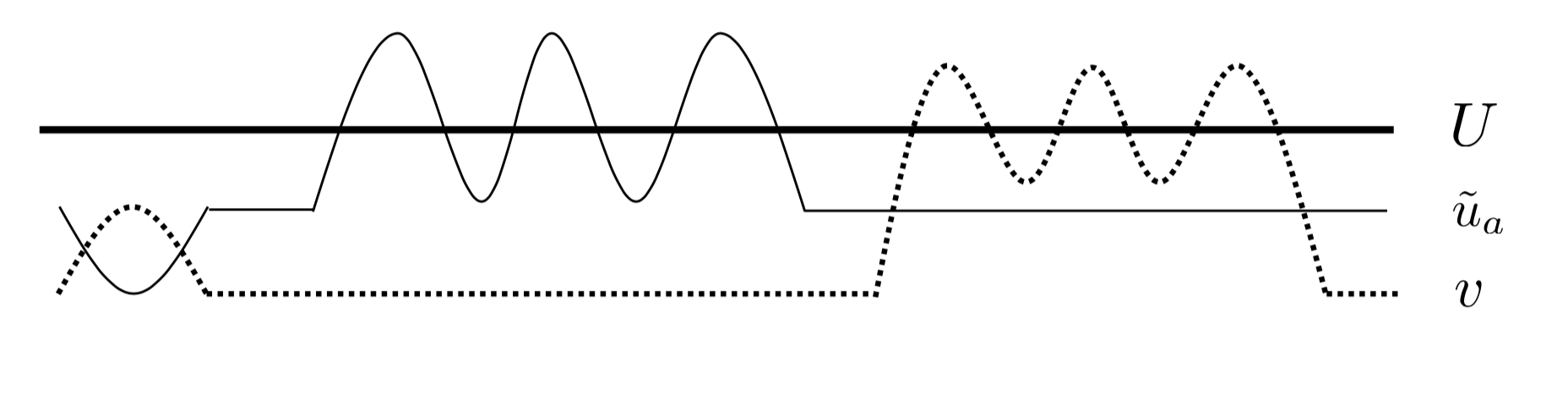}
 \vskip 1mm
  \centerline{Fig.~6}
   \vskip 5mm   
 \noindent Owing to reduction principle for parabolic evolution (cf.~\eqref{reduction-principle}),
   the process from $ t = t_1 $ to $ t = t_3 $ implies that $ \tA_n \Rrightarrow \tB_n $.
But we have $ \tA_n \not\Rrightarrow \tB_n $ by Lemma \ref{lemma:braid-even-odd}. 
This contradiction implies the first inequality in \eqref{liminfsupGBU}.

\smallskip

 {\bf Step 4.} {\it Proof of last inequality in \eqref{liminfsupGBU}.}
In order to prove this, we notice that all zeros of $ v_a (\cdot,t) - U $ (and $ \tv_a(\cdot, t) - U $) locate 
in $ (0, C (T-t)^{1/2} )$ for $ t \in [ t_0, T) $ with some $ C > 0 $ for $ a \gg 1 $ by Theorem \ref{prop:special}. 
If we choose $ t_0 < T $ such that $ C (T-t_0)^{1/2} < D $, then
it suffices to take the same way as above with $ v, u_a $ and the spatial interval $ [ \rho_2, a^{1/2} D] $ replaced 
by $ u, v_a $ and $ [\rho_2, D] $, respectively, where $ v_a $ is defined in \eqref{eq:u_a-def} with $ u $ 
replaced by $ v $.
\end{proof}

\begin{rem} \label{rem_braid}
	
	(i) Whereas the way to apply braid group theory to parabolic PDE is rather simple, 
	a naive choice of three solutions is not successful. It is important how to choose appropriate three solutions.
	In fact, if one applies the braid group theory to $ U, u_a, v$ instead of $ U, \tu_a, v $ above, then there is no contradiction.
	It is crucial to make the tricky choice $ \tu_a $ there.

	(ii) In \cite{mat_cm07}, 
	while the proof of the upper estimate of blow-up rate was given, the  
	proof of the lower estimate was omitted, just stating that one can exchange 
	the roles of the solution under consideration and
	of the special one in the case of lower estimate. 
	As seen above, we must take an interval $ [0, a^{1/2}D] $ with $ a \gg 1 $ in the proof of the lower estimate, whereas a different interval $ [0, D] $ is used in the proof of the upper one.
	In order for the argument of \cite{mat_cm07} to work also for the lower estimate, one needs additional information, 
	for example, like that all zeros of $ u (\cdot,t) - U$  
	approaching $ x = 0 $ as $ t \to T_- $ locate inside backward self-similar region.
	Otherwise, it is not assured that intersections between $ u_a (\cdot,t) $ and $ U$ 
	vanishing at $ (x,t) = (0,T) $ are included in $ [0, D] $, which changes the corresponding braid.  
	It was mentioned in the proof of Theorem 5.7 of \cite{mat_cm07} that 
	complete proofs would be given in a separate paper.
	However, the separate paper has not appeared even as a preprint 
	(and this point was not mentioned in \cite{GMW} either, where the same argument was used).
\end{rem}

\smallskip 
We finally prove Theorem~\ref{mainThm1stab}
which determines the stability/instability of the GBU space-time profile with the continuity/discontinuity of GBU times.

\begin{proof}[Proof of Theorem~\ref{mainThm1stab}] 
{\bf Step 1.} {\it Preliminaries.}
By our assumptions and Lemma~\ref{lemma:zero-curve}, there exists $\eta>0$ such that $u$ is classical at $x=0$ for $t\in [T-\eta,T)$ 
and there exist $\sigma\in(0,\eta)$, $\delta>0$ and $D\in(0,R)$ such that
\be{xnD1} 
x_n (t) < D \  \mbox{ and } \  u(\cdot,t)-U\ne 0 \ \hbox{ in $(x_n(t),D]$,} \quad \mbox{ for all } t\in[T-\sigma, T),
\ee
and
\be{xnD2} 
u(D,t) - U(D) \left\{
\begin{array}{ll}
	\ge 2\delta \quad  &\mbox{ if } n \mbox{ is odd}\\ 
	\noalign{\vskip 1mm}
	\le -2\delta \quad &\mbox{ if } n \mbox{ is even}
\end{array}
\right.
\qquad \hbox{ for all $t\in [T-\sigma,T+\sigma]$.}
\ee
Let $\eps\in(0,\sigma)$. By Proposition~\ref{prop:PS-LBC2}, for
$\hat u_0\in \mathcal{W}$ with $\|\hat u_0-u_0\|_\infty \ll 1$, we get
\be{hatTT2a}	
\hu(D,t) - U(D) \left\{
\begin{array}{ll}
	\ge \delta \quad  &\mbox{ if } n \mbox{ is odd}\\ 
	\noalign{\vskip 1mm}
	\le -\delta \quad &\mbox{ if } n \mbox{ is even}
\end{array}
\right.
\qquad \hbox{ for all $t\in [T-\sigma,T+\sigma]$}
\ee
and
\be{hatTT2a0}	
\sup_{x\in(0,R/2)} \hu_x(x,T-\eps)<\infty\quad\hbox{and}\quad\hu(0,T-\eps)=0.
\ee

\smallskip

{\bf Step 2.} {\it Continuity of GBU time for $n$ odd and stability of the GBU space-time profile for $n=1$.}
If $n$ is odd, then $u(\cdot,T)\ge U$ on $(0,D]$
and it follows from \cite[Proposition~7.1]{MizSou} that $u$ immediately loses BC after $t=T$.
Thus, for $\eps>0$ small, by Proposition~\ref{prop:PS-LBC2}(i), 
we have $\hat u(0,T+\eps)>0$ when
$\|\hat u_0-u_0\|_\infty \ll 1$. We deduce from \eqref{hatTT2a0}
 that $\hu$ undergoes GBU at $(0,\hat T)$ for some time $\hat T\in (T-\eps,T+\eps)$.
This proves continuity of GBU time at $ u_0 $ for $n$ odd.
Note that we may consider the smallest such $\hat T$
(recall that there can be only finitely many), 
hence $\hu$ is classical at $x=0$ on $[T-\eps,\hat T)$.

\smallskip

For $n=1$, by \eqref{xnD1} and Proposition~\ref{prop:PS-LBC2}
 (recalling that the zeros of $u(\cdot,T-\eps) - U$ in $(0, D]$ are nondegenerate), 
we see that $z(\hu(\cdot,T-\eps) - U : [0, D])=1$ when $\|\hat u_0-u_0\|_\infty \ll 1$.
We thus deduce from \eqref{hatTT2a} and Proposition~\ref{propZeroNumber0} that 
$z(\hu(\cdot,t) - U : [0, D])\le 1$ for all $t\in[T-\eps,\hat T)$.
Consequently, the GBU profile of $\hu$ at $(x,t)=(0,\hat T)$ satisfies $n=1$.
Hence the GBU space-time profile with $n=1$ is stable.

\smallskip

{\bf Step 3.} {\it Discontinuity of GBU time for $n$ even and instability of the GBU profile for $n\ge 3$ odd.}
Consider general $n\ge 2$. 
From now on we take $\hu_0=\lambda u_0$ with $\lambda<1$ close to $1$ 
and denote by $\hu_\lambda$ the corresponding solution of \eqref{equ}.
Since $\overline u:=\lambda u$ satisfies $\overline u_t-\overline u_{xx}-|\overline u_x|^p
=\lambda(u_t-u_{xx}-\lambda^{p-1}|u_x|^p)\ge 0$ in $(0,R)\times(0,\infty)$, it follows from the comparison principle
for viscosity solutions that 
\be{orderhulambda}
\hu_\lambda\le \lambda u\quad\hbox{ in $(0,R)\times(0,\infty)$.}
\ee
By Proposition~\ref{prop:PS-LBC}(i), we infer that
\be{huxfinite}
\sup_{(0,R/2]\times[T-\eta,T]} \hu_{\lambda,x}<\infty,
\ee
hence
\be{huxfinite1}
\hbox{$\hu_\lambda<U$ in $(0,\rho]\times[T-\eta,T]$,\  
for some $\rho=\rho_\lambda\in(0,D)$.}
\ee

If $n$ is even, then $\hu_\lambda(\cdot,T)<u(\cdot,T)\le U$ in $(0,D]$
 (hence in particular $u$ becomes immediately classical at $x=0$ after $t=T$, by \cite[Proposition~7.1]{MizSou}).
In view of \eqref{hatTT2a} and \eqref{huxfinite}, there exists $\mu>0$ (depending on $\hat u_0$) such that
$$\hat u\le U_\mu\quad\hbox{on $([0,D]\times\{T\}) \cup (\{D\}\times[T,T+\sigma])$},$$
hence $\hu_\lambda\le U_\mu$ in $[0,D]\times[T,T+\sigma]$ by the comparison principle.
Consequently, $\hu_\lambda$ is classical at $x=0$ for $t\in [T-\eta,T+\sigma]$,
which proves that GBU time is discontinuous at $ u_0 $.

Finally assume that $n\ge 3$ is odd. 
Let $\hat T_\lambda$ be $ \hat T $ in Step 2 for $ \hat u_\lambda $, which now satisfies $\hat T_\lambda \in (T,T+\eps)$.
By \eqref{xnD1}, \eqref{xnD2}, \eqref{orderhulambda} and Proposition~\ref{prop:PS-LBC2},
we see that, for $\lambda<1$ close to $1$,
$$z\bigl(\hu_\lambda(\cdot,T-\eps) - U : (x_{n-1}(T-\eps),D)\bigr)=1.$$ 
From Proposition~\ref{propZeroNumber0}, we first deduce that 
$z\bigl(\hu_\lambda(\cdot,t) - U : (x_{n-1}(t),D)\bigr)=1$ for all $t\in[T-\eps,T)$
and then, also using \eqref{huxfinite1}, that $z(\hu_\lambda(\cdot,T) - U : [0,D])=1$.
By a further application of Proposition~\ref{propZeroNumber0}, we get
$z(\hu_\lambda (\cdot, t) - U : [0, D])\le 1$ for $t\in [T,\hat T_\lambda)$.
Therefore the GBU profile of $\hu_\lambda$ at $(x,t)=(0,\hat T)$ satisfies $n=1$. This proves that the GBU profile is unstable 
for $n\ge 3$ odd.
\end{proof}

\normalcolor

\subsection{Proof of Theorem \ref{th:RBC}(i)}  \label{SecProofBraid2}

Let $0<\tau < \infty$, set $Q=(0,R)\times(0,\tau)$ and 
	let $u\in C^{2,1}(Q)\cap C_b(\overline Q)$ be a solution of problem~\eqref{equREC}, 
	which undergoes RBC at $(x,t) = (0,\tau)$.	
Similarly to the GBU case, we have the following lemma which 
guarantees that the rescaled solution $u_a$ is suitably close to the singular steady state $U$ for large $a$.

\begin{lem}
\label{lemma:LBC}
There exist constants $\eta\in(0,\tau)$ and $C>0$ depending on $u$ such that,
for each $m\in(0,1/4]$, $a>1$ and $D\in(0,1)$,
the solution $u_a$, defined by
\begin{equation}
\label{eq:u_a-def-LBC}
u_a (x,t) := 
a^k u( a^{-1/2} x, \tau + a^{-1}(t- \tau)) \quad \mbox{ in } (0, a^{1/2}R) \times ((1-a)\tau, \tau),
\end{equation}
enjoys the following properties:
\be{eq:u_a-upper2-LBC} 
 - Ca^{- m}
 \le u_a(x,t) - U(x) \le C \Bigl\{1+a^{ -2q }(\tau-t) \Bigr\} a^{-m} \qquad \mbox{in } (0, a^q)\times(\tau -a \eta, \tau), 
\ee
\be{eq:u_a-upper3-LBC} 
-  CD^{ \frac{p}{p-1} }  U(x)
\le u_a(x,t) - U(x) \le C \Bigl\{1+a^{ - 2q }(\tau-t) \Bigr\}D^{ \frac{p}{p-1} } U(x) \quad \mbox{ in } 
(a^q, a^{1/2} D)\times(\tau -a \eta, \tau),
\ee
where $q= \frac{1}{2} \{ \frac{p}{2(p-1)} - m \}\in (0,1/2)$, and  
\be{uU0x-RBC}
\sup_{(x,t)\in (0, a^{1/2} D]\times[\tau-\eta,\tau)} |u_{a,x}(x,t)-U_x(x)|\to 0,\quad\hbox{as } a\to\infty.  
\ee 
\end{lem}

\begin{proof}
From \eqref{eq:u-LBC}, for $\eta\in(0,\tau/2]$, we obtain
\be{eq:u_a-LBC}
| u_a (x,t) - u_a (0,t) - U(x) | \le  \frac{\tilde M}{2} a^{ - \frac{p}{2(p-1)} } x^2 
\quad \mbox{ in } [0, a^{1/2} R] \times (\tau -a \eta, \tau).
\ee
Moreover, by \eqref{eq:u-M-LBC} we have
$$u_a (0,t) \le M a^{ - \frac{p}{2(p-1)} } (\tau -t) \quad \mbox{ in } (\tau -a \eta, \tau),$$
and this combined with \eqref{eq:u_a-LBC} implies
\[
- \frac{\tilde M}{2} a^{ - \frac{p}{2(p-1)} } x^2 \le u_a(x,t) - U(x) 
\le M a^{ - \frac{p}{2(p-1)} } (\tau -t) +  \frac{\tilde M}{2} a^{ - \frac{p}{2(p-1)} } x^2
\quad \mbox{ in } [0, a^{1/2} R] \times (\tau -a \eta, \tau).
\]
Estimates \eqref{eq:u_a-upper2-LBC}, \eqref{eq:u_a-upper3-LBC} then follow similarly as in the proof of Lemma \ref{lemma:u_a}
(see the end of the paragraphs after \eqref{eq:u_a-x-upper} and \eqref{eq:u_a-lower}).
Finally, \eqref{uU0x-RBC} follows from \eqref{eq:u_x-LBC} and the definition of $ u_a $.
\end{proof}

By similar argument as in the proof of Lemma~\ref{lemma:zero-curve}, we then obtain:

\begin{lem}
	\label{lemma:zero-curveRBC}
Let $n$ be the number of vanishing intersections between $ u(\cdot, t) $ and $ U (\cdot) $ at $ (x,t) = (0,\tau) $
(cf.~\eqref{defvanishnumber}). 
 Denoting these intersections by $ 0 < x_1(t) < x_2 (t) < \cdots < x_n (t) $, we have
$ \lim_{t \to \tau_-} x_n (t) = 0 $.
\end{lem}

A key step in the proof of Theorem \ref{th:RBC}(i) is the following lemma, which relies on the braid group argument.

\begin{lem}
\label{lemmaRBCbraid}
We have
\be{liminfsupREC}
0<\liminf_{t\to \tau_-} \, (\tau-t)^{-n} u(0,t)\le
\limsup_{t\to \tau_-} \, (\tau-t)^{-n} u(0,t)<\infty.
\ee
\end{lem}

\begin{proof}
 By Lemma~\ref{lemma:zero-curveRBC} we have $ \lim_{t \to \tau_-}  x_n (t) = 0 $. 
Therefore, for any $ 0 < D \ll 1 $, there exist 
 $ t_0 = t_0(D)<\tau$ and $\delta _0= \delta_0 (D)>0$ such that 
\begin{equation}
\label{eq:zerosRBC} 
 z\bigl( u(\cdot,t) - U:  (0, D] \bigr) = n
\quad \mbox{ and } \quad |u(D,t) - U(D)| \ge \delta_0 U(D),  \quad t\in[t_0, \tau). 
\end{equation}
Let $ v $ be a special solution obtained in Theorem~\ref{prop:special-LBC} for $\ell=n$ 
and $\Omega=(0,\infty)$.
By a time shift we may assume that $v$ undergoes RBC at $(x,t)=(0,\tau)$.

We prove the first inequality in \eqref{liminfsupREC}.
Owing to Theorem~\ref{prop:special-LBC}, Lemma \ref{lemma:LBC} and \eqref{eq:zerosRBC}, we may choose $0<D\ll 1$ and then 
	$a \gg 1 $, $ t_1 < \tau $, $ \delta_1 \in (0,1) $,  such that 
	\be{eqndescr1-REC}
	|U(a^{1/2} D) -  v (a^{1/2} D, t)| > \delta_1 U(a^{1/2} D) > |U(a^{1/2} D) -  u_a (a^{1/2} D, t) | >0, \quad t\in [t_1, \tau),
	\ee
	\be{eqndescr2b-REC}
 z\bigl( u_a(\cdot,t) - U:  (0, a^{1/2} D] \bigr) = z\bigl( v( \cdot, t)-U:  (0, a^{1/2} D]\bigr)= n, \quad t\in [t_1, \tau), 
\ee
\be{eqndescr2-REC1}
z\bigl( u_a( \cdot,t_1) - v( \cdot, t_1):  (0, a^{1/2} D]\bigr)= n,
\ee
\be{eqndescr3-REC1}
  v (0, t_1) > u_a (0, t_1).
\ee
Assume for contradiction that the first inequality in \eqref{liminfsupREC} does not hold.
Then there exists $ t_2 \in (t_1, \tau) $ such that $ v(\cdot,\tau + t) - u_a(\cdot,t) $ loses one zero (or odd number of zeros) at $ x = 0 $ 
and some $\hat t_2< t_2 $ close enough $ t_2 $, and 
\begin{equation}
\label{eq:t_2-RBC}
 v (0, t_2)  < u_a (0, t_2).
\end{equation}
For $ 0 < \lambda < 1 $, let 
\[
\tu_a (x,t) := 
\lambda^k u_a ( \lambda^{-1/2} x, t_1 + \lambda^{-1} (t - t_1) ) 
\quad \mbox{ in } (0, \lambda^{1/2} a^{1/2}) \times ( t_1, \tilde\tau )
\]
with $ \tilde\tau : =t_1 + \lambda (\tau- t_1) $.
For $ 0 < \lambda < 1 $ close enough to $ 1 $, 
 \eqref{eqndescr1-REC}-\eqref{eq:t_2-RBC} hold true with $ u_a $ replaced by $ \tu_a $.
Therefore, the situation of $ U, \tu_a(t), v(t) $ on $ [0, a^{1/2} D ] $ at $ t = t_1 $ is represented by~$ \hA_n $. 
Since $ \tu_a $ recovers boundary condition at $(x,t) = (0, \tilde\tau) $ with $ \tilde\tau < \tau $, there exists $ t_3 \in (t_2, \tilde\tau) $ such that $ v (t) - \tu_a (t) $ loses one zero 
(or odd number of zeros) at $ x = 0 $ and some $\hat t_3 < t_3 $ close to $ t_3 $, and $  v (0,  t_3) > \tu_a (0, t_3) $. 
 By arguing similarly as in the proof of Theorem \ref{mainThm1}(i), we see that\footnote{The alternative argument 
in appendix, for readers who are not familiar with application of braid group theory,
where explicit intersections are considered instead of ``hidden'' intersections, can also be easily modified in this case.}
the situation of $ U, \tu_a(t), v(t) $ 
on $ [0, a^{1/2} D ] $ at $ t = t_3 $ is translated into $ \hB_n $ and that the process from $ t = t_1 $ to $ t = t_3 $ implies that $ \hA_n \Rrightarrow \hB_n $
  (indeed the reduction principle for parabolic evolution stated in Section 6
   is still valid for the viscosity solutions under consideration, owing to the zero-number property in Proposition~\ref{propZeroNumber}(ii)).
On the other hand, we have $ \hA_n \not\Rrightarrow \hB_n $ by Lemma~\ref{lemma:braid-even-odd-2}. 
This contradiction implies the first inequality.

In order to prove the last inequality of \eqref{liminfsupREC}, it suffices to take the same way as above with $ v, u_a $ and the spatial 
interval $ [0, a^{1/2} D] $ replaced by $ u, v_a $ and $ [0, D] $, respectively since all zeros of $ v_a (\cdot, t) - U$  
(and $ \tv_a(\cdot, t) - U $) locate in $ \bigl(0, C (\tau-t)^{1/2} \bigr) $ for $ t \in [t_0, \tau) $ with some $ C > 0 $ for $ a \gg 1 $. 
\end{proof}

 With Lemma~\ref{lemmaRBCbraid} at hand, the proof of Theorem \ref{th:RBC}(i) 
will now be completed by means of dynamical systems arguments.

\begin{proof}[Proof of the space-time profile \eqref{recovery-profile}]
 We note that this will imply \eqref{rateell-REC}.
\smallskip

{\bf Step 1.} {\it First estimates.}
It is sufficient to consider the cases $R=1$ and $R=\infty$.
As before, so as to work with unknown functions defined on the entire half-line, 
letting $s_0=-\log\tau$, we recall the extentions introduced in Lemma~\ref{BernsteinEst}:
$$ 
\begin{aligned}
\tilde u(x,t)&=\zeta(x) u(x,t) &&\quad\hbox{ in $[0,\infty)\times[0,\tau)$,} \\
\tilde w(y,s)&=e^{ks}\tilde u(ye^{-s/2},\tau-e^{-s}) 
&&\quad\hbox{ in $[0,\infty)\times[s_0,\infty)$,} \\
\tilde v(y,s)&=\tilde w(y,s)-U(y) &&\quad\hbox{ in $[0,\infty)\times[s_0,\infty)$,}
\end{aligned}
$$  
where for $R=1$, $\zeta\in C^2([0,\infty))$ is a fixed cut-off function such that $0\le\zeta\le 1$,
 $\zeta=1$ in $[0,\frac{1}{3}]$ and $\zeta=0$ in $[\frac{1}{2},\infty)$,
 whereas for $R=\infty$ we just set $\zeta=1$.
Moreover, $\tilde v$ satisfies the equation
\be{eq:tildeF-LBC-0}
\tilde v_s+\mathcal{L}\tilde v=\tilde F(\tilde v_y,s),\quad y>0,\ s>s_0,
\ee
where 
\be{eq:tildeF-LBC}
\tilde F(\tilde v_y,s):=e^{(k-1)s}g(ye^{-s/2},\tau-e^{-s})-U_y^p-pU_y^{p-1}\tilde v_y,
\quad\hbox{ with } g(x,t):=|u_x|^p\zeta -2u_x \zeta_x-u\zeta_{xx}.
\ee
Recast in terms of $\tilde v$, the sought-for profile \eqref{recovery-profile} is equivalent to 
$$\tilde v(y,s) = \tilde L e^{-\lambda_n s}  \varphi_n (y) + o (e^{-\lambda_n s}) \quad \mbox{ as } s \to \infty,
\quad\hbox{ uniformly for $0\le y\le y_0$,}$$
for each $y_0>0$, with some constant $\tilde L>0$ (note that $\phi_n$ in the statement and $\varphi_n$ (cf.~subsection~\ref{subseceigen})
 differ only by a multiplicative constant).

\smallskip

We start with some basic estimates of $\tilde v$ and $\tilde F$.
We deduce from \eqref{eq:u-LBC} that, for some $s_1>s_0$,
$$\begin{aligned}
|\tilde v(y,s)|
&=\bigl|e^{ks}\zeta(ye^{-s/2})u(ye^{-s/2},\tau-e^{-s})-U(y)\bigr|\\
&\le \zeta(ye^{-s/2})\bigl|e^{ks}u(ye^{-s/2},\tau-e^{-s})-U(y)\bigr|+\bigl(1-\zeta(ye^{-s/2})\bigr)U(y)\\
&\le e^{ks}\bigl(u(0,\tau-e^{-s})+ Cy^2e^{-s}\bigl)+\chi_{\{y\ge \frac13 e^{s/2}\}}U(y),\qquad y>0,\ s>s_1,
\end{aligned}$$
hence, recalling $k=\frac{1-\beta}{2}$, using \eqref{eq:u-M-LBC} and noting that $U(y)\le Cy^2e^{-\alpha s/2}$ for $y\ge \frac13 e^{s/2}$,
\be{eq:v-LBC}
|\tilde v(y,s)|\le Ce^{-\alpha s/2}(1+y^2),\qquad y>0,\ s>s_1.
\ee
Likewise, using also \eqref{eq:u_x-LBC}, we get
$$\begin{aligned}
|\tilde v_y(y,s)|
&=\bigl|e^{(k-\frac12)s}\zeta(ye^{-s/2})u_x(ye^{-s/2},\tau-e^{-s})+e^{ks}e^{-s/2}\zeta'(ye^{-s/2})u(ye^{-s/2},\tau-e^{-s})-U'(y)\bigr| \\
&\le \zeta(ye^{-s/2})\bigl|e^{(k-\frac12)s}u_x(ye^{-s/2},\tau-e^{-s})-U'(y)\bigr|+\bigl(1-\zeta(ye^{-s/2})\bigr)U'(y) \\
&\quad +e^{-s/2}|\zeta'(ye^{-s/2})|e^{ks}u(ye^{-s/2},\tau-e^{-s})\\
&\le Ce^{-\alpha s/2}y+\chi_{\{y\ge \frac13 e^{s/2}\}}U'(y)+Ce^{-s/2}\chi_{\{\frac13 e^{s/2}\le y\le \frac12 e^{s/2}\}}
\bigl(e^{-\alpha s/2}(1+y^2)+U(y)\bigr)
\end{aligned}$$
hence
\be{eq:vy-LBC}
|\tilde v_y(y,s)|\le  Ce^{-\alpha s/2}y,\qquad y>0,\ s>s_1.
\ee
For $R=1$, by \eqref{eq:u-LBC}-\eqref{eq:u_x-LBC}, we may also assume that
\be{eq:uux-LBC}
|u(x,t)|+|u_x(x,t)|\le C,\qquad \ts\frac13\le x\le\ts\frac12,\ \ \tau-e^{-s_1}\le t<\tau.
\ee
 By a time shift, we may assume $s_1=0$ 
 and, by \eqref{eq:tildeF-LBC-0} and Proposition~\ref{locexistvapproxExt}(ii), $\tilde v$ satisfies the variation of constants formula
\be{varconstv2tildeProfile}
\tilde v(s)=e^{-s\mathcal{L}}\tilde v(0)+\int_0^s e^{-(s-\tau)\mathcal{L}}\tilde F(\tilde v_y(\tau))\,d\tau, \quad s>0.
\ee
Next we claim that
\be{eq:Fy-LBC}
|\tilde F(\tilde v_y)|\le C(y^{\beta-1}+y^{p-2})(\tilde v_y)^2+C\chi_{\{y\ge \frac13 e^{s/2}\}},\ \quad y>0,\ s>0.
\ee
To check this, we first note that
$$
0\le F(v_y):=|U_y+v_y|^p-U_y^p-pU_y^{p-1}v_y=\frac{p(p-1)}{2}|U_y+\bar\theta v_y|^{p-2}(v_y)^2
$$
for some $\bar\theta\in(0,1)$, which in view of \eqref{eq:vy-LBC} ensures \eqref{eq:Fy-LBC}, 
except when $R=1$ and $y\ge \frac13 e^{s/2}$.
In the latter case, it suffices to observe that
$|\tilde F(\tilde v_y,s)|\le C+\tilde v_y^2$ for $y\ge \ts\frac13 e^{s/2}$, 
owing to \eqref{eq:tildeF-LBC}, \eqref{eq:uux-LBC}.

\smallskip

{\bf Step 2.} {\it $H^1_\rho$ estimate.}
We show that there exist $\ell\ge 1$ and a constant $K\in\R\setminus\{0\}$ such that
\be{cv1LBC}
\|\tilde v(s)-Ke^{-\lambda_\ell s}\varphi_\ell \|_{H^1_\rho}=o(e^{-\lambda_\ell s}),\quad\hbox{as } s\to\infty.
\ee
This is suggested by dynamical systems theory.
However the latter cannot directly applied because of the inadequate functional framework.
To prove \eqref{cv1LBC} we shall use some arguments from \cite{HV93} (see the proof of Proposition~3.8),
where a related result was proved for type~I blowup 
solutions of the one-dimensional Fujta equation.
However, these arguments need suitable modifications to overcome some specific difficulties 
of our problem (singular vs. regular steady state, gradient nonlinearity).

Let $E=\bigl\{\mu>0;\ \sup_{s>0} e^{\mu s}\|\tilde v(s)\|_{H^1_\rho}<\infty\bigr\}$
and $\mu^*=\sup E$.
Estimates \eqref{eq:v-LBC}-\eqref{eq:vy-LBC} guarantee that 
\be{Enonempty}
\alpha/2\in E,
\ee
hence $E\neq\emptyset$.
Moreover, by \eqref{liminfsupREC} and \eqref{eq:vy-LBC}, for some constants $C_1,C_2>0$, we have
$$\tilde v(y,s)\ge \tilde v(0,s)-Cy^2
\ge 2C_1e^{(k-n)s}-Cy^2
\ge C_1e^{(k-n)s}\quad\hbox{ for all $0<y<C_2e^{(k-n)s/2}$},$$
hence
$$\|\tilde v(s)\|_{L^2_\rho}^2 
\ge Ce^{2(k-n)s}\int_0^{C_2e^{(k-n)s/2}} y^\alpha\,dy=Ce^{(\alpha+5)(k-n)s/2},\quad s>0.$$
Consequently $\mu^*<\infty$.

Fixing $\mu\in (\mu^*-\frac{\alpha}{4},\mu^*)$ such that $\mu+\frac{\alpha}{4}+k$ is noninteger,
we have
\be{hypvAmu}
\|\tilde v(s)\|_{H^1_\rho}\le Ce^{-\mu s},\quad s>0.
\ee
For any $\eps\in(0,1]$, using \eqref{eq:vy-LBC}, \eqref{eq:Fy-LBC}, \eqref{hypvAmu},
H\"older's inequality and setting
$\eta_0(s):=C\bigl(\int_{y>\frac13 e^{s/2}} \rho\,dy\bigr)^2$, we obtain
$$\begin{aligned}
\|\tilde F(\tilde v_y)\|_{L^2_\rho}^2
&\le C\int \rho (y^{\beta-1}+y^{p-2})^2 |\tilde v_y|^4dy+\eta_0(s) \\
&\le Ce^{-2(1+\eps)\frac{\alpha}{2}s}\int \rho |\tilde v_y|^{2(1-\eps)} y^{2(1+\eps)} (y^{\beta-1}+y^{p-2})^2dy+\eta_0(s) \\
&\le Ce^{-(1+\eps)\alpha s} \Bigl(\int \rho |\tilde v_y|^2 \,dy\Bigr)^{1-\eps} 
\Bigl(\int \rho \bigl[y^{\beta+\eps}+y^{p-1+\eps}\bigr]^{2/\eps}\,dy\Bigr)^{\eps} +\eta_0(s)\\
&\le C(\eps)e^{-[(1+\eps)\alpha+(1-\eps)2\mu]s},
\end{aligned}$$
hence in particular
\be{estimFL2}
\|\tilde F(\tilde v_y)\|_{L^2_\rho}\le Ce^{-(\mu+\frac{\alpha}{4})s},\quad s>0.
\ee
Recalling that $\lambda_j=j-k$, there exists a unique integer $J\ge 0$ such that
\be{choiceJ}
\lambda_J<\mu+\frac{\alpha}{4}<\lambda_{J+1}.
\ee
Using \eqref{varconstv2tildeProfile},\eqref{defSG1}, we then split $\tilde v$ as 
\be{defT1234a}
\tilde v(y,s)=T_1+T_2+T_3+T_4,
\ee
where
\be{defT1234}
\begin{cases}
T_1(y,s)&=\ds\sum_{j=0}^{J} a_j e^{-\lambda_j s}\varphi_j(y),\qquad
T_2(y,s)=\ds\sum_{j=J+1}^\infty a_j e^{-\lambda_j s}\varphi_j(y),\\
\noalign{\vskip 1mm}
T_3(y,s)&=\ds\sum_{j=0}^{J} \varphi_j(y)\int_0^s e^{-\lambda_j(s-\tau)}(\tilde F(\tilde v_y(\tau)),\varphi_j)\,d\tau,\\
\noalign{\vskip 1mm}
T_4(y,s)&=\ds\sum_{j=J+1}^\infty \varphi_j(y)\int_0^s e^{-\lambda_j(s-\tau)}(\tilde F(\tilde v_y(\tau)),\varphi_j)\,d\tau
\end{cases}
\ee
and $\sum |a_j|^2<\infty$. To estimate $T_2$, for $s\ge 1$, we write
$$\begin{aligned}
\|T_2(s)\|_{H^1_\rho}^2
&=\ds\sum_{j=J+1}^\infty (1+j)|a_j|^2 e^{-2\lambda_j s}
\le \Bigl(\ds\sum_{j=0}^\infty|a_j|^2 \Bigr)
\sup_{j\ge J+1} (1+j)e^{-2\lambda_j s}\\
&\le C e^{-2\lambda_{J+1} s} \sup_{j\ge J+1} (1+j) e^{-2(j-J-1) s}
=C e^{-2\lambda_{J+1} s} \sup_{i\ge 0} (J+2+i) e^{-2i},
\end{aligned}$$
hence
\be{estimT2H1}
\|T_2(s)\|_{H^1_\rho}\le C e^{-\lambda_{J+1} s},\quad s\ge 1.
\ee
To estimate $T_4$, we first write
$$
\|T_4(s)\|_{L^2_\rho}
\le \ds\sum_{j=J+1}^\infty \int_0^s e^{-\lambda_j(s-\tau)}\bigl|(\tilde F(\tilde v_y(\tau)),\varphi_j)\bigr|\,d\tau
\le \int_0^s \Bigl(\ds\sum_{j=J+1}^\infty e^{-2\lambda_j(s-\tau)}\Bigr)^{1/2}\|\tilde F(\tilde v_y(\tau))\|_{L^2_\rho}\,d\tau.
$$
Noting that 
$$\begin{aligned}
\ds\sum_{j=J+1}^\infty e^{-2\lambda_j(s-\tau)}
&=e^{-2\lambda_{J+1}(s-\tau)}\ds\sum_{i=0}^\infty e^{-2i(s-\tau)}
=e^{-2\lambda_{J+1}(s-\tau)}(1-e^{-2(s-\tau)})^{-1}\\
&\le Ce^{-2\lambda_{J+1}(s-\tau)}\bigl(1+(s-\tau)^{-1}\bigr)
\end{aligned}$$
and using \eqref{estimFL2} and \eqref{choiceJ}, we obtain
\be{estimT4L2}
\begin{aligned}
\|T_4(s)\|_{L^2_\rho}
&\le C\int_0^s e^{-\lambda_{J+1}(s-\tau)}\bigl(1+(s-\tau)^{-1/2}\bigr)e^{-(\mu+\frac{\alpha}{4})\tau}\,d\tau\\
&\le C\int_0^s e^{-\lambda_{J+1}\sigma}(1+\sigma^{-1/2})e^{-(\mu+\frac{\alpha}{4})(s-\sigma)}\,d\sigma \\
&= Ce^{-(\mu+\frac{\alpha}{4})s}\int_0^s (1+\sigma^{-1/2})e^{(\mu+\frac{\alpha}{4}-\lambda_{J+1})\sigma}\,d\sigma
\le Ce^{-(\mu+\frac{\alpha}{4})s}.
\end{aligned}
\ee
Next observe that $Z:=T_4$ solves
$$Z_s-\mathcal{L}Z=H(y,s):=\tilde F(\tilde v_y)- \ds\sum_{j=0}^{J} (\tilde F(\tilde v_y(\tau)),\varphi_j)\varphi_j(y),$$
and that 
\be{estimHH1}
\|H(s)\|_{L^2_\rho}\le \|\tilde F(\tilde v_y(s))\|_{L^2_\rho} \le Ce^{-(\mu+\frac{\alpha}{4})s}.
\ee
Using the variation of constants formula
$Z(s)=e^{\mathcal{L}}Z(s-1)+\int_0^1 e^{(1-\tau) \mathcal{L}}H(s-1+\tau)\,d\tau$ for $s\ge 1$, we get, 
using \eqref{estimT4L2}-\eqref{estimHH1} and \eqref{smoothing0}-\eqref{smoothing1},
\be{estimT4H1}
\|T_4(s)\|_{H^1_\rho}\le \|T_4(s-1)\|_{L^2_\rho}+C\int_0^1 (1-\tau)^{-1/2} \|H(s-1+\tau)\|_{L^2_\rho}\,d\tau
\le Ce^{-(\mu+\frac{\alpha}{4})s},\quad s\ge 1.
\ee
Now, for $0\le j\le J$, we write
\be{estimT3bj}
\begin{aligned}
\int_0^s e^{\lambda_j\tau}(\tilde F(\tilde v_y(\tau)),\varphi_j)\,d\tau
&=\int_0^\infty e^{\lambda_j\tau}(\tilde F(\tilde v_y(\tau)),\varphi_j)\,d\tau
-\int_s^\infty e^{\lambda_j\tau}(\tilde F(\tilde v_y(\tau)),\varphi_j)\,d\tau\\
&\equiv b_j-\int_s^\infty e^{\lambda_j\tau}(\tilde F(\tilde v_y(\tau)),\varphi_j)\,d\tau,
\end{aligned}
\ee
where, owing to \eqref{estimFL2} and \eqref{choiceJ},
\be{estimT3bj2}
\Bigl|\int_s^\infty e^{\lambda_j\tau}(\tilde F(\tilde v_y(\tau)),\varphi_j)\,d\tau\Bigr|
\le \int_s^\infty e^{\lambda_j\tau}\|\tilde F(\tilde v_y(\tau))\|_{L^2_\rho}\,d\tau
\le Ce^{(\lambda_j-\mu-\frac{\alpha}{4})s}
\ee
 and $b_j$ is finite.
Substituting \eqref{estimT3bj} into the identity \eqref{defT1234a}, 
we obtain
$$\tilde v(s)=\ds\sum_{j=0}^{J} (a_j+b_j) e^{-\lambda_j s}\varphi_j(y)
+T_2(y,s)+T_4(y,s)-
\ds\sum_{j=0}^{J} \varphi_j(y) \int_s^\infty e^{-\lambda_j(s-\tau)}(\tilde F(\tilde v_y(\tau)),\varphi_j)\,d\tau.
$$
By the bounds \eqref{estimT2H1}, \eqref{estimT4H1}, \eqref{estimT3bj2}, it follows that
$$\tilde v(s)=\ds\sum_{j=0}^{J} (a_j+b_j) e^{-\lambda_j s}\varphi_j(y)+R(y,s),
\quad\hbox{with } \|R(s)\|_{H^1_\rho} \le Ce^{-(\mu+\frac{\alpha}{4})s},\quad s\ge 1.$$
Let $\tilde E=\bigl\{j\in\{0,\dots,J\};\ a_j+b_j\ne 0\bigr\}$.
We have $\tilde E\ne\emptyset$ since otherwise $\mu^*<\mu+\frac{\alpha}{4}\in E$,
contradicting the definition of $\mu^*$. Letting $\ell=\min\tilde E$, we then have
$$\tilde v(s)=(a_\ell+b_\ell) e^{-\lambda_\ell s}\varphi_\ell(y)+\tilde R(y,s),
\quad\hbox{with } \|\tilde R(s)\|_{H^1_\rho} \le Ce^{-\gamma s},\quad s\ge 1,$$
where $\gamma=\min(\lambda_{\ell+1},\mu+\frac{\alpha}{4})>\lambda_\ell$ owing to \eqref{choiceJ}.
Moreover we necessarily have $\ell\ge 1$ in view of \eqref{Enonempty} and $\lambda_0=-k<0$.
This proves \eqref{cv1LBC}.

\smallskip

{\bf Step 3.} {\it Bootstrap and conclusion.}
Denote $\varphi=\varphi_\ell$, $\lambda=\lambda_\ell$.
We prove that, for all $2\le q<\infty$,
\be{hypboot1}
\|\tilde v(s)-Ke^{-\lambda s}\varphi \|_{W^{1,q}_\rho}=o(e^{-\lambda s})\quad\hbox{as } s\to\infty.
\ee
 By the imbedding \eqref{Imbedd-q},
property \eqref{liminfsupREC} then guarantees that $\ell=n$ and that \eqref{recovery-profile} is true.

\smallskip

Let $\theta:=\tilde v(s)-Ke^{-\lambda s}\varphi$.
The function $\theta$ satisfies the equation $\theta_s-\mathcal{L}\theta=F(v_y)$,
hence the variation of constants formula
$$\theta(s+\bar s)=e^{\bar s \mathcal{L}}\theta(s)+\int_0^{\bar s}e^{(\bar s-\tau) \mathcal{L}}
\tilde F(\tilde v_y(s+\tau))\,d\tau,\quad s, \bar s>0.$$
We shall use a bootstrap argument.
Note that \eqref{hypboot1} is true for $q=2$ by Step 1. Thus fix some $q\in[2,\infty)$ and suppose that
\eqref{hypboot1} is true, i.e. 
$e^{\lambda s} \|\theta(s)\|_{W^{1,q}_\rho}\to 0$ as $s\to\infty$.
 Let $m\in (q,\infty)$. 
For each $\eps\in(0,q)$, by \eqref{eq:vy-LBC}, \eqref{eq:Fy-LBC} and H\"older's inequality,
setting $\eta_1(s):=C\bigl(\int_{y>\frac13 e^{s/2}} \rho\,dy\bigr)^m$, we obtain
$$\begin{aligned}
\|\tilde F(\tilde v_y)\|_{L^m_\rho}^m
&\le C\int \rho (y^{\beta-1}+y^{p-2})^m |\tilde v_y|^{2m}dy+\eta_1(s) \\
&\le Ce^{-(2m-q+\eps)\frac{\alpha}{2}s}\int \rho  (y^{\beta-1}+y^{p-2})^m|\tilde v_y|^{q-\eps} y^{2m-q+\eps} dy +\eta_1(s)\\
&\le Ce^{-(2m-q+\eps)\frac{\alpha}{2}s}\Bigl(\int \rho |\tilde v_y|^q dy\Bigr)^{1-(\eps/q)}
\Bigl(\int \rho  \bigl(y^{(\beta+1)m-q+\eps}+y^{pm-q+\eps}\bigr)^{q/\eps}  dy\Bigr)^{\eps/q}+\eta_1(s)\\
&\le C_\eps e^{-(2m-q+\eps)\frac{\alpha}{2}s}\Bigl(\int \rho |\tilde v_y|^q dy\Bigr)^{(q-\eps)/q}+\eta_1(s).
\end{aligned}$$
Since \eqref{hypboot1} in particular implies $\|\tilde v(s)\|_{W^{1,q}_\rho}\le Ce^{-\lambda s}$, we get 
$$\|\tilde F(\tilde v_y)\|_{L^m_\rho}
\le C_\eps e^{-[\alpha +\frac{q-\eps}{m}(\lambda-\frac{\alpha}{2})]s},\quad s>0.$$
Also, we may choose $r>1$ depending only on $\alpha,\lambda$ and then $\eps>0$ small such that
$m=qr$ satisfies
$\alpha +\frac{q-\eps}{m}(\lambda-\frac{\alpha}{2})>\lambda+\frac{\alpha}{4}$. Consequently,
$$\|\tilde F(\tilde v_y)\|_{L^{qr}_\rho}\le C e^{-[\lambda+\frac{\alpha}{4}]s},\quad s>0.$$
Now let $\bar s:=s^*(q,qr)$ be given by Proposition~\ref{Prop-smoothing}.
By \eqref{smoothing0}--\eqref{smoothing3}, it follows that
$$
\|\theta(s+\bar s)\|_{W^{1,qr}_\rho}
\le C\|\theta(s)\|_{W^{1,q}_\rho}+C\int_0^{\bar s}(\bar s-\tau)^{-3/4}  
\|\tilde F(\tilde v_y(s+\tau))\|_{L^{qr}_\rho}\,d\tau 
\le C\|\theta(s)\|_{W^{1,q}_\rho}+C e^{-[\lambda+\frac{\alpha}{4}]s}.
$$
Consequently, \eqref{hypboot1} is satisfied with $q$ replaced by $qr$.
Since \eqref{hypboot1} is true for $q=2$ (and since the spaces $W^{1,q}_\rho$ decrease with $q$),
it follows that it is true for all finite $q$. This completes the proof.
\end{proof}

\smallskip 

We finally prove Theorem~\ref{mainThmRBCstab}
which determines the stability/instability of the RBC space-time profile with the continuity/discontinuity of RBC times.

\begin{proof}[Proof of Theorem~\ref{mainThmRBCstab}] 
{\bf Step 1.} {\it Preliminaries.}
By our assumptions and Lemma~\ref{lemma:zero-curveRBC}, there exists $\eta>0$ such that $u(0,t)>0$ for $t\in [\tau-\eta,\tau)$, $u(0,\tau)=0$
and there exist $\sigma\in(0,\eta)$, $\delta>0$ and $D\in(0,R)$ such that
\be{xnD1-REC} 
x_n (t) < D \  \mbox{ and } \  u(\cdot,t)-U\ne 0 \ \hbox{ in $(x_n(t),D]$,} \quad \mbox{ for all } t\in[\tau-\sigma, \tau),
\ee
and
\be{xnD2-REC} 
u(D,t) - U(D) \left\{
\begin{array}{ll}
	\ge 2\delta \quad  &\mbox{ if } n \mbox{ is even}\\ 
	\noalign{\vskip 1mm}
	\le -2\delta \quad &\mbox{ if } n \mbox{ is odd}
\end{array}
\right.
\qquad \hbox{ for all $t\in [\tau-\sigma,\tau+\sigma]$.}
\ee
Let $\eps\in(0,\sigma)$. By Proposition~\ref{prop:PS-LBC2}(i), for
$\hat u_0\in \mathcal{W}$ with $\|\hat u_0-u_0\|_\infty \ll 1$, we get
\be{hatTT2a-REC} 	
\hu(D,t) - U(D) \left\{
\begin{array}{ll}
	\ge \delta \quad  &\mbox{ if } n \mbox{ is even}\\ 
	\noalign{\vskip 1mm}
	\le -\delta \quad &\mbox{ if } n \mbox{ is odd}
\end{array}
\right.
\qquad \hbox{ for all $t\in [\tau-\sigma,\tau+\sigma]$,}
\ee
\be{hatTT2a0-REC} 
\hu(0,\tau-\eps)>0.
\ee

\smallskip

{\bf Step 2.} {\it Continuity of RBC time for $n$ odd and stability of the RBC profile for $n=1$.}
If $n$ is odd, then $u(\cdot,\tau)<U$ on $(0,D]$
and it follows from \cite[Proposition~7.1]{MizSou} that $u$ becomes immediately classical at $x=0$ after $t=\tau$.
Thus, for $\eps>0$ small, by Proposition~\ref{prop:PS-LBC2}(ii), we have $\hat u(0,\tau+\eps)=0$ when
$\|\hat u_0-u_0\|_\infty \ll 1$. We deduce from \eqref{hatTT2a0-REC} 
 that $\hu$ undergoes RBC at $(0,\hat \tau)$ for some time $\hat \tau\in (\tau-\eps,\tau+\eps]$.
This proves that RBC time is continuous at $ u_0 $ for $n$ odd.
Note that we may consider the smallest such $\hat \tau$
(recall that there can be only finitely many), 
hence 
\be{hatTT3-REC} 	
\hu(0,t)>0\quad\hbox{on $[\tau-\eps,\hat \tau)$.}
\ee

\smallskip

For $n=1$, by \eqref{xnD1-REC} and Proposition~\ref{prop:PS-LBC2}(i)
 (recalling that the zeros of $u(\cdot,\tau-\eps) - U$ in $[0, D]$ are nondegenerate), we see that
$z(\hu(\cdot,\tau-\eps) - U : [0, D])=1$ when $\|\hat u_0-u_0\|_\infty \ll 1$.
Then 
we deduce from \eqref{hatTT2a-REC}, \eqref{hatTT3-REC} 
and Proposition~\ref{propZeroNumber}(ii) that 
$z(\hu(\cdot,t) - U : [0, D])\le 1$ for all $t\in[\tau-\eps,\hat \tau)$
 (note that Proposition~\ref{propZeroNumber}(ii) applies owing to 
Lemma~\ref{lemsmoothut2} and \eqref{eq:u_x-LBC-0} in Proposition~\ref{prop:PS-LBC}).
Consequently, the RBC profile of $\hu$ at $(x,t)=(0,\hat \tau)$ satisfies $n=1$.
Hence the RBC profile with $n=1$ is stable.

\smallskip

{\bf Step 3.} {\it Discontinuity of RBC time for $n$ even and instability of the RBC profile for $n\ge 2$.}
Consider general $n\ge 2$. 
From now on we take $\hu_0=\lambda u_0$ with $\lambda>1$  close to $1$ 
and denote by $\hu_\lambda$ the corresponding solution
of \eqref{equ}.
Since $\underline u:=\lambda u$ satisfies $\underline u_t-\underline u_{xx}-|\underline u_x|^p
=\lambda(u_t-u_{xx}-\lambda^{p-1}|u_x|^p)\le 0$ in $(0,R)\times(0,\infty)$, it follows from the comparison principle
for viscosity solutions that 
\be{orderhulambda-REC}
\hu_\lambda\ge \lambda u\quad\hbox{ in $(0,R)\times(0,\infty)$.}
\ee
By \eqref{eq:u-LBC-0} in Proposition~\ref{prop:PS-LBC}(i), we infer that
\be{huxfinite-REC} 
\hu_\lambda(0,\tau)>0.
\ee

If $n$ is even, then $\hu_\lambda(\cdot,\tau)>u(\cdot,\tau)\ge U$ on $(0,D]$ 
 (hence in particular $u$ immediately loses BC after $t=T$, by \cite[Proposition~7.1]{MizSou}).
In view of \eqref{hatTT2a-REC} and \eqref{huxfinite-REC}, there exists $\mu>1$ (depending on $\lambda$) such that
$$\hat u_\lambda\ge \mu U\quad\hbox{on $([0,D]\times\{\tau\}) \cup (\{D\}\times[\tau,\tau+\sigma])$}.$$
Since $-(\mu U)''-(\mu U')^p\le 0$, it follows from the comparison principle that
$\hu_\lambda\ge \mu U$ in $[0,D]\times[\tau,\tau+\sigma]$. 
We then deduce from Proposition~\ref{prop:PS-LBC}(i) that $\hu_\lambda(0,t)>0$ for $t\in (\tau,\tau+\sigma]$,
hence for $t\in [\tau-\eta,\tau+\sigma]$,
which proves that RBC time is discontinuous at $ u_0 $.

Finally assume that $n\ge 3$ is odd. 
Let $\hat \tau_\lambda$ be $ \hat \tau $ in Step 2 for $ \hat u_\lambda$, 
which now satisfies $\hat \tau_\lambda\in (\tau,\tau+\eps)$.
By \eqref{xnD1-REC}, \eqref{xnD2-REC}, \eqref{orderhulambda-REC} and Proposition~\ref{prop:PS-LBC2},
we see that, for $\lambda>1$ close to $1$,
$$z\bigl(\hu_\lambda(\cdot,\tau-\eps) - U : (x_{n-1}(\tau-\eps),D)\bigr)=1.$$ 
From Proposition~\ref{propZeroNumber0}, we first deduce that 
$z\bigl(\hu_\lambda(\cdot,t) - U : (x_{n-1}(t),D)\bigr)=1$ for all $t\in[\tau-\eps,\tau)$
and then, also using \eqref{huxfinite-REC}, that $z(\hu_\lambda(\cdot,\tau) - U : [0,D])=1$.
By a further application of Proposition~\ref{propZeroNumber0}, we get
$z(\hu_\lambda (\cdot, t) - U : [0, D])\le 1$ for $t\in [\tau,\hat\tau_\lambda)$.
Therefore the RBC profile of $\hu_\lambda$ at $(x,t)=(0,\hat\tau_\lambda)$ satisfies $n=1$. 
This proves that the RBC profile is unstable for $n\ge 3$ odd.
\end{proof}

\begin{rem} \label{RemStabRBC}
Theorem~\ref{mainThmRBCstab} is formulated for problem \eqref{equ} rather than \eqref{equREC}.
Indeed, continuity and stability properties are usually studied in the context of locally well-posed initial boundary value problems, 
and \eqref{equREC} does not seem to enter in that category, since there is no local existence theory available for this problem.
In this respect, recall that the viscosity existence theory in \cite{BdaLio} requires that there is no LBC at $t=0$ (i.e., $u(0,0)=u(R,0)=0$)
while Proposition~\ref{locexistvapprox} (developed for the specific purpose of the proof of Theorem~\ref{th:RBC}(ii)) requires the strong assumption $u(.,0)\in C^1(0,R)$ with $u_x(.,0)-U'$ bounded.
\end{rem}

\section{Appendix. Alternative argument for Step~3 of the proof of Theorem \ref{mainThm1}{\rm(i)}}

We here present an alternative argument,
making use of explicit intersections instead of ``hidden'' intersections,
to show that, if the first inequality in \eqref{liminfsupGBU} fails, then the evolution from $t_1$ to some suitable time $t\in(t_1,\tilde T)$ leads to 
$\tilde A_n\Rrightarrow \tilde B_n$, hence a contradiction.
The idea is to split the evolution, carrying out a sequence of parabolic reductions along suitable time and space intervals,
chosen in such a way as to satisfy the condition (cf.~after~\eqref{eq:v}) 
that no intersections between $\tilde u_a$ and $v$ appear at the endpoints of the space intervals.

Keeping the notation from the proof of Theorem \ref{mainThm1}{\rm(i)} we set $\bar D=a^{1/2} D$. 
For $t\in [t_1,\tilde T)$ we denote by $\tilde X_1(t),\cdots,\tilde X_n(t)$ the curves of zeros of 
$\tilde u_a (\cdot, t)-v(\cdot,t)$ in $(0,\bar D)$
and $\tau_1,\dots,\tau_n$ their maximal existence times
(cf.~Remark~\ref{rem-zerocurves} and recall that some zeros may vanish or collapse before $t=\tilde T$).
Set 
$$\hbox{$J_d=\bigl\{t\in(t_1,\tilde T);\ [\tilde u_a-v](\cdot,t)$ has a denenerate zero in $(0,\bar D)\bigr\}$
\ and \ $J_{nd}=(t_1,\tilde T)\setminus J_d$}$$
(note that $J_d$ is finite).
Denote ${\N}_n=\{1,\cdots,n\}$ and let $\Sigma=\bigl\{i\in {\N}_n;\ \tau_i\in(t_1,\tilde T)\ \hbox{and}\ \tilde X_i(\tau_i^-)=0\bigr\}$.
We rewrite $\{\tau_i;\,i\in \Sigma\}=\{T_1,\dots,T_q\}$, with $T_1<\cdots<T_q$ and $q\ge 2$.
Set
$$\Sigma_j:=\bigl\{i\in\Sigma;\,\tau_i=T_j\bigl\} 
\quad\hbox{and}\quad \nu_j=|\Sigma_j|,\quad\hbox{ for $j\in {\N}_q$,}$$
and 
$$\bar\rho:=\min\,\Bigl\{\,\displaystyle\min_{t\in [t_1,\tilde T]} X_1(t),
 \displaystyle\inf_{i\in {\N}_n\setminus\Sigma,\,t\in [t_1,\tau_i)}\tilde X_i(t)\Bigr\}.$$
 Since $ \tu_a $ undergoes GBU at  $ (x,t) = (0,\tT) $ with $ \tT < T $ and $\lim_{t\to \tilde T}\tilde x_n(t)=0$,
there exists $\bar T_{q+1}\in (T_q,\tilde T)\cap J_{nd}$ such that 
$\tilde x_n(\bar T_{q+1})<\bar\rho$, and
\be{ordervua2}
v(x, \bar T_{q+1})<\tilde u_a (x, \bar T_{q+1}),\quad x\to 0^+.
\ee
It follows from \eqref{eqndescr3} and \eqref{ordervua2} that $\nu:=\sum_{j=1}^q \nu_j$ is even ($\ge 2$).

Now take any $\rho_{q+1}>0$ such that $\rho_{q+1} <\displaystyle\min_{t\in [t_1,\bar T_{q+1}]} \tilde x_1(t)$ 
and let $T_0=\hat T_0:=t_1$.
By (downward) induction, we may then define times $\hat T_1,\cdots,\hat T_q$ with
$\hat T_j\in (T_{j-1},T_j)\cap J_{nd}$ and numbers $0<\rho_1<\cdots<\rho_q<\rho_{q+1}$ such that
\be{defrhoT}
\displaystyle\sup_{i\in \Sigma_j,\,t\in[\hat T_j,T_j)}\tilde X_i(t)<\rho_{j+1},
\qquad \rho_j < \displaystyle\min_{i\in \Sigma_j,\, t\in [t_1,\hat T_j]} \tilde X_i(t),
\qquad\hbox{ for $j\in {\N}_q$.}
\ee
For given time $t$ and space interval $I$, we denote by $G(t,I)$ the braid associated with the curves $v(\cdot,t)$, $\tilde u_a(\cdot,t)$ and $U$ on $I$.
Owing to \eqref{defrhoT}, for each $j\in {\N}_{q+1}$, 
$ v(\cdot,t), \tu_a(\cdot, t) $ and $ U $ are mutually distinct.
Therefore, by the reduction principle for parabolic evolution (cf.~\eqref{reduction-principle}), we have
$$G(\hat T_{j-1},[\rho_j,\bar D]) \Rrightarrow G(\hat T_j,[\rho_j,\bar D]) = 
X^{\nu_j}G(\hat T_j,[\rho_{j+1},\bar D]),\quad j\in {\N}_q,$$
and $G(\hat T_q,[\rho_{q+1},\bar D]) \Rrightarrow G(\hat T_{q+1},[\rho_{q+1},\bar D]))$.
Consequently, since $\nu$ is even, we obtain
$$ \tilde A_n=G(t_1,[\rho_1,\bar D]) =G(\hat T_0,[\rho_1,\bar D])\Rrightarrow X^\nu G(\hat T_q,[\rho_{q+1},\bar D])
\Rrightarrow X^2G(\hat T_{q+1},[\rho_{q+1},\bar D]).$$

Finally, by \eqref{eqndescr1} and the fact that $\tilde x_n(\bar T_{q+1})<\bar\rho<X_1(\bar T_{q+1})$, 
we see that $\tu_a(\cdot, \hat T_{q+1}) $ and $v(\cdot, \hat T_{q+1})$ have at least two intersections
$x'\in (\tilde x_n(\hat T_{q+1}),X_1(\hat T_{q+1}))$ and $x''\in(X_n(\hat T_{q+1}),a^{1/2} D)$. 
There may be more intersections in $(x',x'')$ (see Fig.~5),
but we can delete them by parabolic reduction, removing the additional factors $X^2$.
It follows that
$G(\hat T_{q+1},[\rho_{q+1},\bar D]) \Rrightarrow Y^{2k} X Y^{2k} X$ or $Y^{2k+1} X^{2k+1}Y$, depending on the parity of $n$.
We conclude that $\tilde A_n\Rrightarrow X^2G(\hat T_{q+1},[\rho_{q+1},\bar D])\Rrightarrow \tilde B_n$,
which provides the desired contradiction.

\medskip

{\bf Acknowledgement.} 
NM is supported by the JSPS Grant-in-Aid for Scientific Research (B) (No.20H01814).
 PhS is partially supported by the Labex MME-DII (ANR11-LBX-0023-01).

\end{document}